\documentclass[11pt]{article}

\usepackage[english]{babel}

\usepackage[letterpaper,top=2cm,bottom=2cm,left=3cm,right=3cm,marginparwidth=1.75cm]{geometry}

\usepackage{bm} %% Bold math formulas 
\usepackage{amsmath}
\usepackage{graphicx}
\usepackage[colorlinks=true, allcolors=black]{hyperref}
\newcommand{\w}{{\text{w}}}
\usepackage{amsthm}
\usepackage{amsmath}
\usepackage{amssymb}
\usepackage{mathrsfs}
\usepackage{graphicx}
\usepackage{caption}
\usepackage{subcaption}
 \usepackage[colorlinks=true]{hyperref}
\usepackage{ucs}
\usepackage{yfonts}
\usepackage{bbm}
\usepackage{geometry}
\usepackage{graphicx}
\usepackage{xifthen}
\usepackage{upgreek} %Para \uprho
\newcommand{\esp}[2]{%
\ifthenelse{\equal{#2}{}}{{E}\left[ #1 \right] }{{E}^{#2}\left[ #1 \right] }%
}

\newcommand{\espk}[2]{%
\ifthenelse{\equal{#2}{}}{\mathbb{E}^{\dag}\left[ #1 \right] }{\mathbb{E}^{\dag}_{#2}\left[ #1 \right] }%
}

\usepackage{mathrsfs}
\usepackage{bm} 
\usepackage{mathabx} %For \widecheck
\usepackage{geometry}
\newtheorem{theo}{Theorem}
\newtheorem{lem}{Lemma}
\newtheorem{prop}{Proposition}
\newtheorem{def1}{Definition}
\newtheorem{cor}{Corollary}
\newtheorem{remark}{Remark}

\usepackage{tcolorbox} % Para hacer cuadro al rededor de TableOfContents

\numberwithin{equation}{section}
\numberwithin{theo}{section}

\newcommand{\dd}{{\mathrm{d}}}

\renewcommand{\epsilon}{\varepsilon} % Usar epsilon bonito por defecto

\usepackage{array} 
\newcounter{hypocounter}
\renewcommand\thehypocounter{(H\arabic{hypocounter})} 
\usepackage{pict2e,picture,graphicx}
\usepackage{mathtools}

\makeatletter
\DeclareRobustCommand{\Arrow}[1][]{%
\check@mathfonts
\if\relax\detokenize{#1}\relax
\settowidth{\dimen@}{$\m@th\rightarrow$}%
\else
\setlength{\dimen@}{#1}%
\fi
\sbox\z@{\usefont{U}{lasy}{m}{n}\symbol{41}}%
\begin{picture}(\dimen@,\ht\z@)
\roundcap
\put(\dimexpr\dimen@-.7\wd\z@,0){\usebox\z@}
\put(0,\fontdimen22\textfont2){\line(1,0){\dimen@}}
\end{picture}%
}
\makeatother

\newcommand{\dminiArrow}{\hspace{.01mm}\scalebox{0.8}{\rotatebox[origin=c]{-90}{ \Arrow[.15cm] \hspace{0.01mm}  }}\hspace{.2mm}}

\newcommand{\TT}{\text{T}}

\title{ The structure of the local time of  Markov processes indexed by Lévy trees}
\author{Armand Riera\footnote{armand.riera@math.uzh.ch } \, and Alejandro Rosales-Ortiz\footnote{alejandro.rosalesortiz@math.uzh.ch}}

\date{University of Zurich}

\newcommand{\Jj}{{\text{J}}}

\newcommand{\Jc}{ \widecheck{J} }

\begin{document}
\maketitle
\newcommand{\toctitlefont}{\color{black}}

\begin{abstract}
We construct an additive functional  playing the role of the local time -- at a fixed point $x$ -- for Markov processes indexed by Lévy trees. We start by proving that  Markov processes indexed by Lévy trees  satisfy a special Markov property which can be thought as a spatial version of the classical Markov property. Then, we construct our additive functional by an approximation procedure and we characterize the support of its Lebesgue-Stieltjes measure. We also give an equivalent construction in terms of a special family of  exit local times. Finally, combining these results, we show that the points at which the Markov process takes the value $x$ encode a new Lévy tree and we construct explicitly its height process. In particular, we recover a recent result of Le Gall concerning the subordinate tree of the Brownian tree  where the subordination function is given by the past maximum process of Brownian motion indexed by the Brownian tree. 
\end{abstract}

\vspace{0.8cm}

\tableofcontents
\newpage

\section{Introduction}
 Excursion theory  plays a fundamental role in the study of $\mathbb{R}_+$--indexed Markov  processes dating back to Itô's  work  \cite{ItoExcursiones}. The purpose of this theory is to describe the evolution of a Markov process between visits to a fixed point in the state space. To be more precise, consider a  Polish space $E$, a strong $E$-valued continuous Markov process $\xi$ and fix  a point $x\in E$,  regular and instantaneous for $\xi$.  The paths of $\xi$ can be decomposed in   excursions away from $x$, where an excursion is a piece of path of random length, starting and ending at $x$, such that in between $\xi$ stays away from $x$.  Formally, they consist of the restrictions of $\xi$ to the connected components of $\mathbb{R}_+ \setminus \{ t \in \mathbb{R}_+ : \xi_t = x \}$. In order to keep track of the ordering induced by the time, the family of excursions is indexed by means of a remarkable  additive functional of $\xi$, called its local time at $x$, and denoted throughout this work by  $\mathcal{L}$. It is well known that $\mathcal{L}$ is a continuous process with Lebesgue-Stieltjes measure supported on the random set:
\begin{equation} \label{introduction:xsMarkov}
    \big\{ t \in \mathbb{R}_+ : \xi_t = x \big\},
\end{equation}
 and that the trajectories of $\xi$ can be recovered from the family of indexed excursions by gluing them together, taking into account the time spent by $\xi$ at $x$.  For technical reasons, we will also assume that the point $x$ is recurrent for $\xi$. We stress that  excursion theory holds under broader assumptions on the Markov process $\xi$,  and we refer to e.g.  \cite[Chapter VI]{BertoinBook} and  \cite{Excursions} for a  complete account.  
\par 
The purpose of this work is to set the first milestone  towards introducing an excursion theory for  Markov processes indexed by random trees. The random trees that we consider are the so-called \textit{Lévy trees}. This family is  canonical, in the sense that Lévy trees are  scaling limits of   Galton-Watson trees \cite[Chapter 2]{Duquesne} and are characterized by a branching property in the same vein as their discrete counterparts \cite{Subor, WeillLevyTrees}. 
At this point, let us mention that Markov processes indexed by Lévy trees  are  fundamental objects in   probability theory -- for instance, they  are  intimately linked  to the theory of   superprocesses \cite{Duquesne,RefSnake}. More recently,  Brownian motion indexed by the Brownian tree has been used as the essential building block in the construction of the universal model of random geometry called the Brownian map \cite{LeGallUnivesality, MiermontUnivesality}, as well as in the construction of other related random surfaces \cite{BMR,Infinite_Spine}. We also stress that  Brownian motion indexed by a stable tree is also a  universal object, due to the fact that  it arises as scaling limit of discrete models \cite{ScallingLimitsStableTrees}.  For the sake of completeness, we shall start with a brief and informal account of our objects of interest.

A Lévy tree can be encoded by a continuous $\mathbb{R}_+$-valued  process $H = (H_t)$  called its \textit{height process}; and for this reason we denote the associated tree by $\mathcal{T}_H$. Roughly speaking, the tree $\mathcal{T}_H$ has a root and $H$ encodes the distances to it when the tree is explored in  "clockwise order". Under appropriate assumptions,  we consider the pair consisting of the Markov process $\xi$ and its local time $\mathcal{L}$,  indexed by a Lévy tree $\mathcal{T}_H$. With a slight abuse of notation, this  process   will be denoted in the rest of this work by:
\begin{equation}\label{intro:MpindexedTree}
    \big( (\xi_{\upsilon}, \mathcal{L}_\upsilon) : \upsilon \in \mathcal{T}_H  \big).
\end{equation}
In short, this process can be thought as a random motion defined on top of $\mathcal{T}_H$ and following the law of $((\xi_t, \mathcal{L}_t) : t \in \mathbb{R}_+)$,  but splitting at every branching point of $\mathcal{T}_H$ into independent copies. The role played by $\{ t \in \mathbb{R}_+ : \xi_t =x \}$  is taken over in this setting  by the following random subset of $\mathcal{T}_H$: 
\begin{equation} \label{introduction:set1}
    \mathscr{Z} := \{ \upsilon \in \mathcal{T}_H :\, \xi_\upsilon = x \}. 
\end{equation}
The definition of the excursions of $( \xi_\upsilon )_{\upsilon \in \mathcal{T}_H}$ away from $x$ should then be clear at an intuitive level -- since it suffices to consider the restrictions of  $( \xi_\upsilon )_{\upsilon \in \mathcal{T}_H}$ to the connected  components of $\mathcal{T}_H \setminus \mathscr{Z}$. Notice however that we  lack a proper  indexing for this family of excursions that would allow  to recover the whole path, as in classical excursion theory. Moreover, one can expect the gluing of these excursions   to be more delicate in our setting,  since  in the time-indexed case the extremities of  an excursion  consist  of only two points, while in the present case,  the extremities are  subsets of $\mathcal{T}_H$ of significantly more intricate nature.  In the same vein, since the set $\mathscr{Z}$ is a subset of $\mathcal{T}_H$, it inherits its tree structure and therefore it possesses  richer   spatial properties than the subset of the real line  \eqref{introduction:xsMarkov}. More precisely, we consider the equivalence relation $\sim_\mathcal{L}$ on $\mathcal{T}_H$ which identifies the components of $\mathcal{T}_H$ where $(\mathcal{L}_\upsilon)_{\upsilon \in \mathcal{T}_H}$ stays  constant.  The resulting quotient space  $\mathcal{T}^{\mathcal{L}}_H := \mathcal{T}_H / \sim_\mathcal{L}$  is also a tree,  encoding   the set $\mathscr{Z}$ and   endowing it  with an additional  tree structure. In the terminology of \cite{Subor}, the  tree $\mathcal{T}^\mathcal{L}_H$ is the so-called  \textit{subordinate tree} of $\mathcal{T}_H$ by $\mathcal{L}$. Since  each component of $\mathcal{T}_H$ where $\mathcal{L}$ stays constant is naturally identified  with an excursion of $\xi$ away from $x$,  a proper understanding of $\mathcal{T}_H^{\mathcal{L}}$ is crucial  to develop an excursion theory for $(\xi_\upsilon)_{\upsilon \in \mathcal{T}_H}$.  This work is devoted to both:  
\begin{enumerate}
    \item Introducing a continuous process  suitable to index  the excursion of $(\xi_\upsilon)_{\upsilon \in \mathcal{T}_H}$ away from $x$;
    \item Studying  the structure of the random set $\mathscr{Z}$.
\end{enumerate}
\par As we shall explain, both questions are intimately related and, as we mentioned before, they lay the foundations for the development of  an excursion theory for $(\xi_\upsilon)_{\upsilon \in \mathcal{T}_H}$. In the case of Brownian motion indexed by the Brownian tree, an excursion theory has  already been  developed in \cite{ALG} and  has turned out  to have multiple applications in Brownian geometry, see e.g. \cite{Disks, Growth}. However, we stress that in  \cite{ALG}  the excursions are not indexed and, in particular,  a reconstruction of the Brownian motion indexed by the Brownian tree in terms of its excursions is still out of reach. Let us now present the general framework of this work.
\par  
 In order to formally define  the tree indexed process \eqref{intro:MpindexedTree},  we  rely on the theory of Lévy snakes  and we shall now give a brief account. The theory of Lévy snakes has mainly  been  developed in the monograph of Duquesne and Le Gall  \cite{Duquesne},  and a detailed presentation of the results that we  need is given in Section~\ref{section:preliminaries}. The process \eqref{intro:MpindexedTree} is  built from two layers of randomness. First, as we already mentioned, the family of random trees that we work with are  called  Lévy trees. 
 If $\psi$ is the Laplace exponent of a spectrally positive Lévy process $X$, under appropriate assumptions on $\psi$, one can define the height process $H$ as a functional  of $X$.  In order to explain how $\mathcal{T}_H$ is encoded  by  $H$,    we work under the excursion measure of $X$ above its running infimum and  we write  $\sigma$ for the duration of an excursion. The relation: 
\begin{equation*}
d_{H}(s,t):=H_s+H_t - 2\cdot \inf_{s\wedge t \leq u \leq s\vee t} H_u,   \quad \text{ for all } (s,t)\in [0,\sigma]^{2},  
\end{equation*}
defines a  pseudo-distance on $[0,\sigma]$, and the associated equivalence relation $\sim_H$ is defined by setting  $s\sim_{H} t$ if and only if $d_{H}(s,t)=0$.  The pointed  metric space $\mathcal{T}_{H}:=([0,\sigma_H]/\sim_{H},d_{H},0)$ is a Lévy tree\footnote{More precisely, since the duration $\sigma$ is random,  $\mathcal{T}_{H}$ is referred to as a \textit{free} Lévy tree.}, where for simplicity  we keep the notation $0$ for the equivalence class of $0$. We also write  $p_H : [0,\sigma] \mapsto \mathcal{T}_H$ for the canonical projection on $\mathcal{T}_H$ and we refer to  Section \ref{sec:trees:1} for more details about this encoding. The point $0$ is called the root of $\mathcal{T}_H$ and, by construction, the height process encodes the distances to it. We stress that the distribution of $\mathcal{T}_H$ is characterized by the exponent $\psi$, and we  say that  $\mathcal{T}_H$ is a $\psi$-Lévy tree.   One of the main technical difficulties of this work is that,  except when $X$ is a Brownian motion with drift, the process $H$ is  not Markovian and we will need to introduce a measure-valued process -- called the exploration process --  which heuristically, carries  the  information needed to make  $H$ Markovian. This process will be  denoted throughout this work by $\rho = (\rho_t: \, t \geq 0)$ and its nature  has a crucial impact on the geometry of $\mathcal{T}_H$. For instance, $\rho$ allows to characterize the multiplicity and genealogy  of points of $\mathcal{T}_H$. More precisely, recall that the multiplicity of a point $\upsilon$ in $\mathcal{T}_H$ is defined as the number of connected components of $\mathcal{T}_H \setminus\{ \upsilon\}$. For $i \in \mathbb{N}^{*} \cup \{ \infty \}$, we write $\text{Multi}_i(\mathcal{T}_H)$ for the set of points of $\mathcal{T}_H$ of multiplicity $i$, and the points of multiplicity strictly larger than $2$ are called \textit{branching points}. For instance, if $X$ does not have jumps, the measures $(\rho_t:~t\geq 0)$ are atomless and  all branching points have multiplicity $3$. In contrast,  as soon as the Lévy measure of $X$ is non-null, the measures $(\rho_t:~t\geq 0)$ have atoms and the set $\text{Multi}_\infty(\mathcal{T}_H)$ is non-empty.   We  also refer to \cite{JLG} for the construction of the exploration process.
 The second layer of randomness consists in defining, given $\mathcal{T}_H$,  a spatial motion indexed by $\mathcal{T}_H$ that roughly speaking behaves like the Markov process $(\xi_t)_{t \in \mathbb{R}_+}$ -- when restricted to an injective path connecting the root of $\mathcal{T}_H$ to a leaf.  This informal description can be formalized by making use of   the theory of random snakes \cite[Section 5]{Duquesne}. More precisely, one can define a  process $(W_s, \Lambda_s : s \in[0,\sigma])$ taking values in the collection of finite $E\times \mathbb{R}_+$--valued  continuous paths, each $(W_s, \Lambda_s)$ having lifetime $H_s$ and  such that, for each $s\in \mathbb{R}_+$ and conditionally on $H_s$, the path $(W_s , \Lambda_s)$ has the same distribution as $(\xi_s, \mathcal{L}_s : s \in [0, H_s] )$. The second main property of $(W, \Lambda)$ is that it satisfies the \textit{snake property}, viz.
 $$\big(W_t(H_t), \Lambda_t(H_t)\big)=\big(W_s(H_s), \Lambda_s(H_s)\big),\quad \text{for every }s\sim_H t. $$
For simplicity, from now on, we will  write  $(\widehat{W}_t, \widehat{\Lambda}_t):=(W_t(H_t), \Lambda_t(H_t))$ for the tip of $(W_t, \Lambda_t)$.  By  the snake property, it follows that the process $(\widehat{W}_t, \widehat{\Lambda}_t : t \in [0,\sigma])$ is well defined in the quotient space $\mathcal{T}_H$, and hence it defines a random function indexed by $\mathcal{T}_H$ which will be denoted by \eqref{intro:MpindexedTree}.  The triplet  $(\rho , W , \Lambda)$ is the so-called  $\psi$-Lévy snake with spatial motion $(\xi, \mathcal{L})$, a Markov process that will be extensively studied throughout this work. 
\par 
Let us now present the statements of our main results. These are stated under the excursion measure of $(\rho , W, \Lambda)$, but let us mention that we will  obtain similar results  under the underlying probability measure. By construction, the study of $\mathscr{Z}$ is closely related to the understanding of the random set:  
\begin{equation} \label{introduction:soporteaA}
    \{  t \in [0,\sigma]  : \widehat{W}_t = x \}, 
\end{equation}
since $\mathscr{Z}$ is precisely its  image under the canonical projection $p_H$  on $\mathcal{T}_H$. However,  note  that these two sets are of radically different natures.   As in  classical excursion theory for  Markov processes, we shall start by constructing  an additive functional $A = (A_t)_{t \in [0,\sigma]}$ of the Lévy snake $(\rho , W, \Lambda)$ with suitable properties and Lebesgue-Stieltjes measure $\dd A$ supported on \eqref{introduction:soporteaA}. The first main result of this work is obtained in Section \ref{section:structureOfTloc}  and is divided in two  parts:
\begin{enumerate}
    \item[\rm{(i)}] The construction of the additive functional  $A$ [Proposition \ref{proposition:aditivaDefinicion}];
    \item[\rm{(ii)}] The characterization of the support  of  $\dd A$ [Theorem \ref{prop:suppA}].
\end{enumerate}
See also  Theorem \ref{theorem:supportIntro} for an equivalent formulation of (ii)  in the terminology of the tree indexed process $(\xi_\upsilon)_{\upsilon \in\mathcal{T}_H}$. Recalling our initial discussion, the process $(A_t)_{t \in \mathbb{R}_+}$ is the natural candidate to index the excursions away from $x$ of $(\xi_\upsilon)_{ \upsilon \in \mathcal{T}_H}$. We are not yet in position in this introduction to formally state  the content of (i) and (ii), but we can  give a general description.  Our construction of $(A_t)_{t \in \mathbb{R}_+}$ relies on the so-called \textit{exit local times} of the Lévy snake  $(\rho , W , \Lambda )$.  More precisely, if we consider the family of domains $\{E \times [0, r) : \, r \in (0,\infty) \}$, for each fixed $r>0$, there exists an additive functional of $(\rho, W, \Lambda)$ that heuristically measures, at every $t \geq 0$,  the number of  connected components of $\mathcal{T}_H \setminus \{ \upsilon \in \mathcal{T}_H : \mathcal{L}_\upsilon \leq r \}$ visited up to time $t$. This description is informal and we refer to Section \ref{Special and intermediate} for  details. We  establish  in Section \ref{subsection:PhiTildeyMarkovTloc}   that the corresponding family of exit local times possesses a  jointly measurable version $(\mathscr{L}_t^{r} : \,  t \geq 0 , r >0 )$,  and in Section \ref{subsection:existenciaAditiva} we define our continuous additive $A$ by setting:
\begin{equation*}
    A_t := \int_0^\infty \dd r \,  \mathscr{L}_t^{r} , \quad t \geq 0.
\end{equation*}
After  establishing  that there is no branching point with label $x$, we   give in Section \ref{subsection:characterizationSupport}  a precise characterization of the support of the measure $\dd A$.  Formally, we  prove that:
\begin{equation*} 
\text{supp} ~\dd A=    \overline{  \big\{ t \in [0,\sigma]  : \xi_{p_H(t)} = x, \,  p_H(t) \in \text{Multi}_2 (\mathcal{T}_H) \cup  \{ 0 \}  \big\}}.  
\end{equation*}
We  also show in Theorem \ref{prop:suppA} that, equivalently, the support of $\dd A$ is the complement of the constancy intervals of $(\widehat{\Lambda}_t: \, t \geq 0)$. In particular, if we denote the right inverse of $A$ by $(A^{-1}_t: t \geq 0)$, the relation:
\begin{equation*}
    H^A_t := \widehat{\Lambda}_{A^{-1}_t}, \quad \quad  t \geq 0, 
\end{equation*}
defines a continuous non-negative process that plays a crucial role in the second part of our work.
\par  In Section \ref{section:treeStructureLocalTime}, we turn our attention to the study of $\mathscr{Z}$ or, equivalently, to  the structure of the subordinate  tree $\mathcal{T}_H^{\mathcal{L}}$. Even if this is an object of very different nature, our analysis relies  deeply  on the results and the machinery developed  in Section \ref{section:structureOfTloc}. The second main result of this work consists in showing that the process $H^A$ satisfies the following properties: 
\begin{enumerate}
    \item[\rm{(i')}] It encodes the subordinate tree $\mathcal{T}_H^{\mathcal{L}}$ [Theorem  \ref{theorem:subordinatedTreebyLocalT} (i)];
    \item[\rm{(ii')}] It is the height function of a Lévy tree, with an exponent $\widetilde{\psi}$ that we   identify [Theorem \ref{theorem:subordinatedTreebyLocalT} (ii)].
\end{enumerate}
In particular, this shows that $\mathcal{T}_H^\mathcal{L}$ is a Lévy tree with exponent $\widetilde{\psi}$. We stress that a continuous function can  fulfill (i') without satisfying (ii'), and it is remarkable that $H^A$ follows the exploration order of a Lévy tree. We also mention that the previous two points  were established -- although  with a different construction of the height process $H^A$ -- in \cite[Theorem 1]{Subor} for the subordination of the Brownian tree by the running maximum of the Brownian motion indexed by the Brownian tree\footnote{ When considering the process $(\xi_\upsilon, \mathcal{L}_\upsilon)_{\upsilon \in \mathcal{T}_H}$ indexed by the Brownian tree, the fact that the subordinate tree $\mathcal{T}_H^{\mathcal{L}}$ is a $\widetilde{\psi}$--Lévy tree  is also proved in \cite[Theorem 16]{Subor} but  the construction of its height process is lacking.}. These approaches are complementary, since the techniques employed in \cite{Subor} rely on a discrete approximation of the height function, while we shall argue directly in the continuum.  We also mention that one of the strengths of our method is that it gives an explicit definition of $H^A$ which is suitable for computations. This point is crucial in order to study the excursions of $(\xi_{\upsilon})_{\upsilon \in \mathcal{T}_H}$ from $x$. Our result shows that the height function of  the subordinate tree $\mathcal{T}^\mathcal{L}_H$ can  be constructed in terms  of functionals of $(\rho , W,\Lambda)$, and that   $A^{-1}$ defines an exploration of $\mathcal{T}_H^{\mathcal{L}}$  compatible with the order induced by $H$. Property (i') will be a  consequence of our previous results  (i), (ii)  and Section \ref{section:treeStructureLocalTime} is mainly devoted to the proof of (ii'). The main difficulty to establish (ii') comes from the fact that, as we already mentioned, the height process of a Lévy tree is not always Markovian. To circumvent this difficulty,  the proof of (ii') relies on the computation of  the so-called marginals of the tree  associated with $H^A$. In particular, it  makes use of all the machinery developed in previous sections as well as  standard properties of Poisson random measures.  
\par 
Let us now close the presentation of our work with a result of independent interest which  is used extensively throughout this paper. In Section \ref{Special and intermediate},  we state and prove the so-called \textit{Special Markov property} of the Lévy snake. This section is independent of the setting of Sections \ref{section:structureOfTloc}  and \ref{section:treeStructureLocalTime}, and  we work with an arbitrary $(\psi, \xi)$-Lévy snake  under  general assumptions on the pair $(\psi, \xi)$. Roughly speaking, the special Markov property  is  a spatial version of the classical Markov property for time-indexed Markov processes. The precise statement is the content of Theorem \ref{Theo_Spa_Markov_Excur_P}, see also Corollary \ref{Corollary:specialMarkovN}.  This result was established   in \cite[Theorem 20]{Subor} for  continuous Markov processes indexed by the  Brownian tree, and a particular case was proved for the first time in  \cite{RefSnake}. Our result is a generalisation of \cite[Theorem 20]{Subor}  holding in the broader setting  of continuous Markov processes indexed by $\psi$-Lévy trees.  The special Markov property of the Brownian motion indexed by the Brownian tree has already played a crucial role in multiple contexts, see for instance \cite{Hull, Growth, Infinite_Spine}  and we expect this result to be useful outside the scope of this  work. We also mention that the special Markov property of the Lévy snake is closely related to the one established by  Dynkin in the context of  superprocesses, see \cite[Theorem 1.6]{Dynkin_Special}. However, we stress that the formulation in terms of the Lévy snake, although less general, gives additional and crucial information  for our purposes. In particular, it takes into account the genealogy induced by the Lévy tree, and hence it caries geometrical information.
\par We conclude this introduction  non-exhaustive summary of related works. First, as we already mentioned, we extend to the general framework of Markov processes indexed by Lévy snakes the work of Le Gall on subordination in the  case of the Brownian motion indexed by the Brownian tree \cite{Subor}.  Moreover, our results on subordination of trees with respect to the local time  are closely related, in the terminology of Lévy snakes, to  Theorem 4 in \cite{BertoinLeGallLeJean} stated in the setting of superprocesses  --  the main difference being that in our work we encode the associated genealogy. For instance, we  recover \cite[Theorem 4]{BertoinLeGallLeJean} in a more precise form in our case of interest. 
We also note that we expect our results to be useful beyond the scope of this work, for instance in Brownian geometry. Finally, in the  case of Brownian motion indexed by the Brownian tree and when $x=0$,  our functional $A$ is closely related to the so-called \textit{integrated super-Brownian excursion} \cite{AldousTreeBasedModels} -- a random measure arising in multiple limit theorems for discrete probability models, but also in the theory of interacting particle systems \cite{SuperBrowniianLimits, Rescaledvotermodels} and in a variety of models of statistical physics \cite{Scalinglatticetrees, Scalinglimitincipient}. More precisely, the total mass $A_\infty$ is the density of the integrated super-Brownian excursion at $0$, see \cite[Proposition 3]{RieraLeGallExplicitDistributions}. In particular, we hope that our construction of the functional $A$ will be useful to obtain new explicit computations regarding the integrated super-Brownian excursion and to generalize these computations to  related models.
\medskip
\\
\par  \textit{The work is organised as follows: }Section 2 gives an overview of the theory of Lévy trees and  snakes.  In Section 3, we state and prove the special Markov property for Lévy snakes and we explore some of its consequences. This section is independent of the rest of the work but is   key for the development of Section 4 and 5. The preliminary results needed for its proof are covered in Section \ref{subsection:exitLocalTime}, and mainly concern approximation results for exit local times.  Section 4 is devoted to first, constructing in Section 4.2 the additive functional $A$ [Proposition \ref{proposition:aditivaDefinicion}],  and afterwards to   the characterization of   the support of the measure $\dd A$ [Theorem \ref{prop:suppA}] in Section 4.3.   We shall give two equivalent descriptions for the support of $\dd A$, one in terms of the pair $(H, W)$, and a second one only depending on $\Lambda$. The latter will be needed in Section \ref{section:treeStructureLocalTime} and we expect the former to be useful to develop an excursion theory --  we plan to pursue this goal in future works. The preliminary results  needed for our constructions are  covered in  Section 4.1. Finally, in Section 5, after recalling preliminary results on subordination of trees by continuous functions, we explore the tree structure of the set $\{ \upsilon \in \mathcal{T}_H :\, \xi_\upsilon = x \}$ by considering the subordinate tree of $\mathcal{T}_H$ with respect to the local time $\mathcal{L}$. The main result of the section is stated in Theorem \ref{theorem:subordinatedTreebyLocalT},  and  consists in proving (i') and (ii'). 
\\
\\
\textbf{Acknowledgments.}  We thank Jean-François Le Gall for stimulating conversations.
\section{Preliminaries} \label{section:preliminaries}
\subsection{The height process and the exploration process}\label{subsection:height}

\noindent 
Let us start by introducing the class of Lévy processes that we will consider throughout this work. We set $X$ a Lévy process  on $\mathbb{R}_+$, and we denote its law started from $0$ by $P$. It will be convenient to assume that $X$ is the canonical process on the Skorokhod space $D(\mathbb{R}_+,\mathbb{R})$ of càdlàg (right--continuous with left limits) real-valued paths equipped with the probability measure $P$. We denote the canonical filtration  by $(\mathcal{G}_t:t\geq  0)$, completed as usual by the class of  $P$-- negligible sets of $\mathcal{G}_\infty=\bigvee_{t\geq 0} \mathcal{G}_t$. We henceforth assume that $X$ verifies $P$-a.s. the following properties:
\begin{itemize}
\item (A1) $X$ does not have negative jumps;
\item (A2) The paths of $X$ are of infinite variation;
\item (A3) $X$ does not drift to $+\infty$.
\end{itemize} 
Since $X$ has no negative jumps the mapping $\lambda\mapsto\mathbb{E}[\exp(-\lambda X_{1})]$ is well defined in $\mathbb{R}_{+}$ and 
we denote the Laplace exponent of $X$  by  $\psi$, viz. the function defined by:
\[\mathbb{E}[\exp(-\lambda X_{1})]=\exp(\psi(\lambda)),\quad \text{ for all } \lambda\geq 0. \]
The function  $\psi$  can be written in  the Lévy-Khintchine form: 
\begin{equation*}
  \psi(\lambda) =  \alpha_0 \lambda + {\beta} \lambda^2 + \int_{(0,\infty)} \pi({\rm{d}} x)~(\exp(-\lambda x) - 1 + \lambda x \mathbbm{1}_{\{ x \leq 1 \}}),  
\end{equation*}
where $\alpha_0 \in \mathbb{R},\,  \beta\in\mathbb{R}_{+}$ and $\pi$ is a sigma-finite measure on $\mathbb{R}_{+}^*$ satisfying $\int_{(0,\infty)}\pi(\dd x)(1\wedge x^{2})<\infty$. Moreover, it is well known  that  condition (A2)  holds  if and only if we have: 
\begin{equation*}
    \beta \neq 0 \quad \quad  \text{ or } \quad \quad \int_{(0,1)}\pi(\dd x) ~ x  = \infty.
\end{equation*}
The Laplace exponent  $\psi$ is infinitely differentiable and strictly convex in $(0,\infty)$ (see e.g.  Chapter 8 in \cite{KyprianouBook}). Since $X$ does not drift towards $\infty$  one has $-\psi'(0+) = E[X_1] \leq 0$ which,  in turn, implies that $X$ oscillates, or drifts towards $-\infty$ and that  $X_t$ has a finite first moment for any $t$. In terms of the Lévy measure, this ensures that the additional integrability condition $\int_{(1,\infty)} \pi(\dd x) ~ x < \infty$ holds. Consequently, $\psi$ can and will be supposed to be of the following form:
 \[ \psi(\lambda)=\alpha\lambda+\beta \lambda^{2}+\int_{(0,\infty)}\pi(\dd x)(\exp(-\lambda x)-1+\lambda x),\]
where now $\pi$  satisfies $\int_{(0,\infty)}\pi(\dd x)(x\wedge x^{2})<\infty$ and $\alpha, \beta \in \mathbb{R}_+$  since $\alpha = \psi'(0+)$. From now on, we will denote the infimum of $X$ by $I$ and remark that, under our current hypothesis, $0$ is regular  and instantaneous for the Markov process $X-I = (X_t - \inf_{[0,t]}X_s : t \geq 0)$.  Moreover, it is standard to see that $P$ --a.s., the Lebesgue measure of $\{ t \in \mathbb{R}_+ : X_t = I_t  \}$ is null.  The process $-I$ is a local time of $X-I$ and we denote  the associated excursion measure from $0$ by $N$. To simplify notation, we write  $\sigma_e$ for the lifetime of an excursion $e$. Finally, we  impose the following additional assumption on $\psi$: 
\begin{equation}\label{1_infinity_psi} 
\int_{1}^{\infty}\frac{\dd \lambda}{\psi(\lambda)}<\infty. \tag*{${(\text{A}{4})}$}
\end{equation}
From now on, we will be working under (A1)~--~(A4). \par 
 Let us now briefly discuss  the main  implications of our assumptions. The condition \ref{1_infinity_psi} is twofold: on the one hand, it ensures that $\lim_{\lambda \rightarrow 
 \infty} \lambda^{-1} \psi(\lambda) = \infty$ which implies that $X$ has paths of infinite variation  \cite[VII-5]{BertoinBook}  (the redundancy in our hypothesis is on purpose for ease of reading).  On the other hand, under our hypothesis (A1)~--~(A3), it is well known that  there exists a  continuous state branching process  with branching mechanism $\psi(\lambda)$ (abbreviated  $\psi$-CSBP) and that  \ref{1_infinity_psi} is equivalent to its a.s. extinction. The $\psi$-Lévy tree can be interpreted as the genealogical tree  of this branching process and is defined in terms of a fundamental functional of $X$, called the height process, that we now introduce. \medskip \\

\noindent \textbf{The height and exploration processes}. 
Let us  turn our attention to the so-called height process -- the  main ingredient needed to define  Lévy trees. Our presentation follows  \cite[Chapter 1]{Duquesne} and we start by introducing some standard notation.  For every $0<s\leq t$, we set 
\begin{equation*}
    I_{s,t}:=\inf_{s \leq u \leq t} X_u,
\end{equation*}
the infimum of $X$ in $[s,t]$ and remark that  when $s = 0$ we  have $I_{t}=I_{0,t}$.  Moreover, since $X$  drifts towards $-\infty$ or oscillates, we must have $I_{t}\to -\infty$ when $t \uparrow \infty$. By  \cite[Lemma 1.2.1]{Duquesne}, for every fixed $t\geq 0$,  the limit:
\begin{equation} \label{definition:alturaH}
    H_t := \lim\limits_{\varepsilon \to 0}\frac{1}{\varepsilon }\int_{[0,t]}\dd r~ \mathbbm{1}_{\{ X_{r}<I_{r,t}+\varepsilon \}}
\end{equation}
exists in probability. Roughly speaking, for every fixed $t$, the quantity $H_t$ measures the size of the set: 
\begin{equation*}
    \{ r  \leq t :~ X_{r-} \leq  I_{r,t} \},
\end{equation*}
and we refer to  $H = (H_t : t \geq 0)$ as the height process of $X$.
By \cite[Theorem 1.4.3]{Duquesne}, condition (A4) ensures that $H$ possesses a continuous modification that we consider from now on and that we still denote by $H$.
\par 
The process $H$ will be the building block to define  Lévy trees. However, $H$ is not Markovian as soon as $\pi \neq 0$ and we will need to introduce a process -- called the exploration process --  which roughly speaking carries  the needed information to make  $H$ Markovian. More precisely, the exploration process is a Markov process and we will write  $H$ as a functional of it. In this direction, we write  $\mathcal{M}_{f}(\mathbb{R}_{+})$ for the set of finite measures on $\mathbb{R}_{+}$ equipped with the topology of  weak convergence and with a slight abuse of notation we write   $0$  for the null measure on $\mathbb{R}_+$. The exploration process $\rho = (\rho_t : t \geq 0)$ is  the random measure defined as: 
\begin{equation} \label{definition:explorationProcess}
    \langle \rho_{t},f \rangle :=\int_{[0,t]}\dd_{s}I_{s,t}\:f(H_{s}), \quad t\geq 0,
\end{equation}
where $\dd_{s} I_{s,t}$ stands for the measure associated with the non-decreasing function $s \mapsto I_{s,t}$.  Equivalently, $\rho$ can be defined  as: 
\begin{equation}\label{rho_atoms}
\rho_{t}(\dd r) := \beta \mathbbm{1}_{[0,H_{t}]}(r)\dd r+\mathop{\sum \limits_{0<s\leq t}}_{X_{s-}<I_{s,t}}(I_{s,t}-X_{s-})\:\delta_{H_{s}}(\dd r),\quad t\geq 0,
\end{equation}
and remark that \eqref{definition:explorationProcess} implies that
\begin{equation*}
 \langle \rho_{t},1 \rangle =I_{t,t}-I_{0,t}=X_{t}-I_{t},\quad t\geq 0. 
\end{equation*} 
In particular, $\rho_t$  takes values in $\mathcal{M}_{f}(\mathbb{R}_{+})$. By  \cite[Proposition 1.2.3]{Duquesne}, the process  $(\rho_t : t \geq 0)$ is an   $\mathcal{M}_{f}(\mathbb{R}_{+})$-valued  càdlàg strong Markov process,  and we briefly recall some of its main properties for later use. 
For every $\mu\in \mathcal{M}_{f}(\mathbb{R}_{+})$, we write  $\text{supp} (\mu)$  for the topological support of $\mu$ and  we set $H(\mu):=\sup \text{supp} (\mu)$ with the convention $H(0) = 0$.\\ The following properties hold: 
\begin{itemize}
    \item[\rm{(i)}] Almost surely, for every $t \geq 0$, we have $\text{supp } \rho_t  = [0, H_t]$ if $\rho_t \neq 0$.
    \item[\rm{(ii)}] The process $t \mapsto \rho_t$ is càdlàg with respect to the total variation distance. 
    \item[\rm{(iii)}] Almost surely, the following sets are equal: 
\begin{equation}  \label{equation:zeros}
       \{ t \geq 0 : \rho_t = 0 \} = \{ t \geq 0 : X_t - I_t = 0 \} = \{ t \geq 0 : H_t = 0 \}.
\end{equation}
\end{itemize}
In particular, note that we have $(H(\rho_t))_{t \geq 0} = (H_t)_{t \geq 0}$ and that point (ii) implies that the excursion intervals away from $0$ of $X-I$, $H$ and $\rho$  coincide. Moreover,  since $I_t \rightarrow -\infty$ when $t \uparrow \infty$, the excursion intervals have finite length and by \cite[Lemma 1.3.2]{Duquesne} and the monotonicity of $t\mapsto I_{t}$ we have:
\begin{equation}\label{temps:local:I}
    \lim \limits_{\epsilon \to 0}\mathbb{E}\big[\sup\limits_{s\in[0,t]}\big| \frac{1}{\epsilon} \int_{0}^{s}\dd u \, \mathbbm{1}_{\{ H_u<\epsilon\}}+I_s\big|\big]=0,\quad \text{ for every }t\geq 0.
\end{equation}
By the previous display, $-I$ can be thought as the local time of $H$ at $0$. 
\par 
The Markov process $\rho$ in our previous definition starts at $\rho_0 = 0$ and, in order to make use of the Markov property, we need to  recall  how to define it's distribution starting from an arbitrary measure $\mu \in \mathcal{M}_f(\mathbb{R}_+)$. In this direction,  we will need to introduce the following two operations: \smallskip \\
\textit{Pruning}. For every  $\mu \in \mathcal{M}_{f}(\mathbb{R}_{+})$ and $ 0\leq a < \langle \mu , 1 \rangle$, we set  $\kappa_{a}\mu$ the unique measure on $\mathbb{R}_+$ such that for every $r \geq 0$:
\[\kappa_{a}\mu([0,r]):=\mu([0,r])\wedge(\langle \mu,1 \rangle -a).\]
If $a \geq \langle \mu , 1 \rangle$  we simply set $\kappa_a\mu := 0$. The operation $\mu\mapsto \kappa_a \mu$ corresponds to a pruning operation "from the right"  and 
note that,  for every $a> 0$ and  $\mu\in \mathcal{M}_{f}(\mathbb{R}_{+})$, the measure $\kappa_{a}\mu$ has compact support. In particular, one has $H(\kappa_{a}\mu)<\infty$ for every $a>0$, even for $\mu$ with unbounded support. \smallskip \\
\textit{Concatenation}. Consider $\mu,\nu\in \mathcal{M}_{f}(\mathbb{R}_{+})$ such that $H(\mu)<\infty$. The concatenation of the measure  $\mu$ with $\nu$ is again an element of $\mathcal{M}_f(\mathbb{R}_+)$, denoted by $[\mu,\nu]$ and  defined by the relation:
\[\langle [\mu,\nu],f \rangle :=\int\mu(\dd r) f(r)+\int\nu(\dd r) f(H(\mu)+r). \]
 Finally, for every $\mu \in \mathcal{M}_{f}(\mathbb{R}_{+})$, the exploration process started from $\mu$ is denoted by  $\rho^{\mu}$ and defined as:
\begin{equation}\label{rho_mu}
\rho^{\mu}_{t}:=[\kappa_{-I_{t}}\mu, \rho_{t}],\quad t> 0 ,
\end{equation}
with the convention  $\rho_{0}^{\mu}:=\mu$. In this definition we used the fact that, $P$-a.s., $I_{t}<0$ for every $t>0$,  since we are not  imposing the condition  $H(\mu) < \infty$ on $\mu$. Remark that the process $\langle \rho^\mu , 1 \rangle := (\langle \rho^\mu_t , 1 \rangle : t \geq 0)$ has the same distribution as $X$ started from $\langle \mu , 1 \rangle$, this fact will be used frequently. 
\par For later use we also need to introduce the dual process of  $\rho$, this is,   the $\mathcal{M}_f(\mathbb{R}_+)$-valued process $(\eta_t :  t \geq 0)$ defined by the formula 
\begin{equation}\label{definition:eta}
\eta_t (\dd r) := \beta \mathbbm{1}_{[0,H_{t}]}(r)\dd r+\mathop{\sum \limits_{0<s\leq t}}_{X_{s-}<I_{s,t}}(X_{s} - I_{s,t} )\:\delta_{H_{s}}(\dd r),\quad t\geq 0.
\end{equation}
This process will  be only needed for  some computations  and the terminology will be justified by the identity \eqref{dualidad:etaRho}  below. Moreover, $\eta$ is càdlàg with respect to the total variation distance and the pair $(\rho, \eta)$ is a Markov process.  We refer to  \cite[Section 3.1]{Duquesne} for a complete account on $(\eta_t : t \geq 0 )$.
\\
\\
Before concluding this section, it will be crucial for our purposes to define the height process and the exploration process under the excursion measure $N$ of  $X-I$. In this direction, if for an arbitrary  fixed $r$ we set $g = \sup \{ s \leq r : X_s-I_s = 0 \}$ and $d = \inf\{ s \geq r : X_s-I_s = 0 \}$, it is straightforward to see that $(H_t : t \in [g,d])$  can be written in terms of a  functional of the excursion of $X-I$ that straddles $r$, say $e_j =( X_{(g + t) \wedge d} - I_{g} : t \geq 0)$, and this functional  does not depend on the choice of $r$. Informally, from the initial definition \eqref{definition:alturaH} this should not come as a surprise since the integral \eqref{definition:alturaH} for $t \in [g,d]$ vanishes on $[0,g]$, we refer to the discussion appearing before Lemma 1.2.4 in \cite{Duquesne} for more details. We denote this functional  by $H(e_j)$ and it satisfies that 
$P$--a.s.,  $H_t = H_{t-g}(e _j)$ for every $t \in [g,d]$.   Furthermore, if we denote  the connected components of $\{ t \geq 0: X_t - I_t = 0 \}$ by $\big((a_i, b_i) : i \in \mathbb{N} \big)$ and  the corresponding excursions by $(e_i : i \in \mathbb{N})$, then we have
$H_{(a_i + t) \wedge b_i } = H_t(e_i)$, for all $t \geq 0$.  By considering the first excursion $e$ of $X-I$ with duration $\sigma_e > \epsilon$ for every $\epsilon > 0$, it follows that the functional $H(e)$ in $D(\mathbb{R}_+ , \mathbb{R})$ under  $N( \dd e  \,| \sigma_e > \epsilon)$ is  well defined, and hence it is also well defined  under the excursion measure $N$.
\par 
Turning now our attention to the exploration process and its dual,  observe that for $t \in [a_i, b_i]$ the mass of the atoms in \eqref{rho_atoms} and \eqref{definition:eta} only depend on the corresponding excursion  $e_i$. We deduce by our previous considerations on $H$ that we can also write   $\rho_{(a_i+t) \wedge b_i} = \rho_t (e_i)$ and $\eta_{(a_i+t) \wedge b_i} = \eta_t (e_i)$, for all $t \geq 0$, where the functionals $\rho(e)$, $\eta(e)$ are still defined by  \eqref{rho_atoms} and \eqref{definition:eta} respectively, but replacing $X$ by $e_i$ and  $H$ by $H(e_i)$ -- translated in time appropriately. By the same arguments as before, we deduce that $\rho(e)$ and $\eta(e)$  under $N(\dd e)$ are well  defined $\mathcal{M}_{f}(\mathbb{R}_{+})$-valued functionals. 
From now on, when working under $N$, the dependency on $e$ is omitted from $H$, $\rho$ and $\eta$. Remark that under $N$, we still have $H(\rho_t) = H_t$   and  $\langle \rho_t , 1 \rangle = X_t$, for every $t \geq 0$, where now $X$ is an excursion of the reflected process.   
By excursion theory for  the reflected Lévy process $X-I$ we deduce that the random measure in $\mathbb{R}_+ \times \mathcal{M}_f(\mathbb{R}_+)$ defined as 
\begin{equation} \label{definition:PoissonRMExcursionesRho}
    \sum_{i \in \mathbb{N}}\delta_{(-I_{a_i} , \rho_{(a_i + \cdot)\wedge{b_i}}, \eta_{(a_i + \cdot)\wedge{b_i}} )}
\end{equation}
is a Poisson point measure with intensity $\mathbbm{1}_{\ell\geq 0}\dd \ell  \, N( \dd \rho, \dd \eta  )$. Finally, we recall for later use the equality in distribution under $N$:
\begin{equation}\label{dualidad:etaRho}
     \big( (\rho_t, \eta_t) : t \geq 0 \big) \overset{(d)}{=} \big( (\eta_{(\sigma -t )-}, \rho_{(\sigma -t )-} ) : t \geq 0 \big), 
\end{equation}
and we refer to \cite[Corollary 3.1.6]{Duquesne} for a proof. This identity is the reason why $\eta$ is called the dual process of $\rho$. 
 
\subsection{Trees coded by excursions and Lévy trees}\label{sec:trees:1}
The height process $H$ under $N$ is the main ingredient needed to define Lévy trees, one of the central objects studied in this work. Before giving a formal definition, we shall  briefly recall  some standard notation and notions related to (deterministic) pointed $\mathbb{R}$-trees. \medskip \\
\noindent \textbf{Real trees.} In the same vein as the construction of planar (discrete) trees in terms of  their  contour functions, there exists a canonical construction of  pointed $\mathbb{R}$-trees in terms of positive continuous functions. In order to be more precise, we introduce some  notation. Let $e:\mathbb{R}_+ \mapsto\mathbb{R}_{+}$  be a  continuous function, set $\sigma_e$ the functional $\sigma_e := \inf \{ t > 0 : e(s) = 0, \,  \text{ for every } s \geq t \}$ with the usual convention $\inf \{ \emptyset \} := \infty$. In particular, when  $e(0) = 0$, $\sigma_{e} < \infty$ and  $e(s)> 0$ for all $s \in (0,\sigma_e)$, the function $e$  is  called an excursion with lifetime $\sigma_e$. Note that these notations are compatible with the ones introduced in the previous section. For convenience, we take  $[0,\sigma_e]  := [0,\infty)$ if $\sigma_e = \infty$. For every $s, t \in [0,\sigma_e]$ with $s \leq t$ set
\begin{equation*}
    \displaystyle m_{e}(s,t):= \inf_{s \leq u \leq t}e(u),
\end{equation*}
and consider  the pseudo-distance on $[0,\sigma_e]$ defined by:
\begin{equation*}
d_{e}(s,t):=e(s)+e(t)-2\cdot  m_{e}(s\wedge t,s\vee t),  \quad \text{ for all } (s,t)\in [0,\sigma_{e}]^{2}. 
\end{equation*}
 The pseudo-distance $d_{e}$ induces an equivalence relation $\sim_{e}$ in $[0,\sigma_e]$ according to the following simple rule: for every $(s,t)\in [0,\sigma_{e}]^{2}$ we write $s\sim_{e} t$ if and only if $d_{e}(s,t)=0$, and we keep the notation $0$ for the equivalency class of the real number $0$. The pointed  metric space $\mathcal{T}_{e}:=([0,\sigma_e]/\sim_{e},d_{e},0)$ is an $\mathbb{R}$-tree, called the tree encoded by   $e$ and we denote  its canonical projection by $p_{e}:[0,\sigma_{e}]\to \mathcal{T}_{e}$. We stress that if $\sigma_e<\infty$, then $\mathcal{T}_e$ is a compact $\mathbb{R}-$tree.
 \\
 \\
 Let us now give some standard properties and notations. We recall that in an $\mathbb{R}$-tree there is only one continuous injective path connecting any two points $u,v \in \mathcal{T}_e$, and  we denote its image in $\mathcal{T}_e$ by $[u,v]_{\mathcal{T}_e}$. We say that $u$ is an ancestor of $v$ if $u \in [0,v]_{\mathcal{T}_e}$ and we write $u\preceq_{\mathcal{T}_e} v$.  One can check directly from the definition that we have $u\preceq_{\mathcal{T}_e} v$ if and only if there exists $(s,t)\in[0,\sigma_{e}]^{2}$ such that  $(p_{e}(s),p_{e}(t))=(u,v)$ and $e(s)=m_{e}(s\wedge t,s\vee t)$. In other words, we have:
\[[0,v]_{\mathcal{T}_e}=p_{e}\big(\big\{s\in [0,\sigma_{e}]:~ e(s)=m_{e}(s\wedge t,s\vee t)\big\}\big),\]
where $t$ is any preimage of $v$ by $p_{e}$. To simplify notation, we  write $u\curlywedge_{\mathcal{T}_e} v$ for the unique element on the tree verifying $[0,u\curlywedge_{\mathcal{T}_e} v]_{\mathcal{T}_e}=[0,u]_{\mathcal{T}_e}\cap [0,v]_{\mathcal{T}_e}$. The element   $u\curlywedge_{\mathcal{T}_e} v$ is known as  the common ancestor of $u$ and $v$. Finally, if $u\in\mathcal{T}_e$, the number of connected components of $\mathcal{T}_e\setminus \{u\}$  is called the multiplicity of $u$. For every $i\in\mathbb{N}^{*}\cup \{\infty\}$, we will denote  the set of points $u\in\mathcal{T}_e$ of multiplicity equal to $i$ by $\text{Mult}_{i}(\mathcal{T})$. The points of multiplicity larger than $2$ are called \textit{branching points}, and the points  of multiplicity $1$ are called \textit{leaves}. 
\\
\\
\textbf{Lévy trees.}  We are now in position to introduce: 
\begin{def1}
The random metric space $\mathcal{T}_{H}$ under the excursion measure $N$ is the (free) $\psi$-Lévy tree.
\end{def1}
\noindent The term free refers to the fact that the lifetime of $H$ is not fixed under $N$ and it  will be omitted from now on. Note that the metric space $\mathcal{T}_{H}$ can be considered under $P$ without any modifications. Since, under $P$, we have $\sigma_H = \infty$, the tree $\mathcal{T}_{H}$ stands for the space $(\mathbb{R}_+/\sim_{H},d_{H},0)$, and in particular it is no longer a compact space. The rest of the properties however remain valid and we will use the same notations indifferently under $P$ and $N$. Moreover, since the point $0$ is recurrent for the process $X-I$, it is also recurrent for $H$ by point (ii) of the previous section. This gives a natural interpretation of  $\mathcal{T}_{H}$ as the concatenation at the root of infinitely many  trees $\mathcal{T}_{H^{i}}$, where $(H^{i})_{i\in\mathbb{N}} = (H(e_i))_{i \in \mathbb{N}}$ are the excursions of $H$ away from $0$, and where the concatenation follows the order induced by the local time $-I$. For this reason, we will say that $\mathcal{T}_{H}$ under $P$ is a $\psi$-forest (made of $\psi$-Lévy trees). In particular, remark that under $P$ (resp. $N$), the root $p_{H}(0)$ is a branching point  of multiplicity $\infty$ (resp. a leaf).
\\
\\
Before concluding the discussion on $\mathbb{R}$-trees, we recall that, under $P$ or $N$, $\text{Mult}_{i}(\mathcal{T}_{H})=\emptyset$ for every $i\notin \{1,2,3,\infty\}$. Moreover, we have $\text{Mult}_{\infty}(\mathcal{T}_{H})\setminus\{p_H(0)\}=\emptyset$ if and only if $\pi=0$ or, equivalently, if $X$ does not have jumps. More precisely, $p_{H}$ realizes a bijection  between $\{t\geq 0: \: \Delta X_{t}>0\}$ and $\text{Mult}_{\infty}(\mathcal{T}_{H})\setminus\{p_{H}(0)\}$.

\subsection{The Lévy snake}\label{secsnake} 
In this section, we  give a short introduction  to the so-called Lévy snake, a path-valued Markov process that allows to formalize the notion of a "Markov process indexed by a Lévy tree". We follow the presentation of  \cite[Chapter 4]{Duquesne}. However,  beware that in this work we consider continuous  paths  defined in closed intervals,  and hence  our framework differs slightly with the one considered in  \cite[Chapter 4]{Duquesne}\footnote{
The paths considered in \cite[Section 4.1]{Duquesne} are càdlàg and defined in intervals of the form $[0,\zeta)$, for $\zeta \in (0,\infty)$. 
}.  
\medskip
\\
\noindent \textbf{Snakes driven by continuous functions.} Fix  a Polish space $E$ equipped with a distance $d_{E}$  inducing its topology and we let   $\mathcal{W}_{E}$ be   the set of $E$-valued killed continuous functions. Each  $\text{w} \in \mathcal{W}_E$  is a continuous path $\text{w}:[0,\zeta_{\text{w}}]\to E$, defined in a compact interval $[0,\zeta_{\text{w}}]$. The functional  $\zeta_{\text{w}} \in [0,\infty)$ is called the lifetime of $\text{w}$ and it will  be convenient to  denote the endpoint of $\text{w}$  by  $\widehat{\text{w}}:=\text{w}(\zeta_{\text{w}})$. Further, we write  $\mathcal{W}_{E,x}:=\{\text{w} \in \mathcal{W}_E:~ \text{w}(0)=x\}$ for the subcollection of paths in $\mathcal{W}_E$ starting at $x$, and we identify the trivial element of $\mathcal{W}_x$ with zero lifetime with the point $x$.  We  equip $\mathcal{W}_{E}$ with the distance
\[d_{\mathcal{W}_{E}}(\text{w},\text{w}^{\prime}):=|\zeta_{\text{w}}-\zeta_{\text{w}^{\prime}}|+\sup \limits_{r\geq 0}d_{E}\big(\text{w}(r\wedge \zeta_{\text{w}}),\text{w}^{\prime}(r\wedge \zeta_{\text{w}^{\prime}})\big),\]
and it is straightforward to check that $(\mathcal{W}_{E},d_{\mathcal{W}_{E}})$ is a Polish space.  Let us insist that the notation $e$ is exclusively used  for  continuous $\mathbb{R}_+$-valued  functions  defined in $\mathbb{R}_+$, and $\w$ is reserved for  $E$-valued continuous paths defined in compact intervals $[0, \zeta_\w]$, viz. for the elements of $\mathcal{W}_E$. \par  
 We will now endow $\mathcal{W}_{E}^{\mathbb{R}_+}$ with a  probability measure. In this direction, consider an $E$-valued  Markov process $\xi = (\xi_t : t \geq 0) $ with continuous sample paths.  For every $x\in E$, let $\Pi_{x}$ denote the distribution of $\xi$ started at $x$ and  also assume that $\xi$ is time-homogeneous (it is implicitly assumed in our definition that the mapping $x \mapsto \Pi_x$ is measurable).
Now, fix  a deterministic continuous function $h:\mathbb{R}_{+}\mapsto \mathbb{R}_{+}$. The first step towards defining the Lévy snake consists in introducing a $\mathcal{W}_E$-valued process referred as  \textit{the snake driven by $h$ with spatial motion $\xi$}. In this direction, we also fix  a point $x\in E$ and a path $\w \in \mathcal{W}_{E,x}$.  For every $a,b$ such that $0\leq a\leq \zeta_{\text{w}}$ and $b \geq a$,  there exists a unique probability measure $R_{a,b}(\text{w},\dd\text{w}^{\prime})$ on $\mathcal{W}_{E,x}$ satisfying the following properties: 
\begin{itemize}
    \item[(\rm{i})] $R_{a,b}(\text{w},\dd \text{w}^{\prime})$-a.s., $\text{w}^{\prime}(s)=\text{w}(s)$ for every $s\in[0,a]$.
    \item[(\rm{ii})] $R_{a,b}(\text{w},\dd \text{w}^{\prime})$-a.s., $\zeta_{\text{w}^{\prime}}=b$.
    \item[(\rm{iii})] Under $R_{a,b}(\text{w},\dd \text{w}^{\prime})$, $(\text{w}^{\prime}(s+a))_{s\in[0,b-a]}$
is distributed as $(\xi_{s})_{s\in[0,b-a]}$ under $\Pi_{\text{w}(a)}$.
\end{itemize}
Denoting  the canonical process on $\mathcal{W}^{\mathbb{R}_{+}}_{E}$  by $(W_{s})_{s\geq 0}$,   it is easy to see  by Kolmogorov's extension theorem that, for  every $\text{w}_{0} \in \mathcal{W}_{E,x}$ with $\zeta_{\text{w}_{0}}=h(0)$, there exists a unique probability measure $Q^{h}_{\text{w}_{0}}$ on $\mathcal{W}_E^{\mathbb{R}_+}$ satisfying  that
\begin{align*}\label{Q_h_w}
&Q^{h}_{\text{w}_{0}}\big(W_{s_{0}}\in A_{0}, W_{s_{1}}\in A_{1},...,W_{s_{n}}\in A_{n}\big)\\
&\hspace{0.2cm}=\mathbbm{1}_{\{\text{w}_{0}\in A_{0}\}}\int_{A_{1}\times A_{2}\times \cdots\times A_{n}} R_{m_{h}(s_{0},s_{1}),h(s_{1})}(\text{w}_{0}, \dd\text{w}_{1})R_{m_{h}(s_{1},s_{2}),h(s_{2})}(\text{w}_{1}, \dd\text{w}_{2})...R_{m_{h}(s_{n-1},s_{n}),h(s_{n})}(\text{w}_{n-1}, \dd\text{w}_{n}). 
\end{align*}
 for every $0=s_{0}\leq s_{1}\leq ...\leq s_{n}$ and $A_{0},..., A_{n}$ Borelian sets of $\mathcal{W}_{E}$. 
The canonical process $W$ in $\mathcal{W}^{\mathbb{R}_+}_E$ under $Q^{h}_{\text{w}_{0}}$ is called the snake driven by $h$ with spatial motion $\xi$ started from $\w_0$.  The value $W_s = (W_s(t) : t \in [0,h(s)])$ of the Lévy snake at time $s$ coincides with $\w_0$ for $0 \leq t  \leq m_h(0,s)$ while for $m_{h}(0,s)\leq t \leq h(s)$, it is distributed as the Markov process $\xi$ started at $\w_0(m_h(0,s))$ and stopped at time $h(s)-m_h(0,s)$. 
 Furthermore, informally, when $h$ decreases, the path is erased from its tip and, when $h$ increases, the path is extended by adding “little pieces” of trajectories of $\xi$ at the tip. The term snake refers to the fact that,  the definition of $Q_{\text{w}_0}^h$ entails that for every $s<s^\prime$ we have:
\begin{equation}\label{weak:snake}
W_{s}(r)=W_{s^\prime}(r),\quad r\in [0,m_{h}(s,s^\prime)],\quad Q_{\text{w}_0}^h\text{--a.s.} 
\end{equation}
 Note however that this property only holds  for fixed $s,s^\prime$ $Q_{\w_0}^h$-a.s. A priori, under  $Q^{h}_{\text{w}_{0}}$, the process $W$ does not have a continuous modification  with respect to the metric $d_{\mathcal{W}_E}$, but it will be crucial for our work to find suitable conditions guaranteeing the existence of such  modification.  This question will be addressed in the following proposition. We start by introducing some notation. First recall 
the convention $[a,\infty]:=[a,\infty)$ for $a < \infty$. Next,  consider a $\mathcal{J}$-- indexed family  $a_i, b_i \in \mathbb{R}_+\cup \{\infty\}$,  $\mathcal{J} \subset \mathbb{N}$, with $a_i < b_i$ and suppose that the intervals  $([a_i, b_i], i \in \mathcal{J})$ are  disjoint. A continuous function $h:\mathbb{R}_+ \mapsto \mathbb{R}_+$  is said to be locally $r$-Hölder continuous in $([a_i, b_i], i \in \mathcal{J})$ if, for every $n\in \mathbb{N}$, there exists a constant $C_n$ satisfying that  $|h(s) - h(t)| \leq C_n|s-t|^r$, for every $i \in \mathcal{J}$ and  $s,t \in [a_i,b_i]\cap [0,n]$. We insist on the fact that the constant $C_n$ does not depend on the index $i$.
 \begin{prop}\label{lemma:regularidadSnakeDriven}
 Suppose that there exists  a constant $C_\Pi> 0$ and two positive numbers $p,q > 0$ such that, for every $x \in E$ and $t\geq 0$, we have:
\begin{equation} \label{regularidad:estimate}
    \Pi_{x}\big(  \sup_{0 \leq u \leq t } d_E(\xi_u , x)^p  \big) \leq C_\Pi\cdot t^{q}.
\end{equation}
Further, consider a continuous function $h: \mathbb{R}_+ \mapsto \mathbb{R}_+$ and denote by $((a_i, b_i):~ i \in \mathcal{J})$ the excursion intervals above its running infimum. If  $h$ is locally $r$-Hölder continuous in $([a_i, b_i]: i \in \mathcal{J})$ with $qr > 1$ then, for every $\emph{w}\in \mathcal{W}_E$ with $\zeta_{\emph{\w}} = h(0)$, the process $W$ has a continuous modification under $Q_\text{\emph{w}}^h$. 
\end{prop}
\begin{proof}
With the notation introduced in the statement of the proposition, we fix a  continuous driving function  $h: \mathbb{R}_+ \mapsto \mathbb{R}_+$ locally $r$-Hölder continuous in $([a_i, b_i]:~ i \in \mathcal{J})$, an initial condition $\w \in \mathcal{W}_E$ with $\zeta_\w = h(0)$, and  we consider an arbitrary $n \in \mathbb{N}$.  By definition,   for  any $s,t \in [a_i,b_i] \cap [0,n]$, we have  $|h(s) - h(t)| \leq C_n\cdot |s-t|^r$ for a constant $C_n$ that does not depend on $i$.  Next, we consider $W$,  the snake  driven by $h$ under $Q^h_{\w}(\dd W)$.  The first step of the proof consists in showing  that the process $(W_s : s \in \bigcup_{i\in \mathcal{J}}[a_i,b_i])$ has a locally Hölder-continuous  modification  on  $\big([a_i,b_i]:~ i \in \mathcal{J} )$. In this direction, we remark that the definition of $d_{\mathcal{W}_E}$ gives: 
\begin{align*} 
    Q^h_\w \big(  d_{\mathcal{W}_E}(W_s, W_t)^p   \big) 
    &\leq  2^{p} \cdot    Q^h_\w \Big(   \sup_{m_h(s,t) \leq u}d_{E} \big(W_s(u \wedge h(s) ), W_t(u \wedge h(t) )\big)^p   \Big) + 2^{p}\cdot|h(s) - h(t)|^p, 
\end{align*}
for every  $s,t \in [a_i, b_i] \cap [0,n]$. Next, note  that the first term on the right hand side can be bounded above by: 
\begin{align*}
& Q^h_\w \Big(   \sup_{m_h(s,t) \leq u} d_{E}\big(W_s(u \wedge h(s) ), W_t(u \wedge h(t))\big)^p   \Big)  \\
& \leq  
2^{p}\cdot  Q^h_\w \Big(   \sup_{m_h(s,t) \leq u} d_{E} \big( W_s(u \wedge h(s) ), W_s( m_h(s,t)) \big)^p   \Big) 
+ 2^p\cdot  Q^h_\w \Big(   \sup_{m_h(s,t) \leq u} d_{E}\big(W_t( m_h(s,t) ), W_t(u \wedge h(t))\big)^p   \Big) \\
&\leq 2^{p}\cdot  Q^h_\w \Big(   \Pi_{W_s(m_h(s,t))} \big(  \sup_{ u \leq  h(s) - m_h(s,t) } d_{E}( \xi_u , \xi_0 )^p \big)  \Big) 
+ 2^{p}\cdot  Q^h_\w \Big(   \Pi_{W_t(m_h(s,t))} \big(  \sup_{ u \leq  h(t) - m_h(s,t) } d_{E}( \xi_0 , \xi_u )^p \big)  \Big)  \\
&\leq 2^{p} C_{\Pi}\cdot\Big( \big|h(s)-m_h(s,t)\big|^{q}+\big|h(t)-m_h(s,t)\big|^{q}\Big), 
\end{align*} where in the second inequality we  applied the Markov property at time $m_h(s,t)$, and in the last one we used the upper bound \eqref{regularidad:estimate}. By our  assumptions on $h$ we derive that, for every $n>0$, there exists a constant $C'_n$ such that:
\begin{equation*}
    Q^h_\w \left(  d_{\mathcal{W}_E}(W_s, W_t)^p   \right)  \leq C'_n \cdot \big( | t-s |^{ q r  }+| t-s |^{ p r  }\big) , \quad \quad \text{ for any } s,t \in [a_i, b_i] \cap [0,n], 
\end{equation*}
and we stress that the constant  $C'_n$  does not depend on $i$. Recall that  $qr>1$ and note that we can assume as well that $p r > 1$,   since by replacing the distance $d_{\mathcal{W}_E}$ by $1\wedge d_{\mathcal{W}_E}$, we can take $p$ as large as wanted.  Now, fix $r_0 \in (0, (qr-1)/p )$. We deduce by a standard Borel-Cantelli argument, similar to the  proof of Kolmogorov's lemma, that there exists a modification of $(W_s : s \in [0,n] \cap \bigcup_{i\in \mathcal{J}}[a_i,b_i] )$, say $(W_s^* : s \in [0,n] \cap \bigcup_{i\in \mathcal{J}}[a_i,b_i] )$, satisfying that $Q_\w^h$-- a.s., for every  $i \in  \mathcal{J}$ 
\begin{equation} \label{equation:regularity_eq_2}
     d_{\mathcal{W}_E}(  W_s^*, W_t^*  ) \leq K_n |s-t|^{r_0},  \quad \quad \text{ for every } s,t \in [a_i, b_i] \cap [0,n], 
\end{equation}
 where the (random) quantity $K_n$ does not depend on $i$. To simplify notation, set $\mathcal{V}:=\mathbb{R}_+\setminus \bigcup_{i\in \mathcal{J}}[a_i,b_i]$ and remark that if $t\in \mathcal{V}$, then $h(t)=\inf\{h(u):~u\in [0,t]\}$. For every $t \in \mathcal{V}$, we set $W^*_t := ( \w(u) : u \in [0,h(t)] )$ and we consider the process $(W^*_t: t\in [0,n])$.  Notice that by the very construction of $W^*$, we have  $Q_\w^h(W_t = W^*_t)=1$ for every $t \in [0,n]$, which shows that $W^*$ is a modification of $W$ in $[0,n]$. 
 \par Let us now show that $W^*$ is continuous on $[0,n]$. The continuity for $t \in [0,n]\cap \bigcup_{i\in \mathcal{J}}(a_i,b_i) $ follows by \eqref{equation:regularity_eq_2} and we henceforth  fix $t\in [0,n]\setminus \bigcup_{i\in \mathcal{J}}(a_i,b_i) $. In particular, we have  $h(t)=\inf\{h(u):~u\in [0,t]\}$. On one hand, if $(s_k:~k\geq 1)$ is a sequence with $s_k \rightarrow t$ as $k \uparrow \infty$,  the continuity of $\w$ and $h$ ensures that  $(\w( u) : u \in [0, h(s_k)] ) \rightarrow W^*_t$ with respect to $d_{\mathcal{W}_E}$. Consequently, if the subsequence $(s_{k}:~k\geq 1)$ takes values in $[0,n]\setminus \bigcup_{i\in \mathcal{J}}(a_i,b_i)$,  it holds that:
 $$\lim \limits_{k\to \infty }d_{\mathcal{W}_E}(W_{s_k}^*, W_t^*)=\lim \limits_{k\to \infty }d_{\mathcal{W}_E}\Big(  \big(\w(u) : u \in [0, h(s_k)] \big)  , W^*_{t} \Big) =0.$$
On the other hand, for every $s \in [a_j,b_j] \cap [0,n]$ for some $j \in \mathcal{J}$ with $s < t$, we have 
\begin{align*}
    d_{\mathcal{W}_E}(W_s^*, W_t^*) 
    \leq d_{\mathcal{W}_E}(W_s^*, W_{b_j}^*)+ d_{\mathcal{W}_E}\big( W^*_{b_j} ,  W^*_{t} \big)
    \leq K_n|s-t|^{r_0} + d_{\mathcal{W}_E}\Big(  \big(\w(u \wedge \zeta_\w) : u \in [0,h(s)] \big)  , W^*_{t} \Big),
\end{align*}
which goes to $0$ as $s \uparrow t$ since $W^*_t = \big( \w(u) : u \in [0,h(t)] \big)$. The case $s>t$ can be treated similarly by replacing  $b_i$ with  $a_i$ and  it follows that, for any subsequence $(s_k:~k\geq 1)$ with $s_k \rightarrow t$, we have $d_{\mathcal{W}_E}(W_{s_k}^*, W_t^*) \rightarrow 0$. Consequently, $W^*$ is continuous on $[0,n]$.  Since this holds for any $n$, we can define a continuous modification of $W$ in $\mathbb{R}_+$. 
\end{proof}
 Under the conditions of Proposition \ref{lemma:regularidadSnakeDriven}, the measure $Q^h_{\rm{w}}$ can be defined in the Skorokhod space of $\mathcal{W}_E$-valued right-continuous paths $\mathbb{D}(\mathbb{R}_+, \mathcal{W}_E)$ and,  with a slight abuse of notation, we still denote it  by $Q^h_{\rm{w}}$.  From now on, we shall  work under these conditions and $Q^h_\w$ will always be considered as a measure in $\mathbb{D}(\mathbb{R}_+ , \mathcal{W}_E)$. In particular, remark that if we write $W$ for the canonical process in $\mathbb{D}(\mathbb{R}_+ , \mathcal{W}_E)$, then  $W$ is $Q^h_\w$--a.s. continuous.  Finally, we point out that the regularity of  $W$ was partially addressed in the proof of \cite[Proposition 4.4.1]{Duquesne}, for initial conditions of the form  $x$ with $x \in E$,  when working with paths $\w$ defined in the half open interval  $[0,\zeta_\w)$.\medskip
\\
\noindent \textbf{The Lévy snake with spatial motion $\xi$.} The driving function $h$ of the random snake that we have considered so far was deterministic, and  the next step consists in  randomising  $h$.   We write $\mathcal{M}_f^0$ for the subset of $\mathcal{M}_f(\mathbb{R}_+)$ defined as  
\[\mathcal{M}^{0}_{f}:=\big\{\mu\in \mathcal{M}_{f}(\mathbb{R}_{+}):\:H(\mu)<\infty \ \text{ and } \text{supp } \mu = [0,H(\mu)]\big\}\cup\{0\},\] 
and we introduce $\Theta$ the collection of pairs  $(\mu, \w) \in \mathcal{M}_f^0 \times \mathcal{W}_{E}$ such that $H(\mu)=\zeta_{\text{w}}$. {Fix a Laplace exponent $\psi$ satisfying (A1) -- (A4), and set 
\begin{equation} \label{defintiion:coladeBallenaconcaradepalemera}
    \Upsilon:=\sup \big\{ r \geq 0 : \lim_{\lambda \rightarrow \infty} \lambda^{-r}\psi(\lambda) = \infty \big\}. 
\end{equation} 
In particular, by the convexity of $\psi$ we must have $\Upsilon \geq 1$. For every $\mu \in \mathcal{M}^{0}_{f}$, write $\textbf{P}_{\mu}$ for the distribution of the exploration process started from $\mu$  in $\mathbb{D}(\mathbb{R}_+, \mathcal{M}_f(\mathbb{R}_+))$ -- the space of right-continuous $\mathcal{M}_f(\mathbb{R}_+)$-valued paths.  With a slight abuse of notation we  denote  the canonical process in $\mathbb{D}(\mathbb{R}_+, \mathcal{M}_f(\mathbb{R}_+))$ by $\rho$ and observe that, by Definition \ref{rho_mu},  the process $\rho$ under $\mathbf{P}_\mu$ takes values in $\mathcal{M}^{0}_{f}$. Notice that $H(\rho)$ under  $\textbf{P}_{\mu}$ is continuous since $\mu \in \mathcal{M}_f^0$. We can now state the hypothesis we will be working with. 
\medskip \\
\noindent In the rest of this work, we will always assume that:
\begin{enumerate}
    \item[\mbox{}] \textbf{Hypothesis $(\textbf{H}_{0})$}. There exists  a constant $C_\Pi> 0$ and two positive numbers $p,q > 0$ such that,\\ for every $x \in E$ and $t\geq 0$, we have:
\begin{equation} \label{continuity_snake}
    \hspace{-15mm}\Pi_{x}\big(  \sup_{0 \leq u \leq t } d_E(\xi_u , x)^p  \big) \leq C_\Pi\cdot t^{q}, \hspace{7mm} \text{ and } \hspace{7mm} q\cdot (1-\Upsilon^{-1})>1. \tag*{$(\textbf{H}_{0})$}
\end{equation}
\end{enumerate}
For instance, it can  be checked that condition \ref{continuity_snake} is fulfilled if the Lévy tree has exponent $\psi(\lambda)= \lambda^\alpha$ for $\alpha \in (1,2]$ and $\xi$ is a Brownian motion.
Let us discuss the implications of \ref{continuity_snake}. 
%%%
Under  $\textbf{P}_{\mu}$, denote the excursion intervals of $H$ above its running infimum by $(\alpha_i, \beta_i)$. Recall from \eqref{rho_mu} that $(\rho^\mu_t :=  [k_{-I_t}\mu , \rho_t ]:~t \geq 0)$,  under $\textbf{P}_0$,  is distributed according to $\textbf{P}_\mu$, and note that  $H_t(\rho^\mu) = H(k_{-I_t}\mu) + H(\rho_t)$, for $t \geq 0$. By    \cite[Theorem 1.4.4]{Duquesne}, under $\mathbf{P}_0$  the process $H(\rho)$ is locally Hölder continuous  of exponent $m$ for any $m \in (0, 1 - \Upsilon^{-1})$. In particular, this holds for some  $m:=r$  verifying $qr > 1$ by the second condition in \ref{continuity_snake}. Since $\big(H(k_{-I_t}\mu): t \geq 0 \big)$ is constant on each excursion interval $(\alpha_i, \beta_i)$ and $(H(\rho_t): t \geq 0)$ is locally $r$-Hölder continuous,  we deduce that $H(\rho^\mu)$ is locally $r$-Hölder continuous on  $([\alpha_i, \beta_i]:\, i \in \mathbb{N})$. Said otherwise, $\textbf{P}_\mu$-a.s., the paths of  $H(\rho)$ satisfy the conditions of Proposition  \ref{lemma:regularidadSnakeDriven} and we  will henceforth assume  that the condition is  satisfied for every path, and not only outside of a negligible set. \par 
Finally, consider the canonical process $(\rho , W)$ in $\mathbb{D}(\mathbb{R}_+, \mathcal{M}_f(\mathbb{R}_+)\times \mathcal{W}_E )$, the space of $\mathcal{M}_f(\mathbb{R}_+)\times \mathcal{W}_E$-valued,   right continuous paths. By our previous discussion we deduce that  we can define a probability measure in  $\mathbb{D}(\mathbb{R}_+, \mathcal{M}_f(\mathbb{R}_+)\times \mathcal{W}_E )$  by setting  
\begin{equation*}
\mathbb{P}_{\mu,\text{w}}(\dd \rho,\: \dd W):=\textbf{P}_\mu(\dd \rho)\:Q^{H(\rho)}_{\text{w}}(\dd W), 
\end{equation*} 
for every $(\mu,\text{w})\in \Theta$. The process $(\rho , W)$ under $\mathbb{P}_{\mu , \w}$ is called the $\psi$-Lévy snake with spatial motion $\xi$ started from $(\mu , \w)$. We denote its canonical filtration by $(\mathcal{F}_t:~t\geq 0)$ and observe that by construction, $\mathbb{P}_{\mu , \w}$--a.s., $W$ has continuous paths.} Now, the proof  of  \cite[Theorem 4.1.2]{Duquesne} applies without any change to our framework and gives that the process $((\rho , W), (\mathbb{P}_{\mu , \w} : (\mu , \w) \in \Theta ))$ is a strong Markov process with respect to the filtration $(\mathcal{F}_{t+})$. It should be noted that assumption \ref{continuity_snake} is the same  as the one appearing in \cite[Proposition 4.4.1]{Duquesne}, for paths defined in $[0,\zeta_\w)$ and  started from $x\in E$. In the particular case  $\psi(\lambda) = \lambda^2/2$, the path regularity of $W$  was already addressed in  \cite[Theorem 1.1]{APathValuedMarkovProcess}. 
\par  Let us conclude our discussion concerning  regularity issues by  introducing the notion of \textit{snake paths}, which summarises the regularity properties of $(\rho,W)$  as well as some  related notation that will be used throughout this work. Recall that $\mathcal{M}_f(\mathbb{R}_+)$, equipped with the topology of weak convergence, is a Polish space  \cite[Lemma 4.5]{KallenbergRandMeas}.  We denote systematically  the elements of the path space  $\mathbb{D}(\mathbb{R}_+ , \mathcal{M}_f(\mathbb{R}_+) \times \mathcal{W}_E)$ by:  
\begin{equation*}
    (\uprho, \omega) = \big(  (\uprho_s, \omega_s) : \, s \in \mathbb{R}_+  \big), 
\end{equation*}
and by definition, we have $(\rho_s(\upvarrho), W_s(\omega)) = (\upvarrho_s, \omega_s )$ for $s \in \mathbb{R}_+$. 
For each fixed $s$, $\omega_s$ is an element of $\mathcal{W}_{E}$ with lifetime  $\zeta_{\omega_s}$, and the $\mathbb{R}_+$-valued  process $\zeta(\omega) := (\zeta_{\omega_s}: \, s \geq 0)$  is called the lifetime process of $\omega$. We   will occasionally use the notation $\zeta_s(\omega)$ instead of $\zeta_{\omega_s}$, and in such cases we will drop the dependence on $\omega$  if there is no risk of confusion. 
\begin{def1}\label{definition:snakePath}
A snake path started from $(\mu, \emph{w}) \in  \Theta$ is an element  $(\uprho , \omega) \in \mathbb{D}(\mathbb{R}_+ , \mathcal{M}_f(\mathbb{R}_+) \times \mathcal{W}_E)$ such that the mapping $s \mapsto \omega_s$ is continuous,  and satisfying the following properties: 
 \begin{itemize}
     \item[\rm{(i)}] $(\uprho_0 , \omega_0 ) = (\mu,\emph{w})$.
     \item[\rm{(ii)}] $(\uprho_s , \omega_s) \in \Theta$ for all $s \geq 0$, in particular   $H(\rho) = \zeta(\omega)$.
     \item[\rm{(iii)}] $\omega$ satisfies the snake property: for any $0\leq s \leq s'$,  
     \begin{equation*} 
         \omega_s(t) = \omega_{s'}(t) \, \,  \text{ for all } \, \,  0 \leq t \leq \inf_{[s,s']} \zeta(\omega).
     \end{equation*}
 \end{itemize}
\end{def1}
\noindent  A continuous $\mathcal{W}_E$-valued path $\omega$ satisfying (iii)  is called    a snake trajectory. We point out that this notion had  already been introduced   in the context of the Brownian snake \cite[Definition 6]{ALG}. However, in the Brownian case the process $W$ is Markovian and there is no need of working with pairs $(\upvarrho, \omega)$ -- this is the reason why we have to introduce the notion of snake paths.  We denote  the collection of snake paths started from  $(\mu , \w) \in \Theta$ by $\mathcal{S}_{\mu, \w}$ and simply write $\mathcal{S}_x$ instead of $\mathcal{S}_{0,x}$. Finally,  we set: 
 \begin{equation*}
     \mathcal{S} := \bigcup_{(\mu , \w) \, \in \, \Theta } \mathcal{S}_{\mu, \w}.
 \end{equation*}
For any given $(\uprho, \omega) \in \mathcal{S}$, we denote indifferently its duration by 
 \begin{equation*}
     \sigma_{H(\uprho)} = \sigma(\omega) = \sup \{t \geq 0 : \, \zeta_{\omega_t} \neq 0  \}.
 \end{equation*}
Remark that,  by continuity and the definition of $Q^{h}_{\text{w}}$, the process  $((\rho , W), (\mathbb{P}_{\mu , \w} : (\mu , \w) \in \Theta ))$ takes values in $\mathcal{S}$ -- it satisfies the snake property by \eqref{weak:snake} and the continuity of $W$. Said otherwise,  $\mathbb{P}_{\mu,\text{w}}\text{-a.s.}$, we have 
\[\zeta_{s}=H(\rho_{s}), \:\: \text{ for every } s \geq 0, \]
and  for any $t \leq t'$
\[W_{t}(s)=W_{t^{\prime}}(s), \:\: \text{ for all } s\leq m_{H}(t,t^{\prime}). \]
 We stress that when working on $\mathcal{S}$ the equivalent notations $\zeta_s$, $H(\rho_s)$  and $H_s$ will be used indifferently.  The snake property implies that, for every $t,t^{\prime} \geq 0$ such that $p_{H}(t)=p_{H}(t^{\prime})$, we have $W_{t}=W_{t^{\prime}}$. In particular, for such times it holds that  $\widehat{W}_{t}=\widehat{W}_{t^{\prime}}$ and hence $(\widehat{W}_t:~t\geq 0)$ can be defined in the quotient space $\mathcal{T}_H$. More precisely, under $\mathbb{P}_{\mu,\text{w}}$, the function defined with a slight abuse of notation for all  $\upsilon \in \mathcal{T}_H$ as 
\[
\xi_{\upsilon}:=\widehat{W}_{t}, \quad \text{where $t$ is any element of } p_{H}^{-1}(\upsilon),
\]
is  well defined and leads us to the notion of  tree indexed processes. When $(\mu , \w) = (0,x)$,  the process $(\xi_\upsilon)_{\upsilon \in \mathcal{T}_H}$ is known as the Markov process $\xi$  indexed by the tree $\mathcal{T}_H$ and started from $x$.\footnote{With the terminology  introduced in \cite[Definition 7]{ALG}, the pair of processes  $(H,\widehat{W})$ is called a treelike-path.} In this work, we will need to consider the restriction of $(\rho,W)$ to different intervals and therefore,  it will be convenient to  introduce a formal notion of subtrajectories.
\medskip \\
{
\noindent \textbf{Subtrajectories. }Fix $s<t$ such that $H_s = H_t$ and $H_r > H_s$ for all $r \in (s,t)$.  The subtrajectory of $(\rho,W)$ in $[s,t]$ is the process taking values in $\mathbb{D}(\mathbb{R}_+,\mathcal{M}_{f}(\mathbb{R}_+) \times\mathcal{W}_E)$, denoted by $(\rho^{\prime}_{r},W^{\prime}_{r})_{r\in[0,t-s]}$  and  defined as follows:  for every $r\in[0,t-s]$, set 
\begin{align*}
   \langle \rho^{\prime}_{r} , f \rangle  & :=\int \rho_{r+s}(\dd h)f(h - H_s)\mathbbm{1}_{\{ h  > H_s \}} \quad \quad \text{ and } \quad \quad 
    W^{\prime}_{r}(\cdot) :=W_{s+r}(H_{s}+\cdot \,).
\end{align*}
In particular, we have 
\[ \zeta(W^{\prime}_{r})=H_{s+r}-H_{s}=H(\rho^{\prime}_{r}), \quad \text{ for all }  r\in [0,t-s]. \]
}
\noindent Remark that if $(\rho,W)$ is a snake path, then the subtrajectory $(\rho^{\prime}, W^{\prime})$ is also in $\mathcal{S}$. Informally, $W^{\prime}$ encodes the labels $(\xi_{v}:~v\in p_{H}([s,t]))$.
\subsection{Excursion measures of the Lévy snake}
Fix $x\in E$ and consider the Lévy snake $(\rho,W)$ under $\mathbb{P}_{0,x}$. By (\ref{equation:zeros}), the measure $0$ is a regular recurrent point for the Markov process $\rho$, which implies that $(0,x)$ is on its turn regular and recurrent for the Markov process $(\rho,W)$. Moreover,  $(-I_{t}: t \geq 0)$ is a local time at $0$ for $\rho$ and hence it is a local time at $(0,x)$ for $(\rho,W)$. We set   $\mathbb{N}_{x}$ the excursion measure of $(\rho,W)$ away from $(0,x)$  associated with  the local time $-I$.  We stress that $\mathbb{N}_{x}$ is a measure in the canonical space  $\mathbb{D}(\mathbb{R}_+ , \mathcal{M}_f(\mathbb{R}_+) \times \mathcal{W}_E)$. By excursion theory of the Markov process $(\rho , W)$, if $\{(\alpha_i, \beta_i):~i \in \mathcal{I}\}$ stands for the excursion intervals of $(\rho , W)$ and $(\rho^i, W^i)$ are the corresponding subtrajectories then, under $\mathbb{P}_{0,x}$, the measure 
\begin{equation}\label{definition:poissonExcursiones}
    \sum \limits_{i\in \mathcal{I}}\delta_{(-I_{\alpha_i},\rho^{i},W^{i})}, 
\end{equation}
is a Poisson point measure with intensity $\mathbbm{1}_{[0,\infty)}(\ell)\dd \ell\:\mathbb{N}_{x}(\dd\rho, \:\dd\omega).$ 
Recalling the interpretation of the restrictions $\mathbb{N}_x( \, \cdot \, | \sigma > \epsilon)$ as the law of the  first excursion with length greater than $\epsilon$, it follows that  under $\mathbb{N}_{x}$, $W$  satisfies the snake property and $(\rho , W) \in \mathcal{S}$. In particular, we can still make use of the definition of subtrajectories and  $(\xi_{\upsilon})_{\upsilon\in \mathcal{T}_{H}}$ under the excursion measure $\mathbb{N}_{x}$, and for simplicity we will use the same notation. \par
By the previous discussion, it is straightforward to verify that 
\begin{equation}\label{N:H(rho)}
\mathbb{N}_{x}(\dd \rho,\: \dd \eta, \: \dd W)=N(\dd \rho, \, \dd \eta)\: Q^{H(\rho)}_{x}(\dd W).
\end{equation}
Said otherwise,  under $\mathbb{N}_{x}$: 
\begin{itemize}
    \item The distribution of $(\rho, \eta)$ is $N(\dd \rho, \, \dd \eta)$;
    \item The conditional distribution of $W$ knowing $(\rho, \eta)$ is $Q_{x}^{H(\rho)}$. 
\end{itemize}
Remark that by construction and \eqref{dualidad:etaRho}, under $\mathbb{N}_x$ we have 
\begin{equation}\label{dualidad:etaRhoW}
     \big( (\rho_t, \eta_t, W_t) : t \in [0,\sigma] \big) \overset{(d)}{=} \big( (\eta_{(\sigma -t )-}, \rho_{(\sigma -t )-} , W_{\sigma -t } ) :t \in [0,\sigma] \big),
\end{equation}
where we used that by continuity, we have $W_{\sigma -t }=W_{(\sigma -t )-}$ for every $t\in[0,\sigma]$.
\par 
When starting from an arbitrary $(\mu , \w) \in \Theta$, the following variant of \eqref{definition:poissonExcursiones} will be used frequently in our computations: 
let $\mathbb{P}_{\mu , \w}^{\dag}$ be the distribution of $(\rho, W)$ killed at time  $\sigma := \inf \{ t \geq 0 : H(\rho_s) = 0 \text{ for every } s\geq t \}$.  For instance, it will be worth noting that by \eqref{rho_mu}, the process  $\langle \rho , 1 \rangle$ is a Lévy process started from $\langle \mu,1 \rangle$ and stopped when reaching 0. Write $\big( (\alpha_i, \beta_i) : \, i \in \mathbb{N} \big)$ for the excursion intervals over the running infimum of  $\langle \rho,1 \rangle$ under $\mathbb{P}^\dag_{\mu,\text{w}}$ and denote  the corresponding subtrajectory associated with $[\alpha_i , \beta_i]$ by $(\rho^i , W^i)$. If for $t \geq 0$ we write $I_t := \inf_{s \leq t}\langle \rho_s ,  1 \rangle - \langle \mu , 1\rangle$, the measure
\begin{equation}
    \label{PoissonRandMeasure_Inf}
    \sum_{i \in \mathbb{N}}\delta_{(-I_{\alpha_i}, \rho^i , W^i)}, 
\end{equation}
is a Poisson point measure with  intensity $\mathbbm{1}_{[0, \langle \mu ,1 \rangle ]} (u) \,  \dd u \, \mathbb{N}_{\text{w}( H( \kappa_{u} \mu ) )}(\dd  \rho, \dd W)$.   Moreover,  if  $h_i := H_{\alpha_i} = H_{\beta_i}$, by \eqref{rho_mu} we have  $h_i = H(\kappa_{-I_{\alpha_i}} \mu )$ and since the image measure of $\mathbbm{1}_{[0,\langle \mu , 1 \rangle]}(u) \, \dd u$ under the mapping $u \mapsto H(\kappa_u \mu)$ is precisely $\mu$, we deduce that  under $\mathbb{P}^\dag_{\mu,\text{w}}$  the  measure
\begin{equation} \label{PoissonRandMeasure}
    \sum_{i \in \mathbb{N}}\delta_{(h_i, \rho^i , W^i)} 
\end{equation}
is a Poisson point measure with  intensity $\mu(\dd h)\mathbb{N}_{\text{w}(h)}(\dd  \rho, \dd W)$.  We refer to  \cite[Lemma 4.2.4]{Duquesne} for additional details. 
\\
\\
We close this section by recalling a many-to-one formula that will be used frequently to obtain explicit computations. We start with some preliminary notations: consider  a 2-dimensional subordinator $(U^{(1)}, U^{(2)})$ defined in some auxiliary probability space $(\Omega_0, \mathcal{F}_0, P^0 )$ with Laplace exponent  given by 
\begin{equation} \label{identity:exponenteSubord}
    - \log E^0  \Big[ \exp \big( - \lambda_1 U^{(1)}_1 - \lambda_2 U^{(2)}_1 \big) \Big] := 
    \begin{cases}
    \frac{\psi(\lambda_1) - \psi(\lambda_2)}{\lambda_1 - \lambda_2} -\alpha\quad \quad \text{if } \lambda_1 \neq  \lambda_2 \\
    \psi'(\lambda_1) - \alpha    \quad \quad \hspace{7.5mm} \text{if } \lambda_1 = \lambda_2, 
    \end{cases}
 \end{equation}
 where $E^0$ stands for the expectation taken with respect to $P^0$. Notice that in particular  $U^{(1)}$ and $U^{(2)}$ are subordinators with Laplace exponent $\lambda \mapsto \psi(\lambda)/\lambda -  \alpha$.
Let $(J_a , \widecheck{J}_a)$ be the pair or random measures  defined by 
\begin{equation*}
   (J_a , \widecheck{J}_a)  := \big(\mathbbm{1}_{[0,a]}(t)~ \dd U^{(1)}_t , \mathbbm{1}_{[0,a]}(t) ~\dd U^{(2)}_t \big),
\end{equation*}
with the convention  $(J_\infty , \widecheck{J}_\infty)  := (\mathbbm{1}_{[0,\infty)}(t) ~\dd U^{(1)}_t , \mathbbm{1}_{[0,\infty)}(t) ~\dd U^{(2)}_t )$. The following many-to-one  equation will play a central role in all this work:
\begin{lem}
For every $x \in E$ and every non-negative measurable functional $\Phi$ taking values in $\mathcal{M}_f(\mathbb{R}_+)^2 \times \mathcal{W}_E$, we have: 
\begin{equation}\label{tirage_au_hasard_N}
     \mathbb{N}_{x} \Big( \int_0^{\sigma} \dd s \, \Phi\big(\rho_s , \eta_s,  W_s \big) \Big)  
    = \int_0^\infty \dd a \, \exp\big(- \alpha a\big)\cdot  
    E^0 \otimes \Pi_{x} \Big(  \Phi \big( J_{a}, \widecheck{J}_a , (\xi_t : t \leq a ) \big)  \Big).
\end{equation}
\end{lem}
\begin{proof}
First, remark that we have
$$ \mathbb{N}_{x} \Big( \int_0^{\sigma} \dd s \, \Phi\big(\rho_s , \eta_s,  W_s \big) \Big)=\int_{0}^{\infty} \dd s \, \mathbb{N}_{x} \Big( \mathbbm{1}_{\{s< \sigma_H \}} \, \Phi\big(\rho_s , \eta_s,  W_s \big) \Big).$$
Next, we  use \eqref{N:H(rho)} to write the previous display in the form:
$$\int_{0}^{\infty} \dd s \, \mathbb{N}_{x} \Big( \mathbbm{1}_{\{s < \sigma_H\}} \,\Pi_{x}\Big[ \Phi\big(\rho_s , \eta_s,  (\xi_r:~r\leq H(\rho_s)) \big)\Big] \Big) = N \Big(\int_{0}^{\sigma} \dd s  \,\Pi_{x}\Big[ \Phi\Big(\rho_s , \eta_s,  \big(\xi_r:~r\leq H(\rho_s)\big) \Big)\Big] \Big). $$
Since now $\Pi_x \big[ \Phi\big(\rho_s, \eta_s , (\xi_r: r \leq H(\rho_s))\big) \big]$ is a functional of $(\rho_s, \eta_s)$, it  suffices to establish \eqref{tirage_au_hasard_N} for a functional only depending on the pair $(\rho_s, \eta_s)$. However, this is precisely formula 
 (18) in   \cite{FractalAspectsofLevyTrees}.
\end{proof}

\section{Special Markov property }\label{Special and intermediate} 
In this section we state and prove  the (strong) special Markov property for the Lévy snake. This result was originally introduced in \cite[Section 2]{Laplacian_u_2} in the special case of the Brownian motion indexed by the Brownian tree, viz. when the Lévy exponent of the tree is of the form $\psi(\lambda)=\beta\lambda^{2}$ and the spatial motion $\xi$ is a Brownian motion. This result plays a fundamental role in the study of Brownian motion indexed by the Brownian tree, see for example \cite{Laplacian_u_2,Subor,Growth,Infinite_Spine}. 
More recently, a stronger version was proved in \cite{Subor} still for  $\psi(\lambda)=\beta\lambda^{2}$ but holding for more general spatial motions $\xi$. In this section we extend this result to   an arbitrary exponent $\psi$ of a Lévy tree. Even if  we follow a similar strategy  to the one introduced in \cite{Subor},  general Lévy trees are significantly less regular than the Brownian tree -- in particular the height process $H$ is not Markovian.   The arguments need to be carefully reworked and for instance, the existence of points with infinite multiplicity hinder considerably the proof. 
\\
\\
We start by introducing some standard notation that will be used in the rest of the section and recalling the preliminaries needed for our purpose. Fix $x \in E$ and for an arbitrary open subset $D \subset E$ containing $x$ and $\text{w}\in \mathcal{W}_{E,x}$, set 
\[\:\tau_{D}(\text{w}):=\inf\big\{t\in[0,\zeta_{\text{w}}] : ~ \text{w}(t)\notin D\big\},\] 
with the usual convention  $\inf \{ \emptyset \}=\infty$. Similarly, we will write $\tau_{D}(\xi):=\inf\{t\geq 0: ~ \xi_{t}\notin D\}$ for the exit time from $D$ of the spatial motion $\xi$. When considering the later, the dependency on $\xi$ is usually dropped when there is no risk of confusion. In the rest of the section, we will always assume that:
\begin{equation}\label{tau_infinity}
 \Pi_{x}(\tau_{D}<\infty)>0. \tag*{$\mathbf{(H_{1})}$}
\end{equation}
The special Markov property is roughly speaking a spatial version of the Markov property. In order to state it, we need to properly define the notion of  paths "inside $D$" and "excursions outside $D$", as well as a notion of measurability with respect to the information generated by the trajectories staying inside of $D$. Section \ref{subsection:exitLocalTime} is devoted to the study of paths inside $D$ and to a fundamental functional of the Lévy snake, called the exit local time. The study of the excursions outside $D$ is postponed to Section  \ref{subsection:specialMarkovProof}.
\subsection{The exit local time}\label{subsection:exitLocalTime}
 Let us begin by introducing some useful operations and notation. \smallskip  \\ 
\noindent \textbf{Truncation.} We start by defining the \textit{truncation} of a  path $(\upvarrho,\omega) \in \mathbb{D}(\mathbb{R}_+ , \mathcal{M}_f(\mathbb{R}_+) \times \mathcal{W}_{E,x} )$ to $D$ -- we stress that we have $\omega_s(0) = x$ for every $s \geq 0$. In this direction, define the functional 
\begin{equation}\label{definition:VD}
    V^D_t(\upvarrho , \omega) := \int_0^t \dd s \,  \mathbbm{1}_{\{  \zeta_{\omega_s} \leq \tau_D(\omega_s)    \}}, \quad \quad t \geq 0,
\end{equation}
measuring the amount of time spent by $\omega$ without leaving  ${D}$ up to time $t$. Let us be more precise: at time $s$, we will say that $\omega_s$ doesn't leave  $D$ (or stays in $D$) if 
$\omega_s ( [0,\zeta_s) ) \subset {D}$ (notice that $\widehat{\omega}_s$ might be in $\partial D$) 
 and on the other hand, we say that the trajectory exits $D$ if  $\omega_s([0,\zeta_s)) \cap D^c \neq \emptyset$. Observe that a trajectory $(\omega_s(t) : t \in [0,\zeta_s] )$ might exit the domain $D$ and return to it before the lifetime $\zeta_s$, but such a trajectory will not be accounted by $V^D$.  Write $\mathcal{Y}_{D}(\upvarrho , \omega):= V^D_{\sigma(\omega)} (\upvarrho, \omega )$ for the total amount of time spent in $D$, and for every $s\in [0, \mathcal{Y}_{D}(\upvarrho, \omega))$ set  
\[
\Gamma_s^{D}(\upvarrho, \omega):=\inf\big\{t\geq 0 :  V_t^D(\upvarrho , \omega) > s\big\},\:
\]
with the convention $\Gamma^D_s (\upvarrho,\omega):= \sigma(\omega)$, if $s \geq \mathcal{Y}_D(\upvarrho,\omega)$.  The  truncation of $(\upvarrho,\omega)$   to $D$ is the element of $\mathbb{D}(\mathbb{R}_+ , \mathcal{M}_f(\mathbb{R}_+) \times \mathcal{W}_{E,x} )$  with lifetime $\mathcal{Y}_{D}(\upvarrho, \omega)$ defined as follows:
\begin{equation*}
\text{tr}_{D}\big(\upvarrho ,\omega\big) :=(\upvarrho_{\Gamma_s^{D}(\upvarrho, \omega)},\omega_{\Gamma_s^{D}(\upvarrho , \omega)})_{s\in \mathbb{R}_+}.
\end{equation*}
 Indeed, observe that the trajectory $(\upvarrho_{\Gamma^D},\omega_{\Gamma^D})$ is càdlàg since $\upvarrho, \omega$  and $\Gamma^{D}$ are càdlàg. For simplicity, we set $\text{tr}_{D} (\omega ) = 
(\omega_{\Gamma_s^{D}(\omega)})_{s\in \mathbb{R}_+ }$ and we write  $ \text{tr}_D (\widehat{\omega})$ for    $ \widehat{\omega}_{\Gamma^D}$.  Roughly speaking, $\text{tr}_{D}(\omega)$ removes the trajectories $\omega_s$ from $\omega$ leaving $D$, glues the remaining endpoints,  and hence encodes the trajectories  $\omega_s$ that stay in $D$. Let us stress that when $(\upvarrho, \omega)$ is an element of $\mathcal{S}_x$, the truncation $\text{tr}_D(\upvarrho , \omega)$ is still in $\mathcal{S}_x$  since   $\text{tr}_{D} (\omega)$ is  a snake trajectory   taking values in $D\cup \partial D$ by   \cite[Proposition 10]{ALG}, and  condition (ii) in Definition \ref{definition:snakePath} is clearly satisfied. Recall that $(\rho, W)$ stands for the canonical process in $\mathbb{D}(\mathbb{R}_+, \mathcal{M}_f(\mathbb{R}_+)\times \mathcal{W}_{E,x} )$, and that it takes values in $\mathcal{S}_x$ under $\mathbb{P}_{\mu , \w}$ for $(\mu , \w) \in \Theta$ or under $\mathbb{N}_{y}$ for $y \in E$. We will also need to introduce the sigma field 
\begin{equation} \label{definition:sigmaFieldTroncature}
    \mathcal{F}^D := \sigma \big( \text{tr}_D( \rho, W)_s : s \geq 0 \big)
\end{equation}
in $\mathbb{D}(\mathbb{R}_+ , \mathcal{M}_f(\mathbb{R}_+) \times \mathcal{W}_E )$, which roughly speaking, contains the information generated by the trajectories that stay in $D$. The following technical lemma will be often useful. It states that, under $\mathbb{N}_x$,  when a trajectory $W_s$ exits the domain $D$, then the measure $\rho_s$ does not have an atom at  level $\tau_D(W_s)$. More precisely:

\begin{lem}\label{lemma:nomassaTD} Let $D$ be an arbitrary  open subset $D\subset E$ containing $x$. Then, $\mathbb{N}_{x}$--a.e.
\begin{equation*}
     \rho_s(\{ \tau_D(W_s)\}) = 0 ,\quad   \text{ for all } s \geq 0.
\end{equation*} 
\end{lem}
\begin{proof}
First, remark that  the many-to-one formula \eqref{tirage_au_hasard_N} gives:
\begin{equation*}
    \mathbb{N}_{x}\Big( \int_0^\sigma \dd s \, \mathbbm{1}_{\{ \tau_D(W_s) < \infty \}} \rho_s(\{\tau_D(W_s)\})  \Big) 
    =
    \int_0^\infty \dd a \, \exp(-\alpha a) E^0\otimes \Pi_x \Big(  \mathbbm{1}_{\{ \tau_D((\xi_u: u \leq a)) < \infty \}}  J_a(\{\tau_D(\xi_u : u \leq a)\})  \Big),
\end{equation*}
which vanishes by the independence between $\xi$ and $J_a$. This shows that $\mathbb{N}_x$--a.e., the Lebesgue measure of the set $\{ s\in [0,\sigma] :  \rho_s(\{\tau_D(W_s)\}) \neq 0 \}$ is null and now we claim  that this implies that $\mathbb{N}_{x}$--a.e. $\rho_s(\{ \tau_D(W_s)\}) = 0$ for all $s \geq 0$.  We argue by contradiction to prove this claim. Suppose that for some $s>0$, we have $\rho_s(\{ \tau_D(W_s)\}) > 0$. In this case, recalling that the exploration process $\rho$ is càdlàg with respect to the total variation distance,  we must have 
\begin{equation*}
\lim_{\epsilon \downarrow 0}\big|\rho_s(\{ \tau_D(W_s)\}) -  \rho_{s+\epsilon}(\{ \tau_D(W_s)\})\big|
\leq  \lim_{\epsilon \downarrow 0} \sup_{A \in \mathcal{B}(\mathbb{R}) } \big|\rho_s(A) - \rho_{s+\epsilon}(A)\big| = 0. 
\end{equation*}
We infer that for some $\delta > 0$,  it holds that   $\rho_u(\{ \tau_D(W_s)\})> 0$ for all  $u \in [s,s+\delta)$. In particular, we have  $H_u \geq H_s$ for all $u \in [s,s+\delta)$. By the snake property, we deduce that, for every $u \in [s,s+\delta)$,  $\tau_D(W_s) = \tau_D(W_u)$ and consequently:
 $$\rho_{u}(\{ \tau_D(W_{u})\})=\rho_u(\{ \tau_D(W_s)\})> 0.$$
However, this is in contradiction with the first part of the proof and the desired result follows.  
\end{proof}
\noindent \textbf{Exit local time.} As in classical excursion theory, we will need to properly index the excursions outside $D$ but we will also ask the indexing to be compatible with the order induced by $H$. To achieve it, we will make use of the  \textit{exit local time} from $D$. We briefly recall its definition and main properties and we refer to  \cite[Section 4.3]{Duquesne} for a more detailed account. By   Propositions 4.3.1 and  4.3.2 in \cite{Duquesne}, under $\mathbb{N}_{x}$ and $\mathbb{P}_{0,x}$, the limit
\begin{equation} \label{definition:aproximationExitlocal}
    L_{s}^{D}:=\lim \limits_{\varepsilon\to 0} \frac{1}{\varepsilon} \int_{0}^{s} \dd r \mathbbm{1}_{\{\tau_{D}(W_{r})<H_{r}<\tau_{D}(W_{r})+\varepsilon \}},
\end{equation}
exists for every $s\geq 0$, where  the convergence holds uniformly in compact intervals in  $L_1(\mathbb{P}_{0,x})$ and $L_1(\mathbb{N}_{x})$.   This defines a continuous non-decreasing process $L^{D}$ called the exit local time from $D$ of $(\rho,W)$.  We insist that,  under $\mathbb{N}_x$ and $\mathbb{P}_{0,x}$, the process $(\rho, W)$ takes values in $\mathcal{S}_x$ which yields that  $H_s = \zeta_s$ for every $s \geq 0$.  We also recall the first moment formula: 
\begin{equation} \label{tirage_au_hasard_ExitN}
    \mathbb{N}_{x} \left( \int_0^\sigma \dd L_s^D~ \Phi(\rho_s , \eta_s,  W_s ) \right)  
    = E^0 \otimes \Pi_{x} \bigg( \mathbbm{1}_{\{ \tau_D < \infty \}} \exp(- \alpha \tau_D) \Phi\big( J_{\tau_D} , \widecheck{J}_{\tau_D} ,(\xi_t : t \leq \tau_D )\big)  \bigg),
\end{equation}
see \cite[Proposition 4.3.2]{Duquesne} for a proof of this identity.
In particular, remark that we have
$$\text{supp } \dd L_s^D \subseteq  \{ s\geq 0 :~ \tau_D(W_s) = H_s \}, \quad \mathbb{N}_{x} \text{--a.e.}$$ 
 We stress that  $L^D$ is constant at every interval at which $W_s$ stays in $D$ and in each connected component of 
$$ \{s\geq 0:~\tau_D(W_s)< H_s\}.$$
We call such a  connected component an excursion interval from $D$. This family of intervals will be studied in detail in the next section. The process $L^D$ is not measurable with respect to $\mathcal{F}^{D}$, the informal reason being that it contains the information on the lengths of the excursions from $D$. However, as we are going to show in  Proposition \ref{L_eta_measurable}, the time-changed process 
\begin{equation*}
    \widetilde{L}^D :=  \big(L_{\Gamma^D_{s}}^{D} \big)_{s\in \mathbb{R}_+}
\end{equation*}
is $\mathcal{F}^{D}$-measurable -- notice that we removed precisely from $L^D$ by means of the time change the constancy intervals generated by  excursions from  $D$. This measurability property will be crucial for  the  proof of the special Markov property and the rest of this section is devoted to its proof.\par 
First remark that we have only defined the exit local time  under the measures $\mathbb{P}_{0,x}$ and $\mathbb{N}_x$ for $x \in D$. In order to be able to apply the Markov property, we  need to  extend the definition to more general initial conditions $(\mu , \w) \in \Theta$. This construction will also be essential for the results of Section \ref{section:structureOfTloc}. The precise statement is given in the following proposition:   \par

\begin{prop} \label{proposition:aproxLDPmu}  Fix $(\mu,\rm{w})\in \Theta$ such that $\rm{w}(0)$ $\in D$ and suppose that  $\mu(\{\tau_D(\mathrm{w})\})=0$.  Then, under $\mathbb{P}_{\mu , \rm{w}}$ there exists a continuous, non-decreasing process $L^D$ with associated Lebesgue-Stieltjes measure $\dd L^D$ supported on $\{ t \in \mathbb{R}_+ : \widehat{W}_t \in \partial D \}$, such that, for every $t \geq 0$ 
\begin{equation} \label{definition:ExitPMu}
 L^D_t= \lim \limits_{\epsilon\to 0}  \frac{1}{\epsilon} \int_0^t \dd s \,  \mathbbm{1}_{\{ \tau_D(W_s) < H_s < \tau_D(W_s) + \epsilon \}}, 
\end{equation}
where the convergence holds uniformly in compact intervals in  $L^{1}(\mathbb{P}_{\mu,\rm{w}})$. Moreover: 
\begin{enumerate}
    \item[\rm{(i)}] Under $\mathbb{P}_{\mu , \emph{w}}$, if $\tau_D(\emph{w}) < \infty$, we have $L^D_t = 0$ for every $t \leq \inf \{s \geq 0 : H_s < \tau_D(\emph{w}) \}$.
    \item[\rm{(ii)}]  Under $\mathbb{P}^{\dag}_{\mu,\emph{w}}$,  recall the definition of the random point measure $\sum_{i \in \mathbb{N}}\delta_{(h_i, \rho^i , W^i)}$ defined in \eqref{PoissonRandMeasure}.  Then we have:
\begin{equation} \label{equation:ExitPmuExcursionesDescomp}
    L^D_\infty(\rho,W) = \sum \limits_{h_i<\tau_D(\emph{w})}L^D_\infty(\rho^i,W^i), \quad \mathbb{P}^{\dag}_{\mu,\emph{w}}\text{--a.s.} 
\end{equation}
\end{enumerate}
\end{prop}
\begin{proof} Let us start with preliminary remarks and introducing some needed notation.  Fix $(\mu,\w)\in \Theta$ with $\w(0) \in D$ satisfying $\mu(\{\tau_D(\w)\}) = 0$. We write
$$T_r:=\inf\{t\geq 0:~H_t=r\},\text{ for every }r \geq 0, ~~~~ \text{ and } ~~~~T_0^{+}:=\inf\{t\geq 0:~\langle \rho_t, 1\rangle=0\}.$$ By \eqref{definition:aproximationExitlocal} and the strong Markov property, we already know that $\epsilon^{-1} \int_{T_0^+}^{T_0^++t} \dd s \,  \mathbbm{1}_{\{ \tau_D(W_s) < H_s < \tau_D(W_s) + \epsilon \}}$  converges as $\epsilon \downarrow 0$ uniformly in compact intervals in  $L^{1}(\mathbb{P}_{\mu,\text{w}})$ towards a non-decreasing continuous process supported on $\{t\geq T_0^+ : \widehat{W}_t \in \partial D\}$. Consequently, it suffices to prove the proposition under $\mathbb{P}_{\mu,\w}^{\dag}$. In this direction,   we set 
\begin{align*}
I(t,\epsilon) :=  \frac{1}{\epsilon} \int_0^t \dd s \, \mathbbm{1}_{\{ \tau_D(W_s) < H_s < \tau_D(W_s) + \epsilon  \}}, 
\end{align*}
 for every $\epsilon > 0$. Recall now that under $\mathbb{P}^\dag_{\mu , \w}$, the process $\langle\rho,1\rangle$ is a killed Lévy process started at $\langle \mu,1\rangle$ and stopped at its first hitting time of $0$. Write  $((\alpha_i, \beta_i):~i \in \mathbb{N})$, for the excursion intervals of  $\langle \rho,1 \rangle$ over its running infimum,  and  let  $(\rho^i, W^i)$ be the subtrajectory associated with the excursion interval $[\alpha_i , \beta_i]$. To simplify notation, we also set $h_i:=H(\alpha_i)$ and recall from \eqref{PoissonRandMeasure}  that the measure $\mathcal{M} := \sum_{i \in \mathbb{N}}\delta_{(h_i, \rho^i , W^i)}$ is a Poisson point measure with intensity $\mu(\dd h)\mathbb{N}_{\w(h)}(\dd \rho, \dd W)$.\par 
We suppose first that $\tau_D(\w) \geq \zeta_\w$.  We shall prove that the collection $\big(I(t,\epsilon), t \geq 0\big)$ for $\epsilon > 0$ is Cauchy  in $L_1(\mathbb{P}^\dag_{\mu , \w})$ uniformly in compact intervals as $\epsilon \downarrow 0$, viz. 
\begin{equation}\label{equation:cauchyLd}
    \lim_{\delta, \epsilon \rightarrow 0}\mathbb{E}_{\mu , \w}^{\dag} \big[ \sup_{s \leq t} |I(s,\epsilon) - I(s,\delta)| \big] = 0.
\end{equation}
This implies directly  the existence of $L^D$ defined as in  \eqref{definition:ExitPMu} as well as point (i). We  shall then deduce (ii), and the remaining case $\tau_D(\w) < \zeta_\w$ is treated afterwards. Let us proceed with the proof of \eqref{equation:cauchyLd}. Since the Lebesgue measure of $\{ t \in [0,\sigma] : \langle \rho_t,1 \rangle = \inf_{s \leq t} \langle \rho_s , 1 \rangle  \}$ is null,  we can write  
\begin{align*}
I(t,\epsilon) = \frac{1}{\epsilon} \sum_{i \in \mathbb{N}}   \int_{\alpha_i \wedge t  }^{\beta_i \wedge t} \dd s \, \mathbbm{1}_{\{ \tau_D(W_s) < H_s < \tau_D(W_s) + \epsilon  \}}, 
\end{align*}
which yields the following upper bound:
\begin{align*}
    &\mathbb{E}_{\mu , \w}^{\dag} \big[ \sup_{s \leq t} |I(s,\epsilon) - I(s,\delta)| \big] \\
    &\quad \leq 
    \mathbb{E}_{\mu , \w}^{\dag} \Big[ \sum_{i \in \mathbb{N}} \sup_{s \leq t  } \big|\frac{1}{\epsilon} \int_{\alpha_i \wedge s}^{\beta_i \wedge s} \dd u \, \mathbbm{1}_{\{ \tau_D(W_u) < H_u < \tau_D(W_u) + \epsilon  \}} -    \frac{1}{\delta} \int_{\alpha_i \wedge s}^{\beta_i \wedge s} \dd u \, \mathbbm{1}_{\{ \tau_D(W_u) < H_u < \tau_D(W_u) + \delta  \}} \big| \Big] \\
    & \quad \leq 
    \mathbb{E}_{\mu , \w}^{\dag} \Big[ \sum_{i \in \mathbb{N}} \sup_{s \leq \sigma(W^i)  } \big|\frac{1}{\epsilon} \int_0^{s \wedge t} \dd u \, \mathbbm{1}_{\{ \tau_D(W^i_u) < H(\rho_u^i) < \tau_D(W^i_u) + \epsilon  \}} -    \frac{1}{\delta} \int_0^{s \wedge t} \dd u \, \mathbbm{1}_{\{ \tau_D(W^i_u) < H(\rho_u^i) < \tau_D(W^i_u) + \delta  \}} \big| \Big]. 
\end{align*} 
 Since $\mu(\{ \tau_D(\w) \}) = 0$, the last display  is given by 
\begin{equation} \label{equation:aproxPmuLdeqe}
   \int_{[0,\tau_D(\w))} \mu(\dd h) \,  \mathbb{N}_{\w(h)}\Big(  \sup_{s \leq t}  | I(s,\epsilon)  - I(s,\delta) | \Big). 
\end{equation}
Let us now  show that \eqref{equation:aproxPmuLdeqe}  converges towards $0$ when  $\epsilon, \delta \downarrow 0$. Since for every $h \in [0,\tau_D(\w))$  we have  $\w(h) \in D$,  the term inside the integral in \eqref{equation:aproxPmuLdeqe} converges towards $0$ as $\epsilon, \delta  \downarrow 0$  by the approximation of exit local times   under the excursion measure given in \eqref{definition:aproximationExitlocal}. Knowing that $\mu$ is a finite measure,  it suffices to show that the term,  
\begin{equation*}
  \mathbb{N}_{\w(h)}\Big(  \sup_{s \leq t}  | I(s,\epsilon)  - I(s,\delta) | \Big),  
\end{equation*}
can be  bounded uniformly in $\epsilon, \delta$. However, still under $\mathbb{N}_{\w(h)}$, we have the simple upper bound: 
\begin{equation*}
     \sup_{s \leq t}  | I(s,\epsilon)  - I(s,\delta) | 
     \leq    I(\sigma,\epsilon)  + I(\sigma,\delta),
\end{equation*}
and by the many-to-one formula \eqref{tirage_au_hasard_N}, we deduce that 
\begin{align*} 
    \mathbb{N}_{\w(h)}\big(  I(\sigma,\epsilon)  \big) 
    =\epsilon^{-1} E^0 \otimes \Pi_{\w(h)}\Big[  \int_0^\infty \dd a \, \exp(-\alpha a) \mathbbm{1}_{\{ \tau_D(\xi) < H(J_a) < \tau_D(\xi) + \epsilon \}}  \Big] \leq 1,
\end{align*}
for every $\epsilon>0$, where to obtain the previous inequality we use that $H(J_a)=a$.
In particular, we have   $\mathbb{N}_{\w(h)}\big(  I(\sigma,\epsilon)  + I(\sigma,\delta)  \big) \leq 2$ and  \eqref{equation:cauchyLd} follows.  Still under our assumption  $\tau_D(\w) \geq \zeta_\w$ we now turn our attention to  \eqref{equation:ExitPmuExcursionesDescomp}. We know that for any $(h_i,W^i, \rho^i) \in \mathcal{M}$ we have the limit in probability:
\begin{equation*}
    L^{D}_{\sigma_i} ( \rho^i, W^i) 
    = \lim_{\epsilon \rightarrow 0} \epsilon^{-1} \int_{a_i}^{b_i} \dd s \mathbbm{1}_{\{\tau_D(W_s) < H_s < \tau_D(W_s)+\epsilon\}}.  
\end{equation*}
It then follows from our definitions that for every $r> 0$,
\begin{equation*}
    L_\sigma^{D} - L^{D}_{T_{\zeta_{\w}-r}} = \sum_{h_i \leq \zeta_\w-r}L^{D}_{\sigma_i}(\rho^i, W^i), 
\end{equation*}
observing that the number of non-zero terms on the right-hand side is finite.
By taking the limit as $r \downarrow 0,$ we deduce \eqref{equation:ExitPmuExcursionesDescomp} by monotonicity.
\par 
Let us now assume that $\tau_D(\w) < \zeta_\w$. To simplify notation, set $a:=\tau_D(\w)$  and notice that
\begin{equation*}
    (\rho_{T_a}, W_{T_a}) = \big(\mu \mathbbm{1}_{[0,\tau_D(\w)]} , (\w(h) : h \in [0,\tau_D(\w)])\big), 
\end{equation*}
where we recall that $\mu(\{ \tau_D(\w) \}) = 0$.  By our previous discussion and the strong Markov property, we deduce that $(I(t,\epsilon) -I(T_a,\epsilon) :t \geq T_a)$ converges as $\epsilon \downarrow 0$ uniformly in compact intervals in $L^{1}(\mathbb{P}_{\mu,\w})$ towards a continuous process. To conclude our proof, it suffices to show that: 
\begin{equation*}
    \lim_{\epsilon \rightarrow 0} 
    \frac{1}{\epsilon} \mathbb{E}_{\mu , \w}^{\dag} \Big[ \int_0^{T_a} \dd s \, \mathbbm{1}_{\{ \tau_D(W_s) < H_s < \tau_D(W_s) + \epsilon  \}} \Big] 
    = 0.
\end{equation*}
 To obtain the previous display,  write 
$$\int_{0}^{T_{a}} \dd s \,  \mathbbm{1}_{\{ \tau_D(W_s) < H_s < \tau_D(W_s) + \epsilon \}}=\sum \limits_{ h_i\geq     a}\int_{\alpha_i}^{\beta_i}\dd s \,  \mathbbm{1}_{\{ \tau_D(W_s) < H_s < \tau_D(W_s) + \epsilon \}},$$
where we have $h_i\neq a$ for every $i\in \mathbb{N}$, since $\mu(\{a\})=0$. Moreover, for every $i$ with $h_i > a$ notice that $\tau_{D}(W_s)=a$. This implies:
\begin{align*}
    \int_{0}^{T_{a}} \dd s \,  \mathbbm{1}_{\{ \tau_D(W_s) < H_s < \tau_D(W_s) + \epsilon \}}\leq \sum \limits_{a\leq h_i\leq a+\epsilon}\int_{0}^{\sigma(W^i)} \dd s \,  \mathbbm{1}_{\{ 0 < H(\rho_s^i) <  \epsilon \}}, 
\end{align*}
and we can now use that $\mathcal{M}$ is a  Poisson point measure with intensity $\mu(\dd h)\mathbb{N}_{\w(h)}(\dd \rho, \dd W)$ to obtain:
\begin{align}\label{equation:aproxLDPmu_eq11}
 \mathbb{E}^{\dag}_{\mu,\w}\big[\int_{0}^{T_{a}} \dd s \,  \mathbbm{1}_{\{ \tau_D(W_s) < H_s < \tau_D(W_s) + \epsilon \}}  \big] 
 \leq \mu([a,a+\epsilon])N(\int_{0}^{\sigma}\dd s \mathbbm{1}_{\{0\leq H(\rho_s)<\epsilon\}}).
\end{align}
Finally, by the many-to-one formula \eqref{tirage_au_hasard_N}, the previous display is equal to 
$\epsilon\cdot\mu([a,a+\epsilon])$, giving:
\begin{equation*}
    \limsup \limits_{\epsilon\to 0}\frac{1}{\epsilon}\mathbb{E}^{\dag}_{\mu,\w}\Big[ \int_{0}^{T_{a}} \dd s \,  \mathbbm{1}_{\{ \tau_D(W_s) < H_s < \tau_D(W_s) + \epsilon \}}\Big]=\mu(\{a\})=0,
\end{equation*}
where in the last equality we use that $\mu\in \Theta$ which ensures that $\mu(\{a\} )=0$.
\end{proof}
Now that  we have defined the exit local time under more general initial conditions, let us turn our attention to the measurabliliy properties of $\widetilde{L}^D$. From now on, when working under $\mathbb{P}_{0,x}$ or $\mathbb{N}_{x}$, the sigma field $\mathcal{F}^D$ should be completed  with the $\mathbb{P}_{0,x}$-negligible and   $\mathbb{N}_{x}$-negligible sets respectively -- for simplicity we use the same notation.
\begin{prop}\label{L_eta_measurable}
Under $\mathbb{P}_{0,x}$ and $\mathbb{N}_x$, the process $\widetilde{L}^D$ is $\mathcal{F}^{D}$-measurable. 
\end{prop}
\noindent In particular, the proposition implies that, under $\mathbb{N}_x$, the total mass $L^D_\sigma = \widetilde{L}^D_\infty$ is $\mathcal{F}^D$-measurable. The proof will mainly rely on the  two following technical lemmas. 
\begin{lem} \label{lemma:techincalLemma0} Consider an open subset $D \subset E$ containing $x$.   Fix an arbitrary   $(\mu , \emph{w})\in \Theta$ with $\emph{w}(0) = x$ and satisfying $\mu( \{ \tau_D(\emph{w}) \}) = 0$ if  $\tau_D(\emph{w})<\infty$. Then, for every $K > 0$,  we have:
\begin{align*}
      \mathbb{E}_{\mu, \emph{w}}^{\dag}\Big[ \int_0^\sigma \dd L_s^D \mathbbm{1}_{\{ \langle \rho_s, 1 \rangle \leq K \}} \Big] = 
     \int_0^{ \mu\left([0,\tau_D(\emph{w}))\right) } \dd u ~E^{0} \otimes \Pi_{\emph{w}( H(\kappa_{\langle \mu , 1 \rangle -u} \mu ))} \left( \mathbbm{1}_{\{ \tau_D < \infty \}} \exp(-  \alpha \tau_D) \mathbbm{1}_{\{   \langle J_{\tau_D} , 1 \rangle  \leq K -u \}} \right).
\end{align*}
\end{lem}
\begin{proof} 
Recall that, under $\mathbb{P}^\dag_{\mu , \w}$, the process $\langle\rho,1\rangle$ is a Lévy process started at $\langle \mu , 1 \rangle$ and stopped at its first hitting time of $0$. As usual, write $\{(\alpha_i, \beta_i): \, i \in \mathbb{N}\}$ for the excursion intervals  of $\langle\rho,1\rangle-\langle \mu,1\rangle$ over its  running infimum, that we still denote by $I$. We write $(\rho^i,W^i)$ for the subtrajectory associated with $[\alpha_i,\beta_i]$. As explained in \eqref{definition:poissonExcursiones}, the measure:
$$\sum_{i \in \mathbb{N}}\delta_{(-I_{\alpha_i}, \rho^i , W^i)}, $$ 
is a Poisson point measure with intensity 
$\mathbbm{1}_{[0,\langle \mu , 1 \rangle]}(u) \dd u \, \mathbb{N}_{\text{w}(H( \kappa_{u}\mu ))}(\dd  \rho, \dd W).$ Furthermore, for every $i\in \mathbb{N}$, we have  $H(\kappa_{-I_{\alpha_i}} \mu ) =H_{\alpha_i} = H_{\beta_i}$ and to simplify notation we denote this quantity by $h_i$. Next, we  notice that, by Proposition \ref{proposition:aproxLDPmu}, we have   $\int_0^\sigma \dd L_s^D \mathbbm{1}_{\{ \langle \rho_s,1\rangle -\langle \mu,1 \rangle = I_s \}} = 0$ and  $L^D_{t} = 0$, for every  $t\leq  \inf \{s \geq 0 : H_s < \tau_D(\w) \}$.  From our previous observations, we get:
\begin{align*}
\int_0^\sigma \dd L_s^D \mathbbm{1}_{\{ \langle \rho_s, 1 \rangle \leq K \}}=\sum\limits_{h_i<\tau_{D}(\w)} \int_{\alpha_i}^{\beta_i} \dd L_s^D \mathbbm{1}_{\{ \langle \rho_s, 1 \rangle \leq K \}}   =\sum\limits_{H(\kappa_{-I_{\alpha_i}}\mu)<\tau_{D}(\w)} \int_{0}^{\beta_i-\alpha_i} \dd L_s^D(\rho^i, W^i) \mathbbm{1}_{\{ \langle \rho_s^i, 1 \rangle \leq K- \langle \mu,1 \rangle - I_{\alpha_i} \}},
\end{align*}
where  we used in the second identity that $\langle \rho_{s+\alpha_i}, 1 \rangle =\langle \rho_s^i, 1 \rangle +\langle \rho_{\alpha_i}, 1 \rangle=\langle \rho_s^i, 1 \rangle+I_{\alpha_i} +\langle \mu,1 \rangle$, for every $s\in [0,\beta_i-\alpha_i]$. This implies that:
\begin{align*} 
 & \mathbb{E}_{\mu, \text{w}}^{\dag}\Big[\sum\limits_{H(\kappa_{-I_{\alpha_i}}\mu)<\tau_{D}(\w)} \int_{0}^{\beta_i-\alpha_i} \dd L_s^D(\rho^i, W^i) \mathbbm{1}_{\{ \langle \rho_s^i, 1 \rangle \leq K- \langle  \mu,1 \rangle -I_{\alpha_i} \}}\Big]  \\
 & \hspace{55mm}=\int_{ \mu([\tau_{D}(\w),\infty))}^{\langle \mu,1\rangle} \dd u \,    \mathbb{N}_{\text{w}( H(\kappa_{u} \mu ))} \Big( \int_0^\sigma \dd L_s^D \mathbbm{1}_{\{  \langle \rho_s , 1 \rangle \leq K - \langle  \mu,1 \rangle +u\}} \Big),  
\end{align*}
and the desired result now follows by performing the change of variable $u  \text{ \reflectbox{$ \longmapsto $}} \,  \langle \mu,1 \rangle  -u$ and applying the many-to-one formula \eqref{tirage_au_hasard_ExitN}. 
\end{proof}

\begin{lem} \label{lemma:technicalLema2} Consider an increasing sequence of open subsets $(D_n:~n\geq 1)$ containing $x$, such that $\cup_n D_n = D$ and  $\overline{D_n} \subset D$. There exists a subsequence $(n_k:~k\geq 0)$ converging towards infinity, such that
\begin{align} \label{equation:uniformConvExtis}
\lim \limits_{k\to \infty}\sup\limits_{s\in[0,\sigma]}|L_{s}^{D_{n_k}}-L_{s}^{D}|=0,\:\: \quad \mathbb{N}_x \text{--a.e. } 
\end{align}
\end{lem}
\begin{proof}
The  proof of this lemma will be achieved by using similar techniques as in  \cite[Proposition 2.3]{Laplacian_u_2} in the Brownian setting. We start  by showing  that, for a suitable subsequence, the total mass $L_\sigma^{D_n}$ converges towards $L_\sigma^{D}$, $\mathbb{N}_x$-a.e. The uniform convergence will then be deduced by standard techniques. Notice however that in   \cite{Laplacian_u_2},  this  is mainly done by establishing an $L_2(\mathbb{N}_x)$ convergence of $L^{D_n}_\sigma$ towards $L^{D}_\sigma$, and that  we do not have a priori  moments of order $2$ in our setting. In order to overcome this difficulty, we need to localize the tree  by the use of a truncation argument. We start by showing  that,  for any fixed $K>0$, we have:
\begin{equation} \label{equation:L2convergenceExitLocal}
    \lim_{n \rightarrow \infty} \int_0^\sigma \dd L_s^{D_n} \mathbbm{1}_{\{\langle \rho_s,1 \rangle\leq K\}} = \int_0^\sigma \dd L_s^{D} \mathbbm{1}_{\{\langle \rho_s,1 \rangle \leq K\}}, \quad \text{ in } L_2(\mathbb{N}_x).
\end{equation}
In this direction, we write $\mathbb{N}_x\Big(  \big|\int_0^\sigma \dd L_s^D \mathbbm{1}_{\{ \langle \rho_s,1 \rangle \leq K \}} 
    - \int_0^\sigma \dd L_s^{D_n} \mathbbm{1}_{\{ \langle \rho_s,1 \rangle \leq K \}} \big| ^2  \Big)$ in the following form
\begin{align}\label{equation:FdmedibilidadCuentass}  \mathbb{N}_x\Big( \big(\int_0^\sigma \dd L_s^D \mathbbm{1}_{\{ \langle \rho_s,1 \rangle \leq K \}} \big)^2  \Big)\nonumber
    &+ \mathbb{N}_x\Big(\big( \int_0^\sigma \dd L_s^{D_n} \mathbbm{1}_{\{ \langle \rho_s,1 \rangle \leq K \}} \big)^2\Big)\\&
    - 2 \mathbb{N}_x \Big(  \big( \int_0^\sigma \dd L_s^{D_n} \mathbbm{1}_{\{ \langle \rho_s,1 \rangle \leq K \}} \big)\cdot 
    \big( \int_0^\sigma \dd L_s^{D} \mathbbm{1}_{\{ \langle \rho_s,1 \rangle \leq K \}} \big)   \Big),
\end{align}
and the proof of \eqref{equation:L2convergenceExitLocal} will follow by computing  each term separately and by taking the limit as $n \uparrow \infty$. First, we  remark that
\begin{align*}
\big( \int_0^\sigma \dd L_s^D \mathbbm{1}_{\{ \langle \rho_s,1 \rangle \leq K \}} \big)^2 
&= 2 \int_0^\sigma \dd L_s^{D} \mathbbm{1}_{\{ \langle \rho_s,1 \rangle \leq K \}}  \int_s^\sigma \dd L_u^D~ \mathbbm{1}_{\{ \langle \rho_u,1 \rangle \leq K\}},
\end{align*}
and the idea now is to apply the Markov property. For convenience, we let  $\Theta_D$ be the subset of $\Theta$ of all the pairs $(\mu , \w)$ satisfying the condition  $\mu(\{\tau_D(\w)\}) = 0$ when $\tau_D(\w) < \infty$, and we define $\Theta_{D_n}$ similarly replacing $D$ by $D_n$. Notice that by Lemma  \ref{lemma:nomassaTD}, we have, $\mathbb{N}_{x}$--a.e.,  $(\rho_{t},W_t)\in \Theta_{D}\cap(\cap_{n\geq 1} \Theta_{D_n})$ for every $t\geq 0$. For $(\mu , \w)\in \Theta_D$,  we set
\begin{align*}
    \phi_{D}(\mu, \text{w}) 
    &:= \mathbb{E}_{\mu, \text{w}}^{\dag}\Big[ \int_0^\sigma \dd L_s^D \mathbbm{1}_{\{ \langle \rho_s, 1 \rangle \leq K \}} \Big] \\
     &=    \int_0^{ \mu\left([0,\tau_D(\w))\right) } \dd u ~E^{0} \otimes \Pi_{\w( H(\kappa_{\langle \mu,1\rangle-u} \mu ))} \left( \mathbbm{1}_{\{ \tau_D < \infty \}} \exp(-  \alpha \tau_D) \mathbbm{1}_{\{   \langle J_{\tau_D} , 1 \rangle  \leq K -u \}} \right),  
\end{align*}
 where in the second equality we used  Lemma \ref{lemma:techincalLemma0}. Note that the dependence of $\phi_D$ on $K$  is being omitted to simplify the notation. By our previous discussion,   an application of the Markov property  gives:  
\begin{align}\label{eq:tau_d:lim}
    \mathbb{N}_x \left( \left(\int_0^\sigma \dd L_s^D \mathbbm{1}_{\{ \langle \rho_s ,1 \rangle  \leq K \}} \right)^2   \right) 
    &= 2 \mathbb{N}_x \left(  \int_0^\sigma \dd L_s^{D} \mathbbm{1}_{\{  \langle \rho_s, 1 \rangle \leq K \}}  \phi_{D}(\rho_s,W_s)  \right)\nonumber \\
    &=  2 E^{0} \otimes \Pi_{x} \left( \mathbbm{1}_{\{ \tau_{D} < \infty \}} \exp(-\alpha \tau_{D}) \mathbbm{1}_{\{ \langle J_{\tau_{D}}, 1 \rangle \leq K \}} \phi_{{D}}(J_{\tau_{D}}, \xi^{\tau_{D}})  \right),   
\end{align}
where to simplify notation, we write 
$\xi^{\tau_{D}} := (\xi_{t}: 0 \leq  t \leq \tau_{D})$. Observe that $(J_{\tau_{D}}, \xi^{\tau_D}) \in \Theta_D$  since by independence, we have $\mathbbm{1}_{\{ \tau_D < \infty \}}J_{\tau_D}(\{ \tau_D \}) = 0$, $P^0\otimes\Pi_x$--a.s. 
Replacing $D$ by $D_n$, we also have $(J_{\tau_{D_n}}, \xi^{\tau_{D_n}}) \in \Theta_{D_n}$ and  we obtain   
\begin{align}  \label{equation:aproxLD_eq}
     \mathbb{N}_x \Big( \big(\int_0^\sigma \dd L_s^{D_n} \mathbbm{1}_{\{ \langle \rho_s ,1 \rangle  \leq K \}} \big)^2   \Big) 
     = 2 E^{0} \otimes \Pi_{x} \left( \mathbbm{1}_{\{ \tau_{D_n} < \infty \}} \exp(-\alpha \tau_{D_n}) \mathbbm{1}_{\{ \langle J_{\tau_{D_n}}, 1 \rangle \leq K \}} \phi_{{D_n}}(J_{\tau_{D_n}}, \xi^{\tau_{D_n}})  \right), 
\end{align}
where for $(\mu , \w) \in \Theta_{D_n}$, we write 
\begin{equation*}
    \phi_{D_n}(\mu, \text{w})  
    =   \int_0^{ \mu\left([0,\tau_{D_n}(\w))\right) } \dd u ~E^{0} \otimes \Pi_{\w( H(\kappa_{\langle \mu,1\rangle -u} \mu ))} \left( \mathbbm{1}_{\{ \tau_{D_n} < \infty \}} \exp(-  \alpha \tau_{D_n}) \mathbbm{1}_{\{   \langle J_{\tau_{D_n}} , 1 \rangle  \leq K -u \}} \right).  
\end{equation*}
Our goal now is to take the limit in  \eqref{equation:aproxLD_eq} as $n \uparrow \infty$ and to show that this limit is precisely \eqref{eq:tau_d:lim}. In this direction, we remark that under $\{ \langle J_{\tau_{D_n}}, 1 \rangle \leq K \}$, we have the trivial bound   $\phi_{D_n}(J_{\tau_{D_n}} , \xi^{\tau_{D_n}})\leq K$. Thanks to the dominated convergence theorem, it is then enough to show that, $P^0\otimes \Pi_x$-a.s.,  the following convergence holds:
$$ \lim \limits_{n\to\infty}\mathbbm{1}_{\{ \tau_{D_n} < \infty \}} \exp(-  \alpha \tau_{D_n}) \mathbbm{1}_{\{   \langle J_{\tau_{D_n}} , 1 \rangle  \leq K \}}\phi_{{D_n}}(J_{\tau_{D_n}}, \xi^{\tau_{D_n}})=\mathbbm{1}_{\{ \tau_{D} < \infty \}} \exp(-  \alpha \tau_{D}) \mathbbm{1}_{\{   \langle J_{\tau_{D}} , 1 \rangle  \leq K  \}}\phi_{{D}}(J_{\tau_{D}}, \xi^{\tau_{D}}).$$
In order to prove it, we start noticing that we always have $\tau_{D_n}\uparrow\tau_{D}$ as $n\to\infty$.
In particular, since $\langle J_{\infty}, 1 \rangle=\infty$, we see that the limit in the previous display is $0$ under   $\{\tau_D=\infty\}$. Let us focus now on the event $\{\tau_D<\infty\}$. First remark that
 \begin{equation*}
    \kappa_{\langle J_{\tau_{D_n}},1\rangle -u }J_{\tau_{D_n}} = \kappa_{\langle J_{\tau_{D}},1\rangle - u }J_{\tau_{D}},
\end{equation*}
for every $u \leq \langle J\tau_{D_n} , 1 \rangle$. This combined with the independence between $J$ and $\xi$ ensures that,  under $\{\tau_D<\infty\}$, the quantities $\langle J\tau_{D_n} , 1 \rangle$ and $\phi_{D_n}(J_{\tau_{D_n}},  \xi^{\tau_{D_n}})$  convergence 
respectively to $\langle J\tau_{D} , 1 \rangle$ and $\phi_{D}(J_{\tau_{D}},  \xi^{\tau_{D}})$, giving the desired convergence under $\{\tau_D<\infty\}$. Consequently, we  get:
\begin{equation*}
    \lim_{n \rightarrow \infty}  \mathbb{N}_x \Big( \big(\int_0^\sigma \dd L_s^{D_n} \mathbbm{1}_{\{ \langle \rho_s ,1 \rangle  \leq K \}} \big)^2   \Big)  
    = 
     \mathbb{N}_x \Big( \big(\int_0^\sigma \dd L_s^D \mathbbm{1}_{\{ \langle \rho_s ,1 \rangle  \leq K \}} \big)^2   \Big). 
\end{equation*}

Turning our attention to the cross-term, we can apply similar steps and  the Markov property as before to obtain 
\begin{align*}
     &\mathbb{N}_x \Big(  \big( \int_0^\sigma \dd L_s^{D_n} \mathbbm{1}_{\{ \langle \rho_s ,1 \rangle  \leq K \}} \big)\cdot 
    \big( \int_0^\sigma \dd L_s^{D} \mathbbm{1}_{\{ \langle \rho_s ,1 \rangle  \leq K \}} \big)   \Big)\\
    &  =  \mathbb{N}_x \Big( \int_0^\sigma \dd L_s^{D_n} \mathbbm{1}_{\{ \langle \rho_s ,1 \rangle  \leq K\}} \int_s^\sigma \dd L_u^D \mathbbm{1}_{\{\langle \rho_u ,1 \rangle  \leq K \}} \Big)
    +
    \mathbb{N}_x \Big( \int_0^\sigma \dd L_s^{D} \mathbbm{1}_{\{  \langle \rho_s ,1 \rangle  \leq K\}} \int_s^\sigma \dd L_u^{D_n} \mathbbm{1}_{\{\langle \rho_u ,1 \rangle  \leq K \}} \Big)  \\
    &  = E^0 \otimes \Pi_x \left( \mathbbm{1}_{\{ \tau_{D_n} < \infty \}} \exp(- \alpha \tau_{D_n}) \mathbbm{1}_{\{ \langle J_{\tau_{D_n}},1 \rangle  \leq K \}} \phi_{D}(J_{\tau_{D_n}} , \xi^{\tau_{D_n}}) \right)   \\
    &  \hspace{40mm} + E^0 \otimes \Pi_x \left( \mathbbm{1}_{\{ \tau_D < \infty \}} \exp(- \alpha \tau_D) \mathbbm{1}_{\{ \langle J_{\tau_D},1 \rangle \leq K \}} \phi_{D_n}(J_{\tau_D} , \xi^{\tau_D}) \right),
\end{align*}
and using the same method as before we get:
$$\lim_{n \rightarrow \infty}\mathbb{N}_x \Big(  \big( \int_0^\sigma \dd L_s^{D_n} \mathbbm{1}_{\{\langle \rho_s ,1 \rangle  \leq K \}} \big)\cdot 
    \big( \int_0^\sigma \dd L_s^{D} \mathbbm{1}_{\{ \langle \rho_s ,1 \rangle  \leq K \}} \big)   \Big)= \mathbb{N}_x \Big( \big(\int_0^\sigma \dd L_s^D \mathbbm{1}_{\{ \langle \rho_s ,1 \rangle  \leq K \}} \big)^2   \Big).$$
Taking  the limit as $n \uparrow \infty$ in  \eqref{equation:FdmedibilidadCuentass}  we deduce the claimed $L_2(\mathbb{N}_x)$ convergence \eqref{equation:L2convergenceExitLocal}. Now that the  convergence of the truncated total mass  has been established, to derive the statement of the  proposition  we proceed  as follows. First, we introduce the processes
\begin{equation*}
    A^n_t := \int_0^{t} \dd L_s^{D_n} \mathbbm{1}_{\{\langle \rho_s , 1  \rangle \leq K \}} \:\:\:\text{ and }\:\:\: A_t := \int_0^{t} \dd L_s^{D} \mathbbm{1}_{\{\langle \rho_s , 1  \rangle \leq K \}},
\end{equation*}
which are continuous additive functionals of the Markov process $(\rho, W)$. Then using the Markov property, we get
\begin{align} \label{equation:martingaleCAF}
   \mathbb{N}_x \left( A^n_\infty | \mathcal{F}_s  \right) = A_{s \wedge \sigma}^{n} + \phi_{D_n}(\rho_{s \wedge \sigma} , W_{s \wedge \sigma}) \:\:\:\text{ and }\:\:\:  \mathbb{N}_x \left( A_\infty | \mathcal{F}_s  \right) = A_{s \wedge \sigma } + \phi_D(\rho_{s \wedge \sigma} , W_{s \wedge \sigma}), 
\end{align}
since   $\phi_{D_n}(\mu,\text{w}) = \mathbbm{E}^\dag_{\mu, \text{w}}[A^n_\infty]$, $\phi_D(\mu, \text{w}) = \mathbbm{E}^\dag_{\mu, \text{w}}[A_\infty]$ and  $\phi_{D_n}(\rho_\sigma, W_\sigma) = \phi_{D}(\rho_\sigma, W_\sigma) =0$, $\mathbb{N}_x$-a.e. To simplify notation, we denote respectively by $M^n_s= \mathbb{N}_x(A^n_\infty | \mathcal{F}_s)$ and $M_s= \mathbb{N}_x(A_\infty | \mathcal{F}_s)$  for $s\geq 0$ the  martingales in \eqref{equation:martingaleCAF}. Next, we  apply  Doob's inequality  to  derive:
\begin{equation} \label{cuentas:markovineq_0}
    \mathbb{N}_{x}\big(\sup \limits_{s>0}|M_{s}^{n}-M_{s}|>\delta\big)\leq\delta^{-2} \mathbb{N}_{x}\big(|  A^n_\sigma-A_\sigma |^2\big).
\end{equation}
Indeed, even if $\mathbb{N}_x$ is not a finite measure, we can argue as follows: fix $a>0$ and observe that $(M_{a+ t})_{t \geq 0}$,  $(M^n_{a+ t})_{t \geq 0}$ under $\mathbb{N}_x( \, \cdot \,  |\sigma > a)$ are uniformly integrable martingales, from which we obtain  
   \begin{equation*}
    \mathbb{N}_{x}\big(\sup \limits_{s\geq a}|M_{s}^{n}-M_{s}|>\delta~ \big|~ \sigma > a\big)\leq\delta^{-2} \mathbb{N}_{x}\big(|   A^n_\sigma-A_\sigma|^2~ \big|~ \sigma > a\big), 
\end{equation*}
and we deduce \eqref{cuentas:markovineq_0} by multiplying both sides by $\mathbb{N}_x(\sigma > a)$ and by taking the limit as $a \downarrow 0$ -- using monotone convergence. \par 
By  \eqref{equation:L2convergenceExitLocal}, the right-hand side of  \eqref{cuentas:markovineq_0}  converges towards $0$ as $n\uparrow \infty$ and we deduce that 
\[
\lim \limits_{k\to\infty} \sup \limits_{s >0}|M_{s}^{n_k}-M_{s}| =0, \quad \mathbb{N}_{x}\text{ --a.e. }
\]
for a suitable subsequence $(n_k:~k\geq 1)$ increasing towards infinity.
Since $\lim \limits_{n\to\infty}\phi_{D_n}(\rho_{s},W_{s})=\phi_{D}(\rho_{s},W_{s})$, we obtain that $\mathbb{N}_x$-a.e.,  for every $t\geq 0$, $\int_0^t \dd L_s^{D_{n_k}} \mathbbm{1}_{\{\langle \rho_s , 1  \rangle\leq K \}} \rightarrow \int_0^{t}\dd L^{D}_s \mathbbm{1}_{\{ \langle\rho_s , 1  \rangle \leq K \}}$ as $k \rightarrow \infty$. By continuity, monotonicity and the fact that $\sigma <\infty$ $\mathbb{N}_x$--a.e.,  we can apply Dini's theorem to get: 
\begin{equation*}
    \lim_{k \rightarrow \infty} \sup_{t > 0} \big|\int_0^t \dd L_s^{D_{n_k}} \mathbbm{1}_{\{ \langle \rho_s , 1  \rangle\leq K\}} - \int_0^t \dd L_s^D \mathbbm{1}_{\{ \langle \rho_s , 1  \rangle \leq K \}} \big| = 0, \quad \quad \mathbb{N}_x \text{-- a.e.}
\end{equation*}
Consequently,  we deduce that on the event $\{ \sup_{s\geq 0}  \langle\rho_s , 1  \rangle \leq K \}= \{\sup X\leq K\}$, the $\mathbb{N}_x$-a.e. uniform convergence (\ref{equation:uniformConvExtis}) holds under a subsequence $(n_k)$, which depends on $K$. Since this holds for arbitrary $K$, we can use a diagonal argument to find a deterministic subsequence that we still denote by $(n_k:k\geq 1)$ converging towards infinity such that
\begin{equation*}
    \lim_{k \rightarrow \infty} \sup_{t\in[0,\sigma]}|L_{t}^{D_{n_k}}-L_{t}^{D}|=0, \quad \quad \mathbb{N}_x \text{-- a.e.}
\end{equation*}
\end{proof}
\noindent We are now in position to prove that the process $\widetilde{L}^D$ is $\mathcal{F}^D$-measurable.
\begin{proof}[Proof of Proposition \ref{L_eta_measurable}]
Until further notice, we argue under $\mathbb{P}_{0,x}$. By \eqref{definition:aproximationExitlocal} and monotonicity,  a diagonal argument gives that we can find a  subsequence $(\varepsilon _k:~k\geq 1)$, with $\varepsilon _k \downarrow 0$ as $k \rightarrow \infty$, such that:
$$L_{\Gamma_s^{D}}^{D_{n}}=\lim \limits_{k \to \infty }\frac{1}{\varepsilon _k}\int_{0}^{\Gamma_s^{D}} \dd r \mathbbm{1}_{\{\tau_{D_{n}}(W_{r})< H_{r}<\tau_{D_{n}}(W_{r})+\varepsilon _k \}}, $$
for every $n\geq 1$ and $s\geq 0$. Our goal  is now to show that:
\begin{align}\label{cambio:(*)}
L_{\Gamma_s^{D}}^{D_{n}}= \lim \limits_{k \to \infty }\frac{1}{\varepsilon _k}\int_{0}^{s} \dd r \mathbbm{1}_{\{\tau_{D_{n}}(W_{\Gamma_r^D})< H_{\Gamma^D_r}<\tau_{D_{n}}(W_{\Gamma^D_r})+\varepsilon _k \}},
\end{align}
which will imply that $(L_{\Gamma_{s}^D}^{D_{n}})_{s\geq 0}$ is $\mathcal{F}^{D}$-measurable for every $n \in \mathbb{N}$.
In order to establish \eqref{cambio:(*)} we argue for $\omega$ fixed and observe that for $k$ large enough, we have:
\begin{equation*}
    \mathbbm{1}_{\{\tau_{D_{n}}(W_{r})< H_{r}<\tau_{D_{n}}(W_{r})+\varepsilon _k \}} = \mathbbm{1}_{\{\tau_{D_{n}}(W_{r})< H_{r}<\tau_{D_{n}}(W_{r})+\varepsilon _k \}} \mathbbm{1}_{\{ H_r \leq \tau_D(W_r)  \}}, \quad \text{ for all } r \in [0,\Gamma_s^D].
\end{equation*}
To see it, remark that if the previous display did not hold, by a compactness argument and continuity we would have $\tau_{D_n}(W_{r_0})  = \tau_D(W_{r_0})  \leq H_{r}$ for some $r_0$ in $[0,\Gamma_s^D]$. This gives a  contradiction since $\overline{D}_n \subset D$ and $(W_{r_0}(t))_{t \in [0, H_{r_0}]}$ is continuous.  Recalling the notation  $V^D$ given in \eqref{definition:VD}, we deduce that 
\begin{align*}
L_{\Gamma_s^{D}}^{D_{n}}
&=\lim \limits_{k \to \infty }\frac{1}{\varepsilon _k}\int_{0}^{\Gamma_s^{D}} \dd r \mathbbm{1}_{\{\tau_{D_{n}}(W_{r})< H_{r}<\tau_{D_{n}}(W_{r})+\varepsilon _k \}} \\
&=\lim \limits_{k \to \infty }\frac{1}{\varepsilon _k}\int_{0}^{\Gamma_s^{D}} \dd V^D_r \mathbbm{1}_{\{\tau_{D_{n}}(W_{r})< H_{r}<\tau_{D_{n}}(W_{r})+\varepsilon _k \}} 
=\lim \limits_{k \to \infty }\frac{1}{\varepsilon _k}\int_{0}^{s} \dd r \mathbbm{1}_{\{\tau_{D_{n}}(W_{\Gamma_r^D})< H_{{\Gamma_r^D}}<\tau_{D_{n}}(W_{\Gamma_r^D})+\varepsilon _k \}}, 
\end{align*}
giving us \eqref{cambio:(*)}.  The same arguments can be applied under $\mathbb{N}_x$ and, to complete the proof of the proposition, it suffices to show that  for every $t\geq 0$
\begin{align} \label{equation:aproxLD}
\lim \limits_{n\to \infty}\sup\limits_{s\in[0,t]}|L_{\Gamma_s^{D}}^{D_{n}}-L_{\Gamma_s^{D}}^{D}|=0, \quad \quad \text{ under }\mathbb{P}_{0,x} \text{ and } \mathbb{N}_x,
\end{align}
at least along a suitable subsequence. However, note that when working under $\mathbb{N}_x$, this convergence follows by  Lemma \ref{lemma:technicalLema2}. Now, the result under $\mathbb{P}_{0,x}$ is a standard  consequence of excursion theory. More precisely, 
recall that $-I$ is the local time of $(\rho,W)$ at $(0,x)$ and, for fixed $r > 0$, set $T_r := \inf \{ t \geq 0 : ~-I_t > r \}$.  If we let  $T_D := \inf\{ t \geq 0 : \tau_D(W_t) < \infty \}$, by continuity there exists a finite number of excursions $(\rho^i, W^i)$ of $(\rho , W)$  in $[0,T_r]$ satisfying $T_{D}(W^i) < \infty$, and their distribution is $\mathbb{N}_{x,0}(\, \cdot \, | T_D < \infty)$. Since $T_r \uparrow \infty$, the approximation \eqref{equation:aproxLD} under $\mathbb{P}_{x,0}$ now follows from the result under $\mathbb{N}_{x,0}$. This completes the proof of  Proposition \ref{L_eta_measurable}.
\end{proof}
%%% Fin Lemma Ltilde 
\subsection{Proof of special Markov property} \label{subsection:specialMarkovProof}
Now that we have already studied the trajectories staying in $D$, we turn our attention to the complementary side of the picture and we start by introducing formally  the notion of excursions from $D$.\medskip \\
\noindent \textbf{Excursions from $D$.} Observe that (\ref{tirage_au_hasard_N}) and  assumption  \ref{tau_infinity}  imply that 
\begin{equation*} 
   \mathbb{N}_{x}\Big(\int_{0}^{\sigma}\dd s~\mathbbm{1}_{\{\tau_{D}(W_{s})< \zeta_{s} \}}>0\Big)>0.
\end{equation*}
Hence, the set $\big\{s\in[0,\sigma]:\:\tau_{D}(W_{s})<\zeta_{s}\big\}$ is non-empty with non null measure under $\mathbb{N}_x$ and $\mathbb{P}_{0,x}$. If we define 
\begin{equation*}
    \gamma^D_s := \big( \zeta_s - \tau_D(W_s) \big)_+, \quad \quad s \geq 0, 
\end{equation*}
it is straightforward  to show by the snake property and the continuity of $\zeta$ that $\gamma^D$ is continuous. Set  
\[\sigma^{D}_{t}:=\inf\big\{s\geq 0:\:\int_{0}^{s} \dd r\mathbbm{1}_{\{\gamma^D_{r}>0\}} > t\big\},\]
and consider the process $(\rho^{D}_{t})_{t\geq 0}$ taking values in $\mathcal{M}_{f}(\mathbb{R}_{+})$ defined by:
\begin{equation}\label{def:rho:D}
\langle  \rho^{D}_{t}, f \rangle :=\int \rho_{\sigma^{D}_{t}}(\dd h)f\big(h-\tau_{D}(W_{\sigma^{D}_{t}})\big) \mathbbm{1}_{\{ h>\tau_D(W_{\sigma_t^D}) \}}.
\end{equation}
Then, by Proposition 4.3.1 in \cite{Duquesne},  $\rho^D$ and $\rho$ have the same distribution under $\mathbb{P}_{0,x}$. In particular,  $\langle \rho^D , 1 \rangle$ has the same law as the reflected Lévy process $X-I$ and we denote its local time at $0$ by  $(\ell^D(s) : s \geq 0)$. Moreover, it is shown in  \cite[ Section 4.3]{Duquesne} that the process $L^D$ is related to the local time $\ell^D$ by the identity:
\begin{equation} \label{equation:exitlocalTimeChange}
    L^D_t = \ell^D \left(  \int_0^t \dd s \,  \mathbbm{1}_{\{ \gamma^D_s > 0  \}}  \right). 
\end{equation}
The proof of  Proposition 4.3.1 in \cite{Duquesne} shows that $\rho^{D}$ can be obtained as limit of functions which are independent of $\mathcal{F}^{D}$,  implying that $\rho^{D}$ is on its turn independent of $\mathcal{F}^{D}$.
Now,  denote  the connected components of the open set 
\begin{equation*}
    \big\{t\geq 0:\:\tau_{D}(W_{t})< \zeta_t \big\} = \big\{ t \geq 0 : \gamma^D_t > 0 \big\},
\end{equation*}
by  $\big((a_{i},b_{i}) : i\in \mathcal{I} \big)$, 
where  $\mathcal{I}$ is an indexing set that might be empty. By construction, for any $s \in (a_i,b_i)$, the trajectory $W_s$ is a trajectory leaving $D$.  Remark that $H_{a_{i}}=H_{b_{i}}<H_r$ for every $r\in (a_i,b_i)$ and let $(\rho^{i},W^{i})$ be the subtrajectory of $(\rho,W)$ associated with $[a_{i},b_{i}]$ as defined in  Section \ref{secsnake}.  Observe that in our setting,   $(\rho^i, W^i)$ is  defined for each $s \in [0,b_i - a_i]$ and for any measurable function $f:\mathbb{R}_+ \mapsto \mathbb{R}_+$ as
\begin{align*}
    \langle \rho^{i}_{s} , f \rangle  =\int \rho_{a_i+s}(\dd h)f(h - \tau_D(W_{a_i}))\mathbbm{1}_{\{ h  > \tau_D(W_{a_i}) \}}
\end{align*}
and
\begin{equation*}
    W^i_s =  W_{(a_i+s) \wedge b_i }( t + \tau_D(W_{a_i})) \quad  \text{ for } t \in [ 0  , \zeta_{ (a_i+ s)\wedge b_i   } - \tau_D(W_{a_i})  ] ,
\end{equation*}
with respective lifetime process given by 
\begin{equation*}
    \zeta_s^i = \zeta _{(a_i + s) \wedge b_i} - \tau_D(W_{a_i}), 
\end{equation*}
where $\tau_D(W_s)= \tau_D(W_{a_i}) = \zeta_{a_i}$. We say that $(\rho^i , W^i)$ is an \textit{excursion} of $(\rho,W)$ from  $D$. Observe that $W_s^i(0) = W_{a_i}^i(0)$ for all $s \in [a_i, b_i]$ by the snake property and that we have $W_{a_i}^i(0) \in \partial D$. This is the point of $\partial D$ used by  the subtrajectory $W^i$  to escape from $D$.

{ 
In order to state the special Markov property we need to introduce one last notation. Let $\theta$ be the right inverse of $\widetilde {L}^D$, viz.  the $\mathcal{F}^D$-measurable function defined as
\[\theta_{r}:=\inf\big\{s \geq 0 \: :  L_{\Gamma^D_{s}}^{D}> r\big\}, \quad  \text{ for all } r\in[0,L_{\sigma}^{D}).\:\]
%% Inicio Markov Px
Recall that we are considering some fixed  $x\in D$,   the notation $( (\rho^i , W^i) : i \in \mathcal{I} )$ for the excursions outside $D$, and that we are working under  the hypothesis \ref{tau_infinity}. We are now going to state and prove the special Markov property under $\mathbb{P}_{0,x}$,  and we will  deduce by standard arguments  a version   under the excursion measure $\mathbb{N}_{x}$.  Under $\mathbb{P}_{0,x}$ we use the same notation as under $\mathbb{N}_x$, but observing that $\sigma_H = \infty$ and  noticing that  $\mathbb{P}_{0,x}$-a.s., we have   $\mathcal{Y}_{D}=\int_{0}^{\infty} \dd s \mathbbm{1}_{\{H_{s} \leq \tau_{D}(W_s)\}}=\infty$ and $L^D_\infty = \infty$. In particular, this implies that  $\Gamma^D_s$ and $\theta_s$ are finite for every $s < \infty$.   
\begin{theo}[Special Markov property]\label{Theo_Spa_Markov_Excur_P}
Under $\mathbb{P}_{0,x}$, conditionally on $\mathcal{F}^{D}$, the point measure
\[\sum \limits_{i\in \mathcal{I}} \delta_{(L_{a_{i}}^{D},\rho^{i},W^{i})}(\dd \ell, \:\dd\rho, \:\dd\omega)\]
is a Poisson point process with intensity
\begin{equation*} 
    \mathbbm{1}_{[0,\infty)}(\ell)\:\dd\ell \: \mathbb{N}_{\mathrm{tr}_{D}(\widehat{W})_{\theta_\ell }}(\dd\rho, \:\dd\omega).
\end{equation*}
\end{theo}
Recall that we have established in Proposition \ref{L_eta_measurable} that $\widetilde {L}^D$ is  $\mathcal{F}^D$-measurable. It might also be worth observing that if  $F = F(\upvarrho , \omega)$ is a measurable function, when integrating with respect to the intensity measure  $ \mathbbm{1}_{[0,\infty)}(\ell)\: \dd\ell \: \mathbb{N}_{\mathrm{tr}_{D}(\widehat{W})_{\theta_\ell }}(\dd\rho, \:\dd\omega)$ we can re-write the expression in the following more tractable form:
\begin{equation*}
    \int_0^\infty \dd \ell \,  \mathbb{N}_{\text{tr}_{D}\widehat{(W)}_{\theta_\ell }} (F) 
    =
    \int_0^\infty \dd\widetilde {L}_s^D \,  \mathbb{N}_{\text{tr}_{D}(\widehat{W})_{s}} (F) 
    = \int_0^\infty \dd {L}_s^D \,  \mathbb{N}_{\widehat{W}_{s}} (F) 
\end{equation*}
where in the last equality, we applied  a change of variable for Lebesgue-Stieltjes integrals using the fact that $L^D$ is constant on the excursion intervals $[\Gamma^D_{s-} , \Gamma^D_{s} ]$ when $\Gamma^D_{s-} < \Gamma^D_s$.  Let us now prove Theorem \ref{Theo_Spa_Markov_Excur_P}.
\begin{proof}
In this  proof, we work with  $(\rho,W)$ under $\mathbb{P}_{0,x}$.  Let us start with some preliminary constructions and  remarks. First, we introduce the $\mathcal{S}_x$-valued  process $(\rho, W^*)$ defined at each $t \geq 0$ as 
\begin{equation*} 
    \big(\rho_{t},W^{*}_{t}(s)\big)=\Big(\rho_{t}(\dd h)  ,W_{t}\big(s\wedge\tau_{D}(W_{t})\big)\Big), \quad \text{ for } s \in [0,\zeta_{W_t}], 
\end{equation*}
and let $\mathcal{F}^D_*$ be its generated sigma-field on $\mathbb{D}(\mathbb{R}_+, \mathcal{M}_f(\mathbb{R}_+) \times \mathcal{W}_{E,x})$. The snake $(\rho, W^*)$ can be interpreted as the Lévy snake associated with $(\psi,\xi^*)$, where $\xi^{*}$ is the stopped Markov process $(\xi^*_t:~t\geq 0)=(\xi_{t\wedge \tau_D(\xi)}:~t\geq 0)$. Since, for every $t\geq 0$,    \[\big(\zeta_{W_{t}}-\tau_{D}(W_{t})\big)_{+}=\big(\zeta_{W_{t}^{*}}-\tau_{D}(W_{t}^{*})\big)_{+},\]
we derive that the process   $\gamma_{t}^D = (\zeta_{W_{t}}-\tau_{D}(W_{t}))_{+}$ is  $\mathcal{F}_*^D-$measurable. Consequently, we have  $\mathcal{F}^D\subset \mathcal{F}^D_*$ since  $V^{D}$ -- the functional measuring the time spent in $D$  defined in \eqref{definition:VD} --  is $\mathcal{F}^D_*$-measurable and by definition $\rm{tr}_D (\rho,W)=\rm{tr}_D (\rho,W^*)$. Recalling  that  $\big( (a_{i},b_{i}) : i \in \mathcal{I}\big)$ stands for the connected components of the open set 
\begin{equation*}
    \{t\geq 0:\:\tau_{D}(W_{t})< \zeta_{W_t} \} = \{ t \geq 0 : \gamma^D_t > 0 \}, 
\end{equation*}
we deduce by the previous discussion and the identity $\tau_D(W_{a_i}) = \zeta_{a_i}$,  that the variables
$$\widehat{W}_{a_i} = \widehat{W}^*_{a_i},\,  \zeta^{i} =\zeta_{(a_{i}+\cdot \, )\wedge b_{i}}-\zeta_{a_{i}} \text{ and a fortiori } \, \rho^{i} \, 
     \text{ are } \mathcal{F}^D_* - \text{ measurable}. $$
 Informally, $\mathcal{F}_*^D$ encodes the information of the trajectories staying in $D$ and the tree structure.
We claim that conditionally on $\mathcal{F}^D_*$, the excursions $(  W^i :i \in \mathbb{N}  )$ are independent, and that the conditional distribution of $W^{i}$ is $\mathbb{Q}_{\widehat{W}_{a_{i}}^*}^{\zeta^{i}}$, where we recall from Section \ref{secsnake} that we denote  the distribution of the snake driven by $h$ started at $x$ by $\mathbb{Q}^h_x$. 
\\
\\
In order to prove this claim, consider a collection of snake trajectories  $\big( W^{i,\prime}:i\in \mathcal{I} \big)$ such that, conditionally on $(\rho , W^{*})$, they are independent and each one is respectively distributed according to the measure $\mathbb{Q}_{\widehat{W}_{a_{i}}^{*}}^{\zeta^{i}}$.  Next let $W^{\prime}$ be the process defined as follows: for every $t$ such that $\gamma_{t}^D=0$ set  $W^{\prime}_{t}=W^{*}_{t}$, and if $\gamma_{t}^D>0$ we  set:
$$
 W^{\prime}_{t}(s)=
\begin{cases}
W_{t}^{*}(s) \hspace{27mm} \, \, \text{if}\:\:s\in[0,\tau_{D}(W_{t}^{*})]\\
W_{t-a_{i}}^{i,\prime}\big(s-\tau_D(W_{t}^{*})\big)\:\: \hspace{4mm} \text{if}\:\:s\in[\tau_{D}(W_{t}^{*}), \zeta(W_t^*)],
\end{cases}
$$
where $i$ is the unique index such that $t\in(a_{i},b_{i})$. By construction, $(\rho , W^{\prime})$ is in $\mathbb{D}(\mathbb{R}_+, \mathcal{M}_f(\mathbb{R}_+) \times \mathcal{W}_{E,x})$ and  a straightforward computation of its finite marginals  shows that its distribution is $\mathbb{P}_{0,x}$, proving our claim. \par 
Notice that \eqref{definition:aproximationExitlocal} implies that  $L^D$ is constant on the intervals $[\Gamma^D_{s-}, \Gamma^D_{s}]$ when $\Gamma^D_{s-}  < \Gamma^D_s$. Hence, $L^D_s = \widetilde {L}^D_{V^D_s}$ for all $s \geq 0$ and in particular $L^D_{a_i} = \widetilde {L}^D_{V^D_{a_i}}$, 
the latter being $\mathcal{F}^D_*$-measurable. Consider now $U$ a bounded   $\mathcal{F}^{D}$-measurable  random variable,  and remark that to obtain the desired result, it is enough to show that:
\begin{equation*}
    \mathbb{E}_{0,x}\Big[U\exp(-\sum\limits_{i\in \mathcal{I}} F\big(L_{a_{i}}^{D},\rho^{i},W^{i})\big)\Big]= \mathbb{E}_{0,x}\Big[U\exp\Big(-\int_{0}^{\infty}\dd\ell \: \mathbb{N}_{\text{tr}_{D}\widehat{(W)}_{\theta_\ell }}\big(1-\exp(-F(\ell ,\rho,W)\big)\Big)\Big],
\end{equation*}
for every  non-negative measurable function $F$  in $\mathbb{R}_+ \times \mathbb{D}(\mathbb{R}_+, \mathcal{M}_f(\mathbb{R}_+) \times \mathcal{W}_E)$. In order to prove this identity, we start by projecting the left term on $\mathcal{F}^{D}_*$:  by the previous discussion and recalling that $\mathcal{F}^{D}\subset\mathcal{F}^{D}_*$, we get
$$\mathbb{E}_{0,x}\Big[U\exp(-\sum\limits_{i\in \mathcal{I}} F\big(L_{a_{i}}^{D},\rho^{i},W^{i})\big)\Big]=\mathbb{E}_{0,x}\Big[U\prod\limits_{i\in\mathcal{I}}\mathbb{Q}_{\widehat{W}_{a_{i}}^{*}}^{\zeta^i}\big(\exp(-F(L_{a_{i}}^{D},\rho^i,W)\big)\Big].$$
Moreover, it is straightforward to see that
\[ \widehat{W}^*_{a_i} = \widehat{W}_{a_{i}}=\text{tr}_{D}\widehat{(W)}_{\theta_{L_{a_{i}}^{D}}},\]
we omit the details of this identity  since the  argument used in (23) of \cite[Theorem 20]{Subor} for the Brownian snake applies directly to our framework.
Consequently, we have: 
$$\mathbb{E}_{0,x}\Big[U\prod\limits_{i\in\mathcal{I}}\mathbb{Q}_{\widehat{W}_{a_{i}}^{*}}^{\zeta^{i}}\big(\exp(-F(L_{a_{i}}^{D},\rho^i,W)\big)\Big]=\mathbb{E}_{0,x}\Big[U\prod\limits_{i\in\mathcal{I}}\mathbb{Q}_{\text{tr}_{D}\widehat{(W)}_{\theta_{L_{a_{i}}^{D}}}}^{\zeta^{i}}\big(\exp(-F(L_{a_{i}}^{D},\rho^i,W)\big)\Big]. $$
Now, we need to take the projection on $\mathcal{F}^D$. 
Recalling that $H(\rho^i)=\zeta^i$, observe that for every $i\in \mathcal{I}$,
$$\mathbb{Q}_{\text{tr}_{D}\widehat{(W)}_{\theta_{L_{a_{i}}^{D}}}}^{\zeta^{i}}\big(\exp(-F(L_{a_{i}}^{D},\rho^i ,W)\big) $$
is a measurable function  of the pair $(L_{a_i}^D,\rho^i)$ and the process $(\text{tr}_{D}(W)_{\theta_r}:~r\geq 0)$, the latter being $\mathcal{F}^D$-measurable. We are going to conclude by showing that the point measure 
\[\sum\limits_{i\in \mathcal{I}}\delta_{(L_{a_{i}}^{D},\rho^{i})}\]
is a Poisson point measure with intensity $\mathbbm{1}_{ [0,\infty)}(\ell)\dd \ell \,  N(\dd \rho)$ independent of $\mathcal{F}^D$. Remark that once this has been established, an application of the exponential formula for functionals of Poisson random measures  yields
$$\mathbb{E}_{0,x}\Big[U\prod\limits_{i\in\mathcal{\mathbb{N}}}\mathbb{Q}_{\text{tr}_{D}\widehat{(W)}_{\theta_{L_{a_{i}}^{D}}}}^{\zeta^{i}}\big(\exp(-F(L_{a_{i}}^{D},\rho^i,W)\big)\Big]=\mathbb{E}_{0,x}\Big[U\exp\Big(-\int_{0}^{\infty} \dd\ell \: \mathbb{N}_{\text{tr}_{D}\widehat{(W)}_{\theta_\ell }}\big(1-\exp(-F(\ell ,\rho,W)\big)\Big)\Big],$$
giving the desired result. In this direction, recall the definition of $\rho^{D}$ given in \eqref{def:rho:D}, and that $\ell^D$ stands for the local time of $\rho^D$ at 0.
We denote the connected component of the open set $\{t\geq 0 : \,  \langle \rho^{D}_{t} , 1 \rangle\neq 0\} = \{ t \geq 0 : H(\rho^D_t) > 0 \}$ by  $\big((c_{j},d_{j}) : j \in \mathbb{N} \big)$  -- the latter equality holding since $\rho^D_t(\{ 0 \}) =0$ --   and  observe that these  are precisely the excursion intervals of $\langle \rho^D, 1 \rangle$ from $0$. It follows by \eqref{definition:PoissonRMExcursionesRho} and the discussion before the proof that 
\begin{equation*}
    \sum_{j\in \mathbb{N}}\delta_{ (\ell^D(c_j) , \, \rho^{D}_{(c_{j}+\cdot)\wedge d_{j}}) }
\end{equation*}
is a Poisson point measure  with intensity $\mathbbm{1}_{ [0,\infty)}(\ell)\dd \ell \,  N(\dd \rho)$ and observe that this measure is  independent of $\mathcal{F}^D$ --  since $\rho^D$ is  independent of $\mathcal{F}^D$. 
Furthermore, by \eqref{equation:exitlocalTimeChange} we have:

$$L_{\sigma_{s}^{D}}^{D}=  \ell^D \Big(  \int_0^{\sigma_{s}^{D}} \dd r \,  \mathbbm{1}_{\{ \gamma^D_r > 0  \}}  \Big)=\ell^D (s),  $$
for every $s\geq 0$. 
 It is now straightforward  to deduce from our last observations that:
\begin{equation*} 
    \big\{ ( L_{a_i}^D , \rho^i ) : i \in \mathcal{I} \big\} =  
   \big\{ ( \ell^D(c_j) , \rho^{D}_{(c_{j}+\cdot)\wedge d_{j}} ) : j \in \mathbb{N} \big\}, 
\end{equation*}
concluding the proof. 

\end{proof}
}
Setting $T_D = \inf\{ t \geq 0 : \tau_D(W_t) < \infty \}$, we infer from our previous result  a version of the special Markov property holding under the probability measure 
\begin{equation*}
    \mathbb{N}^D_{x} := \mathbb{N}_{x}( \, \cdot \,  | T_D < \infty ).
\end{equation*}
Observe that $\mathbb{N}_x(T_D < \infty)$ is finite: if this quantity was infinite, by excursion theory, the process $(\rho , W)$ under  $\mathbb{P}_{0,x}$ would have infinitely many excursions exiting $D$ on compact intervals, contradicting the continuity of its   paths. Finally, note that $(\rho, W)$ under $\mathbb{N}^D_x$ has  the distribution of the first  excursion  exiting the domain $D$.  As a straightforward consequence  of Theorem \ref{Theo_Spa_Markov_Excur_P}, this observation allows us to deduce: 
{
\begin{theo}\label{Theo_Spa_Markov_Excur}
Under $\mathbb{N}^D_{x}$ and conditionally on $\mathcal{F}^{D}$, the point measure:
\[
\sum \limits_{i\in \mathcal{I}} \delta_{(L_{a_{i}}^{D},\rho^{i},W^{i})}(\dd \ell, \: \dd\rho, \: \dd \omega)
\]
is a Poisson point process with intensity
\[\mathbbm{1}_{ [0, L_{\sigma}^{D}]}(\ell)\: \dd \ell \: \mathbb{N}_{\mathrm{tr}_{D}(\widehat{W})_{\theta_\ell}}(\dd \rho, \: \dd\omega).
\]
\end{theo}

%%%% FIn Markov Px

Recall that the measure $\dd L_s^D$ is supported on $\{ s\geq 0 : \widehat{W}_s \in \partial D \}$ and consider a measurable function $g:\partial D\to \mathbb{R}_{+}$. Under $\mathbb{N}_x$, we define the \textit{exit measure from $D$}, denoted by $\mathcal{Z}^D$  as: 
\begin{equation*}
    \langle \mathcal{Z}^D , g \rangle := \int_0^\sigma \dd L^D_s g(\widehat{W}_s).
\end{equation*}
The total mass of $\mathcal{Z}^D$ is $L_\sigma^D$ and, in particular,  $\mathcal{Z}^D$ is non-null only in  $\{ T_D < \infty \}$. Again by a standard change of variable, we get  
\begin{equation*} 
    \langle \mathcal{Z}^D , g \rangle 
    = \int_0^\sigma \dd \widetilde{L}^D_s g( \text{tr}_D (\widehat{W}_s)) = \int_0^{L^D_\sigma} \dd \ell \, g({\text{tr}_D \widehat{W}_{\theta_\ell }}),  \quad \quad \mathbb{N}_x \text{--a.e.}
\end{equation*}
and this implies that $\mathcal{Z}^D$ is $\mathcal{F}^D$-measurable  since $L^D_\sigma \in \mathcal{F}^D$ by Proposition \ref{L_eta_measurable}. In this work, we shall frequently make  use  of the following  simpler  version of the special Markov property. By Theorem \ref{Theo_Spa_Markov_Excur}, we have
\begin{cor} \label{Corollary:specialMarkovN}
Under $\mathbb{N}_x^D$ and conditionally on $\mathcal{F}^D$,  the point measure 
\begin{equation} \label{specialMarkov_v3}
    \sum_{i \in \mathcal{I}} \delta_{(\rho^i , W^i)} (\dd \rho, \,  \dd \omega)
\end{equation}
is a Poisson random measure with intensity $\int \mathcal{Z}^D(\dd y) \mathbb{N}_{y}(\dd \rho, \, \dd \omega)$.
\end{cor}
Let us close this section by recalling some well-known properties of 
$\mathcal{Z}^D$ that will be needed,  and by introducing some useful notations. Remark  by \eqref{tirage_au_hasard_ExitN} that, for any measurable $g : \partial D \mapsto \mathbb{R}_+$ and for every $y \in D$, we have
\begin{equation*}
\:\mathbb{N}_{y}\big(\langle \mathcal{Z}^{D},g \rangle \big)=\Pi_{y}\Big(\mathbbm{1}_{\{\tau_{D}<\infty\}}\exp(-\alpha \tau_{D}) g(\xi_{\tau_{D}})\Big),
\end{equation*}
and for such $g$,  we set:
\begin{equation} \label{definition:u_g_lambda}
    \:u_{g}^{D}(y):=\mathbb{N}_{y}\big(1-\exp(-\langle \mathcal{Z}^{D},g\rangle )\big), \quad  \text{ for all } y\in D.
\end{equation}
Theorem 4.3.3 in \cite{Duquesne}
states that for every $g:\partial D\to \mathbb{R}_{+}$ bounded measurable function,  $u_{g}^{D}$ solves the integral equation:
\begin{equation}\label{integral_equation}
u_{g}^{D}(y)+\Pi_{y}\Big(\int_{0}^{\tau_{D}}\dd t \, \psi(u_{g}^{D}(\xi_{t}))\Big)=\Pi_{y}\big(\mathbbm{1}_{\{\tau_{D}<\infty\}}g(\xi_{\tau_{D}})\big).
\end{equation}
By convention, we set $u_{g}^{D}(y):=g(y)$ for every $y\in \partial D$,   and we stress that this convention is compatible with \eqref{integral_equation}.
}

\section{Construction of a measure supported on
\texorpdfstring{$\{ t \in \mathbb{R}_+ : \widehat{W}_t = x \}$}{Lg}}\label{section:structureOfTloc}

From now on, we fix $x\in E$ and we consider the random set:
\begin{equation} \label{equatioN:introsupport}
     \{ t \in \mathbb{R}_+ : \widehat{W}_t = x  \}, 
     \quad  \text{ as well as its image on the tree $\mathcal{T}_H$, viz.   }\quad 
     \{ \upsilon \in \mathcal{T}_H : \xi_\upsilon = x  \}.
\end{equation}
In order to study the latter,  we shall construct an additive functional $A := (A_t)_{t \in \mathbb{R}_+}$ of the Lévy snake supported on $\{ t \in \mathbb{R}_+ : \widehat{W}_t = x \}$. The present section is devoted to the construction of $A$ and to  develop the machinery needed for our analysis.
 The study  of $\{ \upsilon \in \mathcal{T}_H : \xi_\upsilon = x  \}$ is delayed to Section \ref{section:treeStructureLocalTime} and will heavily rely on the results of this section. Let us discuss now in detail the framework we will consider in the rest of this work.
\\
\\
\textbf{Framework of Section \ref{section:structureOfTloc}  and \ref{section:treeStructureLocalTime}: }With the same notations as in previous sections, consider a strong Markov process $\xi$ taking values in $E$ with a.s.  continuous sample paths and we make the following assumptions:
\begin{equation}\label{Asssumption_2}
\bm{x} \:\:\textbf{is regular, instantaneous and recurrent for}\:\:  {\xi}, \tag*{$\mathbf{(H_{2})}$}
\end{equation}
and
\begin{equation}\label{Asssumption_3}
\mathbf{\int_{0}^{\infty}\dd t \:\mathbbm{1}_{\{\xi_{t}=\bm{x}\}} =0}, \quad  \quad \Pi_x -\text{a.s. } \tag*{$\mathbf{(H_{3})}$}
\end{equation}
Under  \ref{Asssumption_2}  the local time of $\xi$ at $x$ is well defined up to a multiplicative constant  (that we fix arbitrarily) and we denote it by $\mathcal{L}$. The recurrence hypothesis  is assumed for convenience and we expect our results to hold with minor modifications without it.   
Set $E_* := E \setminus \{ x\}$ and for $\w \in \mathcal{W}_E$, with the  notation  of Section \ref{Special and intermediate} write
\begin{equation*}
    \tau_{E_*}(\w) = \inf \{ h \in [0,\zeta_\w] : \w(h) = x \},
\end{equation*}
for the  exit time of $\w$ from the open set $E_*$. Observe that since $x$ is recurrent for $\xi$, we have 
\begin{equation} \label{equation:recurrent}
    \Pi_y(\tau_{E_*}< \infty) = 1 
\end{equation}
for every $y \in E_*$, and in particular \ref{tau_infinity} holds. This will allow us to make use of the special Markov property established in the previous section. Assumption \ref{Asssumption_3} might seem a technicality but it plays a crucial role in our study:   it will ensure,  under $\mathbb{N}_{y}$ and $\mathbb{P}_{0,y}$, that the set of branching points of $\mathcal{T}_H$ and  $\{ \upsilon \in \mathcal{T}_H \setminus \{ 0 \} :  \xi_\upsilon = x \}$ are disjoint. We will explain properly this point  after concluding the presentation of the section.  

\par  Let  $\mathcal{N}$ be the excursion measure of $\xi$ at $x$ associated with $\mathcal{L}$ and, with a slight abuse of notation,  still write    $\sigma_\xi$ for the lifetime of $\xi$ under $\mathcal{N}$.  The pair 
\begin{equation*}
    \overline{\xi}_s = ( \xi_s , \mathcal{L}_s ), \quad s \geq 0, 
\end{equation*}
  is a strong Markov process taking values in the Polish space $\overline{E}:=E\times \mathbb{R}_+$ equipped with the product metric $d_{\overline{E}}$.  We  set $\Pi_{y,r}$ for  its law started from an arbitrary point $(y,r)\in \overline{E}$.  Recall that we always work under the assumptions \ref{continuity_snake}, which for $(\psi, \overline{\xi})$ takes the following form:
\begin{enumerate}
    \item[\mbox{}] \textbf{Hypothesis $(\textbf{H}'_{0})$}. There exists  a constant $C_\Pi> 0$ and two positive numbers $p,q > 0$ such that,\\ for every $y \in E$ and $t\geq 0$, we have:
\begin{equation} \label{continuity_snake_2}
    \hspace{-15mm}\Pi_{y,0}\Big(  \sup_{0 \leq u \leq t } d_{\overline{E}}\big((\xi_u, \mathcal{L}_u) , (y,0) \big)^p  \Big) \leq C_\Pi\cdot t^{q}, \hspace{7mm} \text{ and } \hspace{7mm} q\cdot (1-\Upsilon^{-1})>1, \tag*{$\mathbf{(H_{0}^{\prime})}$}
\end{equation}
\end{enumerate}
where we recall the definition of $\Upsilon$ from \eqref{defintiion:coladeBallenaconcaradepalemera}.
We will use respectively the notation   $\overline{\Theta}$, $\overline{\mathcal{S}}$ for  the sets defined as $\Theta$,  $\mathcal{S}$ in Section \ref{secsnake} but  replacing the Polish space $E$ by $\overline{E}$. It will be convenient to write the elements of $\mathcal{W}_{\overline{E}}$ as pairs  $\overline{\text{w}} = (\text{w}, \ell)$, where $\w \in \mathcal{W}_E$ and $\ell : [0,\zeta_\w]\mapsto \mathbb{R}_+$ is a continuous function.
 Recall that under \ref{continuity_snake_2}, the family of measures $(\mathbb{P}_{\mu , \overline{\w}}: (\mu , \overline{\w}) \in \overline{\Theta}  )$ are defined in the canonical space $\mathbb{D}(\mathbb{R}_+,  \mathcal{M}_f(\mathbb{R}_+) \times \mathcal{W}_{\overline{E}} )$ and we denote  the canonical process by $(\rho , W, \Lambda)$, where  $W_{s}:[0,\zeta_{s}(\overline{W}_{s})]\mapsto E$ and $\Lambda_{s}:[0,\zeta_{s}(\overline{W}_{s})]\mapsto \mathbb{R}_+$.  Said otherwise, for each $(\mu , \overline{\w})\in \overline{\Theta}$, under $\mathbb{P}_{\mu , \overline{\w}}$ the process 
 \begin{equation*}
      (\rho_s ,  W_{s},\Lambda_{s} ), \quad  s \geq 0, 
 \end{equation*}
 is the $\psi$-Lévy snake with spatial motion $\overline{\xi}$ started from $(\mu , \overline{\w})$ and  
   we  simply write $\overline{W}_s := (W_{s},\Lambda_{s})$. For every $(y,r_0) \in \overline{E}$, we denote the excursion measure of $(\rho , \overline{W})$   starting from $(0,y,r_0)$  by ${\mathbb{N}}_{y,r_0}$. 
  \par
  Recall that  under $\mathbb{P}_{0,y,r_0}$ or $\mathbb{N}_{y,r_0}$, for each $s\geq 0$ and  conditionally on $\zeta_s$, the pair  $$(W_s,\Lambda_s) = \big( (W_s(h) , \Lambda_s(h)\big): h \in [0,\zeta_s]\big)$$ has the distribution of $(\xi, \mathcal{L})$ under $\Pi_{y,r_0}$  killed at $\zeta_s$. In particular,  the associated  Lebesgue-Stieltjes  measure of $\Lambda_s$ is  supported on the closure of  $\{ h \in [0,\zeta_s) : W_s(h) = x \}$,  $\mathbb{P}_{0,y,r_0}$ and $\mathbb{N}_{y,r_0}$--a.e. We will restrict our analysis to the collection of initial conditions $(\mu , \overline{\w}):=(\mu,\w,\ell) \in \overline{\Theta}$ satisfying that: 
 \begin{enumerate} 
     \item[\rm{(i)}] $\ell$   is a non-decreasing continuous function and the support  of its Lebesgue-Stieltjes  measure is $$\overline{\big\{ h \in [0,\zeta_\w) : \w(h) = x \big\}}.$$ 
     \item[\rm{(ii)}] The measure $\mu$ does not charge the set $\{ h \in [0,\zeta_\w] : \w(h) = x\}$, viz.  
     \begin{equation*}
         \int_{[0, \zeta_{\w}]} \mu(\dd h) \, \mathbbm{1}_{\{ \w(h) = x  \}} = 0.
     \end{equation*}
 \end{enumerate} 
 This subcollection of $\overline{\Theta}$ is denoted by $\overline{\Theta}_x$ and we will work with the process  $\big( (\rho, \overline{W}) , (\mathbb{P}_{\mu , \overline{\w}} : (\mu, \overline{\w}) \in \overline{\Theta}_x) \big)$.  Conditions (i) and (ii)  are natural, since as a particular consequence of the next lemma,   under $\mathbb{P}_{0,y,r_0}$ and $\mathbb{N}_{y,r_0}$  the Lévy  snake $(\rho, \overline{W})$ takes values in $\overline{\Theta}_x$.
\begin{lem}\label{te_quedas_en_Theta}
For every  $(\mu, \overline{\emph{w}})\in \overline{\Theta}_x$ and $(y,r_0) \in \overline{E}$,  the process $(\rho,\overline{W})$ under  $\mathbb{P}_{\mu,\overline{\emph{w}} }$ and $\mathbb{N}_{y,r_0}$  takes values in $\overline{\Theta}_x$. 
\end{lem}
\begin{proof}
First, we argue  that $\mathbb{N}_{y,r_0}$--a.e. the pair $(\rho_t, \overline{W}_t)$  satisfies (i) and (ii) for each $t \in [0, \sigma]$. 
On the one hand, by formula \eqref{tirage_au_hasard_N},  for every $(y,r_0) \in \overline{E}$ we have :
\begin{align*} 
\mathbb{N}_{y,r_0}\left(\int_{0}^{\sigma_{H}} \dd t \,  \langle \rho_{t},\{h \in [0,H_t]: \:W_{t}(h)=x\} \rangle \right)&=\int_{0}^{\infty}\dd a\:\exp(-\alpha a)\: E^0 \otimes \Pi_{y,r_0}\big[\int_{0}^{a} J_a(\dd h)  \, \mathbbm{1}_{\{\xi_{h}=x\}}\big] = 0.
\end{align*}
In the last equality we used that, by \ref{Asssumption_3} and  the  independence  between $\xi$ and $J_\infty$,  Campbell's formula yields 
$E^0 \otimes \Pi_{y,r_0}\big[\int_{0}^{\infty} J_{\infty}(\dd h) \,  \mathbbm{1}_{\{\xi_{h}=x\}}\: \big]=0$.  On the other hand, by construction of the Lévy snake, for each fixed $t \geq 0$ the support of $\Lambda_t(\dd h)$ is  the closure of  $\{ h \in [0,H_t) : W_t(h) = x \}$ in $[0,H_t]$, $\mathbb{N}_{y,r_0}$--a.e. 
\par 
Consequently, $\mathbb{N}_{y,r_0}$--a.e.~, we can  find a countable dense set $\mathcal{D} \subset [0,\sigma]$ such that  we have
\\
\\
\centerline{$\langle \rho_{t},\{h \in [0,H_t]: \:W_{t}(h)=x\} \rangle = 0$ \hspace{0.5cm} and\hspace{0.5cm} $\text{supp } \Lambda_t(\dd h)$ is the closure of  $\{  h \in [0,H_t) : W_t(h) = x \}$,}\\
\\
for every $t \in \mathcal{D}$. For instance, one can construct the set $\mathcal{D}$ by  taking an infinite sequence of independent uniform points in $[0,\sigma]$.  We now claim  that $\rho$ satisfies that $\mathbb{N}_{y,r_0}$--a.e., for every $s<t$, we have   $\rho_{s}\mathbbm{1}_{[0,m_H(s,t))} = \rho_{t}\mathbbm{1}_{[0,m_H(s,t))}$, where we recall the notation $m_{H}(s,t)=\min_{[s,t]}H$. Indeed,  remark that for fixed $s<t$, this holds by the Markov property and we can extend this property to all $0\leq s<t\leq \sigma$ since $\rho$ is  right-continuous with respect to the total variation distance. Now, by the snake property we deduce  that $\mathbb{N}_{y,r_0}$--a.e, for every $t \in [0,\sigma]$, we have
\begin{equation} \label{equation:NomassaRhox_2}
   \langle \rho_{t},\{h \in [0,H_t): \:W_{t}(h)=x\} \rangle = 0  
    \quad \quad \text{ and }\quad \quad \{ h \in [0,H_t) : {W}_t(h) = x \} = \text{supp } \Lambda_t(\dd h) \cap [0,H_t) .
\end{equation}
Taking the closure in the second equality  we deduce that the closure of  $ \{ h \in [0,H_t) : {W}_t(h) = x \}$ is exactly  $\text{supp } \Lambda_t(\dd h)$. However,  to conclude  that $\mathbb{N}_{y,r_0}$-a.e. 
\begin{equation} \label{equation:NomassaRhox_}
    \langle \rho_{t},\{h \in [0,H_t] : \:W_{t}(h)=x\} \rangle =0, \quad \quad \text{ for all } \,  t\in  [0, \sigma], 
\end{equation}
we still need an additional step.  Arguing  by contradiction,  suppose that for some $t>0$ the quantity  \eqref{equation:NomassaRhox_} is non-null. Then, by \eqref{equation:NomassaRhox_2} we must have  $\rho_t(\{ H_t \})>0$ and $W_t(H_t)=x$. By right-continuity of  $\rho$ with respect to the total variation metric, we get
\begin{equation*}
\lim_{\epsilon \rightarrow 0}|\rho_t(\{ H_t \}) -  \rho_{t+\epsilon}(\{ H_t \})|  = 0, 
\end{equation*}
and we deduce   that for $\epsilon$ small enough, $\rho_u(\{ H_t\})> 0$ for all  $u \in [t,t+\epsilon)$;  in particular $H(\rho_{u})\geq H(\rho_t)$ for all  $u \in [t,t+\epsilon)$. Since $W_t(H_t) = x$, the snake property ensures that ${W}_{u}(H_t) = x$ for all $u \in [t,t+\epsilon)$ and, since   $\rho_u(\{ H_t\})> 0$ for every   $u \in [t,t+\epsilon)$, we obtain a contradiction with the fact that $\langle \rho_{u},\{h \in [0,H_u]: \:W_{u}(h)=x\} \rangle = 0$ for every $u \in \mathcal{D}$. 
\par 
Let us now deduce this result  under $\mathbb{P}_{\mu , \overline{\w}}$.    First, observe that the statement of the lemma follows directly  under $\mathbb{P}_{0,y,r_0}$ by excursion theory. Next, fix $(\mu,\overline{\text{w}})\in \overline{\Theta}_{x}$ with $\overline{\w}(0) = (y,r_0)$, consider $(\rho,\overline{W})$ under $\mathbb{P}_{\mu,\overline{\w}}$ and set $T_{0}^+:=\inf\{t\geq 0:\: \langle \rho_{t} ,1 \rangle  =0\}$. The strong Markov property gives us that $((\rho_{T_{0}^+ +t},\overline{W}_{T_{0}^++t}):~t\geq 0)$ is distributed according to $\mathbb{P}_{0,y,r_0}$ and consequently,   $(\rho_{T_{0}^+ + t},\overline{W}_{T_{0}^++t})_{t\geq 0}$ takes values in $\overline{\Theta}_x$. To conclude,  it remains to prove the statement of the lemma under $\mathbb{P}^{\dag}_{\mu , \overline{\w}}$. In this direction, under $\mathbb{P}^{\dag}_{\mu , \overline{\w}}$, consider $\big((\alpha_i,\beta_i):~i\geq 0\big)$ the excursion intervals  of $\langle \rho,1\rangle$ from its running infimum. Then, write    $(\rho^i, \overline{W}^i)$ for the subtrajectories of $(\rho , \overline{W})$ associated with $[\alpha_i,\beta_i]$ and set $h_i:=H_{\alpha_i}$. We recall from \eqref{PoissonRandMeasure}  that the measure:
\begin{equation*} 
\sum \limits_{i\in \mathbb{N}}\delta_{(h_{i},\rho^{i},\overline{W}^{i})}, 
\end{equation*}
is a Poisson point measure with intensity $\mu(\dd h)\:\mathbb{N}_{\overline{\w}(h)}(\dd\rho,\:\dd \overline{W}).$ Since $(\mu,\overline{\w})\in\overline{\Theta}_x$, it follows by the result under the excursion measures $(\mathbb{N}_{y,r_0}: (y,r_0) \in \overline{E} )$ that $\mathbb{P}^\dag_{\mu , \overline{\w}}$--a.s. the pair  $(\rho_{t},\overline{W}_{t})$  belongs to $\overline{\Theta}_x$ for every  $t\in [0,T_{0}^+]$, as  wanted. 
\end{proof}

 Finally, recall that the snake property ensures that the function $(\widehat{W}_s, \widehat{\Lambda}_s )_{s \geq 0}$ is  well defined in the quotient space $\mathcal{T}_H$.  Hence, we can  think of $\overline{W}$ as a tree-indexed process, that we write with the usual abuse of notation as 
\begin{equation*}
    (\xi_\upsilon, \mathcal{L}_\upsilon)_{\upsilon \in \mathcal{T}_H}. 
\end{equation*}
\textbf{Main results of Section 4: }Now that we have introduced our framework, we can  state the  main results of this section. Much of our effort is devoted to  the construction of a measure supported on the set $\{ t\in \mathbb{R}_+ : \widehat{W}_t = x \}$ and  satisfying suitable  properties. In this direction, for every $r\geq 0$,   we set $\tau_r(\overline{\w}) := \inf \{  h \geq 0 :~\overline{\w}(h)=(x, r) \}$ and remark that,  for every  $(\mu , \overline{\w}) \in \overline{\Theta}_x$, it holds that  $\mu(\{\tau_r(\overline{\w})\}) = 0,$ with the convention $\mu(\infty) = 0$. We can now state the main result of this section:
\begin{theo}\label{thm:existenciaAditivaprelim}
Fix $(y,r_0)\in\overline{E}$ and $(\mu,\overline{\emph{w}})\in\overline{\Theta}_x$.
The  convergence 
\begin{equation} \label{equation:framework-aditiva}
 A_t = \lim_{\varepsilon \downarrow 0} \frac{1}{\varepsilon}  \int_0^t \dd u  \int_{\mathbb{R}_+} \dd r \,  \mathbbm{1}_{\{ \tau_r(\overline{W}_u ) < H_u < \tau_r(\overline{W}_u) + \varepsilon \}} ,
\end{equation}
 holds uniformly in compact intervals  in measure under $\mathbb{P}_{\mu, \overline{\emph{w}} }$ and $\mathbb{N}_{y,r_0}( \, \cdot \, \cap \{ \sigma > z \} )$  for every   $z >0$. Moreover, \eqref{equation:framework-aditiva}  defines a continuous additive functional $A = (A_t)$ for the Lévy snake $(\rho , \overline{W})$  whose Lebesgue-Stieltjes measure $\dd A$ is supported on $\{ t \in \mathbb{R}_+ : \widehat{W}_t = x \}$.
\end{theo}

We will  give another equivalent construction of the additive functional $A$  in  Proposition \ref{proposition:aditivaDefinicion} but we are   not yet in position to formulate the precise statement. Both constructions will be needed for our work. Next,  the second main result of the section characterizes the support of the measure $\dd A$ as follows: 
\begin{theo} \label{theorem:supportIntro}
Fix $(y,r_0)\in\overline{E}$,  $(\mu,\overline{\emph{w}})\in\overline{\Theta}_x$ and 
denote  the support of the Stieltjets  measure of $A$ by $\emph{supp }\dd A$. Under $\mathbb{N}_{y,r_0}$ and $\mathbb{P}_{\mu,\overline{\emph{\w}}}$, we have:
\begin{equation} \label{equation:supportIntro}
    \emph{supp } \, \dd A = \overline{\big\{ t \in [0,\sigma] : \, \xi_{p_H(t)} = x \emph{ and } p_H(t) \in \emph{Multi}_2(\mathcal{T}_H) \cup \{ 0 \}  \big\}}.
\end{equation}
\end{theo}
Observe that under $\mathbb{P}_{\mu , \overline{\w}}$ with $\w(0) = x$, the root of $\mathcal{T}_H$ has infinite multiplicity and this is  why we had to consider it separately in the previous display.  This result is stated in a slightly different but equivalent form  in Theorem \ref{prop:suppA}.  The identity \eqref{equation:supportIntro} can be also formulated in terms of constancy intervals of $\widehat{\Lambda}$. More precisely, we will also establish in Theorem \ref{prop:suppA} that under $\mathbb{N}_{y,r_0}$ and $\mathbb{P}_{\mu,\overline{\w}}$, we have:
\begin{equation} \label{identity:framework:soporteConstanciaLambda}
  \text{supp } \dd A = [0,\sigma] \setminus \big\{ t \in [0,\sigma] : ~ \sup_{(t-\varepsilon, t+\varepsilon )} \widehat{\Lambda}_{s} =  \inf_{(t-\varepsilon, t+\varepsilon )} \widehat{\Lambda}_{s}, \quad  \text{ for some } \varepsilon > 0  \big\}.
\end{equation}
We conclude the presentation of our framework  with a consequence of Lemma \ref{te_quedas_en_Theta}.  Roughly speaking it states that, with the exception of the root under $\mathbb{P}_{0,x,0}$, the process $(\xi_\upsilon)_{\upsilon \in \mathcal{T}_H}$ can not take the value $x$ at the branching points of $\mathcal{T}_H$.  The precise statement is the following: 
\begin{prop}\label{prop:branchingNotocax} For every $(y,r_0) \in \overline{E}$ and $(\mu,\overline{\mathrm{w}})\in \overline{\Theta}_x$, we have: 
\begin{equation*}
    \{ t \in [0,\sigma] : \widehat{W}_t = x  \} \cap  \{ t \in [0,\sigma]: p_H(t) \in \emph{Multi}_3 (\mathcal{T}_H) \cup \emph{Multi}_\infty(\mathcal{T}_H) , \, {p}_H(t) \neq 0 \} = \emptyset,
\end{equation*}
under $\mathbb{N}_{y,r_0}$ and  $\mathbb{P}_{\mu,\overline{\emph{w}}}$.
\end{prop}

\begin{proof}
We start by proving our result under $\mathbb{N}_{y,0}$.  First, introduce the measure $\mathbb{N}^\bullet_{y,0}(\dd\rho, \dd \overline{W}, \, \dd s  )$ supported on $\mathbb{D}(\mathbb{R}_+, \mathcal{M}_f(\mathbb{R}_+) \times \mathcal{W}_{\overline{E}}) \times \mathbb{R}_+$ defined by $\mathbb{N}^\bullet_{y,0} = \mathbb{N}_{y,0}(\dd \rho , \dd \overline{W}) \, \dd s \mathbbm{1}_{\{ s \leq \sigma \}}$ and write  $U : \mathbb{R}_+ \mapsto \mathbb{R}_+$ for the identity function $U(s) = s$. The law under $\mathbb{N}^\bullet_{y,0}$ of $(\rho , \overline{W}, U)$ is therefore  given by
\begin{equation*}
    \mathbb{N}^{\bullet}_{y,0}\Big( \Phi(\rho, \overline{W}, U) \Big) = \mathbb{N}_{y,0}\Big( \int_0^\sigma \dd s \,  \Phi( \rho, \overline{W} , s) \Big). 
\end{equation*} The measure  $\mathbb{N}^{\bullet}_{y,0}$ can be seen as a pointed version of $\mathbb{N}_{y,0}$. In particular, conditionally on $(\rho, \overline{W})$, the random variable $U$ is  a uniform point in $[0,\sigma]$. Under $\mathbb{N}_{y,0}^\bullet$ we still write $X_t := \langle \rho_t ,1  \rangle$ and $H_t:=H(\rho_t)$. Furthermore,  we set  $X^{\bullet}_t:= X_{U+ t} - X_U$ and   $I^{\bullet }_t := \inf _{s \leq t } X^{\bullet}$ , for every $t \geq 0$, and  we  denote  the excursion intervals over the running infimum of $X^{\bullet}$ by $\big( (\alpha_i, \beta_i) : \, i \in \mathbb{N} \big)$. The dependence on $U$ is dropped to simplify notation. Finally, set  
\begin{align*}
     h_i^{\bullet}:= H\big( \kappa_{-I^{\bullet}_{\alpha_i}} \rho_U \big) ,
\end{align*}
and write  $(\rho^{\bullet , i}, \overline{W}^{\bullet , i})$ for the corresponding subtrajectory associated with $(\alpha_i , \beta_i)$ occurring at  height $h_i^{\bullet}$. Under $\mathbb{N}^\bullet_{y,0}$,  the Markov property applied at time $U$ and \eqref{PoissonRandMeasure} gives that, conditionally on $(\rho_U, W_U)$,  the random measure 
\begin{equation*}
    \mathcal{\mathcal{M}}^\bullet := \sum_{i\in \mathbb{N}} \delta_{ ({ h_i^{\bullet}, \,  \rho^{\bullet , i},\, \overline{W}^{\bullet , i}})},  
\end{equation*}
is a Poisson point measure  with intensity 
$ \rho_U(\dd h) \,  \mathbb{N}_{\overline{W}_U(h)} (\dd \rho, \dd {W} ). $ In particular,  the functional
\begin{equation*}
    F(\mathcal{M}^\bullet) = \# \Big\{ (h_i^{\bullet},\rho^{\bullet,i}, \overline{W}^{\bullet,i} ) \in \mathcal{M}^\bullet : W^{\bullet,i}(0) = x \Big\}, 
\end{equation*}
conditionally on $(\rho_U , W_U)$,  is a Poisson random variable with parameter $\int \rho_U(dh) \mathbbm{1}_{\{ W_U(h) = x \}}$. However, by Lemma \ref{te_quedas_en_Theta}, we have 
$\int \rho_U(dh) \mathbbm{1}_{\{ W_U(h) = x \}} =0$ and we derive  that, $\mathbb{N}^{\bullet}_{y,0}$ --a.e.,  $F(\mathcal{M}^\bullet)$ is null. Heuristically,   the previous argument shows that if we take --  conditionally on $\sigma$ -- a  point  uniformly at random in $\mathcal{T}_H$, there is no branching point $\upsilon$ with label $x$ on the right of the branch connecting the root to $\upsilon$. Let us now show that this ensures that
\begin{equation*}
    \big\{ t \in [0,\sigma] : \widehat{W}_t = x  \big\} \cap  \big\{ t \in [0,\sigma] : p_H(t) \in \text{Multi}_3 (\mathcal{T}_H) \cup \text{Multi}_\infty(\mathcal{T}_H) \big\} = \emptyset,  \quad \mathbb{N}_{y,0} \text{--a.e.}
\end{equation*}
Since $\mathbb{N}_{y,0}^{\bullet}(\Phi(\rho,\overline{W})) = \mathbb{N}_{y,0}(\sigma \cdot  \Phi(\rho,\overline{W}))$, it suffices to prove  the previous display  under $\mathbb{N}^\bullet_{y,0}$. Let $(\upsilon_i:i\in \mathbb{N})$ be an indexation of the branching points of $\mathcal{T}_H$ -- an indexation measurable with respect to $H$.
Pick a  branching point  $\upsilon_i \in \text{Multi}_3 (\mathcal{T}_H) \cup \text{Multi}_\infty(\mathcal{T}_H)$ and let  $t_i$ be  the smallest element of $p_H^{-1}(\upsilon_i)$. Arguing  by contradiction,  suppose that $\widehat{W}_{p_H(t_i)} = x$. Still  under $\mathbb{N}^{\bullet}_{y,0}$,  since $\upsilon_i$ is a branching point,  we can find $0 \leq s_* < t_*\leq \sigma$ in $p_H^{-1}(\{\upsilon_i\})$    such that the event
\begin{equation*}
    \{  \widehat{W}_{p_H(t_i)} = x  \} \cap \{ s_* < U <t_* \}, 
\end{equation*}
is included in $\{ F(\mathcal{M}^U) > 0 \}$.
However  $F(\mathcal{M}^U)= 0, \, \mathbb{N}^{\bullet}_{y,0}$--a.e. and we deduce that $\mathbb{N}^{\bullet}_{y,0}\big(   \widehat{W}_{p_H(t_i)} = x  , s_* < U <t_* \big) = 0$.  Next, since conditionally on $(\rho,\overline{W})$, the variable $U$ is uniformly distributed on $[0,\sigma]$, we conclude that  $\mathbb{N}^{\bullet}_{y,0}\big(   \widehat{W}_{p_H(t_i)} = x \big) = 0$. The desired result now follows, since the collection of branching points $(\upsilon_i:i\in \mathbb{N})$ is countable. Finally, we deduce the statement under $\mathbb{N}_{y,r_0}$ by the translation invariance of the local time and under $\mathbb{P}_{\mu,\overline{\w}}$  by excursion theory -- we omit the details  since this is standard and one can apply  the method described in Lemma  \ref{te_quedas_en_Theta}.
\end{proof}

%%% End prelim+Framework secction 4

\noindent \textit{The section is organised as follows:} In Section \ref{subsection:PhiTildeyMarkovTloc} we address several preliminary results needed to prove  Theorems  \ref{thm:existenciaAditivaprelim} and \ref{theorem:supportIntro}  and our results of Section \ref{section:treeStructureLocalTime}. More precisely,   Section \ref{subsection:PhiTildeyMarkovTloc} is essentially devoted to the study  of a family of exit local times of $(\rho , \overline{W})$ that will be used as  building block for our second construction of  $A$. Then in Section \ref{subsection:existenciaAditiva} we  shall prove Theorem \ref{thm:existenciaAditivaprelim}   and establish  our second construction of $A$ in terms of the family of exit times studied in Section \ref{subsection:PhiTildeyMarkovTloc}.  The rest of the section is dedicated to  the study of basic properties of $A$ that we will frequently use in our computations. Finally,  in Section \ref{subsection:characterizationSupport} we  turn our attention  to the study of  the support of the measure  $\dd A$ and it will lead us to the proof of Theorem \ref{theorem:supportIntro} and the characterisation  \eqref{identity:framework:soporteConstanciaLambda}.

\subsection{Special Markov property of the local time}\label{subsection:PhiTildeyMarkovTloc}
The first step towards constructing our additive functional $A$, with associated Lebesgue-Stieltjes measure  $\dd A$ supported  in $\{ t\in \mathbb{R}_+ : \widehat{W}_t = x \}$, consists in the study of a particular family of $[0,\infty)$-indexed  exit local times of $(\rho,\overline{W})$. More precisely,  for each $r \geq  0$, let $D_r \subset \overline{E} := E \times \mathbb{R}_+$ be the open domain 
\begin{equation*}
    D_r := \overline{E} \setminus\{(x,r)\}
    \quad  \text{ and recall the notation  } \quad
    \tau_r(\overline{\w}) := \inf \{  h \geq 0 :~\overline{\w}(h)=(x,r)\},
\end{equation*}
for every $\overline{\w}\in \mathcal{W}_{\overline{E}}$. 
Notice that $\mathcal{\tau}_r(\overline{\w})$ is the exit time from $D_r$ and we write $\mathcal{\tau}_r$ instead of making use of the more cumbersome notation $\mathcal{\tau}_{D_r}$. We also recall that since   $\tau_{r}(\overline{\w}) \in \{ h \in [0,\zeta_\w] : \w(h) = x \}$ as soon as $\tau_r(\overline{\w}) < \infty$, for $(\mu , \overline{\w})\in \overline{\Theta}_x$ we have $\mu(\{ \tau_{r}(\overline{\w}) \}) = 0$.  Proposition \ref{proposition:aproxLDPmu} now yields that for any $(\mu , \overline{\w}) \in \overline{\Theta}_x$ with $\overline{\w}(0) \neq (x,r)$ we have 
\begin{equation}\label{L:r}
L^{D_r}_t(\rho,\overline{W}):=\lim \limits_{\epsilon\to 0}  \frac{1}{\epsilon} \int_0^t \dd s \,  \mathbbm{1}_{\{ \tau_{r}(\overline{W}_s) < H_s < \tau_{r}(\overline{W}_s) + \epsilon \}},
\end{equation}
 where the convergence holds uniformly in compact intervals in $L^1(\mathbb{P}_{\mu , \overline{\w}})$ and $L^1(\mathbb{N}_{\overline{\w}(0)})$. Let us be more precise: recalling the notation $\overline{\w} = (\w , \ell)$,  
first remark that if $\ell(0) < r$,  for any $\w(0) \in E$  we have $\Pi_{\w(0) , \ell(0)}(\tau_r < \infty) = 1$ and in consequence \ref{tau_infinity} holds.  On the other hand, if $r<\ell(0)$,  we simply  have  $L^{D_r} = 0$ since $\tau_{r}(\overline{W}_s)=\infty$ for every $s \geq 0$. Finally, if $ \w(0)  \neq x$ and $r \geq 0$, we have $\tau_{D_{r}}(\overline{W}_s)=\tau_{E_*}( W_s)$ for every $s\geq 0$, and  recalling  \eqref{equation:recurrent} it follows that    $L^{D_r}(\rho,\overline{W}) = L^{E_*}(\rho,W)$. It will be usefully for our purposes to extend the definition to the remaining case $\overline{\w}(0) = (x,r)$, and that is precisely the content of the following lemma: 
\begin{lem}\label{L:r:Lemma}
For $r \geq 0$,  fix $(\mu,\overline{\rm{w}})=(\mu,\rm{w},\ell)$ $\in \overline{\Theta}_x$ with $\overline{\rm{w}}(0) = (x,r)$. Then, the limit \eqref{L:r} 
exists for every $t\geq 0$, where the convergence holds uniformly in compact intervals  in  $L^1(\mathbb{P}_{\mu,\overline{\rm{w}}})$ and $L^1(\mathbb{N}_{x,r})$, and defines a continuous non-decreasing process that we still denote by $L^{D_r}$. Moreover, under $\mathbb{P}^\dag_{\mu , \overline{\rm{w}}}$ and $\mathbb{N}_{x,r}$, we have $L^{D_r}_\sigma = 0$.
\end{lem}
\begin{proof}
{ We work under $\mathbb{P}_{\mu,\overline{\rm{w}}}$ since the result under $\mathbb{N}_{x,r}$ follows directly by excursion theory.
For every $a \geq 0$ we set $T_a:=\inf\{t\geq 0:~H_t=a\}$  and let $T_0^{+}:=\inf\{t\geq 0:~\langle \rho_t, 1\rangle=0\}$. Since $\mathcal{\tau}_{r}(\overline{W}_s) = 0$ for every $s \geq 0$,  we have
$$ \int_0^t \dd s \,  \mathbbm{1}_{\{ \tau_{r}(\overline{W}_s) < H_s < \tau_{r}(\overline{W}_s) + \epsilon \}}=\int_0^t \dd s \,  \mathbbm{1}_{\{0 < H_s <  \epsilon \}}=\int_{T_\epsilon \wedge t}^{T_0^+ \wedge t} \dd s \,  \mathbbm{1}_{\{ 0 < H_s <  \epsilon \}}+\int_{T_0^+ \wedge t}^{T_0^+ \wedge t} \dd s \,  \mathbbm{1}_{\{0< H_s <  \epsilon \}}.$$
Furthermore, by the strong Markov property and  \eqref{temps:local:I}, we already know that $\epsilon^{-1} \int_{T_0^+}^{T_0^++t} \dd s \,  \mathbbm{1}_{\{ 0 < H_s <  \epsilon \}}$  converges as $\epsilon \downarrow 0$ uniformly in compact intervals in  $L^{1}(\mathbb{P}_{\mu,\text{w}})$. To conclude, it suffices to show that:
\begin{equation}\label{L:0:0}
   \lim \limits_{\epsilon\to 0}\epsilon^{-1}\cdot \mathbb{E}_{\mu,\overline{\w}}\big[\int_{T_\epsilon}^{T_0^+} \dd s \,  \mathbbm{1}_{\{ 0 < H_s <  \epsilon \}}\big]=0. 
\end{equation}
Write  $(\alpha_i, \beta_i)$ for $i \in \mathbb{N}$  the excursion intervals of the killed process $(\langle \rho_t,1 \rangle:~t\in[0,T_0^{+}])$ over its running infimum  and  let  $(\rho^i, \overline{W}^i)$ be the subtrajectory associated with the excursion interval $[\alpha_i , \beta_i]$. To simplify notation, we also set $h_i=H(\alpha_i)$ and recall from \eqref{PoissonRandMeasure}  that the measure $\mathcal{M} := \sum_{i \in \mathbb{N}}\delta_{(h_i, \rho^i , \overline{W}^i)}$ is a Poisson point measure with intensity $\mu(\dd h)\mathbb{N}_{\overline{\w}(h)}(\dd \rho, \dd \overline{W})$. Next, notice that:
\begin{align*}
    \int_{T_\epsilon}^{T_{0}^+} \dd s \,  \mathbbm{1}_{\{ 0 < H_s < \epsilon \}}\leq \sum \limits_{0\leq h_i\leq \epsilon}\int_{0}^{\sigma(\overline{W}^i)} \dd s \,  \mathbbm{1}_{\{ 0 < H(\rho_s^i) <  \epsilon \}}, 
\end{align*}
and we can now use that $\mathcal{M}$ is a Poisson point measure with intensity $\mu(\dd h)\mathbb{N}_{\overline{\w}(h)}(\dd \rho, \dd \overline{W})$ to get that:
\begin{align*}
 \mathbb{E}_{\mu,\overline{\w}}\big[\int_{T_\epsilon}^{T_{0}^+} \dd s \,  \mathbbm{1}_{\{ 0 < H_s <  \epsilon \}}  \big] 
 \leq \mu([0,\epsilon])N(\int_{0}^{\sigma}\dd s \mathbbm{1}_{\{0 < H(\rho_s)<\epsilon\}}).
\end{align*}
Finally, by the many-to-one formula \eqref{tirage_au_hasard_N}, the previous display is 
$\epsilon\cdot\mu([0,\epsilon])$, which gives:
\begin{equation*}
    \limsup \limits_{\epsilon\to 0}\epsilon^{-1}\cdot\mathbb{E}^{\dag}_{\mu,\w}\Big[ \int_{T_\epsilon}^{T_{0}^+} \dd s \,  \mathbbm{1}_{\{ 0 < H_s < \epsilon \}}\Big]=\mu(\{0\}).
\end{equation*}
  Now \eqref{L:0:0} follows since we have $\mu(\{0\})=0$, which holds since  $\w(0)=x$  and $(\mu,\overline{\w})\in \overline{\Theta}_x$.
}
\end{proof}
Now, we give  a regularity result for the double-indexed family $({L}_{s}^{D_r} : (s,r)\in \mathbb{R}_+^2)$ that will be needed  in the next section. 
\begin{lem}\label{def:lem:scr:L}
Let $(\mu,\overline{\rm{w}})\in \overline{\Theta}_x$ with $\overline{\rm{w}}=(\rm{w},\ell)$. There exists a $\mathcal{B}(\mathbb{R}_+)\otimes \mathcal{B}(\mathbb{R}_+) \otimes \mathcal{F}$-measurable function 
 $(\mathscr{L}_t^r : \,  (r,t) \in \mathbb{R}_+^2 )$ satisfying the following properties:
\begin{itemize}
    \item [\rm{(i)}] For every $r\geq 0$, the processes    ${L}^{D_r}$ and $\mathscr{L}^r$  are indistinguishables under $\mathbb{P}_{\mu,\overline{\rm{w}}}$.
    \item [\rm{(ii)}]  $\mathbb{P}_{\mu,\overline{\rm{w}}}$ almost surely, the mapping $t \mapsto \mathscr{L}^r_t$ is continuous for every $r \geq 0$.
\end{itemize}
\end{lem}
The result also holds under the measure $\mathbb{N}_{y,r_0}$, for every $(y,r_0)\in \overline{E}$, by the same type of arguments and we omit the details. 
\begin{proof}
Fix an initial condition $(\mu,\overline{\rm{w}})=(\mu, \w, \ell)\in \overline{\Theta}_x$. Since under $\mathbb{P}_{\mu,\w,\ell}$, the distribution of  $(\rho,W,\Lambda-\ell(0))$  is exactly  $\mathbb{P}_{\mu,\w,\ell-\ell(0)}$,  without loss of generality we might assume that $\ell(0)=0$.
 Next, by \eqref{L:r} and Lemma \ref{L:r:Lemma}, for every $r\geq 0$ we have
\begin{equation*}
    \lim_{\varepsilon \downarrow 0} \mathbb{E}_{\mu,\overline{\rm{w}}}\Big[ \sup_{s \leq t} 
    |{L}^{D_r}_s - \varepsilon^{-1}\int_0^s \dd u \mathbbm{1}_{\{ \tau_r(\overline{W}_u ) < H_u < \tau_r(\overline{W}_u) + \varepsilon \}}  |
    \Big] = 0, 
\end{equation*}
and hence for any subsequence $(\epsilon_n)$ converging to $0$, the sequence of processes 
\begin{equation*}
 Y_n(r, t) :=   \varepsilon_n^{-1}\int_0^t \dd u~ \mathbbm{1}_{\{ \tau_r(\overline{W}_u ) < H_u < \tau_r(\overline{W}_u) + \varepsilon_n \}}, \quad \quad t \geq 0,
\end{equation*}
converges uniformly in compact intervals in probability towards ${L}^{D_r}$. Now, to simplify notation write $\bm{\omega} := (\upvarrho, \omega)$ for the elements of  $\mathbb{D}(\mathbb{R}_+, \mathcal{M}_f(\mathbb{R}_+)\times \mathcal{W}_{\overline{E}} )$.  Remark now that the mapping $(u,r,\bm{\omega})\mapsto \tau_r(\overline{W}_u(\omega))$ is jointly measurable since for each $(u,\bm{\omega})$  it is càdlàg in $r$, while for each fixed $r$ the mapping $ (u,\bm{\omega})\mapsto \tau_r( \overline{W} _u(\omega))$ is measurable. Consequently, by Fubini, for every fixed $t$ the application 
\begin{equation*}
    (r,\bm{\omega}) \mapsto \int_0^t \dd u~ \mathbbm{1}_{\{ \tau_r(\overline{W}_u ) < H_u < \tau_r(\overline{W}_u) + \varepsilon_n \}} (\omega)
\end{equation*}
is measurable  while for fixed $(r,\bm{\omega})$ it is continuous in $t$, and we deduce that $Y_n$ is jointly measurable in $(r,t,\bm{\omega})$. It is now standard -- see e.g. \cite[Theorem 62]{ProtterStochIntegration} and its proof -- to deduce that  there exits a jointly measurable process $(r,t,\bm{\omega}) \mapsto Y(r,t, \bm{\omega})$ such that for every $(r, \bm{\omega}) \in \mathbb{R}_+ \times \mathbb{D}(\mathbb{R}_+, \mathcal{M}_f(\mathbb{R}_+)\times \mathcal{W}_{\overline{E}})$, the mapping  $t \mapsto Y(r,t, \bm{\omega} )$ is continuous and for each fixed $r\geq 0$, $Y_n(r,\cdot) \mapsto Y(r, \cdot)$ as $n \uparrow \infty$ uniformly in compact intervals in probability. In particular for each $r\geq 0$, the process $(Y(r,t) : t \geq 0 )$ is indistinguishable from $({L}^{D_r}_t:~ t \geq 0)$ and we shall write $(\mathscr{L}^{\,r}_t:~ t \geq 0, \, r \geq 0)$ instead of  $(Y(r,t) : \, t \geq 0, \, r \geq 0 )$. 
\end{proof}
We now turn our attention to the Markovian properties of $(\mathscr{L}^{\,r}_\sigma : r \geq  0)$ under the excursion measure $\mathbb{N}_{x,0}$. To simplify notation, for every $y\neq x$, we set:
\begin{equation} \label{definition:ulambda}
    u_\lambda(y) :=\mathbb{N}_{y,0} \left(1-\exp (- \lambda \mathscr{L}^{\,0}_\sigma) \right), \quad \text{ for }  y \in E_*,
\end{equation}
and remark that with the notation of \eqref{definition:u_g_lambda} we have $u_{\lambda}(y)=u_{\lambda}^{E_*}(y)$. We shall use the usual convention
$u_\lambda(x) = \lambda$. 
\par 
Before stating our next result, we briefly recall from \cite[Chapter II-1]{RefSnake} that an $\mathbb{R}_+$--valued Markov process with semigroup $(P_t(y,\dd z) :~ t, y \in \mathbb{R}_+ )$ is called a branching process if its semigroup satisfies the branching property, viz. if for any $y,y' \in \mathbb{R}_+$, we have $P_t(y,\cdot)*P_t(y',\cdot) = P_t(y+y', \cdot)$. In order to fall in the framework of \cite[Chapter II- Theorem 1]{RefSnake} we also assume that $\int_{\mathbb{R}_+} P_t(y,\dd z) z \leq y$ for every $t,y \in \mathbb{R}_+$. By the branching property it follows that for any $t,y\in \mathbb{R}_+$  the distribution  $P_t(y,\dd z)$ is  infinitely divisible and  non-negative, and  consequently, for every $t,y \in \mathbb{R}_+$,  the Laplace transform of  $P_t(y,\dd z)$ takes the Lévy-Khintchine form:
\begin{equation*}
    \int_{\mathbb{R}_+} \,  P_t(y, \dd z)\exp(-\lambda z)= \exp \big(-y a_t(\lambda) \big), \quad \quad \text{ for } \lambda \geq 0,
\end{equation*}
for some function $(a_t(\lambda): t, \lambda \geq 0)$.
By \cite[Chapter II- Theorem 1]{RefSnake},  the function $(a_t(\lambda): t  , \lambda \geq 0)$ is the unique non-negative solution of the integral equation 
\begin{equation} \label{equation:branchingEquation}
    a_t(\lambda) + \int_0^t \dd u \,  \Psi(  a_u(\lambda) )  = \lambda,  
\end{equation}
for a function $(\Psi(\lambda) : \lambda \geq 0)$ of the form,  
\begin{equation*}
    \Psi(\lambda) = c_1 \lambda + c_2 \lambda^2 + \int_{\mathbb{R}_+} \nu(\dd y) \,  (\exp(-\lambda y) -1 + \lambda y  ), \quad \quad \text{ for } \lambda \geq 0, 
\end{equation*}
where $c_1, c_2 \in \mathbb{R}_+$ and $\nu$ is a Lévy measure supported on $(0,\infty)$ satisfying $\int_{(0,\infty)} \nu(\dd y) (y \wedge y^2) < \infty$. By \eqref{equation:branchingEquation}, it follows  that $a_t(\lambda)$ is the unique function that satisfies  
\begin{equation*}
    \int_{a_t(\lambda)}^\lambda \frac{\dd s}{\Psi(s)} = t,\quad \quad t,\lambda\geq 0.  
\end{equation*}
The Markov process with semigroup $(P_t)$ is then called a CSBP with branching mechanism $\Psi$, or in short a $\Psi$-CSBP. The exponent $\Psi$ clearly fulfils (A1) and so does (A3) by \cite[Corollary 2]{BertoinBook}. So to fall in our framework, it only need to satisfy (A4) -- since  as explained in the preliminaries (A4)  implies (A2). 
\begin{prop} \label{lemma:divergExpo} 
Under $\mathbb{N}_{x,0}$, the process $(\mathscr{L}^{\,r}_\sigma: r > 0)$ is a branching process with entrance measure $\nu_r( \dd x) = \mathbb{N}_{x,0}(\mathscr{L}^{\,r}_\sigma \in \dd x)$, for $r > 0$, and branching mechanism 
\begin{equation} \label{definition:PhiTilde}
    \widetilde{\psi}(\lambda) = \mathcal{N}\Big( \int_0^\sigma\dd h ~ \psi\big(u_\lambda(\xi_h)\big)\Big), \quad \text{ for } \, \lambda \geq 0. 
\end{equation}
 Moreover, $\widetilde{\psi}$ satisfies the assumptions \emph{(A1)} -- \emph{(A4)} introduced in Section \ref{subsection:height} and consequently we can associate to it a Lévy tree.
\end{prop}
Our result is a particular case, in the terminology of Lévy snakes, of  Theorem 4 in \cite{BertoinLeGallLeJean} stated in the setting of superprocesses.   Theorem 4 in \cite{BertoinLeGallLeJean} is  more general and the family $(\mathscr{L}^r_\sigma)_{r > 0}$ in our result correspond precisely to the total mass process of the superprocess considered in \cite{BertoinLeGallLeJean}, for the same branching mechanism $\widetilde{\psi}$.
\begin{proof} 
The proof is structured as follows: we start by introducing a family of probability kernels  $(P_t)$ and  by showing that they form a semigroup of operators  associated with a branching process. We then establish that $(\mathscr{L}^{\,r}_\sigma :~ r > 0)$ is a Markov process associated with the semigroup $(P_t)$, with entrance measure $(\nu_r: \, r >0)$. Finally, we conclude the proof by establishing that its branching mechanism is $\widetilde{\psi}$ and that it fulfils (A4). 
\par
We stress that we  are only interested in the finite-dimensional distributions of $(\mathscr{L}^{r}_\sigma : r >  0)$. Recalling the notation \eqref{definition:u_g_lambda}, for any $r > 0$ and $\lambda \geq 0$, we write  
\begin{equation*}
     u_\lambda^{D_r}{(x,0)}
= \mathbb{N}_{x,0}\big(1-\exp (-\lambda \mathscr{L}^{\,r}_\sigma ) \big)
= \int \nu_r(\dd y) \,  ( 1-\exp(-\lambda y) ).  
\end{equation*}
Moreover, since  $\mathbb{N}_{x,0}(\mathscr{L}^{\,r}_\sigma > 0) \leq \mathbb{N}_{x,0}(\sup \widehat{\Lambda} \geq  r ) < \infty$, we have $\int_{(0,\infty)}\nu_r(\dd y) (1 \wedge y)  < \infty$, and we deduce that the function  $\lambda\mapsto u_\lambda^{D_r}(x,0)$ is the Laplace exponent of an infinitely divisible random variable with Lévy measure $\nu_r( \cdot \cap (0,\infty))$.  For each $t>0$ and $y \in \mathbb{R}_+$ denote by $P_t(y,\dd z)$ the  probability  measure with Laplace transform 
\begin{equation}\label{equation:laplaceBranchingSemigroup}
    \int P_t(y,\dd z ) \, \exp(-\lambda z )= \exp\big(-y\cdot u_\lambda^{D_t}(x,0)\big), \quad \quad \lambda \geq 0.
\end{equation}
Remark now that the translation invariance of the local time of $\xi$ implies that, under $\mathbb{P}_{0,x,r}$ (resp. $\mathbb{N}_{x,r}$) for  $r\geq 0$,  the distribution of $(W,\Lambda-r)$ is $\mathbb{P}_{0,x,0}$ (resp. $\mathbb{N}_{x,0}$). In particular, for every $s,t\geq 0$, we have
\begin{equation*}
    u_\lambda^{D_{t+s}}(x,s) = u_\lambda^{D_t}(x,0). 
\end{equation*}
We deduce that the family $(P_t(y,\dd z), t > 0,  y \in \mathbb{R}_+)$ is a semigroup since, by the special Markov property applied at the domain $D_s$, it holds that
\begin{align*}
    \int P_{t+s}(y,\dd z) \exp (- \lambda z)
    &= \exp \Big( -y \, \mathbb{N}_{x,0}\Big(1-\exp \big(-\lambda \mathscr{L}^{\,t+s}_\sigma\big) \Big) \Big) \\
    &=  \exp \Big( -y \, \mathbb{N}_{x,0}\Big(1-\exp \big(- \mathscr{L}^{\,s}_\sigma \cdot u_\lambda^{D_{t+s}}(x,s) \big) \Big) \Big) \\
    &=  \exp \Big( -y \, \mathbb{N}_{x,0}\Big(1-\exp \big(- \mathscr{L}^{\,s}_\sigma \cdot u_\lambda^{D_t}(x,0) \big) \Big) \Big) 
    = \exp\Big( -y \cdot u^{D_s}_{ u_\lambda^{D_t}(x,0)}\big(x,0\big)  \Big),
\end{align*}
 which coincides with the Laplace transform of the measure $ \int_{u \in \mathbb{R}_+} P_s(y,\dd u)P_t(u,\dd z)$. Since $\mathbb{N}_{x,0}(\mathscr{L}^{\,r}_\sigma) \leq 1$ by \eqref{tirage_au_hasard_ExitN} and $1-\exp (-\lambda \mathscr{L}^{\,r}_\sigma) \leq \lambda \mathscr{L}^{\,r}_\sigma$, we deduce by dominated convergence and \eqref{equation:laplaceBranchingSemigroup} that 
 $
     \int_{\mathbb{R}_+} P_t(y,\dd z) \, z = y \cdot \mathbb{N}_{x,0}(\mathscr{L}^{\,r}_\sigma) \leq y.  
 $
 Since  the semigroup clearly fulfils  the branching property, it follows that there exits a CSBP associated with the semigroup  $(P_t)$.
\par 
 Recall the notation $T_{D_\epsilon} :=  \inf \{ t \geq 0 : \tau_{\epsilon}(\overline{W}_t) < \infty \}$ as well as the definition of the sigma field $\mathcal{F}^{D_\epsilon}$ from \eqref{definition:sigmaFieldTroncature}.  We will now show that for any $\varepsilon >0$,  the process $(\mathscr{L}^{\,\epsilon+r}_\sigma : r \geq 0 )$ under the probability measure $\mathbb{N}_{x,0}^{D_\varepsilon } := \mathbb{N}_{x,0}(\, \cdot \, | T_{D_\epsilon} < \infty )$ has  transition kernel $(P_t)$.  Fix $\varepsilon < a < b$; by considering the point process  of excursions \eqref{specialMarkov_v3} outside $D_a$, we deduce by an application of the special Markov property that 
\begin{align*}
    \mathbb{N}_{x,0}^{D_\varepsilon } \left( \exp \left(- \lambda \mathscr{L}^{\,b}_\sigma  \right) \big| \mathcal{F}^{D_a} \right) 
    = \exp \left(- \mathscr{L}^{\,a}_\sigma \mathbb{N}_{x,a}\left( 1-\exp(-\lambda \mathscr{L}^{\,b}_\sigma ) \right)   \right) 
    =  \exp \left(- \mathscr{L}^{\,a}_\sigma \cdot u_\lambda^{D_{b-a}}(x,0)    \right),  \quad \mathbb{N}_{x,0}^{D_\varepsilon } \text{ --a.s.},
\end{align*}
where in the last equality we used the translation invariance of the local time of $\xi$. We have obtained that, for every $\varepsilon> 0$,   $(\mathscr{L}^{\,r+\epsilon}_\sigma :~ r \geq 0)$ under $\mathbb{N}^{D_\varepsilon }_{x,0}$ is a CSBP with  Laplace functional $(u_\lambda^{D_r}(x,0):~r> 0)$ and initial distribution $\mathbb{N}^{D_\epsilon}_{x,0}(\mathscr{L}^{\,\epsilon}_\sigma \in \dd x )$ with respect to the filtration $(\mathcal{F}^{D_{\epsilon + r}}: r \geq 0)$ (recall that $\mathscr{L}^{\,r}_\sigma$ is $\mathcal{F}^{D_r}$-measurable by Proposition \ref{L_eta_measurable} and Lemma \ref{def:lem:scr:L}). 
Now, we claim that for any $0<r_1< \dots < r_k$  and any collection of non-negative measurable functions $f_i : \mathbb{R}_+ \mapsto \mathbb{R}_+$, 
\begin{equation} \label{equation:fddscsbp}
    \mathbb{N}_{x,0}\left( \prod_{i=1}^k f_i(\mathscr{L}^{\,r_i}_\sigma) \right) 
    =
    \int_{\mathbb{R}_+} \nu_{r_1}(\dd z_1) f_{1}(z_1) \int_{\mathbb{R}_+} P_{r_2-r_1}(z_1,\dd z_2)f_2(z_2) 
    \dots \int_{\mathbb{R}_+} P_{r_k-r_{k-1}}(z_{k-1},\dd z_k)f_k(z_k).
\end{equation}
This follows from the previous result, by observing that for any $\epsilon < r_1$ we have 
\begin{align*} 
    &\mathbb{N}_{x,0}\left( \prod_{i=1}^k f_i(\mathscr{L}^{\,r_i}_\sigma) \mathbbm{1}_{\{ T_{D_\epsilon} < \infty \}} \right) \\ 
    &\quad \quad \quad =  \mathbb{N}_{x,0} \left( \mathbbm{1}_{\{ T_{D_\epsilon } < \infty \}}  f_1(\mathscr{L}^{\,r_1}_\sigma)  \int_{\mathbb{R}_+} P_{r_2-r_1}(\mathscr{L}^{\,r_1}_\sigma,\dd z_2)f_2(z_2) 
    \dots \int_{\mathbb{R}_+} P_{r_k-r_{k-1}}(z_{k-1},\dd z_k)f_k(z_k) \right),
\end{align*}
  and we conclude  taking the limit as $\epsilon \downarrow 0$. The fact that the family $(\nu_t: t \geq 0)$  satisfies that $\nu_{t+s} = \nu_s P_t$ for $t,s \geq 0$ now follows  from \eqref{equation:fddscsbp}.  Let us now identify $\widetilde{\psi}$. Recall from our  discussion in \eqref{equation:branchingEquation}  that the Laplace exponent $(u_\lambda^{D_r}(x,0) : r,\lambda \geq 0  )$   is the unique  solution to the equation
\begin{equation} \label{equation:branchingIdentif}
    u_\lambda^{D_r}(x,0) + \int_0^r {\rm{d}} u \, \Psi\big( u_{\lambda}^{D_u}(x,0) \big)  = \lambda,
\end{equation}
where $\Psi$ is  the branching mechanism associated with $(P_t)$, and that it is  defined in a unique way by \eqref{equation:branchingIdentif}. In particular, $\Psi$
characterizes completely the semigroup $(P_t)$.  To identify the branching mechanism we argue as follows: first,  observe that the identity \eqref{integral_equation}  applied at the domain  $D_r$ yields
\begin{equation} \label{equation:limiteInfinito}
u_\lambda^{D_r}(x , 0) + \Pi_{x,0}\left( \int_0^{\tau_{D_r}} \dd t ~\psi( u^{D_r}_\lambda(\xi_t , \mathcal{L}_t) ) \right) = \lambda,
\end{equation}
for every $\lambda\geq 0$ and $r>0$.
Next,  by excursion theory and \ref{Asssumption_3} we get:   
\begin{equation*}
    \Pi_{x,0}\left( \int_0^{\tau_{D_r}}  \dd t~\psi( u^{D_r}_\lambda(\xi_t , \mathcal{L}_t) )  \right)
    = 
    \int_{0}^r {\rm{d}}u~ \mathcal{N}\left( \int_0^\sigma \dd t~\psi\left( u_\lambda^{D_r} (\xi_t , u)\right) \right) 
    =
    \int_0^r {\rm{d}}u~  \mathcal{N}\left( \int_0^\sigma \dd t ~\psi\left( u_\lambda^{D_{r-u}} (\xi_t , 0)\right)  \right) ,
\end{equation*}
where in the last equality we use the invariance by translation of the local time of $\xi$. Moreover, the special Markov property applied at the domain $D_0$ gives
\begin{equation*}
    u_\lambda^{D_{r}}(y,0) = u_{u_{\lambda}^{D_r}(x,0)}(y), 
\end{equation*}
for every $y\in E \setminus \{ x \}$ and $\lambda\geq 0$ -- and the identity also holds for $y=x$.
Putting everything together,  by definition of $\widetilde{\psi}$, the identity  (\ref{equation:limiteInfinito}) can be re-written as follows: 
\begin{equation} \label{equation:characterizationBranchingMec}
u_\lambda^{D_r}(x,0) + \int_0^r {\rm{d}} u \, \widetilde{\psi}( u_{\lambda}^{D_u}(x,0) )  = \lambda.
\end{equation}
Consequently, we deduce that the branching mechanism associated with the Laplace functional $u_\lambda^{D_r}(x,0)$ is $\widetilde{\psi}$. It remains to show that the conditions stated in Section \ref{subsection:height} are satisfied by $\widetilde{\psi}$. As we already mentioned, it only remains to verify (A4). In this direction and recalling the notation $T_{D_r} = \inf \{ t \geq 0: \widehat{\Lambda}_t \geq r \}$,  also by \eqref{equation:characterizationBranchingMec} we obtain that $f(\lambda, r) := u_\lambda^{D_r}(x,0)$ satisfies for every $r$,
\begin{equation} \label{cuentas:expoCSBPchar}
    \int_{f(\lambda , r)}^\lambda \frac{\rm{d} s}{ \widetilde{\psi}(s) } = r,
\end{equation} 
where the limit $f(\infty , r)  = \mathbb{N}_{x,0}(L_\sigma^{D_r} > 0)$ is finite,  since   $\{ L_\sigma^{D_r} > 0 \} \subset  \{T_{D_r}  < \infty \}$ and $\mathbb{N}_{x,0}(T_{D_r} < \infty) < \infty$ by the same argument  used before  Theorem \ref{Theo_Spa_Markov_Excur}.
Hence,   taking the limit as $\lambda \uparrow \infty$ in (\ref{cuentas:expoCSBPchar}), we infer  that  the following conditions are fulfilled: 
\begin{equation*} 
    \widetilde{\psi}(\infty) = \infty \quad \quad \text{ and }\quad \quad\int^\infty_{\cdot} \frac{\dd s }{\widetilde{\psi}(s)} < \infty.
\end{equation*}
To derive the exact form of (A4), recall that $\widetilde{\psi}$ is convex and that we have $\widetilde{\psi}(0) = 0$ and $\widetilde{\psi}'(0+) \geq 0$. 
\end{proof}

Now that we have established that $\widetilde{\psi}$ is the Laplace exponent of a Lévy tree, let us briefly introduce some related notation and a few facts that will be  used frequently in the upcoming sections. From now on, we set $\widetilde{X}$ a  $\widetilde{\psi}$-Lévy process  and we write $\widetilde{I}$ for the running infimum of $\widetilde{X}$. We also denote the excursion  measure of the reflected  process $\widetilde{X}- \widetilde{I}$ by $\widetilde{N}$ -- where the associated local time is  $-\widetilde{I}$. The usual notation introduced in Section \ref{subsection:height}  applied to $\widetilde{X}$ are indicated with a $\sim$.  For instance, we denote   the height process and the exploration process issued from $\widetilde{X}$ respectively by $\widetilde{H}$ and  $\widetilde{\rho}$. 
\par By convexity and the fact that $\widetilde{\psi}'(0+) \geq 0$,     the only solution to $\widetilde{\psi}(\lambda) = 0$ is $\lambda = 0$. This implies that the mapping $\lambda \mapsto \widetilde{\psi}(\lambda)$ is invertible in $[0,\infty)$. By  classical results in the theory of Lévy processes, $\widetilde{\psi}^{-1}$  is the Laplace exponent of the right-inverse of $-\widetilde{I}$ and,  since $\widetilde{X} -\widetilde{I}$ does not spend  time at $0$, the former is a  subordinator with no drift. So, recalling the relation between excursion lengths and jumps of the right-inverse of   $-\widetilde{I}$, we derive that:
\begin{equation}\label{identity:inversePhi}
    \widetilde{\psi}^{-1}(\lambda) = \widetilde{N} (1-\exp(-\lambda \sigma )), \quad \lambda \geq 0.
\end{equation}
For a more detailed discussion, we refer to Chapters IV and VII of \cite{BertoinBook}. \smallskip 
\par We close this section with some useful identities in the same vein of \eqref{definition:PhiTilde},   that will be used frequently in our computations. These identities allow  to express some Laplace-like transforms concerning the process $\big(\psi(u_{\lambda}(\xi_t)):~t\geq 0\big)$, under the excursion measure $\mathcal{N}$,  in terms of the $\widetilde{\psi}$. As an application of these computations, we will identify the drift and Brownian coefficients of $\widetilde{\psi}$. We summarise these identities in the following lemma.
\begin{lem} \label{lemma:cuentas} For every $\lambda_1,\lambda_2 \in \mathbb{R}_+$ with $\lambda_1\neq \lambda_2$,    we have
\begin{equation}\label{u:r:l:1}
         \mathcal{N} \left( 1- \exp\Big(- \int _0^\sigma \dd s ~
         \frac{\psi\big(u_{\lambda_1}(\xi_s) \big) -\psi\big( u_{\lambda_2}(\xi_s) \big)   }{u_{\lambda_1}(\xi_s) - u_{\lambda_2}(\xi_s) } \Big) \right) 
         = \frac{\widetilde{\psi}\big(\lambda_1\big) - \widetilde{\psi}\big( \lambda_2\big)   }{ \lambda_1-\lambda_2 }.  
    \end{equation}
\end{lem}

Recalling the identities \eqref{identity:exponenteSubord}, remark that Lemma 8 allows  to express the Laplace exponent of $(\widetilde{U}^{(1)}, \widetilde{U}^{(2)})$ in terms of $\mathcal{N}$ and $\psi$. 

\begin{proof} First note that the functions $\lambda \mapsto u_\lambda(y)$  and $\lambda \mapsto \psi(u_\lambda(y))$ are non-decreasing. So without loss of generality we can and will assume that $\lambda_1>\lambda_2$.  We set $T_x := \inf \{ t \geq 0 : \xi_t = x \}$ and we write 
\begin{align*}
     &\mathcal{N} \Big( 1- \exp \Big(- \int_0^\sigma \dd s ~
         \frac{\psi\big(u_{\lambda_1}(\xi_s) \big) -\psi\big( u_{\lambda_2}(\xi_s) \big) }{u_{\lambda_1}(\xi_s) - u_{\lambda_2}(\xi_s) } \Big) \Big) \\
     & \hspace{20mm} = \mathcal{N} \left( \int_0^\sigma \dd s~ \frac{\psi\big(u_{\lambda_1}(\xi_s) \big) -\psi\big( u_{\lambda_2}(\xi_s) \big) }{u_{\lambda_1}(\xi_s) - u_{\lambda_2}(\xi_s) } \cdot  \exp \Big(- \int_s^\sigma \dd t~\frac{\psi\big(u_{\lambda_1}(\xi_t) \big) -\psi\big( u_{\lambda_2}(\xi_t) \big) }{u_{\lambda_1}(\xi_t) - u_{\lambda_2}(\xi_t) } \Big) \right)  \\
     & \hspace{20mm} = \mathcal{N} \left( \int_0^\sigma \dd s~\frac{\psi\big(u_{\lambda_1}(\xi_s) \big) -\psi\big( u_{\lambda_2}(\xi_s) \big) }{u_{\lambda_1}(\xi_s) - u_{\lambda_2}(\xi_s) }\cdot  \Pi_{\xi_s} \left( \exp \Big(- \int_0^{T_x} \dd t~ \frac{\psi\big(u_{\lambda_1}(\xi_t) \big) -\psi\big( u_{\lambda_2}(\xi_t)\big) }{u_{\lambda_1}(\xi_t) - u_{\lambda_2}(\xi_t) } \Big) \right) \right)
\end{align*}
where in the last equality we applied the Markov property. On the other hand, the definition of $\widetilde{\psi}$ given in  \eqref{definition:PhiTilde} yields 
\begin{align*}
    \frac{\widetilde{\psi}\big(\lambda_1\big) - \widetilde{\psi}\big( \lambda_2\big)   }{ \lambda_1-\lambda_2 }  
    &= \mathcal{N}\left( \int_0^\sigma \dd s~ \frac{\psi\big(u_{\lambda_1}(\xi_s)\big) - \psi\big(u_{\lambda_2}(\xi_s)\big) }{ \lambda_1-\lambda_2} \right) \\
    &= 
    \mathcal{N}\left( \int_0^\sigma \dd s~  \frac{\psi\big(u_{\lambda_1}(\xi_s)\big) - \psi\big(u_{\lambda_2}(\xi_s)\big) }{ u_{\lambda_1}(\xi_s) - u_{\lambda_2}(\xi_s) } \cdot \frac{ u_{\lambda_1}(\xi_s) - u_{\lambda_2}(\xi_s) }{\lambda_1-\lambda_2}   \right).
\end{align*}
Consequently, the lemma will  follow as soon as we establish the identity:
\begin{equation*}\label{(*)--eq--2}
    \frac{u_{\lambda_1}(y) - u_{\lambda_2}(y) }{\lambda_1-\lambda_2}  
    =
    \Pi_{y} \Big( \exp \Big(- \int_0^{T_x} \dd t~  \frac{\psi\big(u_{\lambda_1}(\xi_t)\big) - \psi\big(u_{\lambda_2}(\xi_t)\big) }{ u_{\lambda_1}(\xi_t) - u_{\lambda_2}(\xi_t) } \Big) \Big).
\end{equation*}
In this direction, recall that under $\mathbb{N}_{y,0}$ with $y\neq x$ the processes $\mathscr{L}^{\,0}(\rho,\overline{W})$ and $L^{E_*}(\rho,W)$ are well defined and indistinguishables, and  remark  that 
\begin{align*}
   u_{\lambda_1}(y) - u_{\lambda_2}(y)&=  \mathbb{N}_{y,0} 
    \Big( \exp\big(- \lambda_1 \int_{0}^{\sigma}\dd \mathscr{L}^{\,0}_u\big)  - \exp\big(- \lambda_2 \int_{0}^{\sigma}\dd \mathscr{L}^{\,0}_u\big)  \Big) 
    \\
    &=(\lambda_1-\lambda_2) \cdot \mathbb{N}_{y,0} 
\Big( \int_0^\sigma \dd \mathscr{L}^{\,0}_s \exp\big(- \lambda_1 \int_0^s \dd \mathscr{L}^{\,0}_u \big) \cdot \exp\big( -\lambda_2 \int_s^\sigma \dd \mathscr{L}^{\,0}_u \big)  \Big).
\end{align*}    
Then,  an application of the Markov property gives:
$$  u_{\lambda_1}(y) - u_{\lambda_2}(y)=(\lambda_1-\lambda_2) \cdot \mathbb{N}_{y,0} 
    \Big(  \int_0^\sigma \dd \mathscr{L}^{\,0}_s  \exp\big(- \lambda_1 \mathscr{L}^{\,0}_s\big)\cdot \mathbb{E}^{\dag}_{\rho_s , \overline{W}_s} \big[ \exp\big( -\lambda_2  \mathscr{L}^{\,0}_\sigma\big) \big]   \Big). $$
We can now apply the duality identity  $\big((\rho_{(\sigma-t)-},\eta_{(\sigma-t)-},\overline{W}_{\sigma-t}):~t\in[0,\sigma]\big) \overset{(d)}{=} \big((\eta_{t},\rho_{t},\overline{W}_{t}):~t\in[0,\sigma]\big)$  under $\mathbb{N}_{y,0}$, to get that the previous display is equal to
\begin{align*}
(\lambda_1-\lambda_2)\cdot  \mathbb{N}_{y,0} 
    \Big(  \int_0^\sigma \dd \mathscr{L}^{\,0}_s & \exp\big(-\lambda_1 \int_s^{\sigma}  \dd \mathscr{L}^{\,0}_t\big) \cdot \mathbb{E}^{\dag}_{\eta_s , \overline{W}_s} \big[  \exp\big( -\lambda_2  \mathscr{L}^{\,0}_\sigma  \big) \big]   \Big)\\
    &=(\lambda_1-\lambda_2)\cdot  \mathbb{N}_{y,0} 
    \Big(  \int_0^\sigma \dd \mathscr{L}^{\,0}_s ~ \mathbb{E}^{\dag}_{\rho_s , \overline{W}_s} \big[ \exp\big(-\lambda_1  \mathscr{L}^{\,0}_\sigma  \big) \big] \cdot \mathbb{E}^{\dag}_{\eta_s , \overline{W}_s} \big[  \exp\big( -\lambda_2  \mathscr{L}^{\,0}_\sigma \big) \big]   \Big).
\end{align*}
Remark that $(\eta, \overline{W})$ takes values in $\overline{\Theta}_x$ by duality and right-continuity of $\eta$ with respect to the total variation distance. 
We are now in position to apply the many-to-one equation \eqref{tirage_au_hasard_N}. In this direction, for $(\mu , \overline{\w}) \in \overline{\Theta}_x$ with $\overline{\w}(0) = (y,0)$ and $y \neq x$ we notice that
\begin{equation*}
     \mathbb{E}_{\mu , \overline{\w}}^{\dag} \Big[ \exp\big(-\lambda  \mathscr{L}^{\,0}_\sigma\big)\Big] = \exp \Big(-\int_0^{\tau_{D_0}(\overline{\w})} \mu(\dd h)~ \mathbb{N}_{\overline{\w}(h)} \big( 1-\exp(- \lambda \mathscr{L}^{\,0}_\sigma) \big) \Big) = \exp \Big(- \int_0^{\tau_{D_0}(\overline{\w})} \mu(\dd h)~ u_\lambda(\w(h)) \Big),
\end{equation*}
for every $\lambda>0$. Consequently, \eqref{tirage_au_hasard_N} gives: 
\begin{align*}
    \frac{u_{\lambda_1}(y) - u_{\lambda_2}(y) }{\lambda_1-\lambda_2} 
    &= E^0 \otimes \Pi_y \Big(   \exp\big(-\alpha T_x\big)   \exp \Big(  -\int_0^{T_x} J(\dd s )~ u_{\lambda_1}(\xi_s) - \int_0^{T_x} \Jc(\dd s ) ~u_{\lambda_2}(\xi_s)  \Big) \Big).
\end{align*}
 Finally  an application of  \eqref{identity:exponenteSubord} yields exactly the desired result \eqref{(*)--eq--2}. 
\end{proof}
As an immediate consequence, we obtain two other useful identities  taking   $\lambda_2 = 0$ and letting $\lambda_2 \downarrow \lambda_1$ respectively. For every $\lambda >0$, we have
\begin{equation} \label{identity:expoenente_1}
          \mathcal{N} \Big( 1- \exp\Big(- \int_0^\sigma \dd h~ \psi\big(u_\lambda(\xi_h)\big)/u_\lambda(\xi_h)  \Big) \Big)  = \widetilde{\psi}(\lambda)/\lambda
\quad \text{ and }\quad 
        \mathcal{N} \Big( 1- \exp\Big(- \int_0^\sigma \dd h~ \psi^{\prime}(u_\lambda(\xi_h))  \Big) \Big)   = \widetilde{\psi}^{\prime}(\lambda) ,
\end{equation} 
where for the first one we used that $u_0(y) = 0$ since $\mathbb{N}_y(L^{E_*}_\sigma = \infty ) = 0$. We also stress that \eqref{identity:expoenente_1}  can be proved independently directly by the same   arguments as the ones applied in the proof of \eqref{u:r:l:1}. \par 
Since by Proposition \ref{lemma:divergExpo} the exponent  $\widetilde{\psi}$ satisfies (A1) -- (A4), it can be written in the following form
\begin{equation*}
    \widetilde{\psi}(\lambda) = \widetilde{\alpha}\lambda + \widetilde{\beta} \lambda^2 +\int_{\mathbb{R}_+} \widetilde{\pi}(\dd x) \, (\exp(-\lambda x)-1+\lambda x),
\end{equation*}
where $\widetilde{\alpha}, \widetilde{\beta} \geq 0$ and $\widetilde{\pi}$ is a  measure in $\mathbb{R}_+$ satisfying $\int \widetilde{\pi}(\dd x) (x \wedge x^2)  < \infty$. In the following corollary, we identify the coefficients $\widetilde{\alpha}$ and $\widetilde{\beta}$. 
\begin{cor} \label{corollary:phitildeParametros} We have    $\widetilde{\alpha} = \mathcal{N}\big(1-\exp(-\alpha\sigma )\big)$ and  $\widetilde{\beta} = 0$. 
\end{cor}
\begin{proof} 
To simplify notation, for $\lambda \geq 0$ set ${\psi}^*(\lambda) := {\psi}(\lambda)/ \lambda$, $\widetilde{\psi}^*(\lambda) := \widetilde{\psi}(\lambda)/ \lambda$. Since $\widetilde{\psi}$ satisfies (A1)--(A4), by Fubini we derive that $\widetilde{\psi}^*$ is the Laplace exponent of a subordinator with exponent:
\begin{equation} \label{equation:phitildetriplet_1}
    \widetilde{\alpha} +  \widetilde{\beta} \lambda + \int_\mathbb{R_+} \dd r \, \widetilde{\pi}([r,\infty)) \big(  1-\exp (-\lambda r ) \big).   
\end{equation}
Next, introduce the  measure $\mathcal{N}^*(\dd \xi):= \mathcal{N}(\exp(-\alpha \sigma) \dd \xi)$ and observe that by \eqref{identity:expoenente_1},  $\widetilde{\psi}^*(\lambda)$ can also be written in the form
\begin{equation} \label{equation:phitildetriplet_2}
    \mathcal{N} \Big( 1- \exp\Big(- \int_0^\sigma \dd h~  \psi^*\big(u_\lambda(\xi_h)\big)  \Big) \Big) 
    = \mathcal{N} \big( 1- \exp (- \alpha\sigma) \big) + 
    \mathcal{N}^* \Big( 1- \exp\Big(- \int_0^\sigma \dd h~  \big(\psi^*\big(u_\lambda(\xi_h)\big) - \alpha \big) \Big) \Big).  
\end{equation}
Comparing with  \eqref{equation:phitildetriplet_1}, our result will follow by showing that the second term on the right-hand side of \eqref{equation:phitildetriplet_2} is the Laplace exponent of some pure-jump  subordinator. In this direction, introduce under $E^0 \otimes \mathcal{N}^*$ and  conditionally on $(J_\infty, \xi)$, a  Poisson point measure 
\begin{equation*}
    \mathcal{M}(\dd h , \dd \rho, \dd \overline{W}) = \sum_{i \in \mathbb{N}} \delta_{(h_i, \rho^i,\overline{W}^i )}, 
\end{equation*}
 with intensity $J_\sigma(\dd h) \mathbb{N}_{\xi(h),0}\big( \dd \rho,\dd \overline{W}  \big)$. This is always possible up  to enlarging the  measure space  and for simplicity we still denote the underlying measure by $E^0\otimes\mathcal{N}^*$. Next, define the functional $\sum_{i \in \mathbb{N}}\mathscr{L}^{\,0}_\sigma( \rho^i, \overline{W}^i)$ and denote its distribution  by $\nu(\dd x)$. By definition, we have:
\begin{align*}
    E^0\otimes \mathcal{N}^*\Big( 1-\exp \Big( -\lambda \sum_{i \in \mathbb{N}}\mathscr{L}^{\,0}_\sigma( \rho^i, W^i) \Big) \Big) 
    &= E^0\otimes \mathcal{N}^*\Big( 1-\exp \Big( - 
    \int_0^\sigma J_\sigma(\dd h) u_\lambda\big(\xi(h)\big)
    \Big) \Big) \\
    &= \mathcal{N}^* \Big( 1- \exp\Big(- \int_0^\sigma \dd h~  \big(\psi^*\big(u_\lambda(\xi_h)\big) - \alpha \big) \Big) \Big), 
\end{align*}
where in the last equality we used that $J_\infty$ is the Lebesgue-Stieltjes  measure of a subordinator with exponent $\psi^*(\lambda) -\alpha$. Since the latter expression is finite, we deduce that $\nu$ is a Lévy measure satisfying $\int \nu(\dd r )\,  (1 \wedge r)< \infty$, and that the second term on the right-hand side of \eqref{equation:phitildetriplet_2} is the Laplace exponent of a driftless subordinator with Lévy measure given by $\nu$.   
\end{proof}

\subsection{Construction of the additive functional \texorpdfstring{$(A_t)$}{Lg}} \label{subsection:existenciaAditiva}
\noindent We are finally in position to introduce our additive function:

\begin{prop}\label{proposition:aditivaDefinicion} Fix $(y,r_0)\in \overline{E}$ and $(\mu,\overline{\emph{w}})\in \overline{\Theta}_x$. Under $\mathbb{N}_{y,r_0}$ and  $\mathbb{P}_{\mu,\overline{\emph{w}}}$,  the process  defined  as 
\begin{equation*}
    A_t = \int_{\mathbb{R}_+} \dd r \mathscr{L}^{\,r}_t, \quad  \quad \text{ for } t \geq 0, 
\end{equation*}
is a continuous   additive functional of the Lévy snake taking finite values. Furthermore, we have   
\begin{equation} \label{equation:aproximacionA}
  A_t = \lim_{\varepsilon \downarrow 0} \frac{1}{\varepsilon}  \int_0^t \dd u  \int_{\mathbb{R}_+} \dd r \,  \mathbbm{1}_{\{ \tau_r(\overline{W}_u ) < H_u < \tau_r(\overline{W}_u) + \varepsilon \}} ,
\end{equation}
where the  convergence  holds uniformly in compact intervals  in measure under $\mathbb{P}_{\mu, \overline{\emph{w}} }$ and $\mathbb{N}_{y,r_0}( \, \cdot \, \cap \{ \sigma > z \} )$  for every   $z >0$.
\end{prop}
\begin{proof}
We start proving the proposition under $\mathbb{P}_{\mu,\overline{\w}}$, where $(\mu,\overline{\w}):=(\mu,\w,\ell)\in \overline{\Theta}_x$. Remark that by the translation invariance of the local time  we might assume that $\ell(0)=0$ without loss of generality. For simplicity, we set $y :=\w(0)$. Next, we  write $\widehat{\Lambda}^*_t := \sup_{s \leq t} \widehat{\Lambda}_s$ and we note that  it suffices to show that for any $t, K > 0$
\begin{equation*} 
    \mathbb{E}_{\mu,\overline{\w}} \Big[\sup_{s \leq t} | \int_{\mathbb{R}_+} \dd r \,  \frac{1}{\varepsilon} \int_0^s \dd u ~\mathbbm{1}_{\{ \tau_r(\overline{W}_u ) < H_u < \tau_r(\overline{W}_u) + \varepsilon \}} - \int_\mathbb{R_+} \dd r \mathscr{L}^{\,r}_s | \,  \cdot \mathbbm{1}_{\{ \widehat{\Lambda}^*_t < K \}} \Big]  \rightarrow 0, 
\end{equation*}
as $\varepsilon\downarrow 0$. In this direction, we remark that the previous expression is bounded above by 
\begin{align*} 
     & \int_{\mathbb{R}_+} \dd r \, \mathbb{E}_{\mu,\overline{\w}} \Big[   \sup_{s \leq t} | \frac{1}{\varepsilon} \int_0^s \dd u~ \mathbbm{1}_{\{ \tau_r(\overline{W}_u) < H_u < \tau_r(\overline{W}_u) + \varepsilon \}} -  \mathscr{L}^{\,r}_s | \,  \cdot \mathbbm{1}_{\{ \widehat{\Lambda}^*_t < K  \}} \Big]  \\
     & \hspace{50mm} \leq  \int_{(0,K]} \dd r\, \mathbb{E}_{\mu,\overline{\w}} \Big[   \sup_{s \leq t} | \frac{1}{\varepsilon} \int_0^s \dd u~ \mathbbm{1}_{\{ \tau_r(\overline{W}_u) < H_u < \tau_r(\overline{W}_u) + \varepsilon \}} -  \mathscr{L}^{\,r}_s | \Big],
\end{align*}
since on the event $\{ \widehat{\Lambda}^*_t < K \}$ we have  $\mathscr{L}^{r} = 0$ for every $r>K$. Now, by Lemma \ref{def:lem:scr:L},  it suffices to show that the expectation under $\mathbb{P}_{\mu,\overline{\w}}$ in the previous display is uniformly bounded on $\epsilon, r > 0$
since the desired result then follows by dominated convergence. To do so, we set $T_0^+ := \inf \big\{ t \geq 0 : \langle\rho_t,1\rangle = 0 \big\}$ and we notice that by the strong Markov property, under $\mathbb{P}_{\mu,\overline{\w}}$, the distribution of $(\rho_{T_0^++s},\overline{W}_{T_0^++s}:~s\geq 0)$ is $\mathbb{P}_{0,y,0}(\dd \rho,\dd \overline{W})$. In particular we have the upper bound:
\begin{align*}
\mathbb{E}_{\mu,\overline{\w}} \Big[   \sup_{s \leq t} | \frac{1}{\varepsilon} \int_0^s \dd u~ \mathbbm{1}_{\{ \tau_r(\overline{W}_u) < H_u < \tau_r(\overline{W}_u) + \varepsilon \}} -  \mathscr{L}^{\,r}_s | \Big]\leq &\mathbb{E}_{\mu,\overline{\w}}^\dag \Big[ \frac{1}{\varepsilon} \int_0^{\sigma} \dd u~ \mathbbm{1}_{\{ \tau_r(\overline{W}_u) < H_u < \tau_r(\overline{W}_u) + \varepsilon \}}+ \mathscr{L}^{\,r}_{\sigma} \Big]\\
&+\mathbb{E}_{0,y,0} \Big[   \sup_{s \leq t} | \frac{1}{\varepsilon} \int_0^s \dd u~ \mathbbm{1}_{\{ \tau_r(\overline{W}_u) < H_u < \tau_r(\overline{W}_u) + \varepsilon \}} -  \mathscr{L}^{\,r}_s | \Big].    
\end{align*}
So to conclude we need  to prove both:
\begin{align*}
       &\text{(i)} \quad  \sup_{\epsilon>0} \sup_{r > 0} \mathbb{E}_{0,y,0} \Big[   \sup_{s \leq t} | \frac{1}{\varepsilon} \int_0^s \dd u~ \mathbbm{1}_{\{ \tau_r(\overline{W}_u) < H_u < \tau_r(\overline{W}_u) + \varepsilon \}} -  \mathscr{L}^{\,r}_s | \Big]< \infty; \\
       &\text{(ii)}\quad  \sup_{\epsilon >0}  \sup_{r > 0}  \mathbb{E}_{\mu,\overline{\w}}^\dag \Big[ \frac{1}{\varepsilon} \int_0^{\sigma} \dd u~ \mathbbm{1}_{\{ \tau_r(\overline{W}_u) < H_u < \tau_r(\overline{W}_u) + \varepsilon \}}+ \mathscr{L}^{\,r}_{\sigma} \Big]<\infty. 
\end{align*}
Let us start showing (i).  We are going to apply similar techniques to the ones used in the proof  of Theorem \ref{Theo_Spa_Markov_Excur_P}. In this direction,  we work under $\mathbb{P}_{0,y,0}$ and  we fix $r,\epsilon>0$. Now, recall the definition of $\gamma^{D_r}$, $\sigma^{D_r}$ and $\rho^{D_r}$ introduced in Section \ref{subsection:specialMarkovProof} (keeping in mind the fact that here we work with $(\rho,\overline{W})$) and set
\begin{equation*}
    R^{D_r}_t := \int_0^t \dd s \mathbbm{1}_{\{ \gamma^{D_r} _s > 0 \}}, \quad \quad \text{ for }t \geq 0,
\end{equation*}
which is the right inverse of $\sigma^{D_r}$. Next,  for every $r > 0$, by definition we have  $\tau_r(\rho_t,\overline{W}_t)=\tau_{D_r}(\rho_t,\overline{W}_t)$ and  we derive that
 \begin{align*}
     \int_0^s \dd u~ \mathbbm{1}_{\{ \tau_r(\overline{W}_u ) < H_u < \tau_r(\overline{W}_u) + \varepsilon \}} 
     &=   \int_0^{R_s^{D_r}} \dd u~ \mathbbm{1}_{\{ 0 < H(\rho^{D_r}_u) <  \varepsilon \}} ,
 \end{align*}
 since 
 on $\{ u\geq 0 : \,  H(\rho_{\sigma^{D_r}_u}) > \tau_{r}(\overline{W}_{\sigma^{D_r}_u})  \}$, we have   $H(\rho_u^{D_r}) = H ( \rho_{\sigma_u^{D_r}}) - \tau_{r}(\overline{W}_{\sigma_u^{D_r}}).$
Recall from  \eqref{def:rho:D} that $\langle \rho^{D_r},1 \rangle$ is distributed as $\langle \rho,1 \rangle$ under $\mathbb{P}_{0,y,0}$, which is a reflected $\psi$-Lévy process,  and that we denote its local time at $0$ by $\ell^{D_r}$. In particular, the distribution of $(\langle \rho^{D_r},1\rangle, \ell^{D_r})$ is the same as $\big((X_t-I_{t}, -I_{t}):~t\geq 0\big)$. 
Recalling from \eqref{equation:exitlocalTimeChange} that $\mathscr{L}^{\,r}_t = \ell^{D_r} ( R^{D_r}_t )$ and noticing that  $R^{D_r}_s \leq s$, we derive the following inequality:
 \begin{align*}
     \mathbb{E}_{0,y,0}\Big[   \sup_{s \leq t} | \frac{1}{\varepsilon} \int_0^s \dd u~ \mathbbm{1}_{\{ \tau_r(\overline{W}_u ) < H_u < \tau_r(\overline{W}_u) + \varepsilon \}} -  \mathscr{L}^{\,r}_s | \Big]
     &=
     \mathbb{E}_{0,y,0} \Big[   \sup_{s \leq t} | \frac{1}{\varepsilon} \int_0^{R^{D_r}_s} \dd u~ \mathbbm{1}_{\{ 0 < H(\rho^{D_r}_u) <  \varepsilon \}} -  \ell^{D_r}(R^{D_r}_s)  | \Big] \\ 
     &\leq \mathbb{E}_{0,y,0}\Big[\sup_{s \leq t} | \frac{1}{\varepsilon} \int_0^{s} \dd u~ \mathbbm{1}_{\{ 0 < H(\rho^{D_r}_u) <  \varepsilon \}} -  \ell^{D_r} (s)  | \Big]\\
     &=\mathbb{E}_{0,y,0}\Big[\sup_{s \leq t} | \frac{1}{\varepsilon} \int_0^{s} \dd u~ \mathbbm{1}_{\{ 0 < H(\rho_u) <  \varepsilon \}} +  I_{s}  | \Big],
\end{align*}
where in the first line we used that for each fixed $r>0$, the processes $\mathscr{L}^{\,r}$ and $L^{D_r}$ are indistinguishable.
The latter quantity does not depend on $r$ and by \eqref{temps:local:I} it converges   to $0$ as $\varepsilon \downarrow 0$, giving (i).
\par We now turn our attention to the proof of  (ii). 
{On the one hand, by Proposition   \ref{proposition:aproxLDPmu} - (ii)  and \eqref{tirage_au_hasard_ExitN}, for every $r>0$ we have
\begin{equation*}
    \mathbb{E}^{\dag}_{\mu,\overline{\w}} \big[  \mathscr{L}^{\,r}_\sigma  \big]
    =
    \int_{(0,\tau_r(\overline{\w}))} \mu(\dd h)~\mathbb{N}_{\overline{\w}(h)}\big( \mathscr{L}^{\,r}_\sigma \big) 
    = 
     \int_{(0,\tau_r(\overline{\w}))} \mu(\dd h)~ E^0 \otimes \Pi_{\overline{\w}(h)}\big[ \mathbbm{1}_{\{ \tau_r(\xi,\mathcal{L})  < \infty \}}\exp\big(-\alpha \tau_r(\xi,\mathcal{L})\big)\big] 
     \leq \langle \mu , 1 \rangle. 
\end{equation*}
On the other hand, the remaining term
$$  \mathbb{E}_{\mu,\overline{\w}}^\dag \Big[ \frac{1}{\varepsilon} \int_0^{\sigma} \dd u~ \mathbbm{1}_{\{ \tau_r(\overline{W}_u) < H_u < \tau_r(\overline{W}_u) + \varepsilon \}}\Big]$$
can be bounded similarly as we did in \eqref{equation:aproxLDPmu_eq11}.  More precisely, consider under $\mathbb{P}_{\mu , \overline{\w}}^{\dag}$ the random measure $\sum_{i\in \mathbb{N}}\delta_{(h_i, \rho^i, \overline{W}^i)}$ defined in \eqref{PoissonRandMeasure},  set $T:= \inf\{ t > 0 : H_t = \tau_r(\overline{\w}) \}$, with the convention $T=0$ if $\tau_r(\overline{\w})=\infty$, and  remark that for every  $s\in [0,T]$ we have $\tau_r(\overline{W}_s) = \tau_r(\overline{\w})$. Recalling $\mu(\{\tau_r(\overline{\w})\}) = 0$, it follows by considering    the excursion intervals of $H$ over its running infimum  and our previous remark, that the integral  $\int_0^\sigma 
    \dd u~ \mathbbm{1}_{\{ \tau_r(\overline{W}_u) < H_u < \tau_r(\overline{W}_u) +\epsilon \}}$ can be written as 
\begin{align*}
    \sum_{h_i > \tau_r(\overline{\w})} \int_0^{\sigma(\overline{W}^i)} \dd u~ \mathbbm{1}_{\{ \tau_r(\overline{\w}) < h_i +  H(\rho^i_u) < \tau_r(\overline{\w}) + \epsilon\}} 
    +  \sum_{h_i < \tau_r(\overline{\w})} 
    \int_0^{\sigma(\overline{W}^i)} \dd u~ \mathbbm{1}_{\{ \tau_r(\overline{W}^i_u) < H(\rho^i_u) < \tau_r(\overline{W}_u^i)\}}, 
\end{align*}
where the first term is now bounded above by
$
    \sum_{h_i > \tau_r(\overline{\w})} \int_0^{\sigma(\overline{W}^i)} \dd u~ \mathbbm{1}_{\{ 0 <  H(\rho^i_u) <  \epsilon\}}. 
$
Consequently, by  \eqref{tirage_au_hasard_N}  we have  
\begin{align*}
     & \mathbb{E}^{\dag}_{\mu,\overline{\w}} \Big[     \int_0^\sigma \dd u~ \mathbbm{1}_{\{ \tau_r(\overline{W}_u) < H_u < \tau_r(\overline{W}_u) + \varepsilon \}}  \Big] \\
     &\hspace{20mm}\leq \mu((\tau_r(\overline{\w}), \infty))N(\int_{0}^{\sigma}\dd s \, \mathbbm{1}_{\{0< H(\rho_s)<\epsilon\}})+\int_{(0,\, \tau_r(\overline{\w}))}\mu(\dd h)\mathbb{N}_{\overline{\w}(h)}\big(\int_{0}^{\sigma} \dd s \,  \mathbbm{1}_{\{ \tau_{r}(\overline{W}_s) < H_s <\tau_{r}(\overline{W}_s)+  \epsilon \}}\big), 
\end{align*}
and again by the many-to-one formula \eqref{tirage_au_hasard_N}, the previous display is bounded by $\epsilon \cdot \langle \mu , 1 \rangle$.}  Putting everything together we deduce the upper  bound   
\begin{equation*}
     \mathbb{E}^{\dag}_{\mu,\overline{\w}} \Big[    \frac{1}{\varepsilon} \int_0^\sigma \dd u~ \mathbbm{1}_{\{ \tau_r(\overline{W}_u) < H_u < \tau_r(\overline{W}_u) + \varepsilon \}} +  \mathscr{L}^{\,r}_\sigma  \Big] \leq 2\cdot \, \langle \mu , 1 \rangle, 
\end{equation*}
which does not depend on the pair $r, \epsilon>0$ and concludes the proof of (ii). 
\par Finally,  we extend the result  under the excursion measure $\mathbb{N}_{y,r_0}$.  Working under $\mathbb{P}_{0,y,r_0}$ fix $z > 0$ and denote by $ (\rho^\prime , \overline{W}^\prime) = (\rho_{(g+\cdot)\wedge d}, \overline{W}_{(g+\cdot)\wedge d})$ the first excursion with length $\sigma > z$. By the previous result, the quantity
\begin{align*}
    &\sup_{s \leq t} \Big|~\epsilon^{-1} \int_{0}^{s} \dd u \int_{\mathbb{R}_+} \dd r \mathbbm{1}_{\{  \tau_r(\overline{W}^{\prime}_u)  < H(\rho_u^{\prime}) < \tau_r(\overline{W}^{\prime}_u) + \epsilon \}}    - \int_{\mathbb{R}_+} \dd r \mathscr{L}^{\,r}_s (\rho^\prime, \overline{W}^\prime) ~\Big|  \\
    &\hspace{30mm}=   \sup_{s \leq t \wedge (d-g)} \Big|~ \epsilon^{-1} \int_{g}^{g+s}\dd u \int_{\mathbb{R}_+} \dd r \mathbbm{1}_{\{  \tau_r(\overline{W}_u)  < H_u < \tau_r(\overline{W}_u) + \epsilon    \}}   
    - \int_{\mathbb{R}_+} \dd r (\mathscr{L}^{\,r}_{g+s} -\mathscr{L}^{\,r}_g) ~\Big|  
\end{align*}
converges  in probability to $0$, and it then follows that  \eqref{equation:aproximacionA} holds in measure under $\mathbb{N}_{y,r_0}( \, \cdot \, \cap \{ \sigma > z\} )$.

\end{proof}

\noindent As a straight consequence of the definition of $A$ we deduce the following many-to-one formula:
\begin{lem}\label{lemLAformula:aditiva}
For any non-negative measurable function $\Phi$ on $M_f(\mathbb{R}_+)\times M_f(\mathbb{R}_+) \times \mathcal{W}_{\overline{E}}$ and $(y,r_0) \in \overline{E}$, we have
\begin{equation} \label{LAformula:aditiva}
    \mathbb{N}_{y,r_0} \left( \int_0^\sigma \dd A_s~ \Phi\left(  \rho_s , \eta_s ,\overline{W}_s \right)  \right) 
    = 
    \int_{r_0}^\infty \dd r~ E^0 \otimes \Pi_{y,r_0} \left( \exp\big(- \alpha \tau_r\big) \cdot\Phi\big( J_{\tau_r} , \widecheck{J}_{\tau_r}  , (\xi_t, \mathcal{L}_t : t \leq \tau_r) \big)  \right) .
\end{equation}
\end{lem}
\begin{proof}
By the translation invariance of the local time it is enough to prove the Lemma for $r_0=0$. Now recall that, under $\mathbb{N}_{y,0}$, for every fixed $r\geq 0$ the processes $\mathscr{L}^{\,r}$  and $L^{D_r}$ are indistinguishable. Consequently,  the left-hand side of \eqref{LAformula:aditiva} can be written in the form: 
\begin{equation*}
 \int_0^\infty \dd r \, \mathbb{N}_{y,0} \left( \int_0^\sigma \dd {L}_s^{D_r}~ \Phi\left(  \rho_s , \eta_s ,\overline{W}_s \right)  \right),
\end{equation*}
and hence we arrive at \eqref{LAformula:aditiva} applying \eqref{tirage_au_hasard_ExitN}. 
\end{proof}
 A first  consequence of Lemma \ref{lemLAformula:aditiva} is that for any $(y,r_0) \in \overline{E}$, we have
\begin{equation} \label{indentity:soporte_1}
    \text{supp} \, \dd A \subset \{ t \in \mathbb{R}_+ : \widehat{W}_t = x \}, \quad  \mathbb{N}_{y,r_0} \text{-- a.e.}
\end{equation}
Indeed, it suffices to observe that by \eqref{LAformula:aditiva}, for any $\varepsilon>0$, it holds that
$
    \mathbb{N}_{y,r_0} \left( \int_0^\sigma \dd A_s \mathbbm{1}_{\{  \dd_E ( \widehat{W}_s ,x ) > \varepsilon  \}}  \right) 
    = 0,
$
where we recall that $d_E$ stands for the metric of $E$. 
\noindent Let us comment on a few useful identities that will be used frequently in our computations: 
\\
\\
\textbf{Remark}.
\normalfont Fix $(y,r_0) \in \overline{E}$ with $y \neq x$.
Under $\mathbb{N}_{y,r_0}$ or $\mathbb{P}_{0,y,r_0}$, let  $(g,d)$ be  an interval such that $H_s>H_g = H_d$, for every $s\in(g,d)$, and  $\widehat{\Lambda}_g = r_0$ --  remark that in particular we have $p_{H}(g)=p_{H}(d)$.   We denote  the corresponding subtrajectory, in the sense of Section \ref{secsnake}, by $(\rho^\prime , \overline{W}^\prime)$ and  its duration by $\sigma^\prime = \sigma(W^\prime)$. 
Since for any $q\geq  r$ and $s \geq 0$:
\begin{align*}
 H_{(g + s)\wedge d}= H_g + H(\rho^i_{s \wedge \sigma_i}) ~~~\text{ and }~~~
 \tau_q(\overline{W}_{(g + s)\wedge d}) = H_g + \tau_q(\overline{W}_{s \wedge \sigma_i}^{i}),
\end{align*}
we deduce by the approximation \eqref{equation:aproximacionA} that the process $(A_{(g + t)\wedge d} - A_g: t \geq 0)$ only depends on $(\rho^\prime, \overline{W}^\prime)$ and it will be denoted by $(A_{s}(\rho^\prime, \overline{W}^\prime) : s \geq 0)$. Now we make the following observations:\\
\\
\indent (i) Working under $\mathbb{N}_{y,r_0}$, we denote  the connected components of  the open set $\{( H_s - \tau_{r_0}(\overline{W}_s) )_+ > 0\}$ by $((\alpha_i, \beta_i) : i \in \mathcal{I} )$ and we set $\sigma_i:=\beta_i-\alpha_i$ its duration. We also write $( \rho^{i},  \overline{W}^{i})$ for the excursions from $D_{r_0}$ corresponding to the interval $(\alpha_i,\beta_i)$. By Proposition \ref{proposition:aditivaDefinicion}, the measure $\dd A$ does not charge the set
$\{ s\geq 0:  H_s \leq \tau_{r_0}(\overline{W}_s) \}$ and we derive that:
\begin{equation}\label{remark:descomponerA}
    A_{\sigma} = \sum_{i \in \mathcal{I}} \int_{(\alpha_i, \beta_i]}\dd A_s = \sum_{i\in \mathcal{I}} A_{\sigma_i}(\rho^{i},  \overline{W}^{i}), \quad \quad \mathbb{N}_{y,r_0}\text{--a.e.} 
\end{equation}

\indent (ii) We will now make similar remarks holding under $\mathbb{P}_{\mu, \overline{\w}}^\dag$, for $(\mu,\overline{\w})\in \overline{\Theta}_x$. Under  $\mathbb{P}^{\dag}_{\mu, \overline{\w}}$,
denote the connected components of $\{s\geq 0:~H_s>\inf_{[0,s]} H \}$ by $((a_i,b_i):i\in \mathbb{N})$ and write $(\rho^i,\overline{W}^i)$ for the subtrajectory associated with $[a_i,b_i]$. We also set $h_i=H_{a_i}$ and recall that the measure $\mathcal{M} = \sum_{i \in \mathcal{I}}\delta_{(h_i, \rho^i, \overline{W}^{i})}$ is  the Poisson point measure \eqref{PoissonRandMeasure} associated with $(\rho, \overline{W})$. Moreover, we have:
\begin{align*}
 \mathbb{E}^{\dag}_{\mu, \overline{\w}}\big[| A_\sigma -\sum_{i\in \mathbb{N}} A_\sigma(\rho^{i},  \overline{W}^{i})|\big] &\leq \int_{\mathbb{R}_+} \dd r~ \mathbb{E}^{\dag}_{\mu, \overline{\w}}\big[| \mathscr{L}_\sigma^{\,r} -\sum_{i\in \mathbb{N}} \mathscr{L}_\sigma^{\,r}(\rho^{i},  \overline{W}^{i})|\big].
\end{align*}
Consequently, by Proposition \ref{proposition:aproxLDPmu} - (ii), the previous quantity is null and it follows that we still have 
\begin{equation} \label{remark:descomponerA_v2}
    A_\sigma = \sum_{i\in \mathbb{N}} A_\sigma(\rho^{i},  \overline{W}^{i}) ,\quad \quad \mathbb{P}^{\dag}_{\mu,\overline{\w}} \text{-- a.s.}
\end{equation}

Recall now the definition \eqref{definition:PhiTilde} of $\widetilde{\psi}$  and the notation  $u_\lambda$ introduced in \eqref{definition:ulambda}.  The following proposition relates the Laplace transform of the total mass $A_\sigma$ under $\mathbb{N}_{y,r_0}$ and the Laplace exponent $\widetilde{\psi}$. This identity  will be needed to characterize the support of $\dd A$ and will also play a central role  in Section \ref{section:treeStructureLocalTime}.
\begin{prop}\label{lemma:InversaExpo}
 For every $r_0,\lambda \geq 0$ and $y \in E$, we have
 \begin{equation*} 
    \mathbb{N}_{y,r_0}\Big( 1-\exp\big(- \lambda A_\infty\big) \Big) = u_{\widetilde{\psi}^{-1}(\lambda)}(y),
\end{equation*}
where we recall the convention $u_\lambda(x)=\lambda$, for every $\lambda\geq 0$. Moreover, for  $(\mu,\overline{\mathrm{w}})\in \overline{\Theta}_{x}$, we have:
\begin{equation*}
    \mathbb{E}^{\dag}_{\mu , \overline{\mathrm{w}}} \Big[ \exp\big(-  \lambda  A_\infty \big) \Big] =
    \exp \Big(- \int \mu(\dd h)~ u_{\widetilde{\psi}^{-1}(\lambda)}(\emph{w}(h))  \Big).
\end{equation*}    
\end{prop}
\noindent The proposition has the following consequence: since $\widetilde{\psi}^{-1}(\lambda) = \widetilde{N}(1-\exp(-\lambda \sigma))$,  the total mass $A_\infty$ under $\mathbb{N}_{x,0}$ and $\sigma$ under $\widetilde{N}$ have the same distribution. This connection is the tip of the iceberg of the results  that will be established in the upcoming section, where we establish that the tree structure of the set $\{\upsilon\in \mathcal{T}_H:~\xi_\upsilon=x\}$ is encoded by a $\widetilde{\psi}$--Lévy tree.
\begin{proof} 
Under $\mathbb{N}_{y,r_0}$ with $y \neq x$ and $r_0\geq 0$, set
\begin{equation*}
    T^* := \inf \{t \geq 0 : \tau_{r_0}(\overline{W}_t) < \infty \}, 
\end{equation*}
which is just the first hitting time of $x$ by $(\widehat{W}_t)_{t \in [0,\sigma]}$. Notice that by \eqref{indentity:soporte_1},  $A_\infty$ vanishes on $\{ T^* = \infty \}$ $\mathbb{N}_{y,r_0}$-a.e.. We set  $G_\lambda := \mathbb{N}_{x,0}(1-\exp(-\lambda A_\infty ))$,  and  remark that  the identity \eqref{remark:descomponerA} and  the special Markov property  applied to the domain $D_{r_0}$ yields:
\begin{equation*} 
    \mathbb{N}_{y,r_0}\Big( 1- \exp\big(-\lambda A_\infty\big) \Big) = \mathbb{N}_{y,r_0}\Big( 1- \exp\Big( -  \mathscr{L}^{r_0}_\sigma\cdot \mathbb{N}_{x, r_0} \big( 1-\exp\big(-\lambda A_\infty\big)\big)  \Big)\Big).
\end{equation*}
Next, by the translation invariance of the local time $\mathcal{L}$, we derive that the previous display is equal to:
$$\mathbb{N}_{y,0}\Big( 1- \exp\Big( -  \mathscr{L}^{0}_\sigma\cdot \mathbb{N}_{x, 0} \big( 1-\exp\big(-\lambda A_\infty\big)\big)  \Big)\Big)=u_{G_\lambda}(y).$$
Moreover, for $(\mu , \overline{\w} ) \in \overline{\Theta}_x$ if we   denote under $\mathbb{P}^{\dag}_{\mu, \overline{\w}}$  the Poisson process introduced in \eqref{PoissonRandMeasure} by $\sum_{i \in \mathcal{I}}\delta_{(h_i, \rho^i, \overline{W}^{i})}$, we get  : 
\begin{align*}
    \mathbb{E}^{\dag}_{\mu , \overline{\w}} \Big[ \exp\big(-  \lambda  A_\infty \big) \Big]  
    &= \mathbb{E}^{\dag}_{\mu , \overline{\w}} \Big[ \exp \big(- \lambda  \sum_{i \in \mathcal{I}} A_\infty (\rho^i , \overline{W}^{i})\big) \Big]   \\
    &= \exp \Big( -\int \mu(\dd h) ~\mathbb{N}_{\overline{\w}(h)}\big( 1-\exp\big(-\lambda A_\infty\big)  \big) \Big)
    =  \exp \Big(- \int \mu(\dd h)~ u_{G_\lambda}(\w(h))  \Big),
\end{align*}
where in the first equality we applied \eqref{remark:descomponerA_v2}, and in the second  we used that $\sum_{i \in \mathcal{I}}\delta_{(h_i, \rho^i, \overline{W}^{i})}$ is a Poisson point measure with intensity $\mu(\dd h) \mathbb{N}_{\overline{\w}(h)}(\dd \rho,\dd W)$. Consequently, the statement of the proposition  will now follow if we establish that $G_\lambda=\widetilde{\psi}^{-1}(\lambda)$. In this direction, for $\lambda > 0$, notice that the Markov property implies that
\begin{align*}
    G_\lambda 
    = \lambda\cdot \mathbb{N}_{x,0} \Big( \int_0^\sigma \dd A_s ~ \exp\big(- \lambda \int_s^\sigma \dd A_u \big)  \Big) 
    =\lambda\cdot  \mathbb{N}_{x,0} \Big( \int_0^\sigma \dd A_s ~ \mathbb{E}^{\dag}_{\rho_s, \overline{W}_s} \Big[  \exp\big(- \lambda \int_0^\sigma \dd A_u\big) \Big] \Big).
\end{align*}
By the previous discussion under $\mathbb{P}^{\dag}_{\mu , \overline{\w}}$ and the many-to-one formula of $A$ given in Lemma \eqref{lemLAformula:aditiva}, we get:
\begin{align*}
    G_\lambda 
    &= \lambda \int_0^\infty \dd r ~ E^{0} \otimes \Pi_{x,0} \Big( \exp\big(- \alpha \tau_r\big) \exp\Big(- \int_0^{\tau_r} J_{\tau_r}(\dd h) ~u_{G_\lambda}\big(\xi(h)\big) \Big)  \Big)\\
&= \lambda \int_0^\infty \dd r ~ \Pi_{x,0} \Big(  \exp \Big(- \int_0^{\tau_r} \dd h~ \frac{\psi\big( u_{G_\lambda}(\xi(h))\big)}{u_{G_\lambda}\big(\xi(h)\big)} \Big)  \Big),
\end{align*}
where we recall that $\tau_r(\xi,\mathcal{L}):=\inf\{s\geq 0:~\mathcal{L}_s\geq r\}$ and in the second equality we used  that $J_\infty(\dd h)$ is the Lebesgue-Stieltjets measure of a subordinator with exponent $\psi(\lambda)/\lambda - \alpha$. Next, under $\Pi_{x,0}$, we consider $(s_i,t_i)_{i\geq 1}$ the connected components of $\{s\geq 0:~\xi_s\neq x\}$ and we remark that:
$$
\int_0^{\tau_r} \dd h~ \frac{\psi\big( u_{G_\lambda}(\xi(h))\big)}{u_{G_\lambda}\big(\xi(h)\big)}= \sum \limits_{i\geq 1, \mathcal{L}_{s_i}<r} \int_{s_i}^{t_i} \dd h~\frac{\psi\big( u_{G_\lambda}(\xi(h))\big)}{u_{G_\lambda}\big(\xi(h)\big)},
$$
since $\int_0^\infty \dd h \mathbbm{1}_{\{ \xi_h = x \}} = 0$ by assumption \ref{Asssumption_3}. Consequently, by excursion theory we get:
$$
\Pi_{x,0} \Big(  \exp \Big(- \int_0^{\tau_r} \dd h~ \frac{\psi\big( u_{G_\lambda}(\xi(h))\big)}{u_{G_\lambda}\big(\xi(h)\big)} \Big)  \Big)=   \exp \Big(- r\cdot  \mathcal{N} \Big(  1-\exp \big(- \int_0^\sigma\dd h ~\frac{ \psi\big(u_{G_\lambda}(\xi_h))}{u_{G_\lambda}(\xi_h)} \big)   \Big)\Big), 
$$
and hence
\begin{align*}
     G_\lambda  &= \lambda \cdot \mathcal{N} \Big(  1-\exp \Big(- \int_0^\sigma\dd h ~ \frac{\psi\big(u_{G_\lambda}(\xi_h))}{u_{G_\lambda}(\xi_h)} \Big)   \Big)  ^{-1}.
\end{align*}
However, by the first identity in \eqref{identity:expoenente_1}, we have 
\begin{align*}
     \mathcal{N} \Big(  1-\exp \Big(- \int_0^\sigma  \dd h \frac{\psi(u_{G_\lambda}(\xi_h))}{u_{G_\lambda}(\xi_h)}  \Big)   \Big)
     &= \frac{\widetilde{\psi}(G_\lambda)}{G_\lambda},
\end{align*}
and we derive that  $\widetilde{\psi}(G_\lambda) = \lambda$ for $\lambda > 0$ and equivalently $G_\lambda =\widetilde{\psi}^{-1}(\lambda)$. Finally,  since $G_0 = 0$  the identity also holds for $\lambda = 0$.
\end{proof}
\subsection{Characterization  of the support of \texorpdfstring{$\dd A$}{Lg}}\label{subsection:characterizationSupport}

The rest of the section is devoted to the characterisation,   under $\mathbb{N}_{y,r_0}$ and $\mathbb{P}_{\mu,\overline{\w}}$, of   the support of the measure $\dd A$. Our characterisation is given in terms of the constancy intervals of $\widehat{\Lambda}$, and  of a family of special times for the Lévy snake that will be named \textit{exit times from} $x$. Before giving a precise statement we will need several preliminary results under $\mathbb{N}_{x,0}$.  First recall that under $\mathbb{N}_{x,0}$, for every $r>0$ the processes $\mathscr{L}^{\,r }$ and $L^{D_r}$ are indistinguishables -- and in particular, by   Proposition \ref{L_eta_measurable},  $\mathscr{L}^{\,r }_\sigma$  is $\mathcal{F}^{D_r}$ measurable. Fix $r>0$,   recall the notation  $\tau_r(\rho_t,\overline{W}_t)=\tau_{D_r}(\rho_t,\overline{W}_t)$ for  $t\geq 0$, and denote  the connected components of the open set 
$\{ t \in [0,\sigma] : \tau_{r}(\overline{W}_t) < H_t \}$ by  $\{ (a^r_i, b^r_i): \,  i \in \mathcal{I}_r \}$. We write  $\{(\rho^{i,r},\overline{W}^{i,r} ): \,  i \in \mathcal{I}_r \}$  for the corresponding subtrajectories, where as usual   $\overline{W}^{i,r}=(W^{i,r},\Lambda^{i,r})$. Next, recall the notation  $\Gamma_s^{D} :=\inf\big\{t\geq 0 :  V_t^D  > s\big\}$ for $V^D$ defined by \eqref{definition:VD} and   we  set:
$$\theta_{u}^{r}:=\inf\big\{s \geq 0 \: :  \mathscr{L}^{\,r}_{\Gamma^{D_r}_{s}}> u \big\}, \quad  \text{ for all } u \in[0,\mathscr{L}^{\,r}_\sigma).$$
Remark that $\text{tr}_{D_r} \widehat{( {W} , {\Lambda})}_{\theta^r_u}=(x,r)$, for every $u\in[0,\mathscr{L}_\sigma^{\,r})$. An application of the special Markov property applied at the domain $D_r$ gives that, conditionally on $\mathcal{F}^{D_r}$, the point measure of the excursions from $D_r$
\begin{equation*}
    \mathcal{M}^{(r)} := \sum_{i \in \mathcal{I}_r} \delta_{(\mathscr{L}^{\,r}_{a_i^r} ,  \rho^{i,r} ,  \overline{W}^{i,r} )} 
\end{equation*}
is a Poisson point measure with intensity $ \mathbbm{1}_{[0, \mathscr{L}^{\,r}_\sigma ]}(u ) \dd u~  \mathbb{N}_{x,r}\left( \dd \rho , \dd \overline{W} \right)$.

\begin{lem} \label{lemma:suppordUnderN}
$\mathbb{N}_{x,0}$--a.e., we have $\{0,\sigma\}\in \mathrm{supp } ~ \dd A$. 
\end{lem}
\begin{proof}
We are going to  show that for any $\varepsilon>0$, we have $\mathbb{N}_{x,0}(A_{\varepsilon \wedge \sigma} =0 ) =0$ -- the Lemma will follow  since the symmetric statement $\mathbb{N}_{x}(A_{\sigma - \varepsilon} = 0) = 0$ will then hold by the duality identity  \eqref{dualidad:etaRhoW}. As  previously we write $$G_\lambda  := \mathbb{N}_{x}\big(1-\exp\big(-\lambda A_\infty)\big)=\widetilde{\psi}^{-1}(\lambda),$$ where the second equality holds by  Proposition \ref{lemma:divergExpo}. For every positive rational numbers $r$ and $q$, we introduce the stopping time 
$ T_{q}^r: = \inf \big\{ s \geq 0 : \mathscr{L}^{\,r}_s > q \big\}, 
$
with the  convention $T_{q}^r = \infty$, if $\mathscr{L}^{\,r}_\sigma \leq q$. Let us prove that 
\begin{equation} \label{equation:supportA}
    \mathbb{N}_{x,0}\big(A_{T_q^{r}} = 0 , \mathscr{L}^{\,r}_\sigma > 0 \big) = 0.
\end{equation}

In this direction, set $\mathbb{N}_{x,0}^r:= \mathbb{N}^{r}_{x,0}(\dd \rho, \, \dd \overline{W} | \mathscr{L}_\sigma^{r} > 0)$ and  using the fact that $\mathcal{M}^{(r)}$ is a Poisson point measure with intensity $\mathbbm{1}_{[0, \mathscr{L}^{\,r}_\sigma ]}( u ) \dd u~  \mathbb{N}_{x,r}\left( \dd \rho , \dd \overline{W} \right)$,  remark that 
\begin{align*}
    \mathbb{N}^{ r}_{x,0}\Big( \exp\big(-\lambda A_{T_q^{r}}\big)   \Big) 
    &\leq 
    \mathbb{N}^{ r}_{x,0}\Big( \exp \big( {-\lambda \sum_{i \in \mathcal{I}_r} A_\sigma(\rho^{i,r} , \overline{W}^{i,r}) \mathbbm{1}_{\{ \mathscr{L}^{\,r}_{a_i} \leq q \}}} \big)  \Big) \\
    &= \mathbb{N}^{ r}_{x,0} \Big( \exp \Big( - (q \wedge \mathscr{L}^{\,r}_\sigma)\cdot \mathbb{N}_{x}\big( 1-\exp(-\lambda A_\infty) \big)  \Big)  \Big) 
    = \mathbb{N}^{ r}_{x,0} \Big( \exp\big(-(q \wedge \mathscr{L}^{\,r}_\sigma)\cdot G_\lambda \big)  \Big), 
\end{align*}
and hence:
\begin{align*}
  \mathbb{N}_{x,0}^{ r} ( A_{T_q^r} = 0 )  + \mathbb{N}^{ r}_{x,0}\big( \exp\big(-\lambda A_{T_q^{r}}\big) \mathbbm{1}_{\{ A_{T_q^r} >0  \}}  \big) \leq \mathbb{N}^{ r}_{x,0} \left( \exp\big(-(q \wedge \mathscr{L}^{\,r}_\sigma) \cdot  G_\lambda \big)  \right).  
\end{align*}
Now \eqref{equation:supportA} follows  taking the limit as $\lambda \uparrow \infty$,  since we are working under $\{ \mathscr{L}^{\,r}_\sigma > 0 \}$ and by Proposition \ref{lemma:divergExpo} the function  $\widetilde{\psi}$ satisfies (A4), which gives that $G_\lambda$ goes to $\infty$ when $\lambda \uparrow \infty$. We stress that  \eqref{equation:supportA}  holds for any positive rational numbers $r$ and $q$. Now fix $\varepsilon>0$, and notice  that  by the monotonicity of $A$, we have
\begin{equation*}
    \big\{ A_{\varepsilon \wedge \sigma } = 0 ~;~ T_{q}^r < \varepsilon  \big\} \subset 
    \big\{ A_{T_q^r} = 0 ~;~ T_{q}^r < \varepsilon ~;~ \mathscr{L}^{\,r}_\sigma >0 \big\}, 
\end{equation*}
where the last set has null $\mathbb{N}_{x,0}$ measure by \eqref{equation:supportA}. The identity $\mathbb{N}_{x,0} (A_{\varepsilon \wedge \sigma} = 0) = 0$  now will follow as soon as we show that, $\mathbb{N}_{x,0}$-a.e.~, there exists two positive rational numbers $r$ and $q$ satisfying that $T_{q}^r < \varepsilon$. Said otherwise, we need to establish  that the origin is an accumulation point of  $\{ T_{q}^r :~ r,q \in \mathbb{Q}_+^* \}$. Arguing by contradiction, write
\begin{equation*}
    \Omega _0 = \bigcap_{r,q \in \mathbb{Q}_+^* } \big\{ T_{q}^r \geq \varepsilon \big\} 
    = \bigcap_{r \in \mathbb{Q}_+^*} \big\{ T_q^r \geq \varepsilon ~:~ \forall q > 0  \big\} 
    = \bigcap_{r\in \mathbb{Q}_+^*} \big\{ \mathscr{L}^{\,r}_\epsilon = 0 \big\} 
\end{equation*}
where in the last equality we used \eqref{equation:supportA}, and suppose that $\mathbb{N}_{x,0}( \Omega_0) >0$. To simplify notation, set $C(r) := \inf \{ s \geq 0 : \widehat{\Lambda}_s > r \}$, and remark that  the special Markov property, as stated in Theorem \ref{Theo_Spa_Markov_Excur}, applied to the domain $D_r$  gives $\{ \mathscr{L}^{\,r}_\epsilon = 0 \} = \{ C(r) \geq \varepsilon \}$. We then derive that
\begin{equation*} 
  0 <  \mathbb{N}_{x,0}\Big( \bigcap\limits_{r\in \mathbb{Q}_+^*} \{ C(r) \geq \varepsilon \}  \Big) 
  =  \mathbb{N}_{x,0}\Big( \widehat{\Lambda}_s = 0, \,  \forall s \in [0,\varepsilon \wedge \sigma]  \Big). 
\end{equation*}
However, recalling the definition \eqref{N:H(rho)} of the excursion measure $\mathbb{N}_{x,0}$ this is in contradiction  with the fact that for every $s\in (0,\sigma)$,  $\mathbb{N}_{x,0}$ a.e., $\widehat{\Lambda}_{s} > 0$. Indeed, by definition of the Lévy snake under $\mathbb{N}_{x,0}$, for any fixed $s$ the process $(W_s(t) , \Lambda_s(t) : t \leq \zeta_s)$ has the distribution of a trajectory of the Markov process $(\xi_t, \mathcal{L}_t : t \geq 0 )$ under $\Pi_{x,0}$ killed at $\zeta_s$. We then have  $\Lambda_s(t)>0$, for every $t>0$, since  $\mathcal{L}_t > 0$,  $\Pi_{x,0}$ a.s., and $\zeta_s = H(\rho_s)$ does not vanish on $(0,\sigma)$. 
\end{proof}
\noindent Define:
$$
\mathcal{C}^* := \Big\{ t \in [0,\sigma] : ~ \sup_{(t-\varepsilon, t+\varepsilon )\cap[0,\sigma]} \widehat{\Lambda} =  \inf_{(t-\varepsilon, t+\varepsilon )\cap[0,\sigma]} \widehat{\Lambda}~, \quad  \text{ for some } \varepsilon > 0  \Big\}, $$
and remark that -- the closure of the -- connected components of $\mathcal{C}^*$ are exactly the constancy intervals of $\widehat{\Lambda}$. We will show that the support of $\dd A$ is precisely the complement of $\mathcal{C}^*$. In this direction,  our goal now is to give  an equivalent definition of $\mathcal{C}^*$ in terms of $H$ and $W$, and for this purpose we  introduce the notion of exit times.

\begin{def1} \label{definition:exitTime} \emph{(Exit times from $x$)} A non negative number $t$ is said to be an exit time from the point $x$ for the process $(\rho,W)$ if $\widehat{W}_{t}=x$ and there exists $s>0$ such that  
\begin{equation*}
    H_t < H_{t+u}, \quad  \text{ for all }  \,  u \in (0,s).
\end{equation*}
  The collection of exit times from $x$ is denoted by $\emph{Exit}(x)$.
\end{def1}

\begin{remark} \normalfont 
Note that, for every $t\in \text{Exit}(x) $, the point $p_H(t)$  corresponds by definition to an interior point of the  Lévy tree and in fact, recalling the result of Proposition \ref{prop:branchingNotocax},  $p_H(t)$ is a point of multiplicity $2$ in $\mathcal{T}_H$. In particular, for every $t\in \text{Exit}(x)$, there exists a unique $s>t$ such that $p_{H}(t)=p_{H}(s)$ and satisfying that:
  \begin{equation*}
 \widehat{W}_s = x\quad \text{ and } \quad  H_{s-u} > H_t=H_s \quad   \text{ for all } \,  u \in (0,v),
\end{equation*}
for some $v>0$ -- in this case, we can take $v:= t-s$.  By analogy, we write  $\dminiArrow \text{Exit}(x)$ for the collection of times in $[0,\sigma]$ satisfying the previous display.
Remark that the correspondence described above between $\text{Exit}(x)$ and $\dminiArrow \text{Exit}(x)$ defines a bijection. We also stress  that the inclusion  $\text{Exit}(x) \cup \dminiArrow \text{Exit}(x) \subset \{t \in \mathbb{R}_+ : \widehat{W}_t = x \} $ is a priori strict since we are excluding in our definition potential times that will be mapped  by $p_H$ into leaves  with label $x$.  
\end{remark}

\noindent Let us now prove the following technical lemma:

\begin{lem} \label{L^rcapExi}
For every fixed $r>0$, under $\mathbb{N}_{x,0}$, we have:
\begin{equation}\label{supp L^r}
 \mathrm{supp } ~\dd \mathscr{L}^{\,r}
 = \overline{\{ a_i^r, b_i^r : \, i \in \mathcal{I}_r \} } 
 = \overline{\mathrm{Exit}(x)\cap \big\{s\in [0,\sigma]:~\widehat{\Lambda}_{s}=r\big\}}, 
\end{equation}
and the same identity holds if we replace $\overline{\mathrm{Exit}(x)}$ by $\overline{\dminiArrow \mathrm{Exit}(x)}$. In particular, the measure  $\dd A$ gives no mass to the complement of  $\overline{\mathrm{Exit}(x)}$ (or  $\overline{\dminiArrow \mathrm{Exit}(x)}$).
\end{lem}
\begin{proof}
First remark that if $\mathscr{L}^{\,r}_\sigma=0$, by the special Markov property applied to the domain $D_r$, all the sets appearing in \eqref{supp L^r} are empty. Hence, it suffices to show \eqref{supp L^r} under $\mathbb{N}_x^{r} := \mathbb{N}_{x}( \cdot \, | \mathscr{L}^{\,r}_\sigma >0)$. Moreover, notice that by definition we have:
\begin{equation*} 
    \{ a_i^r : i \in \mathcal{I}_r \} 
    =
    \mathrm{Exit}(x)\cap \big\{s\in [0,\sigma]:~\widehat{\Lambda}_{s}=r\big\}, \quad \text{ and } \quad 
    \{ b_i^r : i \in \mathcal{I}_r \} 
    =
    {{\dminiArrow}{}{\mathrm{Exit}(x)}} \cap \big\{s\in [0,\sigma]:~\widehat{\Lambda}_{s}=r\big\}. 
\end{equation*}
To deduce \eqref{supp L^r}, it is then enough to show that:
$$ \mathrm{supp } ~\dd \mathscr{L}^{\,r}
 = \overline{\{ a_i^r : \, i \in \mathcal{I}_r \} } ,$$
since the same equality will hold for $\{ a_i^r : i \in \mathcal{I}_r \}$ replaced by $\{ b_i^r : i \in \mathcal{I}_r \}$, using the duality identity \eqref{dualidad:etaRhoW} under $\mathbb{N}_{x,0}$.
\par 
So let us prove the previous display, and we start showing the inclusion $ \mathrm{supp } ~\dd \mathscr{L}^{\,r}
 \subset \overline{\{ a_i^r : \, i \in \mathcal{I}_r \} }$. In this direction, 
consider $s\in \mathrm{supp } ~\dd \mathscr{L}^{\,r}$.   By the special Markov property  the set  $\{\mathscr{L}_{a_i^r}^{\,r}:~i\in \mathcal{I}_r\}$ is dense in $[0,\mathscr{L}^{\,r}_\sigma]$, which gives that for every $\epsilon$ there exists $i\in \mathcal{I}_r$ such that $\mathscr{L}^{\,r}_{(s-\epsilon)+}<\mathscr{L}^{\,r}_{a^r_i}<\mathscr{L}^{\,r}_{s+\epsilon}$. This ensures that $a^r_i\in (s-\epsilon, s+\epsilon)$ by monotonicity of $\mathscr{L}^{\,r}$. Consequently, the set $\text{supp } ~\dd \mathscr{L}^{\,r}$ is included in the closure of  $\{a^r_i:~i\in \mathcal{I}_r\}$. Let us now establish the reverse inclusion by showing that for every $j\in \mathcal{I}_r$, we have $a_j^r\in \mathrm{supp } ~\dd \mathscr{L}^{\,r}$. In order to prove it, set $R_t := \sum_{\mathscr{L}_{a^r_i}^r \leq t }\sigma(\overline{W}^{i,r})$ for $t \geq 0$ and notice that it is a càdlàg  process since $R_\infty \leq \sigma <\infty$. Now remark that by definition, for every $k\in  \mathcal{I}_r$ with $a_k^r<a_j^r$, we have:
\begin{equation*}
a_j^{r}-a_k^r \leq R_{\mathscr{L}_{a_j^r}^r -} - R_{\mathscr{L}_{a_k^r}^r -} + \theta^r_{\mathscr{L}_{a_j^r}^r} -
\theta^r_{\mathscr{L}_{a_k^r}^r - }.
\end{equation*} 
Since $\theta^{r}$ is monotone, it has a countable number of discontinuities  and it follows by  the special Markov property --  using that $\theta^r$ is $\mathcal{F}^{D_r}$-measurable -- that all the points  $\{\mathscr{L}_{a^r_i}^{\,r}:~i\in \mathcal{I}\}$ are continuity points of $\theta^{r}$. Since $R$ is càdlàg,  this implies  that for every $\epsilon>0$ there exists $k\in \mathcal{I}_r$ such that $a_j^{r}-\epsilon<a_k^{r}<a_j^{r}$. All the values $\{\mathscr{L}_{a_i^r}^{\,r}:~i\in \mathcal{I}_r\}$ being distinct we derive that $a_j^{r}\in \mathrm{supp } ~\dd \mathscr{L}^{\,r} $, as wanted. 
 As a consequence of \eqref{supp L^r}, it follows that:
\begin{align*}
     \mathbb{N}_{x,0} \left( \int_0^\sigma \dd A_s \mathbbm{1}_{s\notin\overline{\mathrm{Exit}(x)}}  \right)
    =  
    \int_0^\infty \dd r \, \mathbb{N}_{x,0} \left( \int_0^\sigma \dd \mathscr{L}^{\,r}_s \mathbbm{1}_{s\notin\overline{\mathrm{Exit}(x)}}  \right)  
    =0,
\end{align*}
and we deduce by duality that $\dd A$ gives no mass to the complement of  $\overline{\mathrm{Exit}(x)}$ -- the same result holding for  $\overline{{\dminiArrow}{}{\mathrm{Exit}(x)}}$.
\end{proof}

The next proposition  establishes the connection between  the constancy intervals of $\widehat{\Lambda}$, the exit times from $x$ and the excursion intervals from $D_r$. This is the last result needed to characterise  the support of $\dd A$.

\begin{prop}
\label{proposition:exits_constancyLambda}
$\mathbb{N}_{x,0}$--a.e., we have:
\begin{equation} \label{identity:exit,exitvolteado,constanciaL}
    \overline{\mathrm{Exit}(x)}
    =\overline{{\dminiArrow}{}{\mathrm{Exit}(x)}}
    =\overline{ \{ a_i^r, b_i^r : r \in \mathbb{Q}_+^*\emph{ and }i\in \mathcal{I}_r  \} } 
    =[0,\sigma]\setminus\mathcal{C}^*.
\end{equation}
\end{prop}

\begin{proof}
The first step consists in showing 
\begin{equation} \label{equation:exits_step1.1}
    \overline{\mathrm{Exit}(x)}
    \subset \overline{\{ a_i^r, b_i^r : r \in \mathbb{Q}_+^*\text{ and }i\in \mathcal{I}_r  \}  }. 
\end{equation}
Remark that by Lemma \ref{L^rcapExi} the other inclusion is satisfied and still holds if we replace $\overline{\mathrm{Exit}(x)}$ by $\overline{{\dminiArrow}{}{\mathrm{Exit}(x)}}$. In this direction, recall  that by Lemma \ref{te_quedas_en_Theta} the process $(\rho,\overline{W})$ takes values in $\overline{\Theta}_{x}$. In particular, we have
\begin{equation*} 
  \mathbb{N}_{x,0}\text{--a.e.,  } \text{ for all } q \in (0,\sigma) , \, \,  \overline{\{ h < H_q : W_q(h) = x \}} = \text{supp } \Lambda_q(\dd h),\tag*{${(*)}$}
\end{equation*}
where we recall that  $\text{supp } \Lambda_q(\dd h)$ is precisely the set
\begin{equation*}
\Big\{ t \in [0,\zeta_q] : \,  \Lambda_q(t + h) > \Lambda_q(t) \text{ for any }    0 < h <   (H_q - t) \text{ or } \Lambda_q(t) > \Lambda_q(t - h) \text{ for any }   0 < h < t \Big\}.    
\end{equation*}
We let $\Omega_0 \subset \mathbb{D}(\mathbb{R}_+, \mathcal{M}_f(\mathbb{R}_+) \times \mathcal{W}_{\overline{E}})$ be a measurable subset with $\mathbb{N}_{x,0}(\Omega_0^c) = 0$ at which property $(*)$ holds for every $(\upvarrho, \omega) \in \Omega_0$ and we  argue for fixed $(\upvarrho, \omega) \in \Omega_0$.  Fix $t\in \text{Exit}(x)$; by definition, for any $\epsilon > 0$ we can find  $t<q < t+\epsilon$ such that $H_t<H_{r}$ for every $r\in(t,q]$.
By our choice of $\Omega_0$ and the snake property, it must hold either that:
\begin{itemize}
\item[\rm{(i)}] $H_t$ is a  time of right-increase for $\Lambda_q$ (and in particular  $\widehat{\Lambda}_q > {\Lambda}_q(H_t) = \widehat{\Lambda}_t$), or
\item[\rm{(ii)}] $H_t$ is \textit{not} a time of right-increase for  $\Lambda_q$, (and hence ${\Lambda}_q(H_t) > {\Lambda}_q(H_t-s), \,  \forall \,  0 < s < H_t$).
\end{itemize}
If (i) holds,  set $s_{k}:=\sup\{s\in[t,q]:~\widehat{\Lambda}_{s}\leq 2^{-k}\lfloor 2^{k}\widehat{\Lambda}_{t}\rfloor+2^{-k}\}$ and remark that we have $s_{k}\in\bigcup_{r\in \mathbb{Q}_+^*} \{ a_i^r, b_i^r :i\in \mathcal{I}_r  \} $, as soon as $\widehat{\Lambda}_{s_k}<\widehat{\Lambda}_{q}$. However, this is satisfied for $k$ large enough.  On the other hand, if (ii) holds we must have   $\inf_{[t-\epsilon , t]}H < H_t$ since $t$ can not be a local infimum for $H$ (otherwise, $p_H(t)$  would be a branching point with label $\widehat{W}_t = x$, in contradiction with Proposition \ref{prop:branchingNotocax}). Now, the argument of case (i)  holds by   working with  $s_k^\prime:=\sup\{s\in[0 , t]:~\widehat{\Lambda}_{s}\leq 2^{-k}\lfloor 2^{k}\widehat{\Lambda}_{t}\rfloor\}$. This implies   that   $t$ belongs to  the closure of  $\bigcup_{r\in \mathbb{Q}_+^*} \{ a_i^r, b_i^r :i\in \mathcal{I}_r  \} $ giving \eqref{equation:exits_step1.1}.  Moreover, by duality the  contention \eqref{equation:exits_step1.1} holds replacing $\text{Exit}(x)$ by ${\dminiArrow}{}{\mathrm{Exit}(x)}$, proving the first two equalities  in \eqref{identity:exit,exitvolteado,constanciaL}. Consequently, to conclude  it is enough to show that:
\begin{equation}\label{final:sub}
    \overline{\{ a_i^r, b_i^r : r \in \mathbb{Q}_+^*\text{ and }i\in \mathcal{I}_r  \}  }\subset  [0,\sigma] \setminus \mathcal{C}^*\subset \overline{\text{Exit}(x)\cup {\dminiArrow}{}{\mathrm{Exit}(x)}}.
\end{equation}
In this direction,  notice that for every $r\in \mathbb{Q}_+^*$, under $\mathbb{N}_{x,r}$, we have $\widehat{\Lambda}_t>r$ for every $t\in (0,\sigma)$. Now,  an application of the special Markov property applied to the domain $D_r$  gives that:
$$\{ a_i^r, b_i^r : ~i\in \mathcal{I}_r  \}\subset [0,\sigma] \setminus \mathcal{C}^*,\quad \mathbb{N}_{x,0}-\text{a.e.}, $$
for every $r\in \mathbb{Q}_+^*$, and  the first inclusion $\subset$ in \eqref{final:sub} follows. In order to obtain the remaining inclusion,  let $t\in [0,\sigma]\setminus \mathcal{C}^*$. By definition,  for every  $\epsilon > 0$   there exists $t-\epsilon<t_1< t_2 < t + \epsilon$ such that $\widehat{\Lambda}_{t_1} < \widehat{\Lambda}_{t_2}$ or $\widehat{\Lambda}_{t_1} > \widehat{\Lambda}_{t_2}$. If the first holds, then $\sup \{ s \in [t-\epsilon, t_2] : \,  \widehat{\Lambda}_{s} \leq \widehat{\Lambda}_{t_1} \}$ is an exit time and the other case follows by taking $\inf \{ s \in [t_1,t_2] : \widehat{\Lambda}_s \leq \widehat{\Lambda}_{t_2} \}$. This ensures that $t$ is in the closure of  $\text{Exit}(x)\cup {\dminiArrow}{}{\mathrm{Exit}(x)}$ concluding our proof. 
\end{proof}

\noindent Now, we are in position to state and prove the main result of the section: 

\begin{theo}\label{prop:suppA} Fix $(y,r_0)\in \overline{E}$ and $(\mu,\overline{\emph{w}})\in \overline{\Theta}_x$.  Under $\mathbb{P}_{\mu,\overline{\emph{w}}}$ and $\mathbb{N}_{y,r_0}$, we have 
\begin{equation*} 
   \emph{supp } \dd A =  
   \overline{\mathrm{Exit}(x)}= \overline{{\dminiArrow}{}{\mathrm{Exit}(x)}} = [0,\sigma] \setminus \mathcal{C}^* ,
\end{equation*}
where we recall the convention $[0,\infty]=[0,\infty)$.
\end{theo}

\begin{proof} 
First remark that by the special Markov property combined with \eqref{remark:descomponerA} and \eqref{remark:descomponerA_v2}, it is enough to prove the theorem under $\mathbb{N}_{x,0}$ and $\mathbb{P}_{0,x,0}$. We start by proving the theorem under $\mathbb{N}_{x,0}$ and remark that by Proposition \ref{proposition:exits_constancyLambda} we only have to establish the first equality. Moreover,  by  Lemma \ref{L^rcapExi} it only remains to show that under $ \mathbb{N}_{x,0}$:
\begin{equation} \label{equation:exitsUltimacont}
    \text{supp } \dd A \supset  
   \overline{\text{Exit}(x)}. 
\end{equation}
However,  by Lemma  \ref{lemma:suppordUnderN} we know that  $\mathbb{N}_{x,0}( \{0,\sigma\}\cap \text{supp } \dd A=\emptyset ) = 0$, and  then using that conditionally on $\mathcal{F}^{D_r}$ the measure $\mathcal{M}^{(r)}$ is a Poisson point measure with intensity  $ \mathbbm{1}_{[0, \mathscr{L}^{\,r}_\sigma ]}(\ell ) \dd \ell~  \mathbb{N}_{x,r}\left( \dd \rho , \dd \overline{W} \right)$, we derive that:  
\begin{equation*}
    \mathbb{N}_{x,0}-\text{a.e.,  }  \text{ for all } r \in \mathbb{Q}_+^*, \, \,  \{a_i^r,b_i^r : i \in \mathcal{I}_r \}  \subset \text{supp } \dd A.
\end{equation*}
Consequently, Proposition \ref{proposition:exits_constancyLambda} implies \eqref{equation:exitsUltimacont}. Finally, let us briefly explain how to obtain the result under $\mathbb{P}_{0,x,0}$. In this direction, under $\mathbb{P}_{0,x,0}$,  denote  the connected components of $\{ s \in \mathbb{R}_+ : X_s - I_s \neq  0 \}$ by $\big((\alpha_i, \beta_i) : i \in \mathcal{I} \big)$. Excursion theory and our results under $\mathbb{N}_{x,0}$, give that,  under  $\mathbb{P}_{0,x,0}$, we have:
 $$ \text{supp } \dd A \cap \cup_{i} (\alpha_i, \beta_i) =  
 \overline{\text{Exit}(x)} \cap \cup_{i} (\alpha_i, \beta_i) = \overline{\dminiArrow \text{Exit}(x)}~\cap \cup_{i} (\alpha_i, \beta_i)= \big([0,\sigma] \setminus \mathcal{C}^*\big) \cap \cup_{i} (\alpha_i, \beta_i), $$  
    $\alpha_i\in \text{supp } \dd A\cap \overline{\text{Exit}}(x)\cap \big([0,\infty)\setminus \mathcal{C}^*\big)$ and  $\beta_i\in \text{supp } \dd A\cap  {\dminiArrow}{}{\mathrm{Exit}(x)}\cap \big([0,\infty)\setminus \mathcal{C}^*\big)$ for every $i\in \mathcal{I}$. The desired result now follows since the set $\{\alpha_i:~i\in \mathcal{I}\}$ and $\{\beta_i:~i\in \mathcal{I}\}$ are dense in  $\{ s \in \mathbb{R}_+ : X_s - I_s =  0 \}$. 
\end{proof}

\section{The tree structure of \texorpdfstring{$\{\upsilon\in \mathcal{T}_H:~\xi_\upsilon=x\}$}{Lg} }\label{section:treeStructureLocalTime}

 In this section, we work under  the  framework introduced at the beginning of Section  \ref{section:structureOfTloc}.
 Our goal now is to study the structure of the set $\{ \upsilon \in \mathcal{T}_H : \xi_\upsilon = x \}$ and to do so, we  encode it by \textit{the subordinate tree  of $\mathcal{T}_H$ with respect to the local time $(\mathcal{L}_\upsilon: \upsilon \in \mathcal{T}_H)$}.  In this direction, we need to briefly recall  the notion of subordination of trees defined in \cite{Subor}. 
 \medskip \\
\noindent \textbf{Subordination of  trees by increasing functions. }Let $(\mathcal{T},d_{\mathcal{T}}, \upsilon_0)$ be an $\mathbb{R}$-tree and recall the standard notation $\preceq_\mathcal{T}$ and $\curlywedge_\mathcal{T}$ for the ancestor order and the first common ancestor. Next, consider  $g:\mathcal{T}\to \mathbb{R}_{+}$ a non-negative continuous function. We say that $g$ is non-decreasing if for every $u,v\in \mathcal{T}$:
\begin{equation*}
 u\preceq_{\mathcal{T}} v  \text{ implies that } g(u)\leq g(v).
\end{equation*}
 When the later holds,  we can define a pseudo-distance on $\mathcal{T}$ by setting 
\begin{equation}\label{eq:d:subor}
d_{\mathcal{T}}^{g}(u,v):=g(u)+g(v)-2 \cdot g(u\curlywedge_{\mathcal{T}} v),  \quad \quad (u,v)\in\mathcal{T}\times \mathcal{T}.
\end{equation}
The pseudo-distance $d_\mathcal{T}^g$ induces the following equivalence relation on $\mathcal{T}$: for $u, v \in \mathcal{T}$ we write
\begin{equation*}
u\sim_{\mathcal{T}}^{g}  v \iff d_{\mathcal{T}}^{g}(u,v)=0,
\end{equation*}
and it was shown in \cite{Subor} that $\mathcal{T}^{g}:=(\mathcal{T}/\sim_{\mathcal{T}}^{g},d_{\mathcal{T}}^{g},\upsilon_0)$ is a compact pointed $\mathbb{R}$-tree, where we still denoted  the equivalency class of the root of $\mathcal{T}^g$ by $\upsilon_0$. The tree $\mathcal{T}^{g}$ is called the subordinate tree of $\mathcal{T}$ with respect to $g$ and we write  $p_{g}^{\mathcal{T}}:\mathcal{T}\to\mathcal{T}^{g}$ for  the canonical projection which associates  every $u\in\mathcal{T}$ with  its $\sim_{\mathcal{T}}^{g}$--equivalency class. Observe that any two points $u,v \in \mathcal{T}$ are identified if and only if $g$  stays constant on $[u,v]_\mathcal{T}$ and consequently the subordinate tree is obtained from $\mathcal{T}$ by identifying in a single point the components of $\mathcal{T}$ where $g$ is constant. \smallskip \\ 
\par Getting back to our setting, recall that under $\mathbb{N}_{x,0}$,    $(\mathcal{L}_\upsilon: \upsilon \in \mathcal{T}_H)$ corresponds to   $(\widehat{\Lambda}_t : t  \geq 0)$  in the quotient space $\mathcal{T}_H = [0,\sigma]/ \sim_H$. This entails that the local time $(\mathcal{L}_\upsilon : \upsilon \in \mathcal{T}_H)$ is a non-decreasing function on $\mathcal{T}_H$ and we denote  the induced subordinate tree by $\mathcal{T}_H^{ \mathcal{L} }$.  Recall  that  the exponent 
\begin{equation*} 
    \widetilde{\psi}(\lambda) = \mathcal{N}\left( \int_0^\sigma \dd h \,  \psi(u_\lambda(\xi_h)) \right), \quad \text{ for } \, \lambda \geq 0,  
\end{equation*}
is the exponent of a Lévy tree by Proposition \ref{lemma:divergExpo}.  
Hence, it satisfies  (A1)---(A4)  and by Corollary \ref{corollary:phitildeParametros} it can be written in  the following form:
$$\widetilde{\psi}(\lambda):=\widetilde{\alpha}\lambda+\int_{(0,\infty)}\widetilde{\pi}(\dd x)(\exp(-\lambda x)-1+\lambda x), $$
where $\widetilde{\alpha}=\mathcal{N}\big(1-\exp(-\alpha\sigma)\big)$ and $\widetilde{\pi}$ is a sigma-finite measure on $\mathbb{R}_+\setminus\{0\}$ satisfying $\int_{(0,\infty)}\widetilde{\pi}(\dd x)(x\wedge x^{2})<\infty$. We will also use  the notation $\widetilde{H}$ and $\widetilde{N}$  introduced prior to \eqref{identity:inversePhi} for the height process and the excursion measure of a $\widetilde{\psi}$--Lévy tree.  Finally, we recall that $A$ stands for the additive functional introduced in Proposition \ref{proposition:aditivaDefinicion} and we denote its right inverse by $A_t^{-1}:=\inf\{s\geq 0:A_{s}>t\}$, with the convention $A^{-1}_{t}=\sigma$ for every $t\geq A_\infty = A_\sigma$. Remark that the constancy intervals of $A$ in $[0,\sigma]$ are the connected components of $[0,\sigma] \setminus \text{supp } \dd A$, which  by Theorem \ref{prop:suppA}  are precisely the connected components of $\mathcal{C}^*$ -- the constancy intervals of the process $(\widehat{\Lambda}_{t}:~t\in [0,\sigma])$. In particular,  $(\widehat{\Lambda}_{A^{-1}_t}: t \geq 0)$ is  a continuous non-negative process, with lifetime $A_\infty$. We can now  state the main result of this section:

\begin{theo}\label{theorem:subordinatedTreebyLocalT} The following properties hold:
\begin{itemize}
    \item [\rm{(i)}] Under $\mathbb{N}_{x,0}$, the subordinate tree  of $\mathcal{T}_H$ with respect to the local time $\mathcal{L}$, that we denote by   $\mathcal{T}_H^{ \mathcal{L}}$,  is isometric to the tree coded by the continuous function $(\widehat{\Lambda}_{A^{-1}_t}:~ t \geq 0).$

    \item [\rm{(ii)}] Moreover, we have the equality in distribution 
\begin{equation} \label{equation:leyHtilde}
  \Big( (\widetilde{H}_t: \, t \geq 0), \text{ under }\widetilde{N} \Big)\overset{(d)}{=}\Big( \big( \widehat{\Lambda}_{A^{-1}_t} :\,  t \geq 0\big),  \text{ under } \mathbb{N}_{x,0}\Big).
\end{equation}
In particular, $\mathcal{T}_H^{ \mathcal{L}}$ is a Lévy tree with exponent $\widetilde{\psi}$. 
\end{itemize}
\end{theo}

\begin{remark}
\emph{
Let us mention that when $\psi(\lambda)=\lambda^{2}/2$ and the underlying spatial motion $\xi$ is a Brownian motion in $\mathbb{R}$, the previous theorem implies that under $\mathbb{N}_{0,0}$ the subordinate tree of $\mathcal{T}_H$ with respect to the local time $\mathcal{L}$ at $0$ is a Lévy tree and -- as a direct consequence of the scaling invariance of the Brownian motion -- its exponent is of the form $\widetilde{\psi}(\lambda)=c\lambda^{3/2}$, for some constant $c>0$. This result was already obtained by other methods in \cite[Theorem 2]{Subor}. 
}
\end{remark}

We stress that the key result in (ii) is the identity in distribution \eqref{equation:leyHtilde}:  it entails that not only the function $(\widehat{\Lambda}_{A^{-1}_t}: t \geq 0)$ encodes the subordinate tree, but it is also  the height process of a Lévy tree. The fact that $\mathcal{T}_{H}^{\mathcal{L}}$ is a $\widetilde{\psi}$-Lévy tree is then a direct consequence of (i) and \eqref{equation:leyHtilde}.  By a straightforward application of excursion theory one can deduce a version under  $\mathbb{P}_{0,x,0}$ of Theorem \ref{theorem:subordinatedTreebyLocalT}, where now $\mathcal{T}_H^{ \mathcal{L}}$ is a Lévy forest with exponent $\widetilde{\psi}$. The details are left to the reader. 
\medskip \\
\textit{The rest of the section is organised as follows: }The section is devoted to the proof of Theorem \ref{theorem:subordinatedTreebyLocalT}. In Section  \ref{subsection:5.1} we start by  showing (i)  and we present the strategy  that we  follow  to prove  (ii). The proof of (ii) relies in   all the machinery developed in previous sections combined with standard properties of Poisson point measures and is the content of Section \ref{subsection:prooflawEmbededTree}. 

\subsection{The height process of the  subordinate tree} \label{subsection:5.1}

 In this short section we establish the first claim of Theorem \ref{theorem:subordinatedTreebyLocalT} and settle the ground for the second part of the result. For every $u\in \mathcal{T}_{H}$, recall that $\mathcal{L}_{u}:=\widehat{\Lambda}_{s}$  where $s$ is any element of $p_{H}^{-1}(\{u\})$ (note that the definition is non ambiguous by the snake property) and that $\mathcal{L}$ is non-decreasing on $\mathcal{T}_H$. To simplify notation, we set:
$${H}^A_t:=\widehat{\Lambda}_{A^{-1}_t},\quad t\geq 0,$$
which is a continuous process -- as it was  already mentioned  in the discussion before Theorem \ref{theorem:subordinatedTreebyLocalT}. Let us start with the proof of  Theorem  \ref{theorem:subordinatedTreebyLocalT}-(i). 

\begin{proof}[Proof of Theorem  \ref{theorem:subordinatedTreebyLocalT}-(i)] Our goal is to show that, under $\mathbb{N}_{x,0}$,  the trees $\mathcal{T}_{H^A}$ and $\mathcal{T}_H^{ \mathcal{L}}$ are isometric. In this direction, we start by introducing the pseudo-distance:
$$\widetilde{d}(s,t):=\widehat\Lambda_{t}+\widehat\Lambda_s-2\cdot\min\limits_{s\wedge t,s\vee t} \widehat\Lambda ~,\quad s,t\in[0,\sigma],$$
and we write $s\approx t$ if and only if $\widetilde{d}(s,t)=0$. By the snake property,  we have $s\approx t$ for every $s\sim_H t$. Moreover, since $\mathcal{L}$ is increasing on $\mathcal{T}_H$, we get
\begin{align*}
  \widetilde{d}(s,t) &= \mathcal{L}_{p_H(t)}+\mathcal{L}_{p_H(s)}-2\cdot  \mathcal{L}_{p_H(s)\curlywedge_{\mathcal{T}_H} p_H(t)},
  \end{align*}
for every $s,t\in [0,\sigma]$. 
The right-hand side of the previous display is exactly the definition of the pseudo-distance  associated with the subordinate tree $\mathcal{T}_H^{ \mathcal{L}}$ between $p_{H}(s)$ and $p_{H}(t)$ given in \eqref{eq:d:subor}. We deduce that $([0,\sigma]/\approx, \widetilde{d},0)$ is isometric to $\mathcal{T}^{\mathcal{L}}_H$. It remains to show that $([0,\sigma]/\approx, \widetilde{d},0)$ is also isometric to $(\mathcal{T}_{H^A},d_{H^A},0)$. In order to prove it, we notice that:
$$\widetilde{d}(A_{r_1}^{-1}, A_{r_2}^{-1})=d_{H^A}(r_1,r_2), $$
for every $r_1,r_2\in [0,A_\sigma]$. Furthermore, for every $t\in [0,\sigma]$ there exists  $r\in[0,A_\sigma]$ such that $A_{r-}^{-1}\leq t \leq A_{r}^{-1}$ since by Lemma \ref{lemma:suppordUnderN} the points $0$ and $\sigma$ are in the support of $\dd A$. Moreover we have $\widetilde{d}(A_r^{-1},t)=0$, since by  Theorem \ref{prop:suppA} the process $\widehat{\Lambda}$ stays constant on every interval of the form $[A_{r-}^{-1},A_{r}^{-1}]$. This implies that $[0,\sigma]/\approx~=\{A_{r}^{-1}:~r\in [0,A_\infty]\}/\approx$ and we deduce by the previous display that $([0,\sigma]/\approx, \widetilde{d},0)$ and  $(\mathcal{T}_{H^A},d_{H^A},0)$ are isometric giving the desired result.
\end{proof}
The main difficulty  to establish Theorem \ref{theorem:subordinatedTreebyLocalT} (ii) comes from the fact that $\widetilde{H}$ is not a Markov process. To circumvent this, we   are going to use  the notion of marked trees embedded in a function. \medskip \\

\noindent  \textbf{Marked trees embedded in a function.}
A marked tree is a pair $\textbf{T}:=(\TT, \{ h_v:~ v \in \TT \})$, where $\TT$ is a finite rooted ordered  tree and  $h_v \geq 0$  for every $v \in \TT$ -- the number $h_v$ is called the label of the individual $v$.  For completeness let us give the formal definition of a rooted ordered tree. First, introduce  Ulam's tree:
$$\mathcal{U} := \bigcup_{n=0}^{\infty} \{1,2,...\}^n$$
where by convention $\{1,2,...\}^0 = {\varnothing}$. If $u = (u_1,...u_m)$ and $v = (v_1,...,v_n)$ belong to $\mathcal{U}$, we write $uv$ for the concatenation of $u$ and $v$, viz. $(u_1,...u_m,v_1,...,v_n)$ . In particular, we have $u\varnothing=\varnothing u=u$.
A (finite) rooted ordered tree $\TT$ is a finite subset of $\mathcal{U}$ such that:
\begin{itemize}
\item[\rm{(i)}] $\varnothing\in \TT$;
\item[\rm{(ii)}] If $v\in \TT$ and $v=uj$ for some $u\in \mathcal{U}$ and $j\in \{1,2,...\}$, then $u\in \TT$;
\item[\rm{(iii)}]  For every $u \in \TT$,  there exists a number $k_u(\TT)\geq 0$ such that $uj\in \TT$ if and only if $1 \leq  j \leq k_u(\TT )$.  
\end{itemize}
 If $u \in \TT$ can be written as  $u=v j$ for some $v \in \TT$, $1 \leq j \leq k_v(\text{T})$, we say that $v$ is the parent of $u$. More generally, if   $u=v y$ for some $v \in \TT$ and $y \in \mathcal{U}$ with $y \neq \emptyset$, we say that $v$ is an ancestor of $u$ or equivalently that $u$ is a descendant  of $v$. On the other hand, if $u \in \TT$ satisfies that  $k_u(\TT) = 0$,  $u$ is called a leaf. The element $\varnothing$ is interpreted as the root of the tree and if $v$ is a vertex of $\TT$, the branch connecting the root and $v$ is the set of prefixes of $v$ --  considered with its corresponding family of labels. \par
 Let us also introduce the concatenation of  marked trees. If $\textbf{T}_1,...,\textbf{T}_k$ are $k$ marked trees and $h$ is  a non-negative real number, we write 
$[\textbf{T}_1,...,\textbf{T}_k]_h $ for the marked tree defined as follows. The lifetime of $\varnothing$ is $h$, $k_{\varnothing}=k$, and for $1\leq j\leq k$ the point $ju$ belongs to the tree structure of $[\textbf{T}_1,...,\textbf{T}_k]_h$ if and only if $u\in \TT_j$ and its label  is the label of $u$ in $\textbf{T}_j$. For convenience, we will  identity a marked tree $\textbf{T}:=(\TT, \{ h_v:~ v \in \TT \})$ with the set $\{(v,h_v):~v\in \TT\}$.
\\
\\
We are now in position to define the embedded   marked tree associated with  a continuous function $(e(t))_{t \in[a, b]}$ and a given finite collection of times. We fix a finite sequence  of times 
$a \leq t_1 \leq \dots t_n \leq b$ and we recall the notation $m_e(s,t)=\inf_{[s\wedge t,s\vee t]}e$. The embedded tree associated with the marks $t_1, \dots, t_n$ and the function $e$, $\theta(e,t_1, \dots , t_n)$,
is defined inductively, according to the following steps: 
\begin{itemize}
    \item If $n = 1$, set $\theta(e,t_1) = ( \emptyset,\{e(t_1)\} )$. 
    \item If $n \geq 2$, suppose that we know how to construct   marked trees with less than $n$ marks. Let  $i_1, \dots , i_k$  be the distinct indices satisfying that $m_e(t_{i_q},t_{i_q +1}) = m_e(t_{1},t_{n})$, and
    define the following restrictions for $1 \leq q \leq k-1$
\begin{align*}
    e^{(0)}(t) := (e(t)  : \,  t \in [t_1,t_{i_1}]), \, \, 
    e^{(q)}(t) := (e(t)  : \,  t \in [t_{i_q+1},t_{i_{q+1}}]), \, \, 
    e^{(k)}(t) := (e(t)  : \,  t \in [t_{i_{k}+1},t_{n}]).
\end{align*}
    Next, consider the associated finite labelled trees, 
\begin{align*} 
    \theta(e^{(0)},t_1, \dots , t_{i_1}), \,   \theta(e^{(q)},t_{i_q+1}, \dots , t_{i_{q+1}}), \, 
    \theta(e^{(k)},t_{i_k+1}, \dots , t_{n}), \quad  \text{ for } 1 \leq q \leq k-1, 
\end{align*}
and finally,  concatenate them with a common ancestor with label $m_e(t_1, t_n)$, by  setting  
\begin{equation*}
    \theta(e,t_1, \dots , t_n) 
    := [ \theta(e^{(0)},t_1, \dots , t_{i_1}), \dots , 
    \theta(e^{(k)},t_{i_k+1}, \dots , t_{n})]_{m_e(t_1, t_n)}, 
\end{equation*}
and completing the recursion.
\end{itemize}

We say that the label $h_v$ is the height of $v$ in $ \theta(e,t_1, \dots , t_n)=(\TT, \{h_v:~v\in \TT\})$. Let us justify this terminology. First assume that $e(0)=0$ and consider $\mathcal{T}_e$ the compact $\mathbb{R}$--tree induced by $e$. Then if $v_{1}, \dots , v_{n}$ are the leaves of $\textbf{T}$ in lexicographic order,  we have  
$(h_{v_{1}}, \dots , h_{v_{n}}) = (e(t_1), \dots , e(t_n))$. Moreover,  if we write $v_i\curlywedge_{\TT} v_j$ for the common  ancestor of $v_i$ and $v_j$ in $\textbf{T}$, it holds that $h_{v_j \curlywedge_{\TT} v_i} = \inf_{[t_i \wedge t_j , t_i \vee t_j]} e$.\footnote{The definition of  $ \theta(e,t_1, \dots , t_n)$ is directly connected with the classical notion of marginals trees -- where the label of a point is the increment between its height and the height of its parent.}
\\
\\
\noindent \textbf{Statements and main steps for the proof of Theorem \ref{theorem:subordinatedTreebyLocalT} (ii). }Our argument relies in  identifying the distribution of the discrete embedded tree associated with  $(\widehat{\Lambda}_{{A}^{-1}_t}:{0 \leq t \leq A_\infty})$ when the collection of marks are  Poissonian. In this direction, we denote the law of a Poisson process $(\mathcal{P}_t:~t\geq 0)$ with intensity $\lambda$ by $Q^\lambda$  and 
 we   work with the pair $(H^A_t, \mathcal{P}_t  )_{t \leq A_\infty}$, under the product measure $ \mathbb{N}_{x,0} \otimes \,  Q^\lambda$. 
 For convenience, we denote the law of $(\rho, \overline{W}, \mathcal{P}_{\cdot \wedge A_\infty })$ under $\mathbb{N}_{x,0} \otimes \,  Q^\lambda$ by $\mathbb{N}_{x,0}^\lambda$  
 and we  let  $0 \leq \mathfrak{t}_1< \dots < \mathfrak{t}_M \leq A_\infty$ be the jumping times of $(\mathcal{P}_t)$ falling in the excursion interval $[0,A_\infty]$, where   $M := \mathcal{P}_{A_\infty}$. Finally,  consider the associated embedded tree
\begin{equation*}
  \textbf{T}^A :=  \theta \big(  H^A, \mathfrak{t}_1 , \dots , \mathfrak{t}_M \big), \quad \text{ under } \mathbb{N}^\lambda_{x,0}( \, \cdot \,  | M \geq 1).
\end{equation*}
Remark that the probability measure $\mathbb{N}^\lambda_{x,0}( \, \cdot \,  | M \geq 1)$ is well defined since by Proposition \ref{lemma:InversaExpo} we have
\begin{equation*} 
    \mathbb{N}^\lambda_{x,0}\left( M \geq 1 \right) 
    = \mathbb{N}_{x,0}\left( 1- \exp(-\lambda A_\infty) \right) = \widetilde{\psi}^{-1}(\lambda).
\end{equation*}
Our goal is to show that $\textbf{T}^A$ is distributed as the discrete embedded  tree  of a $\widetilde{\psi}$-Lévy tree associated with Poissonian marks with intensity $\lambda$. To state this formally, recall the notation $\widetilde{N}$ for the excursion measure of a $\widetilde{\psi}$-Lévy process, and that $\widetilde{H}$ stands for the associated height process. We write  $\widetilde{N}^\lambda$ for the law of $(\widetilde{\rho} , \mathcal{P}_{\cdot \wedge \sigma_{\widetilde{H}}})$ under $\widetilde{N} \otimes Q^\lambda$ and  remark that  $\widetilde{M} := \mathcal{P}_{\sigma_{\widetilde{H}}}$ is  the number of Poissonian marks in $[0,\sigma_{\widetilde{H}}]$. For simplicity, we denote the jumping times of $\mathcal{P}$ under $\widetilde{N}^\lambda$ by $\mathfrak{t}_1,\dots, \mathfrak{t}_{\widetilde{M}}$.
\begin{prop} \label{proposition:subordinatedTreeMarginal}
The discrete tree
$\emph{\textbf{T}}^A$ under $\mathbb{N}^{\lambda}_{x,0}( \, \cdot \, | M \geq 1)$ has the same distribution as 
\begin{equation*}
    \widetilde{\emph{\textbf{T}}} := \theta\big( \widetilde{H} , \mathfrak{t}_1, \dots , \mathfrak{t}_{\widetilde{M}}\big) \quad \quad \text{under } \widetilde{N}^\lambda( \, \cdot \,  | \widetilde{M} \geq 1).
\end{equation*}
\end{prop}
The proof of Proposition \ref{proposition:subordinatedTreeMarginal} is rather technical and will be postponed to  Section \ref{subsection:prooflawEmbededTree}.
{The reason behind considering  Poissonian marks to identify the distribution of $H^A$ is to take advantage of the memoryless of Poissonian marks;  this flexibility will allow us to make extensive use of the Markov property and excursion theory.}
Let us now  explain how to  deduce Theorem \ref{theorem:subordinatedTreebyLocalT} (ii) from Proposition  \ref{proposition:subordinatedTreeMarginal}.
\begin{proof}[Proof of Theorem  \ref{theorem:subordinatedTreebyLocalT} (ii)] 
First remark  that the fact that $\mathcal{T}_{H}^{\mathcal{L}}$ is a $\widetilde{\psi}$-Lévy tree is a direct consequence of Theorem \ref{theorem:subordinatedTreebyLocalT} (i) and \eqref{equation:leyHtilde}. To conclude it remains to prove \eqref{equation:leyHtilde}.
In this direction,  notice that the marked trees  considered  are ordered trees -- the order of the vertices being the one induced by the marks. Recall that for every $1\leq i\leq M$, the quantity $H^A_{\mathfrak{t}_i}$ is the label of  the $i$-th leaf in lexicographical order, and  the same remark holds  replacing $(H^A,M, \textbf{T}^A)$ by $(\widetilde{H},\widetilde{M}, \widetilde{\textbf{T}})$. Consequently, the identity $\textbf{T}^A \overset{(d)}{=}\widetilde{\textbf{T}}$ of Proposition \ref{proposition:subordinatedTreeMarginal} yields the following equality in distribution 
\begin{equation*}
    \Big(\big( \widetilde{M} , \widetilde{H}_{\mathfrak{t}_1}, \dots ,  \widetilde{H}_{\mathfrak{t}_{\widetilde{M}}}\big):~\widetilde{N}^\lambda( \, \cdot \,  | \widetilde{M} \geq 1)\Big) 
    \overset{(d)}{=}
    \Big(\big(M , H^A_{\mathfrak{t}_1}, \dots , H^A_{\mathfrak{t}_M}\big):~\mathbb{N}^{\lambda}_{x,0}( \, \cdot \, | M \geq 1) \Big). 
\end{equation*}
Recall from Proposition \ref{lemma:InversaExpo} and the discussion after it, that $A_\infty$ under $\mathbb{N}_{x,0}$ and $\sigma_{\widetilde{H}}$ under $\widetilde{N}$ have the same distribution. This ensures that, up to enlarging the measure space, we can define  the height  process $\widetilde{H}$ under the measure $\mathbb{N}_{x,0}^\lambda$ in such a way that its  lifetime  is precisely $A_\infty$, viz. $\sigma_{\widetilde{H}} = A_\infty$, and then we might and will consider the same  collection of Poisson marks $\mathfrak{t}_1, \dots, \mathfrak{t}_M$ to mark the processes $H^A$ and $\widetilde{H}$. In the rest of the proof, we work with this  coupling. In particular, under $\mathbb{N}^{\lambda}_{x,0}( \, \cdot \, | M \geq 1)$, our previous discussion entails 
\begin{equation*}
    \left( M, \widetilde{H}_{\mathfrak{t}_1}, \dots , \widetilde{H}_{\mathfrak{t}_M} \right ) 
    \overset{(d)}{=}
    \left(  M , H^A_{\mathfrak{t}_1}, \dots , H^A_{\mathfrak{t}_M} \right ). 
\end{equation*}
Let $( U_i:~i\geq 1  )$ be a collection of independent identically distributed  uniform random variables  in $[0,A_\infty]$ -- and independent of all the rest. Remark that,  conditionally on $A_\infty$, $(\mathcal{P}_t : t \leq A_\infty)$ is independent of $\widetilde{H}$ and $H^A$, and the random variable  $M$ is  Poisson with intensity $(\lambda A_\infty)$. By conditioning on $A_\infty$, we deduce that for any $m \geq 1$ and any  measurable function  $f:\mathbb{R}^{m}\mapsto \mathbb{R}_+$, we have
\begin{align} \label{equation:leyHtilde_}
 \mathbb{N}^\lambda_{x,0}\left( f( \widetilde{H}_{U^m_{(1)}}, \dots , \widetilde{H}_{U^m_{(m)}} ) \frac{( \lambda A_\infty )^{m}}{m!}\exp\big(-\lambda A_\infty \big) \right) =
    \mathbb{N}_{x,0}^\lambda \left( f( H^A_{U^m_{(1)}}, \dots , H^A_{U^m_{(m)}}  ) \frac{( \lambda A_\infty )^{m}}{m!}\exp\big(-\lambda A_\infty \big) \right),
\end{align}
where  $( U^m_{(1)}, \dots , U^m_{(m)} )$ stands for  the order statistics of $\{ U_1, \dots , U_m \}$. Since the previous display holds for every $\lambda> 0$ we get that
\begin{equation*}
    \left(A_\infty,  \widetilde{H}_{U^m_{(1)}}, \dots , \widetilde{H}_{U^m_{(m)}} \right ) 
    \overset{(d)}{=}
    \left(  A_\infty , H^A_{U^m_{(1)}}, \dots , H^A_{U^m_{(m)}} \right ),
\end{equation*}
for every $m\geq 1$.
Denote   the unique continuous function vanishing at $\mathbb{R}_+ \setminus  (0,A_\infty)$ and linearly interpolating between the points $\{ ( A_\infty \cdot im^{-1} , \widetilde{H}_{U_i^{(m)}}) : i \in  \{ 1, \dots , m \}\}\cup \{ (0,0) , (A_\infty , 0) \} $ by  $(\widetilde{H}^m_t:~ t \geq 0)$. Similarly, let   ${H}^{A,m}$ be the analogous function defined by replacing $\widetilde{H}$ by $H^A$. Next we remark that \eqref{equation:leyHtilde_} ensures that, for every continuous bounded function $F: \mathbb{R}^{k}  \mapsto \mathbb{R}_+$ and any fixed set of times $0 \leq t_1 \leq \dots \leq t_k$,  we have
\begin{equation*}
    \mathbb{N}_{x,0}^\lambda \left( F ( \widetilde{H}^m_{t_1}, \dots , \widetilde{H}^m_{t_k}  )\cdot\big(1-\exp(-A_\infty)\big) \right) 
    = 
    \mathbb{N}_{x,0}^\lambda \left( F ( H^{A,m}_{t_1} , \dots ,  H^{A,m}_{t_k})\cdot \big(1-\exp(-A_\infty)\big) \right).
\end{equation*}
Using the fact that $U_{ \lfloor t m \rfloor }^{m}\to A_\infty \cdot t$ a.s.  for every $t\in [0,1]$, we derive  that the pointwise convergences ${\widetilde{H}}^m \rightarrow \widetilde{H}$ and ${H}^{A,m} \rightarrow {H}^A$ as $m \uparrow \infty$. Finally since  $\mathbb{N}_{x,0}^\lambda(1-\exp(-A_\infty))<\infty$, we deduce \eqref{equation:leyHtilde} by dominated convergence.
\end{proof}
\subsection{Trees embedded in the  subordinate tree} \label{subsection:prooflawEmbededTree}

\par
This section is devoted to the proof of Proposition  \ref{proposition:subordinatedTreeMarginal}. In short, the idea is to  decompose  inductively $\widetilde{\textbf{T}}$ and $\textbf{T}^A$  starting from their respective "left-most branches" -- viz. the path connecting the root $\emptyset $ and the first leaf  with the corresponding labels -- and to show that they have the same law.  Next, if we remove the left-most branch of $\widetilde{\textbf{T}}$ and $\textbf{T}^A$, we are left with two ordered collections of independent subtrees and we shall establish that they have respectively the same law as $\widetilde{\textbf{T}}$ and $\textbf{T}^A$. This will allow us to iterate this  left-most branch decomposition in such a way that the branches discovered at step $n$ in $\widetilde{\textbf{T}}$ and $\textbf{T}^A$  have the same law.  Proposition  \ref{proposition:subordinatedTreeMarginal} will follow since this procedure leads respectively  to discover $\widetilde{\textbf{T}}$ and $\textbf{T}^A$.  In order to state this formally let us  introduce  some  notation.
\\
\\
If $\textbf{T}:=(\TT,(h_v:~v\in \TT))$ is a discrete labelled tree and $n\geq 0$,  we let $\textbf{T}(n)$ be the set of all couples $(u,h_u)\in \textbf{T}$ such that $u$ has at most $n$ entries in $\{2,3,...\}$. In particular $\textbf{T}(0)$ is the branch connecting the root and the first leaf.
Next,  we introduce the collection
\begin{equation*}
    {\mathbb{S}}(\textbf{T}):=\big( (h_v , k_v(\TT)-1):~v\text{ is a vertex of } \textbf{T}(0)\big), 
\end{equation*}  
where the elements are listed in increasing order with respect to the height and we recall that $k_v(\TT)$ stands for the number of children of $v$. For simplicity, set  $R:=\#\textbf{T}(0)-1$,  write $v_1,...,v_{R+1}$ for the vertices  of $\textbf{T}(0)$ in lexicographic order and observe that  $v_1$ is the root while $v_{R+1}$ is the first leaf -- in particular $k_{v_{R+1}}(\text{T}) =0$. Heuristically,  ${\mathbb{S}}(\textbf{T})$ -- or more precisely the measure $\sum_i (k_{v_i}(\TT)-1)\delta_{h_{v_i}}$ -- is  a discrete version of the  exploration process  when visiting the first leaf of $\TT$ and for this reason  ${\mathbb{S}}(\textbf{T})$ will be called  the left-most spine of $\textbf{T}$.  Now, for every $1 \leq j \leq R$, set 
\begin{equation*}
   K_j(\textbf{T}):= \sum_{i=1}^{j}  (k_{v_i}(\TT)-1), 
\end{equation*}
with the convention $K_0(\textbf{T})=0$ and remark that $K(\textbf{T}):=K_{R}(\textbf{T})$ stands
for the number of subtrees   attached "to the right" of  $\textbf{T}(0)$ in  $\textbf{T}$.
To define these subtrees when $K(\textbf{T}) \geq 1$, we need to introduce the following: for every $1\leq i\leq K_{R}(\textbf{T}) = K(\textbf{T})$,  let $a(i)$ be the unique index such that $K_{a(i)-1}(\textbf{T}) < i \leq  K_{a(i)} (\textbf{T})$. Then, we introduce the marked tree
\begin{equation}\label{T_iiii}
    \textbf{T}_{i}:=\big\{(u,h'_u):~\big(v_{a(i)}(K_{a(i)} + 2 -i  )u,h_{v_{a(i)}} + h'_u \big)\in \textbf{T}\big\}.
\end{equation}
Remark that the labels in each subtree $\textbf{\TT}_i$ have been shifted by their relative height in $\mathbb{S}(\textbf{\TT})$ and that the collection $(\textbf{T}_i : 1 \le i \leq K(\textbf{T}))$  is listed in counterclockwise order. 
\\
\\
We now apply  this decomposition  to $\widetilde{\textbf{T}}$ and $\textbf{T}^A$. For simplicity, we write $\widetilde{K}:= K(\widetilde{\textbf{T}})$ (resp. $K:= K({\textbf{T}}^A)$) for the number of subtrees attached  to the right of  $\widetilde{\textbf{T}}(0)$  (resp. $\textbf{T}^A(0)$). When $\widetilde{K}\geq 1$ (resp. $K\geq 1$), we let $\widetilde{\textbf{T}}_i$ (resp. $\textbf{T}_i^A$) be the marked trees defined by \eqref{T_iiii} using $\widetilde{\textbf{T}}$ (resp. $\textbf{T}^A$).
 Proposition \ref{proposition:subordinatedTreeMarginal}  can now be reduced to the following result:  

\begin{prop} \label{lemma:seccion5_lemafinal} 
\emph{(i)} We have 
\begin{equation*}
   \Big( \mathbb{S}(\widetilde{\emph{\textbf{T}}})~:~\widetilde{N}^\lambda(\cdot|\widetilde{M}\geq 1) \Big)
\overset{(d)}{=} 
     \Big(\mathbb{S}(\emph{\textbf{T}}^A)
   ~:~\mathbb{N}_{x,0}^\lambda(\cdot|M\geq 1)\Big).
   \end{equation*}
\emph{(ii)} Under   $\widetilde{N}^\lambda(  \cdot \,  |\widetilde{K},\widetilde{M} \geq 1)$  and conditionally on $\mathbb{S}( \widetilde{\emph{\textbf{T}}})$,  the  subtrees   $\widetilde{\emph{\textbf{T}}}_{1}, \dots \widetilde{\emph{\textbf{T}}}_{\widetilde{K}}$ are distributed as $\widetilde{K}$ independent copies distributed as $\widetilde{\emph{\textbf{T}}}$ under $\widetilde{N}^\lambda(  \cdot \,  |\widetilde{M} \geq 1)$.  Similarly, under $\mathbb{N}_{x,0}^\lambda(\dd \overline{W}, \dd\mathcal{P}|K,M \geq 1)$ and conditionally on $\mathbb{S}(\emph{\textbf{T}}^A)$, the subtrees  ${\emph{\textbf{T}}}^{A}_1, \dots,  \emph{\textbf{T}}^{A}_K$ are distributed as  $K$ independent copies distributed as  $\emph{\textbf{T}} ^A$ under $\mathbb{N}_{x,0}^\lambda(\dd \overline{W}, \dd\mathcal{P}|M \geq 1)$.
\end{prop}
Let us explain why Proposition \ref{proposition:subordinatedTreeMarginal} is a consequence of the previous result. 
\begin{proof}[Proof of Proposition \ref{proposition:subordinatedTreeMarginal}]
We are going to show by induction that for every $n\geq 0$:
\begin{equation}\label{induction}
\widetilde{\textbf{T}}(n)\text{ under } \widetilde{N}^\lambda(  \cdot \,  |\widetilde{M} \geq 1) ~~~~\text{ is distributed as }~~~~ \textbf{T}^A(n) \text{ under }\mathbb{N}_{x,0}^\lambda(\cdot|M\geq 1).
\end{equation}
First notice that Proposition \ref{lemma:seccion5_lemafinal} - (i) gives the previous identity in the case  $n=0$. Assume now that  \eqref{induction} holds for $n\geq 0$ and let us prove  the identity for $n+1$. First, remark that it is enough to argue with $\widetilde{\textbf{T}}(n+1)$ under $\widetilde{N}^\lambda(  \cdot \,  |\widetilde{K},\widetilde{M} \geq 1)$ and $\textbf{T}^A(n+1)$  under $\mathbb{N}_{x,0}^\lambda(\cdot|K,M\geq 1)$ -- since by Proposition \ref{lemma:seccion5_lemafinal}, the variable $\widetilde{K}$ under $\widetilde{N}^\lambda(  \cdot \,  |\widetilde{M} \geq 1)$ is distributed as $K$ under $\mathbb{N}_{x,0}^\lambda(\cdot|M\geq 1)$. Next, we see that 
 $\widetilde{\textbf{T}}(n+1)$ can be  obtained by gluing  the trees $\widetilde{\textbf{T}}_i(n)$ to $\widetilde{\textbf{T}}(0)$ at their respective positions after translating the labels by the associated heights. Moreover, these positions and heights  are precisely the entries of  $\mathbb{S}(\widetilde{\textbf{T}})$. Since the same discussion holds when replacing $\widetilde{\textbf{T}}$  by $\textbf{T}^A$, the case $n+1$ follows by 
 Proposition \ref{lemma:seccion5_lemafinal} and the case $n$. Finally, since the trees $\widetilde{\textbf{T}}$ and $\textbf{T}^A$ are finite, \eqref{induction} implies the desired result.
\end{proof}

\par Our goal now is to prove
 Proposition \ref{lemma:seccion5_lemafinal}. In this direction,  we will first encode the spines ${\mathbb{S}}(\widetilde{\textbf{T}}),  {\mathbb{S}}(\textbf{T}^A)$ as well as the corresponding subtrees  $\widetilde{\textbf{T}}_i$, $\textbf{T}^A_i$ in terms of $\widetilde{\rho}$, $(\rho , \overline{W})$ and $\mathcal{P}$. This will allow us to identify their law by making use of the machinery developed in previous sections.  
 While $\mathbb{S}(\widetilde{\textbf{T}})$ can be constructed directly in terms of $(\widetilde{\rho}_{\mathfrak{t}_1 +  t} : t \geq 0 )$ and the Poisson marks,  the construction of $\mathbb{S}(\textbf{T}^A)$ is more technical.  Roughly speaking,  the strategy  consists in defining in terms of $(\rho , \overline{W})$  the  exploration process for the subordinate tree at time $\mathfrak{t}_1$, say $\rho^*_{\mathfrak{t}_1}$,  and then  show  --  see Lemma \ref{lemma:explorationAtexpotime} below -- that  $\widetilde{\rho}_{\mathfrak{t}_1}$ and $\rho^*_{\mathfrak{t}_1}$ have the same distribution.
Needless to say that  this statement is informal, since we  have not yet shown that the subordinate tree is a Lévy tree.  We will then  deduce (i)  by  considering   $\mathbb{S}(\widetilde{\textbf{T}})$, $\mathbb{S}(\textbf{T}^A)$ and  conditioning respectively  on $\widetilde{\rho}_{\mathfrak{t}_1}$ and $\rho^*_{\mathfrak{t}_1}$, Point (ii) will then follow easily by construction. For simplicity, from now on  we write $\mathfrak{t} := \mathfrak{t}_1$.  
\\
\\
We first start working under $\widetilde{N}^\lambda(\cdot \, | M \geq 1)$ and we  introduce the following  notation: let $\big((\widetilde{\alpha}_i, \widetilde{\beta}_i):~i\in \mathbb{N}\big)$ be the connected components of the open set
$$
\big\{s\geq \mathfrak{t}:~\widetilde{H}_s>\inf\limits_{[\mathfrak{t},s]} \widetilde{H} \big\}.
$$
As usual, we write  $\widetilde{\rho}^{\,i}$ for the associated subtrajectory of the exploration process in the interval $[\widetilde{\alpha}_i , \widetilde{\beta}_i ]$. 
We also consider $\widetilde{H}^{i}:=(\widetilde{H}_{(\widetilde{\alpha}_i+s)\wedge \widetilde{\beta}_i}-\widetilde{H}_{\widetilde{\alpha}_i}:~s\geq 0)$,  $\widetilde{\mathcal{P}}^{i}:=(\widetilde{\mathcal{P}}_{(\widetilde{\alpha}_i+t)\wedge \widetilde{\beta}_i}-\widetilde{\mathcal{P}}_{\widetilde{\alpha}_i}:~t\geq 0)$ and note that in particular we have $H(\widetilde{\rho}^{\,i}) = \widetilde{H}^i$. Write  $\widetilde{h}_i:=\widetilde{H}(\widetilde{\alpha}_i)$, and consider the marked measure:
\begin{equation*}
    \widetilde{\mathcal{M}}:=\sum\limits_{i\in \mathbb{N}}\delta_{(\widetilde{h}_i, \widetilde{\rho}^{\,i},\widetilde{\mathcal{P}}^i)}.
\end{equation*}
By the Markov property and  \eqref{PoissonRandMeasure}, conditionally on $\mathcal{F}_{\mathfrak{t}}$, the measure $\widetilde{\mathcal{M}}$ is a Poisson point measure with intensity $\widetilde{\rho}_{\mathfrak{t}}(\dd h)\widetilde{N}^\lambda(\dd \rho, \,  \dd \mathcal{P}).$ 
Now we can identify $\mathbb{S}(\widetilde{\textbf{T}})$  in terms of functionals of $\widetilde{\mathcal{M}}$ and $\widetilde{H}_{\mathfrak{t}}$. First, set  $(\widetilde{h}^\circ_p:~ 1\leq p \leq \widetilde{R})$  the collection of  the different heights -- in increasing order -- among $(\widetilde{h}_i:~ i \in \mathbb{N})$ at which  $\widetilde{\mathcal{P}}^{i}_{\sigma(\widetilde{\rho}^{\,i})} \geq 1$.  In particular,  $\widetilde{R}$ gives the number of different heights $\widetilde{h}_j$ at which we can find at least one marked excursion above the running infimum of $(\widetilde{H}_{\mathfrak{t}+ t}: t \geq 0)$.  Next, we write $\widetilde{M}^\circ_p$ for  the number of atoms at level  $\widetilde{h}^\circ_p$ in $\widetilde{\mathcal{M}}$ with at least one Poissonian mark. Now, remark that by construction we have:
\begin{equation}\label{definition:spinePhiTilde}
     \mathbb{S}(\widetilde{\textbf{T}}) =  \big( (\widetilde{h}^\circ_1, \widetilde{M}^\circ_1), \dots , (\widetilde{h}^\circ_{\widetilde{R}},\widetilde{M}^\circ_{\widetilde{R}}), ( \widetilde{H}_{\mathfrak{t}}, -1)\big),  
\end{equation}
  and in particular  $\widetilde{K}  = \sum_{i=1}^{\widetilde{R}} \widetilde{M}^\circ_i$.  Finally, for later use  denote the corresponding marked excursions  arranged  in counterclockwise order by $ \widetilde{\mathscr{E}} := ( (\widetilde{\rho}^q_\circ,\widetilde{H}^q_\circ, \widetilde{\mathcal{P}}^q_\circ): 1 \leq q \leq  \widetilde{K})$.  Notice that the subtrees $(\widetilde{\textbf{T}}_i : 1 \le i \leq \widetilde{K})$ are precisely the respective embedded   marked trees associated with  $((\widetilde{H}^q_\circ, \widetilde{\mathcal{P}}^q_\circ): 1 \leq q \leq  \widetilde{K})$.
\par 
The main step remaining in our analysis under $\widetilde{N}^\lambda(\cdot \, | \widetilde{M} \geq 1)$ consists in  characterizing the law of  $(\widetilde{H}_{\mathfrak{t}},\widetilde{\rho}_{\mathfrak{t}})$, and this is  the content of the following lemma. Since  $\widetilde{\mathcal{M}}$ conditionally on $\mathcal{F}_{\mathfrak{t}}$ is a Poisson point measure with intensity $\widetilde{\rho}_{\mathfrak{t}}(\dd h)\widetilde{N}^\lambda(\dd \rho, \,  \dd \mathcal{P})$, this will suffice to identify the distribution of $\mathbb{S}(\widetilde{\textbf{T}})$.  In this direction, Corollary \ref{corollary:phitildeParametros} ensures that the measure   $\widetilde{\rho}_{\mathfrak{t}}$ is purely atomic  and consequently by \eqref{rho_atoms} it is of the form: 
\begin{equation*}
    \widetilde{\rho}_{\mathfrak{t}}:=\sum_{i \in \mathbb{N}} \widetilde{\Delta}_i \cdot  \delta_{\widetilde{h}^i} \,.
\end{equation*}
 We stress that we have $\{\widetilde{h}^i:i\in \mathbb{N}\}=\{\widetilde{h}_i:i\in \mathbb{N}\}$ -- even though the latter set has repeated elements.
\begin{lem}\label{lemma:Seccion5_PoissoncondHtilde}
Under $\widetilde{N}^\lambda(\cdot |~\widetilde{M}\geq 1)$, the random variable $\widetilde{H}_{\mathfrak{t}}$ is exponentially distributed with intensity $\lambda/\widetilde{\psi}^{-1}(\lambda)$. Moreover, conditionally on $\widetilde{H}_{\mathfrak{t}}$, the measure $\sum \delta_{(\widetilde{h}^i,\widetilde{\Delta}_i)}$   is a Poisson point measure with intensity $\mathbbm{1}_{[0,\widetilde{H}_{\mathfrak{t}}]}(\dd h) \widetilde{\nu}(\dd z)$, where $\widetilde{\nu}(\dd z)$ is the measure supported on $\mathbb{R}_+$ characterised by:
\begin{equation} \label{equation:leynutilde}
    \int \widetilde{\nu}(\dd z)\big(1-\exp(-p z)\big)=\frac{\widetilde{\psi}(p)-\lambda}{p-\psi^{-1}(\lambda)} - \frac{\lambda}{\psi^{-1}(\lambda)}, \qquad p\geq 0.
\end{equation}
\end{lem}
\begin{proof}
Recall that by Proposition \ref{lemma:InversaExpo}, we have  $\widetilde{\psi}^{-1}(\lambda) = \widetilde{N}(1-\exp(-\lambda \sigma))= \widetilde{N}^\lambda(\widetilde{M}\geq 1)$.  Consider two measurable  functions $g: \mathbb{R}_+ \mapsto \mathbb{R}_+$, $F: \mathcal{M}_f(\mathbb{R}_+) \mapsto \mathbb{R}_+$ and remark that
\begin{align*}
    \widetilde{N}^{\lambda}\big( g(\widetilde{H}_{\mathfrak{t}})F(\widetilde{\rho}_{\mathfrak{t}})\mathbbm{1}_{\{\widetilde{M}\geq 1 \}}\big)= \lambda\cdot  \widetilde{N}\big(\int_{0}^\sigma \dd s \exp(-\lambda s) g(\widetilde{H}_{s})F(\widetilde{\rho}_{s})\big).
\end{align*}
By duality \eqref{dualidad:etaRhoW} and the Markov property, the previous expression can be written in the form: 
\begin{align*}
\lambda \cdot  \widetilde{N} \big(\int_{0}^{\sigma} \dd s~ g(\widetilde{H}_{s})F(\widetilde{\eta}_{s}) \exp(-\lambda (\sigma-s)) \big)  &= \lambda\cdot\widetilde{N} \big(\int_{0}^{\sigma} \dd s~ g(\widetilde{H}_{s})F(\widetilde{\eta}_{s})\widetilde{E}_{\widetilde{\rho}_s}[ \exp(-\lambda \sigma)]\big)\\
  &=\lambda\cdot\widetilde{N}\big(\int_{0}^{\sigma} \dd s~ g(\widetilde{H}_{s})F(\widetilde{\eta}_{s})\exp(-\widetilde{\psi}^{-1}(\lambda)\langle \widetilde{\rho}_{s}, 1\rangle )\big),
\end{align*}
where in the last line we use the identity $\widetilde{\psi}^{-1}(\lambda)=\widetilde{N}(1-\exp(-\lambda \sigma)) $.  Consider under $P^0$ the pair of subordinators $(\widetilde{U}^{(1)}, \widetilde{U}^{(2)})$ with Laplace exponent \eqref{identity:exponenteSubord}, defined  replacing $\psi$ by $\widetilde{\psi}$, and denote its Lévy measure by $\widetilde{\gamma}(\dd u_1 , \dd u_2 )$. We stress that since $\widetilde{\psi}$ does not have Brownian part, the subordinators $(\widetilde{U}^{(1)}, \widetilde{U}^{(2)})$  does not have drift.  The  many-to-one formula \eqref{tirage_au_hasard_N} applied to $\widetilde{\psi}$ gives:
\begin{align} \label{equation:section5_leytiempoTpareja}
    \widetilde{N}^{\lambda}\big(g(\widetilde{H}_{\mathfrak{t}})F(\widetilde{\rho}_{\mathfrak{t}})\big|~\widetilde{M}\geq 1\big)=\frac{\lambda}{\widetilde{\psi}^{-1}(\lambda)}\int_{0}^{\infty}\dd a \,   \exp(-\widetilde{\alpha} a) g(a) E^{0}\big[F(\mathbbm{1}_{[0,a]} \dd \widetilde{U}^{(1)}) \exp(-\widetilde{\psi}^{-1}(\lambda) \widetilde{U}^{(2)}_a  )\big].
\end{align}
We shall now deduce from the later identity that the pair $\big(\widetilde{H}_{\mathfrak{t}} , \sum \delta_{(\widetilde{h}^i,\widetilde{\Delta}_i)}\big)$ has the desired distribution. In this direction, let $f:\mathbb{R}_+^{2}\to \mathbb{R}_+$ be a measurable function satisfying $f(h,0)=0$, for every $h\geq 0$. By  \eqref{equation:section5_leytiempoTpareja},  we derive   that 
\begin{align} \label{equation:seccion5_rhotilde}
   \widetilde{N}^{\lambda} \big(g(\widetilde{H}_{\mathfrak{t}})&\exp\big(-\sum_{i \in \mathbb{N}} f(\widetilde{h}_i,\widetilde{\Delta}_i )\big)|\widetilde{M}\geq 1\big) \nonumber \\
   &=\frac{\lambda}{\widetilde{\psi}^{-1}(\lambda)}\int_{0}^{\infty}\dd a \,  g(a)\exp(-\widetilde{\alpha} a) E^{0}\Big[ \exp\Big(-\sum_{h \leq a} \big(f(h,\Delta\widetilde{U}^{(1)}_{h} )+\widetilde{\psi}^{-1}(\lambda)\Delta\widetilde{U}^{(2)}_{h} \big)\Big)\Big].
\end{align}
Moreover, by the exponential formula it follows that the expectation under $E^0$ in the previous display equals
\begin{align*}
\exp \Big(-\int_0^a \dd h \int \widetilde{\gamma}(\dd u_1,\dd u_2) \big(1-\exp(-f(h,u_1)-\widetilde{\psi}^{-1}(\lambda) u_2)\big)\Big), 
\end{align*}
and notice that we can write:  
\begin{align*}
\int \widetilde{\gamma}(\dd u_1,\dd u_2)\big(1-\exp(-f(h,u_1)-\widetilde{\psi}^{-1}(\lambda) u_2)\big)= &\int \widetilde{\gamma}(\dd u_1,\dd u_2)\exp(-\widetilde{\psi}^{-1}(\lambda) u_2)\big(1-\exp(-f(h,u_1))\big)\\
&+ \int \widetilde{\gamma}(\dd u_1,\dd u_2)\big(1-\exp(-\widetilde{\psi}^{-1}(\lambda) u_2)\big).
\end{align*}
To simplify this expression,  introduce   the measure  $\widetilde{\gamma}^{\prime}(\dd u_1) :=\int_{u_2 \in \mathbb{R}}\widetilde{\gamma}(\dd u_1,\dd u_2)\exp(-\widetilde{\psi}^{-1}(\lambda) u_2)$ 
and observe that  \eqref{identity:exponenteSubord} entails  $\int \widetilde{\gamma}(\dd u_1,\dd u_2)(1-\exp(-\widetilde{\psi}^{-1}(\lambda) u_2))={\lambda}/{\widetilde{\psi}^{-1}(\lambda)}-\widetilde{\alpha}$.  We deduce that \eqref{equation:seccion5_rhotilde} can be written in the following form:
\begin{align*}
    \frac{\lambda}{\widetilde{\psi}^{-1}(\lambda)}\int_{0}^{\infty}\dd a \, g(a) \exp(-\frac{\lambda}{\widetilde{\psi}^{-1}(\lambda)} a) \exp \Big( - \int_0^a\dd h\int \widetilde{\gamma}'(\dd u_1) \Big(1- \exp \big( -f(h,u_1) \big) \Big) \Big), 
\end{align*}
and to conclude it suffices to remark that $\widetilde{\gamma}^\prime=\widetilde{\nu}$, since by \eqref{identity:exponenteSubord} we have
\begin{align*}
    \int \widetilde{\gamma}^\prime(\dd u_1)\big(1-\exp(-p u_1)\big)&=\int \widetilde{\gamma}(\dd u_1,\dd u_2) \big(\exp(-\widetilde{\psi}^{-1}(\lambda) u_2)-\exp(-pu_1-\widetilde{\psi}^{-1}(\lambda) u_2)\big)\\
    &=\frac{\widetilde{\psi}(p)-\lambda}{p-\widetilde{\psi}^{-1}(\lambda)} - \frac{\lambda}{\widetilde{\psi}^{-1}(\lambda)},
\end{align*}
for every $p\geq 0$. 
\end{proof}

We now turn our attention to the other side of the picture, and we now work  under  $\mathbb{N}^{\lambda}_{x,0}(\cdot |M\geq 1)$. The objective is to obtain analogue results for the spine $\mathbb{S}(\textbf{T}^A)$. In this direction,  recall the notation $G_\lambda := \widetilde{\psi}^{-1}(\lambda)$ and we start with the following technical lemma  characterizing the law of $(\rho,\overline{W})$ at  time $A^{-1}_\mathfrak{t}$.  

\begin{lem} \label{lemma:CalculosAdditiva} For any non-negative measurable function $f$ in $M_f(\mathbb{R}_+) \times \mathcal{W}_{\overline{E}}$,  we have:
\begin{equation*} 
    \mathbb{N}^\lambda_{x,0} \left( f ( \rho_{A^{-1}_{\mathfrak{t}}} , \overline{W}_{A^{-1}_{\mathfrak{t}}})  \mathbbm{1}_{ \{ M \geq 1 \}} \right) = \lambda \int_0^\infty \dd a \, E^{0} \otimes \Pi_x \Big( \exp(-\alpha \tau_a)  f ( \Jj_{\tau_a} , (\xi_{t} , \mathcal{L}_t)_{t \leq \tau_a} ) \exp \big( {-\int_0^{\tau_a} \Jc_{\tau_a}(\dd h) \,  u_{G_\lambda}(\xi_h)  } \big)\Big).  
\end{equation*}
\end{lem}
\begin{proof} Since $\{ M \geq 1 \} = \{ \mathfrak{t}  \leq A_\infty\}$, we have:
\begin{align*}
    \mathbb{N}_{x,0}^\lambda \left( f ( \rho_{A^{-1}_{\mathfrak{t}}} , \overline{W}_{A^{-1}_{\mathfrak{t}}}) \mathbbm{1}_{\{ M\geq 1 \}}  \right) 
     &=  \lambda \cdot \mathbb{N}_{x,0} \left( \int_0^{A_{\infty}} \dd s \,  f(\rho_{A^{-1}_s}  , \overline{W}_{A^{-1}_s} ) \exp(-\lambda s) \right) \\
     &=\lambda \cdot \mathbb{N}_{x,0} \left( \int_0^{\sigma} \dd A_s \, f(\rho_{s}  , \overline{W}_{s} ) \exp(- \lambda  A_s ) \right) \\
     &=-\lambda \cdot\mathbb{N}_{x,0} \left( \int_0^{\sigma} \dd A_{\sigma - s} \, f(\rho_{\sigma - s}  , \overline{W}_{\sigma - s} ) \exp(- \lambda  A_{\sigma - s} ) \right).
\end{align*}
Moreover, by time reversal \eqref{dualidad:etaRhoW}, we know that:
\begin{equation*}
    (\rho_{(\sigma - s)-} , \overline{W}_{\sigma - s } , A_{\sigma - s} : 0 \leq s \leq \sigma  ) \overset{(d)}{=} (\eta_{s} , \overline{W}_{s } , A_\sigma -
A_{s} : 0 \leq s \leq \sigma  ),
\end{equation*}
and we remark that $\{ s \in [0,\sigma] : \rho_{s} \neq \rho_{s-}  \} \subset \{ s \in [0,\sigma] : \rho_{s}(\{ H_s \}) > 0  \}$ which  has null $\dd A$ measure $\mathbb{N}_{x,0}$-- a.e by the many-to-one formula of Lemma \ref{lemLAformula:aditiva}. This implies:
$$ - \mathbb{N}_{x,0} \left( \int_0^{\sigma} \dd A_{\sigma - s} \, f(\rho_{\sigma - s}  , \overline{W}_{\sigma - s} ) \exp(- \lambda  A_{\sigma - s} ) \right)= \mathbb{N}_{x,0} \left( \int_0^{\sigma} \dd A_{s} \, f(\eta_{ s}  , \overline{W}_{ s} ) \exp(- \lambda \int_{s}^{\sigma} \dd A_{s} ) \right).$$
Next,  by making use of   the strong Markov property, we derive  that
\begin{align*}
\mathbb{N}_{x,0}^\lambda \left( f ( \rho_{A^{-1}_{\mathfrak{t}}} , \overline{W}_{A^{-1}_{\mathfrak{t}}}) \mathbbm{1}_{\{ M\geq 1 \}}  \right) 
&= \lambda\cdot
 \mathbb{N}_{x,0} \left( \int_0^{\sigma} \dd A_{s} \, f(\eta_{s}  , \overline{W}_{s} ) \exp\big(-  \lambda  \int_s^\sigma \dd A_s\big)  \right) \\
&= \lambda\cdot \mathbb{N}_{x,0} \left( \int_0^{\sigma} \dd A_{s} \, f(\eta_{s}  , \overline{W}_{s} ) \mathbb{E}^{\dag}_{\rho_s , \overline{W}_s} \Big[ \exp\big(-  \lambda  \int_0^\sigma \dd A_s\big) \Big]   \right) \\
&=  \lambda \cdot  \mathbb{N}_{x,0} 
\left( \int_0^{\sigma} \dd A_{s} \, f(\eta_{s}  , \overline{W}_{s} ) \exp \Big( {-\int\rho_s(\dd h) \, u_{G_\lambda} (W_s(h)) } \Big) \right), 
\end{align*}
where in the last line we used Proposition \ref{lemma:InversaExpo}. The statement of the lemma now follows applying \eqref{LAformula:aditiva} and  recalling that $(J_\infty, \Jc_\infty) \overset{(d)}{=} ( \Jc_\infty , J_\infty)$, under $P^0$.
\end{proof}
For simplicity, in the rest of the section we write:
\begin{equation*}
    (\rho^A_\mathfrak{t} , \overline{W}_{\mathfrak{t}}^A) 
    :=
    (\rho_{A^{-1}_{\mathfrak{t}}} , \overline{W}_{A^{-1}_{\mathfrak{t}}}),
\end{equation*}
and  $ \overline{W}_{\mathfrak{t}}^A:=( W_{\mathfrak{t}}^A, \Lambda_{\mathfrak{t}}^A)$ -- remark that in particular we have $H^A_\mathfrak{t} = \widehat{\Lambda}^A_{\mathfrak{t}}$.
Let us now decompose  $\overline{W}_{\mathfrak{t}}^A$  in terms of its excursion intervals away from $x$. To be more precise, we need to introduce some  notation. For every $r>0$ and $\overline{\w}:=(\w,\ell)\in \mathcal{W}_{\overline{E}}$, we set:
 $$\tau_r^{+}(\overline{\w}):= \inf \big\{  h \geq 0 : \ell(h) >  r \big\}. $$
Remark that since $\ell$ is continuous, $r \mapsto \tau_r^{+}(\overline{\w})$ is right-continuous. Moreover, $\tau(\overline{\w})$ and $\tau^+(\overline{\w})$ are related by the relation $\tau_r(\overline{\w})=\tau_{r-}^{+}(\overline{\w})$. Similarly, under $\Pi_{y,0}$ for $y\in E$, we will write $\tau_{r}^+(\xi):=\inf\{t\geq 0: ~ \mathcal{L}_t> r\}$. The advantage  of working with $\tau^{+}(\xi)$ instead of  $\tau(\xi)$ is that, under $\Pi_{x,0}$, the process $\tau^{+}(\xi)$ is a subordinator. Moreover, by excursion theory, it is well known that   its Lévy-Itô decomposition  is given by 
\begin{equation*}
    \tau_r^+(\xi) = \sum_{s \leq r} \Delta \tau_s^+(\xi),  \quad \quad r \geq 0,
\end{equation*}
since  \ref{Asssumption_3} ensures that the process $\tau^+(\xi)$ does not have drift part -- equivalently  $\tau^{+}(\xi)$ is purely discontinuous. For simplicity, when there is no risk of confusion the dependency on $\xi$ is dropped.
\par Getting back to our discussion, under $\mathbb{N}^\lambda_{x,0}(\cdot \, | M \geq 1)$,  let $(r_{j}:~j\in \mathcal{J})$ be  an enumeration of the jumping times of the right-continuous process  $(\tau^{+}_r(\overline{W}^A_{\mathfrak{t}}): 0 \leq r < H^A_\mathfrak{t} )$ -- for technical reasons the indexing is assumed to be measurable with respect to $\overline{W}^A_{\mathfrak{t}}$. For each $j \in \mathcal{J}$, set 
\begin{equation*}
    \overline{W}^{A,j}_{\mathfrak{t}}:= \Big( 
    \big(  W^A_{\mathfrak{t}}\big(h+\tau_{r_j}(\overline{W}^A_{\mathfrak{t}})\big), \Lambda^A_{\mathfrak{t}}\big(h+\tau_{r_j}(\overline{W}^A_{\mathfrak{t}})\big)-\Lambda^A_{\mathfrak{t}}(\tau_{r_j}\big(\overline{W}^A_{\mathfrak{t}})\big) \big) :~h\in [0, \tau^{+}_{r_j}(\overline{W}^{A}_{\mathfrak{t}}) - \tau_{r_j}(\overline{W}^{A}_{\mathfrak{t}})  ] \Big),  
\end{equation*}
and
$$\langle \rho^{A,j}_{\mathfrak{t}} , f \rangle   :=\int \rho_{\mathfrak{t}}^A(\dd h)f(h -\tau_{r_j}(\overline{W}^A_{\mathfrak{t}}) )\mathbbm{1}_{\{\tau_{r_j}(\overline{W}^A_{\mathfrak{t}}) <h< \tau^{+}_{r_j}(\overline{W}^A_{\mathfrak{t}})  \}}. 
$$
The first coordinates of the family  $(\overline{W}^{A,j}_{\mathfrak{t}}: j \in \mathcal{J})$ correspond to the excursion of $W^A_\mathfrak{t}$ away from $x$ while the second coordinate is identically zero.  We also stress that since $(x,0)\in \overline{\Theta}_x$, by Lemma  \ref{te_quedas_en_Theta} the support of $\rho^A_{\mathfrak{t}}$ is included in $\bigcup_{j\in \mathcal{J}}(\tau_{r_j}(\overline{W}^A_{\mathfrak{t}}),\tau_{r_j}^+(\overline{W}^A_{\mathfrak{t}}))$. Our goal now is to identify the law of $\sum_{j\in \mathcal{J}}\delta_{(r_j,   \rho_{\mathfrak{t}}^{A,j}, {W}^{A,j}_{\mathfrak{t}}  )}$. As we shall see, the restriction to the first two coordinates  of this  measure is, roughly speaking, a biased version of the excursion  point measure of $\xi$ under $\Pi_{x,0}$. More precisely, let  $( E^{0}\otimes\mathcal{N})_*(\dd J,\dd \xi)$ be the measure in $\mathcal{M}_f(\mathbb{R}_+)\otimes \mathbb{D}(\mathbb{R}_+, E)$ defined by 
\begin{equation*}
    ( E^{0}\otimes\mathcal{N})_*\big[ F(J, \xi) \big] 
    :=   E^{0}\otimes\mathcal{N}\Big[ \exp \big (- \int \Jc_{\sigma}(\dd h) u_{G_\lambda}(\xi_h) -\alpha \sigma \big)  F( J_\sigma,\xi) \Big].
\end{equation*}

\begin{lem}\label{lemma:Seccion5_PoissoncondHA}
Under $\mathbb{N}_{x,0}^{\lambda}(\cdot  |M\geq 1)$, the random variable $H_{\mathfrak{t}}^A$ is exponentially distributed with parameter $\lambda/\psi^{-1}(\lambda)$. Moreover,  conditionally on $H_{\mathfrak{t}}^A$, the measure:
$$\sum \limits_{j\in \mathcal{J}}\delta_{(r_j , \rho_{\mathfrak{t}}^{A,j}, {W}^{A,j}_{\mathfrak{t}})}, $$
is a Poisson point measure with intensity $\mathbbm{1}_{[0,H_{\mathfrak{t}}^A]}(r)\dd r ( E^{0}\otimes\mathcal{N})_*( \dd J,\dd \xi )$.
\end{lem}
\begin{proof}
First, we fix two measurable  functions $g: \mathbb{R}_+ \mapsto \mathbb{R}_+$ and $f:\mathbb{R}_+ \times \mathcal{M}_f(\mathbb{R}_+)\times \mathbb{D}(\mathbb{R}_+ , E) \mapsto \mathbb{R}_+$. The statement of the lemma will follow by establishing that:
\begin{align}\label{equation:section5_PoissoncondH}
& \mathbb{N}_{x,0}^{\lambda}\big(g(H^{A}_{\mathfrak{t}})\exp(-\sum_{j \in \mathcal{J}} f(r_j,\rho_\mathfrak{t}^{A,j},{W}_{\mathfrak{t}}^{A,j}  ))\:|\:M\geq 1\big) \nonumber \\
&\hspace{10mm}=
\frac{\lambda}{\widetilde{\psi}^{-1}(\lambda)}\int_{0}^\infty\dd r \, \exp \big(
   - r \cdot \frac{\lambda}{\widetilde{\psi}^{-1}(\lambda)}
   \big) g(r)  \exp \Big( 
   -\int_0^r \dd s \,  ( E^{0}\otimes\mathcal{N})_*\big[ 1- \exp \big(-  f(s, J,\xi  )  \big) \big] \Big).
\end{align}
 To simplify notation, for every $\mu \in \mathcal{M}(\mathbb{R}_+)$ and $a,b\geq 0$, we write $\phi(\mu,a,b)$ for the measure $\nu$ defined by:
$$\int \nu(\dd h)F(h)=\int_{(a,b)}\mu(\dd h)F(h-a).$$ Next, under $\Pi_{x,0}$, denote  the excursion point measure  of $\xi$  by $\sum_{j}\delta_{(r_j, \xi^{j})}$.
Now an application of Lemma \ref{lemma:CalculosAdditiva} gives
\begin{align*}
& \mathbb{N}_{x,0}^{\lambda}\big(g(H^{A}_{\mathfrak{t}})\exp(-\sum_{j \in \mathcal{J}} f(r_j,\rho_\mathfrak{t}^{A,j},{W}_{\mathfrak{t}}^{A,j}  ))\:|\:M\geq 1\big)\\
&=\frac{\lambda}{\widetilde{\psi}^{-1}(\lambda)}\int_{0}^\infty \dd r \, g(r) E^{0}\otimes \Pi_{x,0} \Big( \exp(-\alpha \tau_r) \exp \big(- \sum\limits_{r_j\leq r} f(r_j, \phi(J_\infty ,\tau_{r_j}, \tau_{r_j}^+ ), \xi^j ) \big) \exp\big( -\int \Jc_{\tau_r}(\dd h) u_{G_\lambda}(\xi_h) \Big) \Big) \\
&= \frac{\lambda}{\widetilde{\psi}^{-1}(\lambda)}\int_{0}^\infty\dd r \, g(r) E^{0}\otimes \Pi_{x,0}\Big( \exp\Big(- \sum\limits_{r_j\leq r}\Big\{  f\big(r_j, \phi(J_\infty ,\tau_{r_j}, \tau_{r_j}^+ ),\xi^j \big) +\int_{\tau_{r_j}}^{\tau_{r_j}^+} \Jc_{\infty}(\dd h) u_{G_\lambda}(\xi_h)+\alpha \sigma(\xi^j)\Big\}\Big) \Big), 
\end{align*}
where in the last equality we used  the fact that $\tau^+$ is purely discontinuous and that thanks to \ref{Asssumption_3}, under $P^0\otimes\Pi_{x,0}$, we can write $\Jc_{\infty}(\dd h) =\sum_{r_j} \Jc_{\infty}(\dd h) \mathbbm{1}_{[\tau_{r_j}, \tau_{r_j}^+]}(h) $. We are going to conclude using standard techniques of excursion theory. First remark that if we  introduce  an i.i.d. collection of measures $(J_\infty^j, \Jc _\infty^j )_{j \in \mathbb{N}}$    distributed as $(J_\infty, \Jc_\infty)$ under $P^0$,  the previous display can  be written in the form:  
\begin{align}\label{(*):Lemma14}
    \frac{\lambda}{\widetilde{\psi}^{-1}(\lambda)}\int_{0}^\infty\dd r \, g(r) E^{0}\otimes \Pi_{x,0}\Big( \exp \Big(- \sum\limits_{r_j\leq r} \Big\{ f(r_j, J^j_{\sigma(\xi^j)},\xi^j  ) +\int \widecheck{J}^j_{\sigma(\xi^j)}(\dd h) u_{G_\lambda}(\xi_h^j) +\alpha \sigma(\xi^j) \Big\} \Big) \Big). 
\end{align}
Since by excursion theory  $\sum_{r_j \leq r}\delta_{(r_j, J^j_\infty , \Jc^j_\infty, \xi^{j} )}$ is a Poisson point measure with intensity $\mathbbm{1}_{[0,r]}(\dd s)  E^{0}\otimes\mathcal{N}(\dd J_\infty,\dd \Jc_\infty, \dd \xi)$,  we deduce that the expectation under $E^0\otimes \Pi_{x,0}$ in \eqref{(*):Lemma14} writes:  
\begin{align*}
&\exp \Big( -\int_0^r \dd s \,   E^{0}\otimes\mathcal{N}\Big[ 1- \exp \big(-  f(s , J_{\sigma}, \xi ) - \int \Jc_{\sigma}(\dd h) u_{G_\lambda}(\xi_h) - \alpha \sigma \big) \Big]  \Big).
\end{align*}
Next, we remark that the previous display equals:
\begin{align*}
\exp \Big( 
   -\int_0^r \dd s \,  ( E^{0}\otimes\mathcal{N})_*\Big[ 1- \exp \big(-  f(s , J, \xi)  \big) \Big] \Big)
   \exp \Big(
   -r\cdot  E^{0}\otimes\mathcal{N}\Big[ 1-\exp\big( -\int \Jc_\sigma (\dd h) u_{G_\lambda}(\xi_h) - \alpha \sigma \big)\Big]
   \Big).
\end{align*}
Moreover, by \eqref{identity:exponenteSubord} the measure   $\Jc_\infty$ is the Lebesgue-Stieltjes measure of a subordinator with Laplace exponent  $p\mapsto\psi(p)/p -  \alpha$, which yields 
\begin{equation*}
     E^{0}\otimes\mathcal{N}\big[ 1-\exp\big( -\int \Jc_{\sigma}(\dd h) u_{G_\lambda}(\xi_h) - \alpha \sigma \big)\big]
    =
    \mathcal{N}\Big(1- \exp\Big( -\int_0^{\sigma}\dd h \frac{\psi( u_{G_\lambda}(\xi_h))}{u_{G_\lambda}(\xi_h)} \Big) \Big)
    =  \frac{\lambda}{\widetilde{\psi}^{-1}(\lambda)}, 
\end{equation*}
where in the first equality we applied \eqref{identity:exponenteSubord} and in last one we  used \eqref{identity:expoenente_1}. Putting everything together we obtain the desired identity \eqref{equation:section5_PoissoncondH}. 
\end{proof}

To identify the law of  $\mathbb{S}(\textbf{T}^A)$, we now define the natural candidate of  the exploration process of  the subordinate tree at time $\mathfrak{t}$  -- as we already mentioned,  this statement is  purely heuristic. Let us start by introducing some notations. Still under $\mathbb{N}^{\lambda}_{x,0}( \cdot |M \geq 1)$  denote the connected components of the open set 
\begin{align*}
    \big\{  s \geq A^{-1}_\mathfrak{t}  :~ {H}_s>\inf\limits_{[A^{-1}_\mathfrak{t},s]} {H} \big\} 
\end{align*}
by  $((\alpha_i , \beta_i):i \in \mathbb{N})$, and  as usual write   $(\rho^i,\overline{W}^{i}):=(\rho^i,W^i,\Lambda^i)$ for   the subtrajectory associated with the excursion interval $[\alpha_i , \beta_i]$. Further, set  $h_i := H_{\alpha_i}$ and consider the measure: 
\begin{equation}\label{delta:5.2:fin}
\sum \limits_{i\in \mathbb{N}}\delta_{(h_i, \rho^{i},\overline{W}^{i} ) }\, . 
\end{equation}
By the strong Markov property and \eqref{PoissonRandMeasure}, conditionally on $({\rho}^A_{\mathfrak{t}} , \overline{W}^A_{\mathfrak{t}})$, the measure \eqref{delta:5.2:fin} is a Poisson point measure with intensity $
    \rho_{\mathfrak{t}}^A(\dd h)\mathbb{N}_{\overline{W}_{\mathfrak{t}}^A(h)}(\dd\rho, \dd\overline{W}).
$ 
Next, for every $j\in \mathcal{J}$ we set:
\begin{equation} \label{defnition:Lis}
    L^{j}:=\sum \limits_{\tau_{r_j}(\overline{W}^A_{\mathfrak{t}})<h_i<\tau^{+}_{r_j}(\overline{W}^A_{\mathfrak{t}})}\mathscr{L}^{\,r_j}_\sigma(\rho^{i},\overline{W}^i),
\end{equation}
which is  the total amount of exit local time from the domain $D_{r_j}$ generated by the excursions glued on the right-spine of $\overline{W}^A_{\mathfrak{t}}$  at the interval $\big(\tau_{r_j}(\overline{W}^A_\mathfrak{t}), \tau_{r_j}^{+}(\overline{W}^A_\mathfrak{t})\big)$. Finally, we introduce the measure   $\rho^{*}_{\mathfrak{t}}:=\sum\limits_{j\in \mathcal{J}} L^j\cdot  \delta_{r_j}$. 
\begin{lem} \label{lemma:explorationAtexpotime} We have the following identity in distribution:
\begin{equation*}
    \big((\widetilde{H}_{\mathfrak{t}},\widetilde{\rho}_{\mathfrak{t}}):\widetilde{N}^{\lambda}(\cdot|\widetilde{M}\geq 1)\big) \overset{(d)}{=} \big((H^{A}_{\mathfrak{t}},\rho^{*}_{\mathfrak{t}}):~\mathbb{N}^{\lambda}_{x,0}(\cdot |M\geq 1)\big).\vspace{1mm}
\end{equation*}
\end{lem}
\noindent In particular, Lemma \ref{lemma:explorationAtexpotime} implies that   $H(\rho^*_{\mathfrak{t}})=H^{A}_{\mathfrak{t}}$.
\begin{proof}
We start noticing that, by Lemmas \ref{lemma:Seccion5_PoissoncondHtilde} and \ref{lemma:Seccion5_PoissoncondHA}, we already have:
\begin{equation*}
    \big(\widetilde{H}_{\mathfrak{t}}:\widetilde{N}^{\lambda}(\cdot|\widetilde{M}\geq 1)\big) \overset{(d)}{=} \big(H^{A}_{\mathfrak{t}}:~\mathbb{N}^{\lambda}_{x,0}(\cdot |M\geq 1)\big).
\end{equation*}
Consequently, again by Lemma \ref{lemma:Seccion5_PoissoncondHtilde}  the desired result will follow by  showing that, under $\mathbb{N}^{\lambda}_{x,0}(\cdot |M\geq 1)$ and conditionally on $H^A_\mathfrak{t}$, the measure 
$$\sum \limits_{j\in \mathcal{J}}\delta_{(r_j,L^j)}$$
is a Poisson point measure with intensity $\mathbbm{1}_{[0,H^A_\mathfrak{t}]}(\dd h) \widetilde{\nu}(\dd z)$, where the measure $\widetilde{\nu}$ is characterised  by \eqref{equation:leynutilde}. In this direction, we work in the rest of the proof under $\mathbb{N}^{\lambda}_{x,0}(\cdot | M\geq 1)$ and  recall that, conditionally on $(\rho^A_{{\mathfrak{t}}} , \overline{W}^A_{\mathfrak{t}})$, the measure \eqref{delta:5.2:fin} is a Poisson point measure with intensity
$
\rho^A_{{\mathfrak{t}}}(\dd h)\mathbb{N}_{\overline{W}^A_{{\mathfrak{t}}}(h)}(\dd\rho, \dd\overline{W}).
$
In particular,  \eqref{defnition:Lis} entails that conditionally on $(\rho^A_{{\mathfrak{t}}} , \overline{W}^A_{\mathfrak{t}})$, the random variables  $(L^{j}: j \in \mathcal{J})$ are independent. Moreover, 
since by definition  $u_p(y)=\mathbb{N}_{y,0}(1-\exp(-p \mathscr{L}^{\,0}_\sigma))$, the translation invariance of the local time $\mathcal{L}$ gives
\begin{equation*}
    \mathbb{N}^{\lambda}_{x,0}\big(\exp(-pL^j)\:|\:\rho^A_{{\mathfrak{t}}} , \overline{W}^A_{\mathfrak{t}})
    =\exp\Big(-\int_{\tau_{r_j}(\overline{W}^A_{\mathfrak{t}})}^{\tau_{r_j}^+(\overline{W}^A_{\mathfrak{t}})}\rho^A_{\mathfrak{t}}(\dd h)u_{p}\big(W^A_{\mathfrak{t}}(h)\big)\Big)
    =\exp\Big(-\int  \rho^{A,j}_{\mathfrak{t}}(\dd h)u_{p}\big(W^{A,j}_{\mathfrak{t}}(h)\big)\Big), 
\end{equation*}
 for every $j\in \mathcal{J}$.  It will be then convenient to  introduce, for $(\mu , \w)\in \mathcal{M}_f(\mathbb{R}_+)\times \mathcal{W}_E$, the probability measure $\rm{m}_{\mu,\w}$ in $\mathbb{R}_+$ defined through its Laplace transform: 
$$\int {\rm{m}}_{\mu,\w}(\dd z)\exp(-p z)
=\exp\Big(-\int \mu(\dd h) u_{p}\big(\w(h)\big)\Big), $$
if $H(\mu)=\zeta(\w)$ and ${\rm{m}}_{\mu,\w}=0$ otherwise. The map $(\mu,\w)\mapsto \rm{m}_{\mu,\w}$ takes values in $\mathcal{M}_f(\mathbb{R}_+)$ and it is straightforward to see that it is measurable. 
Next, remark that by our previous discussion we have: 
\begin{align*}
\mathbb{N}^{  \lambda}_{x,0}\Big( G(H^{A}_{\mathfrak{t}}) \exp(-\sum_{j \in \mathcal{I}} f(r_j,L^{j}))\:|\:M\geq 1\Big)
&=\mathbb{N}^{\lambda}_{x,0}\Big( G(H^{A}_{\mathfrak{t}}) \prod_{j \in \mathcal{I}}  \int {\rm{m}}_{\rho^{A,j}_{\mathfrak{t}},W^{A,j}_{\mathfrak{t}}}(\dd z)\exp(- f(r_j,z)) \:\Big|\:M\geq 1\Big)\\
&=\mathbb{N}^{\lambda}_{x,0}\Big( G(H^A_\mathfrak{t}) \exp(-\sum_{j \in \mathcal{J}} f^*(r_j,\rho^{A,j}_\mathfrak{t}, W_{\mathfrak{t}}^{A,j}))\:\Big|\:M\geq 1\Big), 
\end{align*}
where $f^{*}(r,\mu,\w):=-\log\big(\int {\rm{m}}_{\mu, \w}(\dd z) \exp(-f(r,z))\big).$
Now, we can apply  Lemma \ref{lemma:Seccion5_PoissoncondHA} to get:
\begin{align*}
&\mathbb{N}^{  \lambda}_{x,0}\Big( G(H^{A}_{\mathfrak{t}}) \exp(-\sum_{j \in \mathcal{I}} f(r_j,L^{j}))\:\Big|\:M\geq 1\Big) \\
&\hspace{20mm} =  \mathbb{N}^{\lambda}_{x,0}\Bigg( G(H^{A}_{\mathfrak{t}}) \exp\Big(-\int_0^{H^A_{\mathfrak{t}}} \dd r \, ( E^{0}\otimes\mathcal{N})_* \Big[ \int \mathrm{m}_{ J, \xi}(\dd z) \big(1-\exp-f(r, z) \big) \Big] \Big) \Bigg), 
\end{align*}
and it follows that conditionally on $H^A_{\mathfrak{t}}$ the measure $\sum \limits \delta_{(r_j,L^j)}$ is a Poisson point measure with intensity:
\begin{equation*}
      \mathbbm{1}_{[0,H^A_{\mathfrak{t}}]}(r) \dd r \,  ( E^{0}\otimes\mathcal{N})_*\big[  \mathrm{m}_{ J, \xi} (\dd z) ].  
\end{equation*}
To conclude, we need to show that the measure $( E^{0}\otimes\mathcal{N})_*\big[  \mathrm{m}_{ J, \xi} (\dd z) ]$ is precisely $\widetilde{\nu}( \dd z)$. 
In this direction, remark that:
\begin{align*}
& (  E^{0}\otimes \mathcal{N})_*\big[ \int \mathrm{m}_{ J , \xi}(\dd z) (1-\exp(-p z)) \big] 
=  ( E^{0}\otimes\mathcal{N})_* \big[ 1-\exp(-\int  J (\dd h) u_{p}(\xi(h) )\big]  \\
&= E^{0}\otimes\mathcal{N}\Big(1-\exp\big(-\int J_\sigma (\dd h)~ u_{p}(\xi(h)) - \int \Jc_\sigma (\dd h) u_{G_\lambda}(\xi(h)) -\alpha \sigma \big)\Big)\\
&\hspace{0.5cm} - E^{0}\otimes\mathcal{N}\Big(1-\exp\big( -\int \Jc_\sigma (\dd h) u_{G_\lambda}(\xi(h))-\alpha \sigma \big)\Big).
\end{align*}
Then, \eqref{identity:exponenteSubord} entails that the previous display is equal to
\begin{align*}
\mathcal{N}\Big(1-\exp\big(-\int_0^\sigma \dd h\frac{\psi(u_p(\xi(h)))-\psi(u_{G_\lambda}(\xi(h)))}{u_p(\xi(h))-u_{G_\lambda}(\xi(h))})\Big)-\mathcal{N}\Big(1-\exp \big( -\int_0^\sigma \dd h\frac{\psi(u_{G_\lambda}(\xi(h)))}{u_{G_\lambda}(\xi(h))}\big)\Big).
\end{align*}
  However, by Lemma \ref{lemma:cuentas} the previous display is precisely \eqref{equation:leynutilde}. 
\end{proof}

We can now identify $\mathbb{S}(\textbf{T}^A)$  in terms of our functionals.  In this direction, for every $i \in\mathbb{N}$, we introduce $(\rho^{i,k},W^{i,k},\Lambda^{i,k})_{k\in\mathcal{K}_{i}}$  the excursions from $D_0=\overline{E}~\setminus\{(x,0)\}$ of  $(\rho^i,W^{i},\Lambda^i-\Lambda^i_0)$. 
In particular,  the family $(\rho^{i,k}, \overline{W}^{i,k})_{k\in \mathcal{K}_i}$ is in one-to-one correspondence with the
connected components $[a_{i,k},b_{i,k}]$, $k\in \mathcal{K}_i$, of the open set $\{s\in[0,\sigma(\overline{W}^i)]:\tau_{\Lambda^i_0}(\overline{W}^{i}_s)<\zeta_s(\overline{W}_s^{i})\}$, in such a way that $(\rho^{i,k}, W^{i,k},\Lambda^{i,k}+\Lambda_0^i)$ is the subtrajectory of $(\rho^{i},\overline{W}^{i})$ associated with the interval  $[a_{i,k},b_{i,k}]$. 
In the time scale of $((\rho_{s},\overline{W}_{s}):s\geq 0)$, the excursion $(\rho^{i,k},W^{i,k},\Lambda^{i,k} + \Lambda^i_0 )$   corresponds to the subtrajectory associated with $[\alpha_{i,k},\beta_{i,k}]$,
where $\alpha_{i,k}:=\alpha_i+a_{i,k}$ and $\beta_{i,k}:=\alpha_i + b_{i,k}$. Next, for each $k \in \mathcal{K}_i$, we introduce the point process  $\mathcal{P}^{i,k}_{t} :=\mathcal{P}_{(A_{\alpha_{i,k}}+t)\wedge A_{\beta_{i,k}}}-\mathcal{P}_{A_{\alpha_{i,k}}}$ and we set: 
\begin{equation*}
    \mathcal{M}:=\sum \limits_{i\in \mathbb{N}}\sum\limits_{k\in \mathcal{K}_{i}}\delta_{(\Lambda_{0}^i(0), \rho^{i,k},  \overline{W}^{i,k}, \mathcal{P}^{i,k})}.
\end{equation*}
An application of the Markov property at time $A^{-1}_{\mathfrak{t}}$ and the special Markov property applied to the domain $D_0$ shows that, conditionally on $\rho^*_{\mathfrak{t}}$, the measure $\mathcal{M}$ is a Poisson point measure with intensity $\rho^*_{\mathfrak{t}}(\dd r) \mathbb{N}_{x,0}^\lambda(\dd \rho, \dd \overline{W}, \dd \mathcal{P})$.
For every $j\in \mathcal{J}$, consider  
\begin{equation*}
    M_j:=\#\Big\{ \big(\Lambda_{0}^i(0), \rho^{i,k},  \overline{W}^{i,k}, \mathcal{P}^{i,k}\big)\in \mathcal{M}:~\Lambda_0^{i} (0) =r_j\text{ and }\mathcal{P}^{i,k}_{A_\sigma(\overline{W}^{i,k})} \geq 1\Big\}, 
\end{equation*}
and denote  the elements of $\{ (r_j , M_j),  j\in \mathcal{J}:~M_{j} \geq 1\}$ arranged in increasing order with respect to  $r_j$ by 
$\big( (r^\circ_1, M^\circ_1), \dots , (r^\circ_{R},M^\circ_{R}) \big). $
We now remark that by construction we have:
\begin{equation} \label{definition:spineSubordinado}
     \mathbb{S}(\textbf{T}^A) =\big((r^\circ_1, M^\circ_1), \dots , (r^\circ_{R},M^\circ_{R}), (H^A_{\mathfrak{t}}, -1) \big),
\end{equation}
and, in particular,  $K = \sum_{p=1}^R M^\circ_{p}$ which is   the number of atoms $(\Lambda_{0}^i(0),\rho^{i,k},\overline{W}^{i,k},\mathcal{P}^{i,k}) \in \mathcal{M}$ with at least one Poissonian mark. Finally, we write $\mathscr{E} := ((\rho^{q}_\circ, \overline{W}^q_\circ, \mathcal{P}^q_\circ):~ 1 \leq q \leq K)$ for the collection of these marked  excursions enumerated  in counterclockwise order.  Remark that,  for every $1 \leq q \leq K$,  ${\textbf{T}}^A_i$ is the embedded tree associated with $\widehat{\Lambda}^q_\circ$ -- time changed by $A(\rho^q_\circ , \overline{W}_\circ^q)$ -- and marked by $\mathcal{P}^q_\circ$. We are now in position to prove Proposition  \ref{lemma:seccion5_lemafinal}. \par

\begin{proof}[Proof of Proposition \ref{lemma:seccion5_lemafinal}]
For every $h \geq 0$ with $\widetilde{\rho}_{\mathfrak{t}}(\{ h\}) > 0$, we write  $\widetilde{\mathcal{M}}^{(h)} := \widetilde{\mathcal{M}}\mathbbm{1}_{\{\widetilde{h}^i = h\}}$. Similarly,  for every $r \geq 0$ satisfying ${\rho}^*_{\mathfrak{t}}(\{ r \}) > 0$, we set $\mathcal{M}^{(r)}:=\mathcal{M}\mathbbm{1}_{\{ \Lambda_0^{i} (0) = r \}}$. Next, we introduce the following families respectively under $\widetilde{N}^\lambda(\cdot |\widetilde{M} \geq 1)$ and $\mathbb{N}^{\lambda}_{x,0}(\cdot | M \geq 1)$ :
\begin{equation}\label{def:family_Ntilde}
    \big\{  \big(h \mathbbm{1}_{\{ \widetilde{\mathcal{M}}^{(h)}(\widetilde{M} \geq 1) \geq 1  \}}, \, \widetilde{\mathcal{M}}^{(h)}(\widetilde{M} \geq 1)  \big) : h \geq 0, \, \widetilde{\rho}_{\mathfrak{t}}(\{ h \})> 0  \big\}
    {\cup \{ (\widetilde{H}_{\mathfrak{t}} ,-1 ) \} },
\end{equation}
and
\begin{equation} \label{def:family_Nx}
    \big\{  \big(r \mathbbm{1}_{\{ \mathcal{M}^{(r)}(M \geq 1) \geq 1  \}}, \, \mathcal{M}^{(r)}(M \geq 1)  \big) : r \geq 0, \, \rho^*_{\mathfrak{t}}(\{ r \})> 0  \big\} {\cup \{ ({H}^A_{\mathfrak{t}} ,-1 ) \} }, 
\end{equation}
where by Lemma \ref{lemma:explorationAtexpotime}, we have respectively that  $H(\rho^A_{\mathfrak{t}}) = H^A_\mathfrak{t}$,  $H(\widetilde{\rho}_{\mathfrak{t}}) = \widetilde{H}_\mathfrak{t}$. Recall that,  under $\widetilde{N}^\lambda(\cdot |\widetilde{M} \geq 1,\widetilde{\rho}_{\mathfrak{t}})$, the measure $\widetilde{\mathcal{M}}$ is a Poisson point measure with intensity $\widetilde{\rho}_{\mathfrak{t}}(\dd h)\widetilde{N}^\lambda(\dd \rho, \,  \dd \mathcal{P})$ and similarly,  under $\mathbb{N}^{\lambda}_{x,0}(\cdot | M \geq 1, \rho_{\mathfrak{t}}^*)$,  the measure $\mathcal{M}$ is a Poisson point measure with intensity $\rho^*_{\mathfrak{t}}(\dd r) \mathbb{N}_{x,0}^\lambda(\dd \rho, \dd \overline{W}, \dd \mathcal{P})$. Consequently, by restriction properties of Poisson measures, under $\widetilde{N}^\lambda(\cdot |\widetilde{M} \geq 1,\widetilde{\rho}_{\mathfrak{t}})$, the  variables 
$(\widetilde{\mathcal{M}}^{(h)}(\widetilde{M} \geq 1): \,  \widetilde{\rho}_{\mathfrak{t}}(\{h\}) > 0\big)$ are independent Poisson random variables  with intensity $\widetilde{\rho}_{\mathfrak{t}}(\{  h \})\widetilde{N}^{\lambda}(M \geq 1)$ and, under  $\mathbb{N}^{\lambda}_{x,0}(\cdot | M \geq 1, \rho_{\mathfrak{t}}^*)$, the variables
$(\mathcal{M}^{(r)}(M \geq 1): \,  \rho^*_{\mathfrak{t}}(\{r\}) > 0\big)$ are also  independent Poisson random variables, this time with intensity  $\rho^*_{\mathfrak{t}}(\{  r \}) \mathbb{N}_{x,0}^{\lambda}(M \geq 1)$. Now, recall from  Lemma \ref{lemma:explorationAtexpotime}  the identity 
\begin{equation*}
    (\widetilde{\rho}_{\mathfrak{t}}:\widetilde{N}^{\lambda}(\cdot|\widetilde{M}\geq 1)) \overset{(d)}{=} (\rho^{*}_{\mathfrak{t}}:~\mathbb{N}^{\lambda}_{x,0}(\cdot |M\geq 1)).
\end{equation*}
Since  $\widetilde{N}^\lambda(\widetilde{M} \geq 1) = \mathbb{N}_{x,0}^\lambda(M \geq 1)$, this ensures  that the families   \eqref{def:family_Ntilde} and \eqref{def:family_Nx}  have the same distribution. Moreover, the measures $\widetilde{\rho}_{\mathfrak{t}}$ and $\rho_{\mathfrak{t}}^*$ being atomic,  the families \eqref{definition:spinePhiTilde},  \eqref{definition:spineSubordinado}   correspond respectively to the subset of elements of   \eqref{def:family_Ntilde} and \eqref{def:family_Nx}    with non-null entries.  This gives the first statement of the proposition. 
\par
  To establish  (ii), it suffices to show that conditionally on $\mathbb{S}(\widetilde{\textbf{T}})$,  the  marked excursions $\widetilde{\mathscr{E}}$ are distributed as $\widetilde{K}$ independent copies with law   $\widetilde{N}^\lambda( \dd H, \dd \mathcal{P} |\widetilde{M} \geq 1)$ and that,  conditionally on $\mathbb{S}(\mathbf{T}^A)$, the excursions $\mathscr{E}$ are distributed as $K$ independent copies with law $\mathbb{N}_{x,0}^\lambda(\dd \overline{W}, \dd\mathcal{P}|M \geq 1)$. Remark that our previous reasoning already implies that $\widetilde{\mathscr{E}}$ and $\mathscr{E}$ satisfy  the desired property if we do not take into account the ordering. However, this is not enough  and to keep track of the ordering  we proceed as follows: 
\par 
We start studying $\widetilde{\mathscr{E}}$ under $\widetilde{{N}}^\lambda(\cdot | \widetilde{M} \geq 1)$ and we introduce $(\widetilde{\mathscr{I}}_s:s \geq \mathfrak{t})$,  the running infimum of $(\langle \widetilde{\rho}_{s} , 1 \rangle - \langle \widetilde{\rho}_{\mathfrak{t}} , 1 \rangle: s \geq \mathfrak{t})$. Next, we consider the measure
\begin{equation}\label{definition:medidaordenadatilde}
    \sum_{i \in \mathbb{N}}\delta_{(-\widetilde{\mathscr{I}}_{\widetilde{\alpha}_i},\widetilde{\rho}^{i}, \widetilde{\mathcal{P}}^i) } ,
\end{equation}
and we stress that, by the strong Markov property and the discussion below \eqref{PoissonRandMeasure_Inf},  conditionally on $\mathcal{F}_{\mathfrak{t}}$ this measure  is a Poisson point measure with intensity $\mathbbm{1}_{[0, \langle \widetilde{\rho}_{\mathfrak{t}},1  \rangle ]}(u) \dd u~ \widetilde{N}^\lambda( \dd \rho,\dd \mathcal{P})$. Moreover,  its image  by   the transformation  $s  \mapsto H(\kappa_{s} \widetilde{\rho}_{\mathfrak{t}} )$ on its first coordinate  gives precisely   $\widetilde{\mathcal{M}}$. In particular,  the collection  $\big((\widetilde{h}^\circ_1, \widetilde{M}^\circ_1), \dots , (\widetilde{h}^\circ_{\widetilde{R}},\widetilde{M}^\circ_{\widetilde{R}}), (\widetilde{H}_{\mathfrak{t}} ,-1 )\big)$ only depends on $\widetilde{\rho}_{\mathfrak{t}}$ and  $\big(\widetilde{\mathscr{I}}_{\widetilde{\alpha}_i}:~i\geq 0 \text{ with } \widetilde{\mathcal{P}}^{i}_{\sigma(\widetilde{\rho}^{\,i})}\geq 1\big)$. Remark that  the ordered marked excursions $\widetilde{\mathscr{E}}$ correspond precisely to the atoms $H(\widetilde{\rho}^{\,i})$ of \eqref{definition:medidaordenadatilde} with $\widetilde{\mathcal{P}}^{i}_{\sigma(\widetilde{\rho}^{\,i})}\geq 1$,  when considered in decreasing order with respect to  $-\widetilde{\mathscr{I}}_{\widetilde{\alpha}_i}$. Since $H(\widetilde{\rho}_{\mathfrak{t}}) =  \widetilde{H}_\mathfrak{t}$,  we deduce  by restriction properties of Poisson measures that,  conditionally 
 on $(\widetilde{\rho}_\mathfrak{t},\widetilde{K})$, the collection $\widetilde{\mathscr{E}}$ is independent of $\mathbb{S}(\widetilde{\textbf{T}})$ and  formed by  $\widetilde{K}$    i.i.d. variables  with distribution $\widetilde{N}^\lambda(\dd \rho,\dd\mathcal{P} | \widetilde{M} \geq 1)$, as wanted.
\par

Let us now turn our attention to the distribution of $\mathscr{E}$ under $\mathbb{N}_{x,0}^\lambda(\cdot |M\geq 1)$. Similarly, under $\mathbb{N}_{x,0}^\lambda(\cdot |M\geq 1)$ we consider  $(\mathscr{I}_s:~s \geq A_{\mathfrak{t}}^{-1})$, the running infimum of $(\langle \rho_{s} , 1 \rangle- \langle \rho_{A_\mathfrak{t}^{-1}} , 1 \rangle : s \geq A_\mathfrak{t}^{-1} )$ as well as  the measure
\begin{equation}\label{definition:medidaordenadatilde2}
    \sum_{i \in \mathbb{N}}\delta_{(-\mathscr{I}_{\alpha_i},\rho^{i}, W^i) }. 
\end{equation}
Once again, by the strong Markov property and  \eqref{PoissonRandMeasure_Inf}, conditionally on $\mathcal{F}_{A_\mathfrak{t}^{-1}}$, the measure  \eqref{definition:medidaordenadatilde2} is a Poisson point measure with intensity $\mathbbm{1}_{[0, \langle \rho_{\mathfrak{t}}^A,1  \rangle ]}(u) \dd u~ \mathbb{N}^\lambda_{\overline{W}^A_{\mathfrak{t}}(H( \kappa_{u}\rho^A_{\mathfrak{t}} ))}( \dd \rho,\dd \overline{W})$. We now introduce the process:
$$V_t:=\sum\limits_{i\in \mathbb{N}} \mathscr{L}^{\,\Lambda^i_0}_{t\wedge \beta_i-t\wedge \alpha_i}(\rho^{i}, \overline{W}^i),\quad t\geq 0,$$
where  $V_\infty=\langle \rho_{\mathfrak{t}}^*,1\rangle<\infty$ by Lemma \ref{lemma:explorationAtexpotime}. Recall that $(\rho^{i,k}, \overline{W}^{i,k})_{k\in \mathcal{K}_i}$ stands  for the excursions of $(\rho^{i}, W^{i},\Lambda^{i}-\Lambda^i_0)$ outside $D_0$ and  we stress  that in the time scale of $((\rho_{s},\overline{W}_{s}):s\geq 0)$, the  excursion $(\rho^{i,k}, W^{i,k},\Lambda^{i,k}+\Lambda_0^i)$ corresponds to the subtrajectory associated with $[\alpha_{i,k},\beta_{i,k}]$,
where $\alpha_{i,k}:=\alpha_i+a_{i,k}$ and $\beta_{i,k}:=\alpha_i + b_{i,k}$. To simplify notation set $\text{Tr}(\rho^i,\overline{W}^{i})$ for the truncation of $(\rho^i,\overline{W}^{i})$ to the domain $D_{\Lambda^i_0}$. An application of the strong Markov property combined  with the special Markov property  in the form given in Theorem \ref{Theo_Spa_Markov_Excur}  implies that, conditionally on $\sum \limits_i \delta_{(-\mathscr{I}_{\alpha_i},\text{Tr}(\rho^i, \overline{W}^i))}$, 
the measure:
\begin{equation}
    \label{definition:medidaordenadatilde2_2}
\sum\limits_{i\in \mathbb{N},k\in \mathcal{K}_i} \delta_{(V_{\alpha_{i,k}},\rho^{i,k},\overline{W}^{i,k},\mathcal{P}^{i,k})}
\end{equation}
is a Poisson point measure with intensity $\mathbbm{1}_{[0,\langle \rho_{\mathfrak{t}}^*,1\rangle]}(p)\dd p~\mathbb{N}_{x,0}^{\lambda}(\dd\rho,\dd\overline{W},\dd \mathcal{P})$.  The conclusion is now  similar to the previous discussion on  $\widetilde{\mathscr{E}}$. First,  remark that the collection $((r^\circ_1, M^\circ_1), \dots , (r^\circ_{R},M^\circ_{R}), ( H^A_\mathfrak{t}, -1) )$ can be recovered from 
\begin{equation*}
    \sum \limits_{i \in \mathbb{N}} \delta_{(-\mathscr{I}_{\alpha_i},\text{Tr}(\rho^i, \overline{W}^i))} \quad \text{ and } \quad 
    \Big(V_{\alpha_{i,k}} : i \in \mathbb{N}, k \in \mathcal{K}_i \text{ with }\mathcal{P}^{i,k}_{A_\sigma(\overline{W}^{i,k})}\geq 1\Big)
\end{equation*}
by making use of  the mapping $r \mapsto \sum_{(-\mathscr{I}_{\alpha_i}) \leq  r}\mathscr{L}_\sigma^{\Lambda^i_0}(\rho^i , \overline{W}^i)$ and  the fact  that $\Lambda_0^i(0)$ can be read from   $\text{Tr}(\rho^i,\overline{W}^i)$. In our last claim we used that $\mathscr{L}^{\, \Lambda_0^i}_\sigma(\rho^i, \overline{W}^i)$ is measurable with respect to $\text{Tr}(\rho^i, \overline{W}^i)$ -- by Proposition \ref{L_eta_measurable} -- as well as the equality $H^A_{\mathfrak{t}} = \sup_{i\in \mathbb{N}} \Lambda^i_0(0)$ -- which holds since $\mathcal{M}$ conditionally on $\rho^*_\mathfrak{t}$ is Poisson $\rho^*_\mathfrak{t}(\dd r) \mathbb{N}^\lambda_{x,0}$ and   $H(\rho^*_{\mathfrak{t}})=H^A_\mathfrak{t}$ by Lemma \ref{lemma:explorationAtexpotime}.   Furthermore, the ordered marked excursions $\mathscr{E}$ correspond precisely to the atoms of \eqref{definition:medidaordenadatilde2_2} with $\mathcal{P}^{i,k}_{A_\sigma(\overline{W}^{i,k})}\geq 1$ in decreasing order with respect to  the process $V$ -- since $V$ is non-decreasing and all the values $\{V_{\alpha_{i,k}}:i\in \mathbb{N}, k\in\mathcal{K}_i\}$ are distinct. Putting everything together, we deduce by restriction properties of Poisson measures that, conditionally on $ \sum \limits_i \delta_{(-\mathscr{I}_{\alpha_i},\text{Tr}(\rho^i, \overline{W}^i))}$ and  $K$,
the collection $\mathscr{E}$ is independent of $\mathbb{S}(\textbf{T}^A)$ and composed by $K$   i.i.d. variables with distribution $\mathbb{N}^\lambda_{x,0}(\dd \rho,\dd \overline{W},\dd\mathcal{P} | M \geq 1)$. This completes the proof of Proposition \ref{lemma:seccion5_lemafinal}. 
\end{proof}

\end{document}